\documentclass[12pt]{amsart}
\usepackage{epsfig,color, amsmath}
\usepackage[normalem]{ulem} 
\usepackage{esint}
\usepackage{amssymb}
\usepackage{hyperref, mathrsfs}
\usepackage[dvipsnames]{xcolor}
\usepackage{bm}
\usepackage{enumitem}
\usepackage{mathtools, comment}
\usepackage{scalerel}
\usepackage{accents}

\usepackage{tikz}

\hypersetup{colorlinks}
\hypersetup{citecolor=blue}
\hypersetup{urlcolor=blue}

\makeatother

\theoremstyle{definition}

\theoremstyle{definition}

\def\fnum{equation} 
\newtheorem{Thm}[\fnum]{Theorem}
\newtheorem{Cor}[\fnum]{Corollary}

\newtheorem{Lem}[\fnum]{Lemma}

\newtheorem{Def}[\fnum]{Definition}
\newtheorem{Exa}[\fnum]{Example}
\newtheorem{Rem}[\fnum]{Remark}
\newtheorem{Pro}[\fnum]{Proposition}

\newtheorem{Cla}[\fnum]{Claim}

\numberwithin{equation}{section}
\renewcommand{\rm}{\normalshape}

\newcommand{\Vol}{{\text{Vol}}}

\newcommand{\eps}{{\varepsilon}}

 \newcommand{\N}{\ensuremath{\mathbb{N}}}
 
  \newcommand{\R}{\ensuremath{\mathbb{R}}}
  
 \newcommand{\Z}{\ensuremath{\mathbb{Z}}}

 \newcommand{\ba}{\begin{align*}}
 \newcommand{\ea}{\end{align*}}

\newcommand{\dd}{{\mathsf {d}}}

\newcommand{\veps}{\varepsilon}

\DeclareMathOperator{\RCD}{RCD}
\DeclareMathOperator{\Ric}{Ric}

\DeclareMathOperator{\MCP}{MCP}

\DeclareMathOperator{\Ext}{Ext}
\DeclareMathOperator{\Int}{Int}

\newcommand{\mm}{\mathfrak{m}}

\title{Topology of Two-Dimensional Collapsed Spaces with Lower Ricci
Bounds}
\author{Elia Bruè \and Qin Deng}
\address{\parbox{\linewidth}{Bocconi University, Department of Decision Sciences.\\
	Via Roentgen 1,
	20136 Milano -- Italy\\[-4pt]\phantom{a}}}
\email{elia.brue@unibocconi.it}
\address{\parbox{\linewidth}{Bocconi University, Department of Decision Sciences.\\
	Via Roentgen 1,
	20136 Milano -- Italy\\[-4pt]\phantom{a}}}
\email{qin.deng@unibocconi.it}
\date{}
\newcommand{\mrestr}{\mathbin{\vrule height 1.6ex depth 0pt width
0.13ex\vrule height 0.13ex depth 0pt width 1.3ex}}

\begin{document}

\begin{abstract}
    We prove that collapsed metric measure spaces with Ricci curvature bounded below and essential dimension two are topological surfaces, possibly with boundary.
\end{abstract}

\maketitle

\tableofcontents

\section{Introduction}
The study of collapsing and degeneration of Riemannian manifolds under various curvature constraints has been a central theme in geometric analysis for several decades \cite{Gro78,CG86,Y88,F87b,F88,F89,CG90,Y91,CFG92,CC96,CC97,CT06}. A major objective of the subject is to understand both the metric and topological properties of the resulting limit spaces, as well as the geometric structure underlying the collapse itself.

\medskip
In this work, we focus on {\it collapsing under uniform lower Ricci curvature bounds}, namely on pointed Gromov--Hausdorff limits of $N$-dimensional Riemannian manifolds
\begin{equation}
    (M_i^N,g_i,p_i) \xrightarrow{GH} (X,\dd_X,\mm_X,x)
\end{equation}
satisfying the uniform Ricci lower bound and the volume collapsing condition
\begin{equation}\label{eq:bounds}
    \Ric_{g_i}\ge -(N-1),
    \quad
    \Vol_{g_i}(B_1(p_i))\to 0.
\end{equation}
By Gromov's compactness theorem, the class of $N$-dimensional manifolds with uniform lower Ricci bounds is precompact in the Gromov--Hausdorff topology. Thus, after passing to a subsequence, one always obtains a limit metric measure space. Their systematic study began in the 1990s with the foundational work of Cheeger and Colding \cite{CC96,CC97,CC00,C97}, building on earlier contributions of Fukaya \cite{F87a} and Anderson \cite{A90}.

\medskip
Since \cite{CC96, CC97, CN12}, it has been understood that collapsed Ricci limit spaces carry a well-defined notion of dimension $n<N$, called the \emph{essential dimension} (see Definition~\ref{Def: essential dim}). This coincides with the {\it rectifiable dimension} of the regular set. In the special case where the limit space is itself a smooth manifold--for instance, a flat cylinder collapsing to a line, or Berger metrics on $S^3$ collapsing to the round metric on $S^2$--the essential dimension agrees with the usual manifold dimension.
An important open question in the field \cite[Open Problem~3.4]{N20} asks whether collapsed Ricci limit spaces admit a topological manifold structure, at least in a neighborhood of the regular set. It is well known that this holds in essential dimension $n=1$, as a consequence of the metric and topological characterization \cite{H11,KL16}. On the other hand, recent examples \cite{HNW25,Z24} show that collapsed limits of essential dimension $n\ge 3$ need not admit any topological manifold structure, thus leaving the case $n=2$ open.

\medskip
The main result of this paper shows that the two-dimensional case is indeed more rigid, and that the manifold structure survives collapsing. In fact, our theorem holds in the broader class of $\RCD$ spaces.

\begin{Thm}[Manifold Structure]\label{thm:mainRCD}
    Let $(X,\dd_X,\mm_X)$ be an $\RCD(-(N-1),N)$ space of essential dimension $2$. Then $(X,\dd_X)$ is a topological $2$-manifold, possibly with boundary.
\end{Thm}

We refer readers less familiar with the $\RCD$ formalism to the preliminary Section~\ref{Sec: prelim} and to the survey articles \cite{A18,G23,S23}. Here we only recall that $\RCD$ spaces are the counterpart of Alexandrov spaces in the setting of lower Ricci curvature bounds. They are metric measure spaces endowed with a synthetic curvature-dimension condition \cite{S06a,S06b,LV09}, together with the infinitesimal Hilbertianity condition \cite{AGS14, G15}, which rules out genuinely Finsler structures. This class includes Riemannian manifolds with Ricci curvature bounded below and dimension bounded above, as well as their measured Gromov--Hausdorff limits. In particular, {\it Theorem~\ref{thm:mainRCD} immediately yields the same conclusion for Ricci limit spaces of essential dimension 2.}

\medskip
Another key advantage of working in the $\RCD$ setting is that this class is stable under several natural operations, such as taking cones, spherical suspensions, and quotients by isometric group actions. It is also stable under splitting: if an $\RCD$ space splits isometrically off a Euclidean factor $\mathbb{R}^m$, then the remaining factor is again an $\RCD$ space, now of essential dimension $n-m$. This immediately yields the following  corollary.

\begin{Cor}\label{Cor:n-2symmetries}
    Let
    $(M_i^N,g_i,p_i) \xrightarrow{GH} (X,\dd_X,\mm_X,x)$ satisfy
    $\Ric_{g_i}\ge -(N-1)$. Assume that $(X,\dd_X,\mm_X)$ has essential dimension $n$ and splits isometrically as $X=\mathbb{R}^{n-2}\times Y$. Then $(Y,\dd_Y)$, and hence also $(X,\dd_X)$, is a topological manifold, possibly with boundary.
\end{Cor}

This extension is particularly promising for applications, since Euclidean factors arise in many situations, typically as a consequence of splitting theorems \cite{CG71,G13} combined with additional metric or topological information. 

\medskip
Our main result, Theorem~\ref{thm:mainRCD}, is sharp in several respects. First, as already mentioned, it fails as soon as the essential dimension of the limit space is at least $3$ \cite{HNW25,Z24}. In fact, the failure is quite dramatic: these examples exhibit Ricci limit spaces with highly singular topology. Hupp--Naber--Wang \cite{HNW25} construct a sequence of six-dimensional manifolds with a uniform lower Ricci bound collapsing to a space of essential dimension $4$ with the property that every open subset of the limit has infinite second homology. In \cite{Z24}, the author obtains an analogous infinite-homology phenomenon for a sequence of five-dimensional manifolds collapsing to a limit of essential dimension $3$, by means of a completely different construction.

\medskip
Second, the appearance of boundary is unavoidable, and this phenomenon is already visible in essential dimension $n=1$. For instance, if one views the two-sphere as the metric suspension of $S^1$, one can consider a family of warped product metrics that shrinks the $S^1$ factor while keeping the curvature positive. The resulting limit space is a closed interval.
Moreover, even in essential dimension $2$, the metric structure of the boundary of collapsed Ricci limit spaces can be quite complicated. A key example to keep in mind is the construction of Pan--Wei \cite{PW22a}, which produces a Ricci limit space of essential dimension two whose Hausdorff dimension is strictly larger than $2$. Topologically, the space is a half-plane, and the singular set of large Hausdorff dimension coincides exactly with the boundary. 

\subsection{Nonnegative curvature}

We characterize the topology of surfaces admitting an $\RCD(K,N)$ structure of essential dimension $2$ in the cases of positive $K>0$ and nonnegative curvature $K=0$. In analogy with the theory of smooth surfaces with nonnegative Gauss curvature \cite{CV35}, the existence of such an $\RCD$ structure imposes topological restrictions. 

\begin{Thm}\label{Thm:top_characterization}
Let $(X,\dd_X,\mm_X)$ be an $\RCD(K,N)$ space of essential dimension $2$.

\begin{itemize}
    \item[(i)] If $K>0$, then $X$ is homeomorphic to either $S^2$, $\mathbb{RP}^2$, or the disk $D^2$.

    \item[(ii)] If $K=0$, $X$ is compact, and $X$ is not homeomorphic to one of the spaces in {\rm(i)}, then the universal cover $\widehat X$ is isometric to either $\mathbb{R}^2$ or the strip $\mathbb{R}\times [a,b]$ endowed with the product flat metric. In particular, $X$ is isometric to either a two-torus, a Klein bottle, a M\"obius strip, or a cylinder, each endowed with a flat metric.

    \item[(iii)] If $K=0$ and $X$ is noncompact, then the universal cover $\widehat X$ is either homeomorphic to $\mathbb{R}^2$ or to the half-plane $\mathbb{R}^2_+$, or isometric to the flat strip $\mathbb{R}\times [a,b]$. If $X$ is not simply connected, then it can only be isometric to an open flat cylinder or an open flat M\"obius strip, or homeomorphic to an open cylinder with one boundary component.
\end{itemize}
\end{Thm}
The combination of Theorem~\ref{Thm:top_characterization} and Corollary~\ref{Cor:n-2symmetries} imposes strong topological restrictions on collapsed Ricci limit spaces with at least $n-2$ Euclidean symmetries arising as limits of manifolds satisfying $\Ric_{g_i}\ge -\eps_i$, with $\eps_i\to 0$. This situation appears naturally, for instance, in the study of universal covers of manifolds with {\it almost nonnegative Ricci curvature} \cite{Y88}, especially in low dimensions or under additional assumptions on the Betti numbers.

\medskip
Finally, it is a natural question whether the manifold boundary, and indeed the topology itself, is stable under Gromov--Hausdorff convergence for $\RCD$ spaces of essential dimension $2$. This will be addressed in a forthcoming paper. 

\subsection{Structure of limits and related results}
The starting point of our analysis is the structure theory for Ricci limit spaces and $\RCD$ spaces, see the preliminary section~\ref{Sec: prelim}. If $(X,\dd_X,\mm_X)$ is either the limit of a sequence of $N$-dimensional manifolds with $\Ric\ge K$, or more generally an $\RCD(K,N)$ space, we denote by $\mathcal{R}_n(X)$ its regular set, namely the set of points all of whose tangent cones are isometric to $\mathbb{R}^n$ such that $\mm_X(X\setminus \mathcal{R}_n(X))=0$. In this case, $n$ is the essential dimension. It is well known that $\mathcal{R}_n(X)$ is $n$-rectifiable \cite{CC96,CC97,MN19,BS20}.

\medskip
In the noncollapsed setting, namely when $n=N$, a much stronger regularity theory is available \cite{CC96,DG18,H20,BGHZ23}. Volume convergence \cite{C97,DG18}, together with volume monotonicity, yields an $\eps$-regularity theorem: for every $x\in \mathcal{R}_n$, there exists a neighborhood $B_r(x)$ such that every ball $B_s(y)\subset B_{2r}(x)$ is $\eps$-GH close to a Euclidean ball. This multiscale control is the key input in Reifenberg-type arguments, which lead to the construction of a bi-H\"older parametrization by a Euclidean ball. In particular, one obtains an open neighborhood of $\mathcal{R}_n$ which is homeomorphic to an open manifold. Moreover, by the more recent work \cite{CJN21}, the complement of the manifold part of $X$ is $(n-2)$-rectifiable and has finite $(n-2)$-dimensional Hausdorff measure.

\medskip
By contrast, the global topology may still be very poorly behaved in higher dimensions, even in the noncollapsed setting. For instance, Menguy \cite{M00} constructed a four-dimensional noncollapsed Ricci limit space with locally infinite topological type. Even in the Ricci-flat setting, manifold structure may break down on a set of codimension four, as illustrated by the Eguchi--Hanson example \cite{EH79}. It is conjectured by Cheeger, Colding, and Tian that, in general, manifold structure persists away from a singular set of codimension at least four \cite[Remark 1.19]{CCT02}, \cite[Conjecture 0.7]{CC97}; see also \cite[Conjecture 1.8 and Remark 1.9]{KM21} for the conjecture in the $\RCD$ setting. This conjecture has been confirmed in the setting of bounded Ricci curvature \cite{CN15,JN21}.

\medskip
In low dimensions, however, the picture is much more rigid. Two-dimensional noncollapsed $\RCD(K,2)$ spaces are known to be Alexandrov spaces \cite{LS23}. Perelman's conical neighborhood theorem implies that Alexandrov spaces are manifolds with conical singularities \cite{P93}; see also \cite{P91}. In particular, two-dimensional Alexandrov spaces and noncollapsed $\RCD(K,2)$ spaces are surfaces, possibly with boundary.
In dimension three, noncollapsed Ricci limit spaces are known to be bi-H\"older homeomorphic to smooth Riemannian manifolds \cite{S12,ST21,ST22}; see also \cite{Z93}. On the synthetic side, noncollapsed $\RCD(K,3)$ spaces without boundary are orbifolds with locally finite $C(\mathbb{RP}^2)$ singularities \cite{BPS24,DP25}. A classification of the topology of complete open Riemannian manifolds with $\Ric \ge 0$ was achieved in \cite{L13} using minimal surface techniques; see also \cite{SY82}. This was recently extended to noncollapsed $\RCD(0,3)$ spaces without boundary, using completely different methods, in \cite{CMS26}.

\medskip
When boundary is involved, the situation becomes even subtler. In the noncollapsed setting, boundary points can be characterized in terms of tangent cones \cite{DG18,KM21}. A possible definition of boundary is as the closure of the set of points whose tangent cone is the Euclidean half-space. In \cite{BNS22}, the first author, together with Naber and Semola, proved that this set has dimension $n-1$ and is locally homeomorphic to a manifold with boundary near regular boundary points. Once again, the key ingredient is an $\eps$-regularity theorem at regular boundary points, which relies heavily on the noncollapsed structure. It provides the necessary multiscale control and ultimately leads to a Reifenberg-type argument.

\subsection{Strategy of proof and new ideas}\label{sec:itro_strategy}

In the collapsed setting $n<N$, no comparable $\eps$-regularity theorem is available (see \cite{NZ16} for a conditional result under bounded Ricci curvature). Even in essential dimension $n=2$, it is easy to construct examples where the $\eps$-GH closeness of a ball $B_1(x)$ to a Euclidean ball in $\mathbb{R}^2$ does not propagate to smaller scales.
Moreover, the boundary cannot be characterized purely in terms of tangent cones. As already mentioned, the key example illustrating this difficulty is the construction of Pan--Wei \cite{PW22a}, where the Ricci limit space is topologically a half-plane, yet its boundary has several striking metric features: it has Hausdorff dimension strictly larger than $2$, every tangent cone is isometric to the space itself, and in particular no tangent cone splits off a Euclidean factor.

\medskip
In view of these difficulties, it is clear that a purely metric approach is not sufficient to recover the manifold structure in the collapsed setting. For this reason, we adopt a new strategy, combining metric rigidity with topological arguments. It relies on several new ideas, and its main steps can be summarized as follows:
\smallskip
\begin{itemize}[itemsep=\smallskipamount]
    \item[(i)] \textit{Jordan-type separation theorems.} We prove that simple piecewise-geodesic loops---polygonal loops, in our terminology---contained in sufficiently small balls disconnect the space into two components. These separation-type theorems are the main technical tool of the paper, and a substantial part of the work is devoted to their proof. 
    \item[(ii)] \textit{Characterization of interior and boundary points.} We characterize interior points and boundary points---called \emph{extremal points} in our terminology, see Definition~\ref{Def: ext-int points}---in terms of whether sufficiently small balls centered at the point are disconnected by geodesics emanating from it.  
    Using (i), we prove that the interior set is open, dense, and contains the regular set, see Theorem~\ref{Thm: int set open}.
    \item[(iii)] \textit{Parametrization of the boundary.} We construct a local parametrization of the extremal set around each of its points by means of a simple continuous curve. The argument is of cut-locus type: the extremal set is recovered through a maximal extendibility property of suitable families of geodesics.
    \item[(iv)] \textit{Polygonal neighborhood bases and grid construction.} We construct a special neighborhood basis consisting of interior or boundary polygonal domains, depending on whether the point is interior or extremal. These domains should be viewed as precompact regions bounded by suitable polygonal loops, whose existence is provided by the Jordan-type theorems. Finally, for each polygonal domain we construct a homeomorphism with the Euclidean disk, or with the half-disk in the boundary case. This is achieved through the construction of a family of finer and finer grids, inspired by the classical two-sphere recognition theorem \cite{M16}.
\end{itemize}
% \medskip
% The metric rigidity along geodesic (i) is proven in Section~\ref{Sec: geod} (see Theorem~\ref{Thm: geodesic interior tangent}). It is based on the metric classification of spaces of essential dimension one (Theorem~\ref{Thm: RCD dim 1 class}) with the splitting theorem (Theorem~\ref{Thm: splitting thm}), we prove via a bootstrap argument that if a sufficiently small ball centered at an interior point $p$ of $\gamma$ is sufficiently GH-close to a Euclidean ball, then every smaller ball centered at $p$ is also GH-close to a Euclidean ball. Moreover, the approximation improves at smaller scales due to the splitting induced by the geodesic. It follows that if one tangent space at $p$ is $\R^2$, then all tangent spaces at $p$ must be $\R^2$. Since the tangent space of a tangent space is again a tangent space, this excludes certain possibilities, such as $\R \times S^1$. An analogous analysis for Euclidean half-balls rules out all remaining cases except $\R^2$ and $\R^2_+$. The continuity of tangent spaces in the interior of a geodesic \cite{CN12, D25} then yields the claim. 

\subsection{Jordan-type theorems}
\label{intro:jordan}

The starting point of our analysis is the metric rigidity of geodesics, established in Theorem~\ref{Thm: geodesic interior tangent}. It asserts that every geodesic segment is of exactly one of two types: either it is a {\it regular geodesic}, meaning that its interior is contained in the regular set, or it is a {\it boundary geodesic}, meaning that at each interior point the tangent cone is the Euclidean half-space $\R^2_+$.
The key observation is the following. If a sufficiently small ball centered at an interior point $p$ of a geodesic $\gamma$ is sufficiently Gromov--Hausdorff close to a Euclidean ball, then every smaller ball centered at $p$ is also Gromov--Hausdorff close to a Euclidean ball. Moreover, because of the splitting induced by $\gamma$, the approximation improves at smaller scales. As a consequence, if one tangent cone at $p$ is $\R^2$, then every tangent cone at $p$ must be $\R^2$. Since tangent cones of tangent cones are again tangent cones, this rules out intermediate possibilities such as $\R\times S^1$. A similar argument for Euclidean half-balls excludes all remaining cases except $\R^2$ and $\R^2_+$. Finally, the continuity of tangent cones along the interior of a geodesic \cite{CN12,D25} yields the conclusion.
\medskip

The next step is a {\it local disconnection theorem for regular geodesics}. We prove that a regular geodesic $\gamma$ disconnects sufficiently small balls centered at an interior point $p\in \gamma$ into exactly two path-connected components, see Theorem~\ref{Thm: good components existence}. This statement is partly inspired by \cite{C24}, although our conclusion is stronger and the proof is completely different.
The basic idea is the following. If a sufficiently small ball centered at $p$ is sufficiently close to $\mathbb{R}^2$, then the points outside a small tubular neighborhood of $\gamma$ can be canonically divided into two sides. The metric rigidity described above allows one to zoom in and perform the same construction in smaller balls centered at nearby points of $\gamma$. By proving that these local partitions are compatible across scales and centers, we obtain a global partition of all points in a small ball $B$ around $p$, away from $\gamma$, into two sides. One can then verify that these two sides are precisely the two path-connected components of $B\setminus \gamma$.
In Section~\ref{Sec: corner}, we extend this local disconnection result to balls centered at an intersection point of two regular geodesics, see Theorem~\ref{Thm: good components existence corner}.

\medskip

We then turn to {\it polygonal loops} (see Definition~\ref{Def: poly loop}), namely closed curves obtained by concatenating finitely many regular geodesics. Given a simple polygonal loop $\alpha$, and under a suitable $1$-contractibility assumption (see Remark~\ref{Rem:simply-connectedness}), we construct a path-connected neighborhood $U$ of $\alpha$ such that $U\setminus \alpha$ has exactly two connected components, see Proposition~\ref{Pro: poly loop neigh}. The neighborhood $U$ is built as a union of small balls centered along $\alpha$. By the local disconnection results, each such ball carries a canonical partition of the points away from $\alpha$ into two sides. The main point is then to show that these local sides can be matched consistently as one moves along $\alpha$. The only obstruction would be a global exchange of the two sides after one full turn, and this is ruled out by the $1$-contractibility assumption in Section~\ref{subsec:Mobius}. As a consequence, $U\setminus \alpha$ has exactly two path-connected components.
Finally, an application of Mayer--Vietoris, again using the $1$-contractibility assumption, upgrades this local statement to our first Jordan-type theorem, Theorem~\ref{Thm: Jordan theorem}, which shows that $\alpha$ disconnects the whole space $X$ into two path-connected components.

\subsection{Extremal set and cut-locus argument}
As already pointed out, we use the separation properties of geodesics to distinguish interior points from extremal points. In the end, the extremal points will form the boundary of the surface. This is carried out in Section~\ref{Sec: ext-int points}. The main result there is that the extremal set is closed, while its complement, the interior set, is dense and contains the regular set.

\medskip
In order to show that the extremal set corresponds to the manifold boundary, a first step is to show that any extremal point lies in the interior of a curve contained in the extremal set. We achieve this by a cut-locus argument in Section~\ref{Sec: ext set top}. To illustrate our ideas, consider the following example: Let $\overline D$ be a closed $2$-disk equipped with a smooth metric $g$. Let $p \in D$ and consider the cut locus $C$ of $p$. Then $\partial \overline D \subset C$. For any $x \in C \setminus \partial \overline D$, it follows from topological considerations that there exists a point $y \in C$ and two distinct geodesic $\gamma_1, \gamma_2$ from $p$ to $y$ such that that $x$ is contained in the connected component of $\overline D \setminus (\gamma_1 \cup \gamma_2)$ that does not intersect $\partial \overline D$. Thus by sweeping out the cut locus of $p$ and removing any point $x$ meeting the criterion above, we can identify exactly $\partial \overline D$. Using technical ingredients like the non-branching property (Theorem~\ref{Thm: non-branching}) and maximal extensions of a geodesic (Definition~\ref{Def: max extension}), it is possible to replicate aspects of this construction.

\subsection{Polygonal domains and grid construction}
In Section~\ref{Sec: poly dom}, we introduce the notion of interior polygonal domains, which are open subsets of the interior set bounded by simple polygonal loops. We will show in Theorem~\ref{Thm: poly dom neigh basis} via an explicit construction that every interior point admits a neighborhood basis consisting of interior polygonal domains. In the proof of Theorem~\ref{thm:mainRCD}, these domains act exactly the coordinate neighborhoods for the manifold interior. 
In Section~\ref{Sec: bd poly dom}, we introduce the notion of boundary polygonal domains. As with interior polygonal domains, we show in Theorem~\ref{Thm: bd poly dom neigh basis} that every extremal point admits a neighborhood basis consisting of extremal polygonal domains. These domains will act as the coordinate neighborhoods of the manifold boundary in the proof of Theorem~\ref{thm:mainRCD}.

\medskip

The final step is to show that interior polygonal domains are homeomorphic to open disks, while boundary polygonal domains are homeomorphic to half-disks. This is carried out in Section~\ref{Sec: man struct}.
The idea is to refine a polygonal domain into increasingly finer grids by constructing horizontal and vertical polygonal curves in a grid-like pattern. The Jordan-type theorems then guarantee that the complement of these curves has exactly the expected number of connected components, namely the cells of the grid. The goal is to build a countable family of grids that separates points in the polygonal domain. This separation property is then used to define a homeomorphism from the closure of the polygonal domain onto the closed disk, or onto the closed half-disk in the boundary case. 

\subsection*{Acknowledgements}
The authors are grateful to Daniele Semola for useful conversations on the topics of the paper.
Part of this work was carried out while EB was a von Neumann Fellow at the Institute for Advanced Study, whose excellent working conditions and support he gratefully acknowledges. This material is based upon work supported by the National Science Foundation under Grant No.~DMS-2424441.

\section{Preliminaries}\label{Sec: prelim}
\subsection{Metric space preliminaries}\label{Subsec: metric space} In this paper, a \emph{metric space} $(X, \dd_X)$ is always assumed to be proper, meaning that closed and bounded sets are compact. For brevity, we will often denote a metric space by $X$ in place of $(X, \dd_X)$. A metric space is \emph{pointed} if it is equipped with a distinguished point $x_0 \in X$, and will be denoted by $(X, \dd_X, x_0)$, or simply $(X, x_0)$. 

A continuous curve $\gamma: [a,b] \to X$ with $b > a$ is called a \emph{geodesic} if there exists $c > 0$ such that for any $s, t \in [a,b]$, $\dd_X(\gamma(s),\gamma(t)) = c|t-s|$. The unique $c$ for which this holds is called the \emph{speed} of $\gamma$, and a geodesic is \emph{unit-speed} if $c=1$. In particular, all geodesics in this paper are length-minimizing and constant-speed, and point maps are not geodesics. We say that a metric space $X$ is a \emph{geodesic space} if any $x,y \in X$ can be connected by a geodesic. 

A continuous curve $\gamma: \R \to X$ is called a \emph{line} if $\gamma \lvert_{[-t,t]}$ is a geodesic for all $t > 0$. A continuous curve $\gamma:[a, \infty) \to X$ is called a \emph{ray} if $\gamma \lvert_{[a,t]}$ is a geodesic for all $t > a$. 

A geodesic $\rho:[c,d] \to X$ is a \emph{subsegment} of a geodesic $\gamma:[a,b] \to X$ if there exists some subinterval $I \subset [a,b]$ such that $\rho$ is equal to $\gamma \lvert_I$, up to a linear (possibly orientation-reversing) reparameterization. We say that $\rho$ is a \emph{proper subsegment} of $\gamma$ if the above holds with $I \neq [a,b]$. More generally, we allow $\gamma$ to be a line or a ray as well, with the obvious modifications in the definition. 

Given a metric space $(X, \dd_X)$ and $p \in X$, we denote by $B^{X}_{r}(p)$ and $\overline B^{X}_r(p)$ the open and closed balls of radius $r$ in $X$ centered at $p$, respectively. We will often use the notation $B_r(p)$ and $\overline B_r(p)$ in place of $B_r^X(p)$ and $\overline B^X_r(p)$ if there is no ambiguity in the ambient metric space. Both $B_r(p)$ and $\overline B_r(p)$ are endowed with the extrinsic metric when viewed as metric spaces. 

More generally, given $A \subset X$, we denote by $B^{X}_r(A) := \{x \in X : \dd_{X}(x, A) < r\}$ and $\overline B^X_r(A) := \{x \in X : \dd_{X}(x, A) \leq r\}$ the open and closed $r$-neighborhoods of $A$, respectively. We will use the notations $B_r(A)$ and $\overline B_r(A)$ if the ambient metric space is clear. 

Set $\mathbb R_+ := [0, \infty)$. As we will often work with the plane and the upper half-plane, we denote by $\R^2_+ := \R \times \R_+$ the closed upper half-plane and $0^2 := (0,0)$.

\begin{Def}[Gromov--Hausdorff distance and convergence]\label{Def: GH distance}
Let $(X, \dd_X)$ and $(Y, \dd_Y)$ be compact metric spaces. A triple $(Z, \iota_1, \iota_2)$, where $(Z, \dd_Z)$ is a metric space, and $\iota_1: X \to Z$ and $\iota_2: Y \to Z$ are isometric embeddings, is called \emph{admissible}. 
The \emph{Gromov--Hausdorff distance} between $X$ and $Y$ is
\begin{equation*}
    \dd_{GH}(X, Y) := \inf_{(Z, \iota_1, \iota_2)} \dd^Z_H(\iota_1(X), \iota_2(Y)),
\end{equation*}
where the infimum is taken over all admissible triples $(Z, \iota_1, \iota_2)$, and $\dd^Z_H$ is the Hausdorff distance in $Z$. Recall that given a metric space $(Z, \dd_Z)$ and bounded subsets $E, F \subset Z$, the Hausdorff distance between $E$ and $F$ is defined by
\begin{equation*}
    \dd^Z_H(E, F) := \inf \{r \ge 0 : E \subset B_r(F) \text{ and } F \subset B_r(E)\}.
\end{equation*}
Any admissible triple $(Z, \iota_1, \iota_2)$ such that $\dd^Z_H(X, Y) < \veps$ is called an \emph{$\veps$-realization} (of the Gromov--Hausdorff distance) between $X$ and $Y$. 

Let $(X_n, \dd_n)$, $n \in \N \cup \{\infty\}$, be compact metric spaces. We say that \emph{$(X_n)_{n \in \N}$ converges to $X_\infty$ in the Gromov--Hausdorff sense} (or simply $X_n \overset{GH}{\longrightarrow} X_\infty$) if
\begin{equation*}
    \dd_{GH}(X_n, X_\infty) \to 0 \quad \text{as $n \to \infty$}.
\end{equation*}
We say that $(Z, \{\iota_n : n \in \N \cup \{\infty\}\})$, where $(Z, \dd_Z)$ is a metric space and each $\iota_n: X_n \to Z$ is an isometric embedding, is a \emph{realization of the convergence} $X_n \overset{GH}{\longrightarrow} X_\infty$ if
\begin{equation*}
    \dd_H^Z(\iota_n(X_n), \iota_\infty(X_\infty)) \to 0 \quad \text{as $n \to \infty$}.
\end{equation*}
\end{Def}

%We recall the standard gluing construction for metric spaces (see, for example, \cite{BBI01}). Let $\{(X_n, \dd_n):n \in I\}$ be a finite or countably infinite family of metric spaces. Let $(Z, \dd_Z)$ be a metric space and let $f_n: Z \to X_n$ be isometric embeddings for each $n \in I$. Define $\tilde X := \sqcup_{n \in I} X_n$. Define a semi-metric (\cite[Definition 1.1.4]{BBI01}) $\tilde \dd$ on $\tilde X$ by
%\begin{align*}
%\tilde \dd(x,y):=\begin{cases}
    %\dd_n(x,y) &\text{if $x, y \in X_n$,}\\
	%\inf_{z \in Z} \dd_n(x,f_n(z))+\dd_m(f_m(z), y) &\text{if $x \in X_n$ and $y \in X_m$ for $n \neq m$}.
	%\end{cases}
%\end{align*}
%Let $X$ be the set of equivalence classes of $\tilde X$ with respect to the equivalence relation $\tilde \dd = 0$. We can define a metric $\dd$ on $X$ by setting $\dd([x], [y]) = \tilde \dd(x,y)$ for any $x, y \in X$. Denote by $\pi: \tilde X \to X$ the quotient map, and by $\iota_n: X_n \to \tilde X$ the inclusion map for each $n \in I$. It is straightforward to check that $\pi \circ \iota_n: (X_n, \dd_n) \to (X, \dd)$ is an isometric embedding. 

%The following are easy consequences of the definition. 
%\begin{Pro}[Triangle inequality]\label{Pro: GH triangle ineq}
%If $X, Y,$ and $Z$ are compact metric spaces, then 
%\begin{equation*}
%    \dd_{GH}(X, Z) \leq \dd_{GH}(X,Y) + \dd_{GH}(Y,Z).
%\end{equation*}
%\end{Pro}

\begin{Pro}\label{Pro: GH distance balls center}
Let $(X, \dd_X)$ be a locally compact geodesic space, and let $p, q \in X$. Then
\[
\dd_{GH}(\overline B_r(p), \overline B_r(q)) \leq \dd_X(p,q)
\]
for every $r > 0$.
\end{Pro}
\begin{proof}
Set $s := \dd_X(p,q)$. It suffices to prove that $\dd_{H}^{X}(\overline B_r(p), \overline B_r(q)) \leq s$.
Let $x \in \overline B_r(p)$. Then there exists a geodesic from $x$ to $q$ of length at most $r + s$, and a point $y$ on this geodesic such that $\dd_X(x,y) \leq s$ and $\dd_X(y,q) \leq r$. Therefore, $x \in \overline B_{s}(\overline B_r(q))$. It follows that $\overline B_r(p) \subset \overline B_s(\overline B_r(q))$, and by symmetry, $\overline B_r(q) \subset \overline B_s(\overline B_r(p))$. 
\end{proof}

%Let $(X_n, \dd_n)$, where $n \in \N \cup \{\infty\}$, be compact metric spaces. Assume that $(X_n)_{n \in \N}$ converges to $X_\infty$ in the Gromov--Hausdorff sense, and let $\veps_n \to 0$ be such that $\dd_{GH}(X_n, X_\infty) < \veps_n$. For each $n \in \N$, let $(Z_n, \iota_1^n, \iota_2^n)$ be an $\veps_n$-realization between $X_n$ and $X_\infty$. Since the maps $\iota^n_2: X_\infty \to Z_n$ are isometric embeddings, we may glue the family of metric spaces $\{Z_n : n \in \N\}$ along the subspaces $\iota^n_2(X_\infty)$ to obtain a metric space $(Z, \dd_Z)$. It follows that for each $n \in \N \cup \{\infty\}$, there is a natural isometric embedding of $X_n$ into $Z$, and that $Z$ equipped with these embeddings is a realization of the convergence in the following sense.

The next lemma allows us to use Gromov--Hausdorff closeness at a given scale to obtain Gromov--Hausdorff closeness at smaller scales, with controlled error.
 
\begin{Lem}\label{Lem: realization scale change}
Let $(X, \dd_X)$ and $(Y, \dd_Y)$ be locally compact geodesic spaces, and let $p \in X$ and $q \in Y$. Assume that 
\[\dd_{GH}(\overline B_1(p), \overline B_1(q)) < \veps\]
for some $\veps > 0$, and let $(Z, \iota_1, \iota_2)$ be an $\veps$-realization between $\overline B_1(p)$ and $\overline B_1(q)$. Let $p' \in \overline B_s(p)$ and $q' \in \overline B_s(p)$ for some $s \in (0,1)$. If $\dd_Z(\iota_1(p'), \iota_2(q')) < \veps$, then \[\dd_{GH}(\overline B_r(p'), \overline B_r(q')) < 3\veps\]
for any $r \in (0, 1-s]$. Moreover, $(Z, \iota_1', \iota_2')$ is a $(3\veps)$-realization between $\overline B_r(p')$ and $\overline B_r(q')$, where $\iota_1'$ and $\iota_2'$ are the restrictions of $\iota_1$ and $\iota_2$ to $\overline B_r(p') \subset B_1(p)$ and $\overline B_r(q') \subset B_1(q)$.
\end{Lem}

\begin{proof}
We first show that $\iota_1(\overline B_r(p')) \subset B_{3\veps}(\iota_2(\overline B_r(q')))$.
Let $x \in \overline B_r(p')$. Choose $y \in \overline B_{1}(q)$ such that $\dd_{Z}(\iota_1(x), \iota_2(y)) < \veps$. By the triangle inequality,
\[    
\dd_Y(y,q')=\dd_{Z}(\iota_2(y), \iota_2(q')) \leq \dd_{Z}(\iota_2(y), \iota_1(x))+\dd_{Z}(\iota_1(x), \iota_1(p'))+\dd_{Z}(\iota_1(p'), \iota_2(q')) < r + 2\veps. 
\]
Choose a point $y' \in Y$ on a geodesic from $y$ to $q'$ such that $\dd_Y(y', y) < 2\veps$ and $\dd_Y(y', q') \leq r$. Then $y' \in \overline B_{r}(q') \subset B_1(q)$ is in the domain of $\iota_2$. Therefore, 
\begin{equation*}
    \dd_{Z}(\iota_1(x), \iota_2(y'))  \leq \dd_{Z}(\iota_1(x), \iota_2(y))+\dd_{Z}(\iota_2(y), \iota_2(y')) < \veps+2\veps=3\veps.
\end{equation*}
As $x \in \overline B_r(p')$ was arbitrary, we obtain $\iota_1(\overline B_r(p')) \subset B_{3\veps}(\iota_2(\overline B_r(q')))$. The reverse inclusion holds by symmetry. Therefore, $(Z, \iota_1', \iota_2')$ is a $(3\veps)$-realization between $\overline B_r(p')$ and $\overline B_r(q')$. 
\end{proof}

\begin{Def}\label{Def: regular realization}
Let $\veps, r >0$. Let $(X, \dd_X)$ be a locally compact metric space, and let $\gamma:[-r,r] \to X$ be a unit-speed geodesic. An admissible triple $(Z, \iota_1, \iota_2)$ between $\overline B_r(\gamma(0))$ and $\overline B_r(0^2)$ is called an \emph{$(\veps r)$-regular realization of $(\overline B_r(\gamma(0)), \gamma)$} if the following hold:
\begin{enumerate}
    \item $(Z, \iota_1, \iota_2)$ is an $(\veps r)$-realization between $\overline B_r(\gamma(0))$ and $\overline B_r(0^2)$;
    \item $\dd_Z(\iota_1(\gamma(t)), \iota_2((t,0))) < \veps$ for all $t \in [-r,r]$.
\end{enumerate}

More generally, if $\gamma:[a,b] \to X$ has a subsegment $\bar \gamma$ of length $2r$ with midpoint $\gamma(t_0)$, then we say $(Z, \iota_1, \iota_2)$ is an \emph{($\veps r$)-regular realization} of $(\overline B_r(\gamma(t_0)), \gamma)$ provided that it is an $(\veps r)$-regular realization of $(\overline B_r(\tilde \gamma(0)), \tilde \gamma)$ as defined above, where $\tilde \gamma:[-r, r] \to X$ is the unit-speed reparameterization of the subsegment $\bar \gamma$.
\end{Def}

\begin{Lem}\label{Lem: regular realization exist}
For every $\veps \in (0, 1/10)$, there exists $\delta(\veps) > 0$ such that for every locally compact metric space $(X, \dd_X)$, and every unit-speed geodesic $\gamma:[-r,r] \to X$, if 
\[\dd_{GH}(\overline B_r(\gamma(0)), \overline B_r(0^2))< \delta r,\] 
then there exists an $(\veps r)$-regular realization of $(\overline B_r(\gamma(0)), \gamma)$. 
\end{Lem}

We omit the proof of the previous lemma, since it is a standard compactness argument. 

%\begin{proof}
%By rescaling, we may assume that $r = 1$. Suppose that the proposition does not hold for some $\veps > 0$. Then there exists a sequence of metric spaces $X_n$ and unit-speed geodesics $\gamma_n: [-1,1] \to X_n$ such that for some sequence $\delta_n \to 0$, it holds that $\overline B_1(\gamma_n(0))$ is $\delta_n$-GH close to $\overline B_1(0^2)$, but there does not exist an $\veps$-regular realization of $(\overline B_1(\gamma_n(0)), \gamma_n)$. Let $(Z, \{\iota_n:n \in \N \cup \{\infty\}\})$ be the realization of the convergence $\overline B_1(\gamma_n(0)) \overset{GH}{\longrightarrow} \overline B_1(0^2)$. Up to taking a subsequence, $\iota_n \circ \gamma_n$ converges uniformly to $\iota_\infty \circ \gamma_\infty$, where $\gamma_\infty:[-1,1] \to  \overline B_1(0^2)$ is a unit-speed geodesic. Therefore, for all sufficiently large $n$, $(Z, \iota_n, \iota_\infty)$ is an $\veps$-realization between $\overline B_1(\gamma_n(0))$ and $ \overline B_1(0^2)$ such that $\dd_Z(\iota_n(\gamma_n(t)), \iota_\infty(\gamma_\infty(t))) < \veps$ for all $t \in [-1,1]$. By composing $\iota_\infty$ with the linear map that takes $(t,0)$ to $\gamma_\infty(t)$, we obtain an $\veps$-regular realization of $(\overline B_1(\gamma_n(0)), \gamma_n)$, which yields a contradiction.
%\end{proof}

%We note that by a more direct geometric argument, $\delta = O(\sqrt \veps)$ in the previous lemma. 

Keeping the notation of Definition~\ref{Def: regular realization}, we will use the regular realization of $(B_r(\gamma(0)), \gamma)$ to divide $B_r(\gamma(0))$, excluding a small tubular neighborhood of $\gamma$, into two regions approximating the upper and lower half-balls, as follows: Fix $\veps \in (0, 1/10)$ and $r > 0$. Let $(Z, \iota_1, \iota_2)$ be an $(\veps r/100)$-regular realization of $(B_r(\gamma(0)), \gamma)$. We define $H^{Z}_{\veps r, +}(\gamma([-r,r]))$ by
\begin{equation*}
\begin{aligned}
H^{Z}_{\veps r,+}(\gamma([-r,r])) := 
\Big\{ x \in \overline B_r(\gamma(0)) \setminus B_{\veps r}(\gamma([-r,r]))& \;: \; 
\exists y \in \overline B_r(0^2) \text{ such that } \\
&\dd_Z(\iota_1(x), \iota_2(y)) < \frac{\veps r}{100} \text{ and } \pi_2(y) \ge \frac{\veps r}{2} 
\Big\},
\end{aligned}
\end{equation*}
where $\pi_2: \R^2 \to \R$ is the projection map onto the second coordinate. In other words, $H^{Z}_{\veps r, +}(\gamma([-1,1]))$ is an approximate upper half-ball, as determined by the regular realization $Z$. The approximate lower half-ball $H^{Z}_{\veps r,-}(\gamma([-r,r]))$ is defined in the same way, except we require that $\pi_2(y) \leq -\veps r/2$. 

\begin{Lem}\label{Lem: approx half balls exist}
The sets $H^Z_{\veps r, +}(\gamma([-r,r]))$ and $H^Z_{\veps r, -}(\gamma([-r,r]))$ are non-empty, disjoint, and closed in $\overline B_r(\gamma(0))$. Moreover, \[\overline B_r(\gamma(0)) \setminus B_{\veps r}(\gamma([-r,r])) = H^Z_{\veps r, +}(\gamma([-r,r])) \sqcup H^Z_{\veps r, -}(\gamma([-r,r])).\]
\end{Lem}

\begin{proof}
By rescaling, we may assume that $r = 1$. Set $H_\pm := H^Z_{\veps, \pm}(\gamma([-1,1]))$.

Let $x \in \overline B_1(\gamma(0))$. Then there exists a point $y \in \overline B_1(0^2)$ such that $\dd_Z(\iota_1(x), \iota_2(y)) < \veps/100$. If in addition $x \notin B_\veps(\gamma([-1,1]))$, one can check from the triangle inequality that $\dd_{\R^2}(y, [-1,1] \times \{0\}) \geq \veps/2$. This implies that $\pi_2(y) \geq \veps/2$ or $\pi_2(y) \leq -\veps/2$. In particular, $x$ is contained in at least one of $H_\pm$, and so $\overline B_1(\gamma(0)) \setminus B_\veps(\gamma([-1,1])) = H_+ \cup H_-$. 

Next, we show that $H_+$ is non-empty. Let $p:= (0,1) \in \overline B_1(0^2)$. Then there exists a point $x \in \overline B_1(\gamma(0))$ such that $\dd_Z(\iota_1(x), \iota_2(p)) < \veps/100$. Since $\veps < 1/10$, one can check from the triangle inequality that $x \notin B_\veps(\gamma([-1,1]))$. It follows that $x \in H_+$ and so $H_+$ is non-empty. A symmetric argument shows that $H_-$ is also non-empty. 

Next, we show that the two sets are disjoint. Assume otherwise. Then there exist $y, y' \in \overline B_1(0^2)$, both of which approximate $x$, such that $\pi_2(y) \geq \veps/2$ and $\pi_2(y') \leq -\veps/2$. We have
\begin{align*}
    \dd_{\R^2}(y, y') &= \dd_Z(\iota_2(y), \iota_2(y'))\\
              &\leq \dd_Z(\iota_2(y), \iota_1(x)) + \dd_Z(\iota_2(y'), \iota_1(x))\\
                &< \veps/100 + \veps/100 = \veps/50.
\end{align*}
This is impossible, since $\dd_{\R^2}(y, y') \geq \veps$. 

Finally, it is clear that $H_+$ is open in $\overline B_1(\gamma(0)) \setminus B_\veps(\gamma([-1,1]))$. We have shown that $H_-$ is the complement of $H_+$ in $\overline B_1(\gamma(0)) \setminus B_\veps(\gamma([-1,1]))$. Therefore, $H_-$ is closed in $\overline B_1(\gamma(0)) \setminus B_\veps(\gamma([-1,1]))$, and hence closed in $\overline B_1(\gamma(0))$. By symmetry, $H_+$ is also closed. 
\end{proof}

The following lemma shows that the sets $H^Z_{\veps r, +}(\gamma([-r,r]))$ and $H^Z_{\veps r, -}(\gamma([-r,r]))$ are, up to a permutation of $\pm$, independent of the $(\veps r/100)$-regular realization $(Z, \iota_1, \iota_2)$. 

\begin{Lem}\label{Lem: approx half balls unique}
Let $(Z, \iota_1, \iota_2)$ and $(Z', \iota_1', \iota_2')$ be two $(\veps r/100)$-regular realizations of $(\overline B_r(\gamma(0)), \gamma)$. Then, up to permuting $\pm$, we have $H^Z_{\veps r, \pm}(\gamma([-r,r])) = H^{Z'}_{\veps r, \pm}(\gamma([-r,r]))$.
\end{Lem}

\begin{proof}
By rescaling, we can assume that $r = 1$. Set $H_\pm := H^Z_{\veps, \pm}(\gamma([-1,1]))$ and $H'_\pm := H^{Z'}_{\veps, \pm}(\gamma([-1,1]))$. By symmetry, it suffices to prove the following claim: if $x_1, x_2 \in H_+$ and $x_1 \in H_+'$, then $x_2 \in H_+'$. Assume for the sake of contradiction that there exist $x_1 \in H_+ \cap H_+'$ and $x_2 \in H_+$ such that $x_2 \notin H_+'$. By Lemma~\ref{Lem: approx half balls exist}, we have $x_2 \in H_-'$. 

Since $x_1, x_2 \in H_+$, there exist $y_1, y_2 \in \overline B_1(0^2)$ such that $\pi_2(y_1), \pi_2(y_2) \geq \veps/2$, and 
\[
    \dd_Z(\iota_1(x_1), \iota_2(y_1)) < \veps/100 \quad \text{and} \quad \dd_Z(\iota_1(x_2), \iota_2(y_2)) < \veps/100.
\]
Consider the line segment joining $y_1$ and $y_2$ in $\overline B_{1}(0^2)$. Choose points $\{q_i\}_{1\leq i\leq k}$ on this segment such that $q_1 = y_1$, $q_k = y_2$, and $\dd_{\R^2}(q_i, q_{i+1}) < \veps/100$ for all $i = 1,\dots,k-1$. Since $\pi_2(y_1), \pi_2(y_2) \geq \veps/2$, we obtain $\pi_2(q_i) \geq \veps/2$ for all $i$. Let $\{p_i\}_{1 \leq i \leq k}$ be points in $\overline B_1(\gamma(0))$ such that $p_1 = x_1$, $p_k = x_2$, and $\dd_{Z}(\iota_1(p_i), \iota_2(q_i)) < \veps/100$ for all $i$. By the triangle inequality, $\dd_{X}(p_i, p_{i+1}) = \dd_Z(\iota_1(p_i), \iota_1(p_{i+1})) < 3\veps/100$ for $i = 1, \dots, k-1$. Moreover, we have
\begin{align*}
    &\dd_X(p_i, \gamma([-1,1]))\\
            \geq \; & \dd_{Z}(\iota_2(q_i), \iota_2([-1,1] \times \{0\})) - \dd_H^Z(\iota_2([-1,1]\times\{0\}), \iota_1(\gamma([-1,1]))) - \dd_Z(\iota_1(p_i), \iota_2(q_i))\\
            >\; & \veps/2 - \veps/100 - \veps/100
            > 2\veps/5.
\end{align*}

Since $x_1 \in H_+'$ and $x_2 \in H_-'$, there exist $y_1', y_2' \in \overline B_1(0^2)$ such that $\pi_2(y'_1) \geq \veps/2$, $\pi_2(y'_2) \leq -\veps/2$, and 
\begin{equation*}
    \dd_{Z'}(\iota'_1(x_1), \iota'_2(y_1')) < \veps/100 \quad \text{and} \quad \dd_{Z'}(\iota'_1(x_2), \iota'_2(y'_2)) < \veps/100. 
\end{equation*}
Let $\{q_i'\}_{1\leq i\leq k}$ be points in $\overline B_{1}(0^2)$ such that $q'_1 = y'_1$, $q_k' = y'_2$, and $\dd_{Z'}(\iota_1'(p_i), \iota_2'(q_i')) < \veps/100$. From the previous estimates, we obtain, for $i = 1, \dots, k-1$,
\begin{align*}
    \dd_{\R^2}(q_i', q_{i+1}') &= \dd_{Z'}(\iota_2'(q_i'), \iota_2'(q_{i+1}'))\\
                        &\leq \dd_{Z'}(\iota_2'(q_i'), \iota_1'(p_i))+\dd_{Z'}(\iota_1'(p_i), \iota_1'(p_{i+1})) + \dd_{Z'}(\iota_1'(p_{i+1}), \iota_2'(q_{i+1}'))\\
                        &< \veps/100+3\veps/100+\veps/100 = \veps/20,
\end{align*}
and for $i = 1, \dots, k$,
\begin{align*}
    &\dd_{\R^2}(q_i', [-1,1]\times\{0\})\\
    = \; &\dd_{Z'}(\iota'_2(q_i'), \iota'_2([-1,1] \times \{0\}))\\
    \geq \; & \dd_{Z'}(\iota'_1(p_i), \iota_1'(\gamma([-1,1]))) - \dd^{Z'}_H(\iota_1'(\gamma([-1,1])), \iota_2'([-1,1]\times\{0\})) - \dd_{Z'}(\iota'_1(p_i), \iota'_2(q_i'))\\
    >\; &2\veps/5 - \veps/100 - \veps/100 = \veps/5.
\end{align*}
The second inequality implies that $|\pi_2(q_i')| > \veps/5$ for all $i$. We show that $\pi_2(q_i') > \veps/5$ for all $i$ by induction. Indeed, as the base of induction, we have $\pi_2(q_1') \geq \veps/2 > \veps/5$. Assume for some $i = 1,\dots, k-1$ it holds that $\pi_2(q_i') > \veps/5$. Since $\dd_{\R^2}(q_i', q_{i+1}') < \veps/20$, we obtain $\pi_2(q_{i+1}') > \veps/5 - \veps/20 > 0$. Combined with $|\pi_2(q_{i+1}')| > \veps/5$, we conclude that $\pi_2(q_{i+1}') > \veps/5$, which finishes the induction step. In particular, $\pi_2(q_k') > \veps/5$. However, $\pi_2(q_k') = \pi_2(y_2') \leq -\veps/2$, giving a contradiction. 
\end{proof}

The previous lemma implies a more general consistency result, across two different scales. 
\begin{Lem}\label{Lem: approx half ball consistent}
Let $(X, \dd_X)$ be a locally compact geodesic space, and let $\gamma:[-1,1] \to X$ be a unit-speed geodesic. Let $t \in (-1,1)$ and let $0 < r \leq 1-|t|$. Assume that for some $\veps \in (0,1/10)$, there exist an $(\veps r/10^3)$-regular realization $(Z, \iota_1, \iota_2)$ of $(\overline B_1(\gamma(0)), \gamma)$, and an $(\veps r/100)$-regular realization $(\tilde Z, \tilde \iota_1, \tilde \iota_2)$ of $(\overline B_r(\gamma(t)), \gamma)$. Then, up to permuting $\pm$, we have
\[
H^{\tilde Z}_{\veps r, \pm}(\gamma([t-r, t+r])) = H^{Z}_{\veps r/10, \pm}(\gamma([-1,1])) \, \cap \, \bigr(\overline B_r(\gamma(t)) \setminus B_{\veps r}(\gamma([t-r,t+r]))\bigr)
\]
\end{Lem}
\begin{proof}
By Lemma~\ref{Lem: approx half balls exist}, we have
\[
H^{\tilde Z}_{\veps r,+}\bigl(\gamma([t-r,t+r])\bigr)\sqcup
H^{\tilde Z}_{\veps r,-}\bigl(\gamma([t-r,t+r])\bigr)
=
\overline B_r(\gamma(t))\setminus B_{\veps r}\bigl(\gamma([t-r,t+r])\bigr).
\]
Therefore, up to permuting $\pm$, it suffices to prove the inclusions
\[
H^{\tilde Z}_{\veps r,\pm}\bigl(\gamma([t-r,t+r])\bigr)
\subset
H^{Z}_{\veps r/10,\pm}\bigl(\gamma([-1,1])\bigr).
\]

Let $\iota_1':\overline B_r(\gamma(t))\to Z$ be the restriction of $\iota_1$ to $\overline B_r(\gamma(t))\subset \overline B_1(\gamma(0))$. 
Let $\tau_t:\R^2\to\R^2$ be the translation $\tau_t(x)=x+(t,0)$, and define
$\iota_2':=\iota_2\circ \tau_t\big|_{\overline B_r(0)}$.
By Lemma~\ref{Lem: realization scale change}, $(Z,\iota_1',\iota_2')$ is an $(\veps r/100)$-regular realization of $(\overline B_r(\gamma(t)),\gamma)$.

Let $H^{Z}_{\veps r,\pm}(\gamma([t-r,t+r]))$ denote the corresponding approximate half-balls associated to $(Z,\iota_1',\iota_2')$. 
By Lemma~\ref{Lem: approx half balls unique}, we have
\[
H^{Z}_{\veps r,\pm}(\gamma([t-r,t+r]))
=
H^{\tilde Z}_{\veps r,\pm}(\gamma([t-r,t+r])),
\]
up to permuting $\pm$. 
Thus it suffices to show that
\[
H^{Z}_{\veps r,\pm}(\gamma([t-r,t+r]))
\subset
H^{Z}_{\veps r/10,\pm}(\gamma([-1,1])).
\]

Let $x \in H^{Z}_{\veps r, +}(\gamma([t-r, t+r]))$. Thus there exists $y' \in \overline B_r(0^2)$ such that 
\[
\dd_{Z}(\iota_1'(x), \iota_2'(y')) < \veps r/100 \quad \text{and} \quad \pi_2(y') \geq \veps r/2.
\]
Let $y := y' + (t,0)$. The above conditions are equivalent to 
\[
\dd_{Z}(\iota_1(x), \iota_2(y)) < \veps r/100 \quad \text{and} \quad \pi_2(y) \geq \veps r/2.
\]
According to Lemma~\ref{Lem: approx ball contain} below, we have
\[
\overline B_{r}(\gamma(t)) \setminus B_{\veps r}(\gamma([t-r, t+r])) \subset \overline B_{1}(\gamma(0)) \setminus B_{\veps r/10}(\gamma([-1,1])).
\]
In particular, $x \in \overline B_{1}(\gamma(0)) \setminus B_{\veps r/10}(\gamma([-1,1]))$. Since $\overline B_{1}(\gamma(0)) \setminus B_{\veps r/10}(\gamma([-1,1]))$ is the disjoint union of $H^{Z}_{\veps r/10, +}(\gamma([-1,1]))$ and $H^{Z}_{\veps r/10, -}(\gamma([-1,1]))$ by Lemma~\ref{Lem: approx half balls exist}, the point $x$ must be contained in one of the two sets. Assume for the sake of contradiction that $x \in H^{Z}_{\veps r/10, -}(\gamma([-1,1]))$. By definition, this means that there exists $z \in \overline B_1(0^2)$ such that 
\[
\dd_{Z}(\iota_1(x), \iota_2(z)) < \veps r/1000 \quad \text{and} \quad \pi_2(z) \leq -\veps r/20.
\]
By the triangle inequality, 
\[\dd_{\R^2}(y, z) \leq \dd_{Z}(\iota_2(y), \iota_1(x)) + \dd_{Z}(\iota_1(x), \iota_2(z))< \veps r/100 + \veps r/1000 < \veps r/2.\] On the other hand, we have $\pi_2(y) \geq \veps r/2$ and $\pi_2(z) \leq -\veps r/20$, leading to a contradiction. 
\end{proof}

\begin{Lem}\label{Lem: approx ball contain}
Let $(X, \dd_X)$ be a locally compact geodesic space, and let $\gamma:[-1,1] \to X$ be a unit-speed geodesic. Let $t \in (-1,1)$ and let $0 < r \leq 1-|t|$. Assume that for some $\veps \in (0,1/10)$ there exists an $(\veps r/100)$-regular realization $(Z, \iota_1, \iota_2)$ of $(\overline B_1(\gamma(0)) ,\gamma)$. Then 
\[\overline B_{r}(\gamma(t)) \setminus B_{\veps r}(\gamma([t-r,t+r])) \subset \overline B_1(\gamma(0)) \setminus B_{\veps r/10}(\gamma([-1,1])).\]
\end{Lem}
\begin{proof}
Let $x \in \overline B_{r}(\gamma(t)) \setminus B_{\veps r}(\gamma([t-r,t+r]))$. We have to show that $x \notin B_{\veps r/10}(\gamma([-1,1]))$.

By Lemma~\ref{Lem: realization scale change}, $(Z, \iota_1', \iota_2')$, where $\iota_1'$ and $\iota_2'$ are the restrictions of $\iota_1$ and $\iota_2$ to $\overline B_r(\gamma(t))$ and $\overline B_r((t,0))$ respectively, is an $(\veps r/10)$-realization between $\overline B_r(\gamma(t))$ and $\overline B_r((t,0))$. Therefore, there exists $y \in \overline B_r((t,0))$ such that 
\[\dd_{Z}(\iota'_1(x), \iota'_2(y)) = \dd_{Z}(\iota_1(x), \iota_2(y)) < \veps r/10.\]
By the triangle inequality, for any $s \in [t-r, t+r]$, we have
\begin{align*}
    \dd_{\R^2}(y, (s,0)) &= \dd_{Z}(\iota_2(y), \iota_2((s,0)))\\
    &\geq \dd_{Z}(\iota_1(x), \iota_1(\gamma(s))) - \dd_{Z}(\iota_1(x), \iota_2(y)) - \dd_{Z}( \iota_2((s,0)), \iota_1(\gamma(s)))\\
    &> \veps r - \veps r/10 - \veps r/100 > \veps r/2.
\end{align*}
Since $y \in \overline B_r((t,0))$, this implies that $|\pi_2(y)| > \veps r/2$. Applying the triangle inequality again, for any $s \in [-1,1]$, we have
\begin{align*}
    \dd_{X}(x, \gamma(s)) &= \dd_{Z}(\iota_1(x), \iota_1(\gamma(s)))\\
    &\geq \dd_{Z}(\iota_2(y), \iota_2((s,0))) - \dd_{Z}(\iota_2(y), \iota_1(x)) - \dd_{Z}(\iota_1(\gamma(s)), \iota_2((s,0)))\\
    &> \veps r/2 - \veps r/10 - \veps r/100 > \veps r/10. 
\end{align*}
Therefore, $x \notin B_{\veps r/10}(\gamma([-1,1]))$. 
\end{proof}

\subsection{RCD space preliminaries}\label{Subsec: RCD space}
In this subsection, we collect some prerequisite results on $\RCD$ spaces. As the theory is at this stage very well-developed, we only present background relevant to our arguments, and will not attempt to be comprehensive, referring to \cite{S06a, S06b, LV09, BS10, G13, AGS14, EKS15, G15, K15, DG16, DG18, AMS19, MN19, BS20, KM21, CM21, BNS22, W24a, D25} for details.  

In the development of the theory, various results were proved assuming the $\RCD^{*}(K,N)$ condition. This has since been shown to be equivalent to the $\RCD(K,N)$ condition \cite{CM21, L24} and hence we will use the $\RCD(K,N)$ assumption when citing various results, even if they were originally stated for $\RCD^*(K,N)$ spaces. 

Let $(X, \dd, \mm)$ be an $\RCD(K,N)$ space. Given $x_0 \in X$ and $r \in (0,1)$, we define the rescaled measure
\begin{equation}\label{Eq: rescaled measure}
    \mm^{x_0}_r := \Big(\int_{B_{r}(x_0)} 1 - \frac{\dd(x, x_0)}{r} d\mm(x)\Big)^{-1} \mm. 
\end{equation}
We recall the definition of tangent spaces, referring to \cite{GMS15} for the notion of pointed measured Gromov--Hausdorff convergence used in the definition. 
\begin{Def}[Tangent spaces]\label{Def: tangent space} Let $x_0 \in X$. A pointed metric measure space $(Y, \dd_Y, \mm_Y, y_0)$ is called a \emph{tangent space of $X$ at $x_0$} if there exists a sequence of radii $r_i \to 0$ such that $(X, r_i^{-1}\dd, \mm_{r_i}^{x_0}, x_0)$ converges to $(Y, \dd_Y, \mm_Y, y_0)$ in the pointed measured Gromov--Hausdorff sense as $i \to \infty$. We call $y_0$ the \emph{designated origin} of the tangent space. The collection of all tangent spaces of $X$ at $x_0$ is denoted $\text{Tan}(X, \dd, \mm, x_0)$. 
\end{Def}

A standard compactness argument gives that $\text{Tan}(X, \dd, \mm, x_0)$ is non-empty. In general, $\text{Tan}(X, \dd, \mm, x_0)$ may contain more than one element. It follows from the rescaling and stability properties of $\RCD(K,N)$ spaces \cite{AGS14, GMS15} that every element of $\text{Tan}(X, \dd, \mm,  x_0)$ is a pointed $\RCD(0,N)$ space. The following proposition is immediate from the definition.
\begin{Pro}\label{Pro: tangent of tangent}
Let $(X, \dd, \mm)$ be an $\RCD(K,N)$ space and let $x_0 \in X$. If $(Y, \dd_Y, \mm_y, y_0)$ is a tangent space of $X$ at $x_0$, then any tangent space of $Y$ at $y_0$ is also (pointed metric measure isomorphic to) a tangent space of $X$ at $x_0$. 
\end{Pro}

Let $k$ be any integer so that $1 \leq k \leq N$. Set $c_k := \int_{B_1(0^k)} 1-|x| \, d\mathcal{H}^k(x)$, where $0^k \in \R^k$ is the origin.  The \emph{$k$-dimensional regular set $\mathcal{R}_k$} is defined by
\begin{equation}\label{Eq: reg set}
    \mathcal{R}_k := \Big\{x \in X: \text{Tan}(X, \dd, \mm, x) = \{(\R^k, \dd_{\R^k}, c_k^{-1}\mathcal{H}^k, 0^k)\}\Big\}
\end{equation}

\begin{Rem}\label{Rem: metric Rk rigidity}
We will mostly be interested in metric properties of tangent spaces. In particular, when we say that a tangent space at $x_0$ is $Z$, we usually only mean that it is isometric to $Z$, without any implications on the measure. We note that in the case where the tangent space is isometric to $\R^k$, the splitting theorem (see Theorem~\ref{Thm: splitting thm}) gives measure rigidity as well. In particular, if all tangent spaces at $x_0$ are isometric $\R^k$, then $x_0 \in \mathcal R_k$. 
\end{Rem}

We have the following theorem from \cite{MN19, BS20} (cf. \cite{CN12}).
\begin{Thm}[Constancy of the dimension]\label{Thm: constant dim} 
Let $(X, \dd, \mm)$ be an $\RCD(K,N)$ for some $K \in \R$ and $N \in [1,\infty)$. Assume that $X$ is not a point. Then there exists a unique $n \in \N$ with $1 \leq n \leq N$ such that $\mm(X \setminus \mathcal R_n) = 0$.
\end{Thm}

The previous theorem motivates the following definition.
\begin{Def}[Essential dimension]\label{Def: essential dim}
Let $(X, \dd, \mm)$ be an $\RCD(K,N)$ for some $K \in \R$ and $N \in [1, \infty)$. We define the \emph{essential dimension of $X$}, denoted by $\dim(X, \dd, \mm)$ (or simply $\dim(X)$), as follows:
\[
\dim(X, \dd, \mm) :=
\begin{cases}
n, & \text{if $X$ is not a point, where $n$ is as in Theorem~\ref{Thm: constant dim}},\\
0, & \text{if $X$ is a point}.
\end{cases}
\]
\end{Def}

The $\RCD(K,N)$ spaces of essential dimension one are classified in \cite{KL16} (cf. \cite{H11}).
\begin{Thm}\label{Thm: RCD dim 1 class}
Let $(X, \dd, \mm)$ be an $\RCD(K,N)$ space for some $k \in \R$ and $N \in [1, \infty)$. If $\dim((X, \dd, \mm)) = 1$, then $X$ is isometric to $\R, \R_+$, a closed interval, or a circle. 
\end{Thm}

The splitting theorem of Cheeger-Gromoll for Riemannian manifolds of nonnegative Ricci curvature \cite{CG71} was generalized to the $\RCD$ setting in \cite{G13} (cf. \cite{CC96}). 
\begin{Thm}[Splitting theorem \cite{G13}]\label{Thm: splitting thm}
Let $(X, \dd_X, \mm_X)$ be an $\RCD(0,N)$ space for some $N \in [2,\infty)$. Assume that $X$ contains a line. Then there exists an $\RCD(0,N-1)$ space $(Y, \dd_Y, \mm_Y)$ such that
\[
(X, \dd_X, \mm_X) \cong (\R \times Y, \dd_\R \times \dd_Y, \mathcal{L}^1 \otimes \mm_Y)
\]
as metric measure spaces. Moreover, the given line in $X$ is mapped to $\R \times \{y_0\}$ for some $y_0 \in Y$ under this isomorphism.
\end{Thm}
 
\begin{Thm}\label{Thm: geod splitting thm}
Let $(X, \dd_X, \mm_X)$ be an $\RCD(K,N)$ space for some $K \in \R$ and $N \in [2, \infty)$, and let $\gamma:[-1,1] \to X$ be a unit-speed geodesic. Then for every tangent space $(Z, \dd_Z, \mm_Z, z_0)$  of $X$ at $\gamma(0)$, there exists a pointed $\RCD(0,N-1)$ space $(Y, \dd_Y, \mm_Y, y_0)$ such that
\[
(Z, \dd_Z, \mm_Z, z_0) \cong (\R \times Y, \dd_\R \times \dd_Y, \mathcal{L}^1 \otimes \mm_Y, (0, y_0)).
\]
Moreover, $\dim((Y, \dd_Y, \mm_Y)) \leq \dim((X, \dd_X, \mm_X)) - 1$.  
\end{Thm}
\begin{proof}
The first conclusion follows immediately from the splitting theorem (Theorem~\ref{Thm: splitting thm}) and a compactness argument. To prove the second conclusion, we note that the lower semi-continuity of the essential dimension under pmGH convergence \cite[Theorem 1.5]{K19} gives that $\dim(Z) \leq \dim(X)$. Since the essential dimension is additive under metric measure product for $\RCD$ spaces and $Z \cong \R \times Y$, we conclude. 
\end{proof}

The following theorem is immediate from Theorem~\ref{Thm: geod splitting thm} and a compactness argument. 
\begin{Thm}[Almost-splitting theorem]\label{Thm: almost splitting thm}
Let $(X, \dd_X, \mm_X)$ be an $\RCD(K,N)$ space for some $K \in \R$ and $N \in [2, \infty)$, and let $\gamma:[-1,1] \to X$ be a unit-speed geodesic. For any $\veps > 0$, there exists $r(K, N, \veps) > 0$ such that for any $s \leq r$, there exists a pointed $\RCD(0, N-1)$ space $(Y, y_0)$, and an $(\veps s)$-realization $(Z, \iota_1, \iota_2)$ between $\overline B_s(\gamma(0))$ and $\overline B^{\R \times Y}_s((0,y_0))$ such that $\dd_Z(\iota_1(\gamma(t)), \iota_2((t,y_0))) < \veps s$
for all $t \in [-s,s]$. In particular, 
\[\dd_{GH}(\overline B_s(\gamma(0)), \overline B^{\R \times Y}_s((0, y_0))) < \veps s.\] Moreover, $\dim((Y, \dd_Y, \mm_Y)) \leq \dim((X, \dd_X, \mm_X)) - 1$. 
\end{Thm}

A topological space $X$ is \emph{semi-locally simply connected} if for every $x \in X$, there is a neighborhood $U$ of $x$ such that the inclusion $U \hookrightarrow X$ induces the trivial map on $\pi_1$. It is known that $\RCD(K,N)$ spaces have the following stronger property which in particular implies the semi-locally simply connected property \cite{W24a} (see also \cite{PW22b, W24b}). 
\begin{Thm}\label{Thm: semi-locally simply connected}
Let $(X, \dd, \mm)$ be an $\RCD(K,N)$ space for some $K \in \R$ and $N \in (1, \infty)$. Then for any $x \in X$ and $R > 0$, there exists $r > 0$ such that the inclusion $B_r(x) \hookrightarrow B_R(x)$ induces the trivial map on $\pi_1$. 
\end{Thm}

It follows from Theorem~\ref{Thm: semi-locally simply connected} and \cite[Theorem 1.1 and Remark 2.2]{MW19} that any $\RCD(K,N)$ space admits a simply connected universal cover. 
\begin{Thm}[Existence of simply connected universal cover]\label{Thm: simply connected univ cover}
Let $(X, \dd, \mm)$ be an $\RCD(K,N)$ space for some $K \in \R$ and $N \in (1, \infty)$. Then $X$ admits a simply connected universal cover $(\widehat X, \widehat \dd, \widehat \mm)$ which is itself an $\RCD(K,N)$ space. 
\end{Thm}
We refer to \cite{MW19} for the definition of the universal cover in the context of $\RCD(K,N)$ spaces (see also \cite{SW01, SW04} for earlier work on the universal cover of Ricci limit spaces). 

We say that a geodesic space $(X, \dd)$ has the \emph{non-branching property} if for any pair of geodesics $\gamma, \rho:[0,1] \to X$, it holds: if $\gamma(t_1) = \rho(t_1)$ for some $t_1 \in [0,1]$ and $\gamma(t_2) = \rho(t_2)$ for some $t_2 \in (0,1)$ such that $t_2 \neq t_1$, then $\gamma \equiv \rho$. 
\begin{Thm}[Non-branching property \cite{D25}]\label{Thm: non-branching}
Let $(X, \dd, \mm)$ be an $\RCD(K,N)$ space for some $K \in \R$ and $N \in (1, \infty)$. Then $X$ has the non-branching property. 
\end{Thm}

We will be interested in geodesics that are extendible past their endpoints and in geodesics that are contained in the regular set, as these exhibit additional useful properties. 
\begin{Def}[Extendible geodesics]\label{Def: extend geod}
Let $(X, \dd)$ be a geodesic space. We say that $\gamma:[a,b] \to X$ is \emph{extendible past an endpoint $\gamma(a)$} (resp. $\gamma(b)$) if there exists $\veps > 0$ and a geodesic $\tilde \gamma:[a-\veps ,b] \to X$ (resp. $\tilde \gamma:[a, b+\veps] \to X$) such that $\tilde \gamma(t)=\gamma(t)$ for all $t\in[a,b].$ We say that $\gamma$ is \emph{extendible in both directions} if there exists $\veps>0$ and a geodesic $\tilde\gamma:[a-\veps,b+\veps]\to X$ such that $\tilde \gamma(t)=\gamma(t)$ for all $t\in[a,b]$.
\end{Def}

\begin{Def}[Regular geodesics]\label{Def: reg geod}
Let $(X, \dd, \mm)$ be an $\RCD(K,N)$ space of essential dimension $n$. We say that a geodesic $\gamma:[a,b] \to X$ is \emph{regular} if $\gamma((a,b)) \subset \mathcal R_n$, and \emph{strongly regular} if $\gamma([a,b]) \subset \mathcal R_n$. 
\end{Def}

The following propositions give the existence of many extendible regular geodesics. 
\begin{Pro}\label{Pro: geodesic a.e. extend}
Let $(X, \dd, \mm)$ be an $\RCD(K,N)$ space for some $K \in \R$ and $N \in (1, \infty)$. Then for $\mm$-a.e. $x \in X$, there is a unique geodesic from $x$ to $\mm$-a.e. $y$, and this geodesic is extendible in both directions and strongly regular. 
\end{Pro}

\begin{proof}
By the non-branching property (Theorem~\ref{Thm: non-branching}), any geodesic that is extendible past at least one of its endpoints must be the unique geodesic between its endpoints. Let $A_1 \subset X \times X$ (resp. $A_2 \subset X \times X$) be the set of all pairs $(x,y) \in X \times X$ such that a geodesic from $x$ to $y$ is extendible past $x$ (resp. $y$). For each $x \in X$, we denote by $A(x)$ the set of all points $y \in X$ such that a geodesic from $x$ to $y$ is extendible past $y$. 

Using the arguments of \cite[Section 4]{C14}, $A_1$ and $A_2$ are $(\mm \times \mm)$-measurable and $A(x)$ is $\mm$-measurable. The $\MCP(K,N)$ property \cite{S06b, O07}, satisfied by $\RCD(K,N)$ spaces, then implies that $\mm(X \setminus A(x)) = 0$ for any $x \in X$ by a standard argument. By Tonelli's theorem, it follows that $(\mm \times \mm)(X \times X \setminus (A_1 \cap A_2)) = 0$. Let $R$ be the set of pairs $(x,y) \in (X \times X)$ such that the geodesic from $x$ to $y$ is contained in $\mathcal{R}_n$ (where $n = \dim(X)$). By Theorem~\ref{Thm: constant dim} and the $\mm$-a.e. convexity of $\mathcal{R}_n$ \cite[Theorem 6.5]{D25}, it follows that $(\mm \times \mm)((X \times X) \setminus R)=0$. Therefore, $(\mm \times \mm)((X \times X) \setminus (A_1 \cap A_2 \cap R)) = 0$. Since $A_1 \cap A_2 \cap R$ is exactly the set of pairs $(x,y)$ such that the geodesic between $x$ and $y$ is unique, extendible in both directions and strongly regular, we conclude by Fubini's theorem. 
\end{proof}

\begin{Pro}\label{Pro: every pt have extend reg geod}
Let $(X, \dd_X, \mm)$ be an $\RCD(K,N)$ space for some $K \in \R$ and $N \in (1, \infty)$, and let $\dim(X) = n$. Then for any $x \in X$ there exists a unique geodesic $\rho:[0,1] \to X$ from $x$ to $\mm$-a.e. $y$. Moreover, $\rho((0,1]) \subset \mathcal R_n$ and $\rho$ is extendible past $y$.
\end{Pro}
\begin{proof}
As already discussed in the proof of Proposition~\ref{Pro: geodesic a.e. extend}, it holds that, for $\mm$-a.e. $y \in X$, the geodesic from $x$ to $y$ is unique and extendible past $y$. By \cite[Theorem 6.2.5]{D21}, it also holds that for $\mm$-a.e. $y \in Y$, there exists a geodesic $\rho:[0,1] \to X$ such that $\rho((0,1]) \subset \mathcal R_n$. The proposition follows by combining these two facts. 
\end{proof}

\subsection{Maximal extension of geodesics} In this subsection, we introduce the notion of the maximal extension of a geodesic from an endpoint. This notion will be important when we use cut locus methods in Section~\ref{Sec: ext set top}.

\begin{Lem}\label{Lem: max extension}
Let $(X, \dd_X, \mm_X)$ be an $\RCD(K,N)$ space, and let $\gamma:[a,b] \to X$ be a geodesic. Then exactly one of the following holds:
\begin{enumerate}
    \item There exist a unique $c \geq b$ and a unique geodesic $\rho: [a,c] \to X$ extending $\gamma$ (i.e., $\rho|_{[a,b]} = \gamma$) such that for any geodesic $\nu: [a,d] \to X$ with $d \ge b$ satisfying $\nu|_{[a,b]} = \gamma$, we have $d \le c$ and $\nu = \rho|_{[a,d]}$.
    \item There exists a unique geodesic ray $\rho: [a, \infty) \to X$ extending $\gamma$ (i.e., $\rho|_{[a,b]} = \gamma$).
\end{enumerate}
\end{Lem}

%\begin{proof}
%Let $\mathcal I$ be the collection of geodesics 
%\[
%\rho:[a,c] \to X, \quad c \ge b, \quad \text{such that } \rho|_{[a,b]} = \gamma.
%\] 
%By the non-branching property (Theorem~\ref{Thm: non-branching}), for any two geodesics $\rho_1, \rho_2 \in \mathcal I$, one is a subsegment of the other. This ensures that a maximal element in $\mathcal I$ with respect to the subsegment relation, if it exists, is unique. We have two cases:

%\noindent \textbf{Case 1:} $\sup_{\rho \in \mathcal I} \ell(\rho) = L < \infty$, where $\ell(\rho)$ denotes the length of $\rho$. Then by taking the limit of a converging sequence of geodesics $\rho_i \in \mathcal I$ with $\ell(\rho_i) \to L$, we obtain a geodesic $\rho \in \mathcal I$. It is straightforward to check that $\rho$ satisfies the conclusions of (1) by the non-branching property. Moreover, (2) does not hold in this case. 

%\noindent \textbf{Case 2:} $\sup_{\rho \in \mathcal I} \ell(\rho) = \infty$. Then by taking the limit of a converging sequence of geodesics $\rho_i \in \mathcal I$ with $\ell(\rho_i) \to \infty$, we obtain a ray $\rho:[0, \infty) \to X$. It is straightforward to check that $\rho$ is the unique ray extending $\gamma$ as in (2) by the non-branching property. Moreover, (1) does not hold in this case. 
%\end{proof}
The previous lemma follows from the non-branching property (Theorem~\ref{Thm: non-branching}) and a standard compactness argument, and motivates the following definition. 

\begin{Def}\label{Def: max extension}
Let $(X, \dd_X, \mm_X)$ be an $\RCD(K,N)$ space, and let $\gamma:[a,b] \to X$ be a geodesic. The \emph{maximal extension} of $\gamma$ from the endpoint $\gamma(a)$ is defined as the unique geodesic or ray $\rho$ given by Lemma~\ref{Lem: max extension}.
\end{Def}

We now prove the following technical lemma, which in particular shows that any geodesic can be $\veps$-perturbed so as to intersect a given polygonal loop (see Definition \ref{Def: poly loop}) at only finitely many points. This will simplify several arguments later. 
\begin{Lem}\label{Lem: perturb generic int}
Let $(X, \dd_X, \mm_X)$ be an $\RCD(K,N)$ of essential dimension $2$. Let $p \in X$ and let $\alpha:[0,1] \to X$ be a polygonal loop. Then for $\mm_X$-a.e. $x \in X$, the geodesic from $p$ to $x$ is unique, regular, and intersects $\alpha$ at only finitely many points. 
\end{Lem}
\begin{proof}
Let $\gamma_i$, $i = 1,\dots,m$, be the edges of $\alpha$. By Lemma~\ref{Lem: max extension}, each $\gamma_i$ admits maximal extensions from its two endpoints, which we denote by $\gamma_i^\pm$. Let 
\[
G := \bigcup_{i=1}^m \big(\mathrm{Im}(\gamma_i^+) \cup \mathrm{Im}(\gamma_i^-)\big),
\]
where $\mathrm{Im}(\gamma_i^\pm)$ denotes the image of $\gamma_i^\pm$.

First, assume that $p \notin \alpha([0,1])$. Let $x \in X \setminus G$ and let $\rho$ be any geodesic from $p$ to $x$. We claim that $\rho$ intersects each edge $\gamma_i$ at most once. This immediately implies that $\rho$ intersects $\alpha$ at only finitely many points. Assume for contradiction that there exists $\gamma_i$ intersecting $\rho$ at least twice. Since $x \in X \setminus G$, we have $x \notin \alpha([0,1])$. In particular, the endpoints of $\rho$ do not lie in $\alpha([0,1])$. Since $\rho$ intersects $\gamma_i$ twice, it follows by the non-branching property that $\gamma_i$ is a subsegment of $\rho$. Thus $\rho$, and hence $x$, lies on one of $\gamma_i^\pm$, a contradiction. 

It now follows from Proposition~\ref{Pro: every pt have extend reg geod} and Lemma~\ref{Lem: geod meas 0} below that for $\mm_X$-a.e. $x \in X$, the geodesic from $p$ to $x$ is unique, regular, and intersects $\alpha$ at only finitely many points.

Next, assume that $p \in \alpha([0,1])$. For each $i$ such that $p$ lies in the interior of the edge $\gamma_i:[a_i,b_i]\to X$, subdivide $\gamma_i$ at $p$ into two subsegments. Take the maximal extensions of both subsegments from $p$ and add their images to $G$. By Lemma~\ref{Lem: geod meas 0}, the set $G$ remains $\mm_X$-null.
It is not difficult to check that the previous argument goes through for this $G$. 
\end{proof}

%If $p$ lies outside of the interior of $\gamma_i$, the previous argument applies. If $p$ lies in the interior, it follows by the non-branching property that one of the subsegments of $\gamma_i$ obtained by subdividing $\gamma_i$ at $p$ must be a subsegment of $\rho$. Therefore, $\rho$, and hence $x$, lies on the maximal extension of that subsegment from $p$. This implies $x \in G$, contradicting our assumption. The rest of the proof continues as in the previous case. 

\begin{Lem}\label{Lem: geod meas 0}
Let $(X, \dd_X, \mm_X)$ be an $\RCD(K,N)$ space of essential dimension $2$. Then the image $S$ of any geodesic, ray, or line is an $\mm_X$-null set. 
\end{Lem}

\begin{proof}
By \cite[Theorem 1.2]{KM16}, the restricted measure $\mm_X \mrestr \mathcal R_2$ is absolutely continuous with respect to $\mathcal H^2$. Therefore, $\mm_X(S) = \mm_X(S\cap \mathcal R_2) = 0$.
\end{proof}

\section{Geodesics in spaces of essential dimension two}\label{Sec: geod}

In this section, we study the properties of geodesics in $\RCD(K,N)$ spaces of essential dimension $2$. The main result of the section is that, for any geodesic in such a space, either all tangent spaces along its interior are isometric to $\R^2$, or all are isometric to $\R^2_+$.

\begin{Lem}\label{Lem: R2 persistence}
For every $\veps > 0$, there exist $\bar r(N, \veps), \delta(\veps) > 0$ such that the following holds: Let $(X, \dd_X, \mm_X)$ be an $\RCD(-(N-1),N)$ space of essential dimension $2$, and let $\gamma: [-1,1] \to X$ be a unit-speed geodesic. If for some $0 < r \leq \bar r$ and some $t \in [-9/10, 9/10]$, it holds that
\[\dd_{GH}(\overline B_{r}(\gamma(t)), \overline B_{r}(0^2)) < \delta r,\]
then
\[\dd_{GH}(\overline B_{s}(\gamma(t)), \overline B_{s}(0^2)) < \veps s\]
for every $0 < s \leq r$. Moreover,
\[
\lim_{s \to 0^+}\frac{\dd_{GH}(\overline B_s(\gamma(t)), \overline B_s(0^2))}{s} = 0.
\]
\end{Lem}

\begin{proof}
By the classification Theorem~\ref{Thm: RCD dim 1 class}, there exists $\bar\veps > 0$ such that if $(Y,y_0)$ is a pointed $\RCD(0,N)$ space of essential dimension $\le 1$ and 
\[\dd_{GH}(\overline B^{\R \times Y}_1((0,y_0)), \overline B_1(0^2)) \le \bar \veps,
\]
then $B^{\R \times Y}_{1/10}((0,y_0))$ is isometric to $B_{1/10}(0^2)$. 

It suffices to prove the lemma for any $\veps \le \bar \veps$. Choose $\delta > 0$ so that for any $x \in X$ and $r > 0$, if 
\begin{equation*}
\dd_{GH}(\overline B_{r}(x), \overline B_{r}(0^2)) < \delta r,
\end{equation*}
then 
\begin{equation*}
\dd_{GH}(\overline B_{s}(x), \overline B_{s}(0^2)) < \veps s,
\end{equation*}
for all $s \in [r/10, r]$.

Suppose that the first part of the lemma does not hold. Then there exists a sequence of spaces $X_i$ with geodesics $\gamma_i$, and sequences $t_i \in [-9/10, 9/10]$ and $r_i \to 0$ such that 
\[
\dd_{GH}(\overline B_{r}(\gamma_i(t_i)), \overline B_{r}(0^2)) < \delta r
\]
but
\[\dd_{GH}(\overline B_{s}(\gamma(t_i)), \overline B_{s}(0^2)) \ge \veps s\]
for some $0 < s \leq r$. Let $s_i$ be the largest scale $\le r_i$ for which the latter occurs. Then $s_i \to 0$ and it follows from our choice of $\delta$ that $s_i < r_i/10$.

Passing to a subsequence of the rescaled spaces $(\frac{1}{10s_i}X_i, \gamma_i(t_i))$, we obtain as a pointed Gromov--Hausdorff limit a space $(X_\infty, \gamma_\infty(0))$, where $\gamma_\infty$ is a line, such that 
\[
\dd_{GH}(\overline B_{1}(\gamma_\infty(0)), \overline B_{1}(0^2)) \le \veps \le \bar \veps
\]
and 
\[
\dd_{GH}(\overline B_{1/10}(\gamma_\infty(0)), \overline B_{1/10}(0^2)) \ge \veps/10.   
\]
On the other hand, it follows from Theorem~\ref{Thm: geod splitting thm} that $X_\infty$ splits as $\R \times Y$, where $Y$ is an $\RCD(0,N)$ space of essential dimension $\le 1$. Together with our choice of $\bar \veps$, this yields a contradiction with the two inequalities above. This proves the first part of the lemma.

Suppose that the second part of the lemma does not hold. Then there exist $s_i \to 0$ and $\veps' > 0$ such that
\[
\dd_{GH}(\overline B_{s_i}(x), \overline B_{s_i}(0^2)) \ge \veps' s_i. 
\]
for every $i$. On the other hand, it follows from the first part of the lemma that
\[
\dd_{GH}(\overline B_{10s_i}(x), \overline B_{10s_i}(0^2)) < 10\veps s_i \le 10 \bar \veps s_i
\]
for every $i$. It is now straightforward to obtain a contradiction by taking a limit of the rescaled spaces $(\frac{1}{10s_i}X_i, \gamma(t))$ and arguing as in the proof of the first part of lemma. 
\end{proof}

Lemma~\ref{Lem: R2 persistence} can be interpreted as follows. Suppose that a ball of radius $r$, centered in the interior a geodesic, is sufficiently Gromov--Hausdorff close to a Euclidean $r$-ball. Assume moreover that for all scales $s \leq r$, the geodesic provides sufficient almost splitting. Then any smaller ball of radius $s \leq r$ centered at the same point is also Gromov--Hausdorff close to a Euclidean $s$-ball, and the approximation improves as $s$ decreases. 

%\begin{Rem} 
%On a non-collapsed $\RCD(K,N)$ space $X$, it holds that for any $\veps > 0$, there exists $\delta(K, N, \veps) > 0$ such that for any $p \in X$, if 
%\[\dd_{GH}(\overline B_r(p), \overline B_r(0^2)) < \delta r,\] then
%\[\dd_{GH}(\overline B_s(p), \overline B_s(0^2)) < \veps s\]
%for any $0 < s \leq r$. We give a sketch of the argument. Assume that $\overline B_r(p)$ is Gromov--Hausdorff close to $\overline B_r(0^2)$. By the volume convergence theorem (\cite{DG18}, see also \cite{C97}), the volume of $\overline B_r(p)$ is close to that of $\overline B_r(0^2)$, and hence almost maximal. The Bishop-Gromov inequality then forces the volume of all smaller balls $\overline B_s(p)$ to be almost maximal as well. Finally, it follows from volume rigidity (\cite{DG18}, see also \cite{C96, CC96, CC97}) that $\overline B_s(p)$ must be Gromov--Hausdorff close to $B_s(0^2)$.
%This propagation of regularity property was important in the application of the Colding-Colding Reifenberg theorem \cite[Theorem A.1.1]{CC97} to establish a manifold structure on the almost-regular set (see \cite{CC97} for the Ricci limit case and \cite{MN19} for the non-collapsed $\RCD$ case). The first part of Proposition~\ref{Pro: R2 persistence} can be thought of as a weaker version of this property, holding only in the interior of geodesics. It will also be vital for the proof of the manifold structure in our setting. 
%\end{Rem}

A similar statement holds when the $r$-ball is close to a ball in the upper half-plane $\R^2_+$. In this case, it is not generally true that smaller balls of radius $s \leq r$ remain Gromov--Hausdorff close to $B_s^{\R^2_+}(0^2)$. For example, if $X = \R^2_+$ and the geodesic $\gamma$ lies slightly above the boundary, then the $r$-ball $B_r(\gamma(t))$ may be close to $B_r^{\R^2_+}(0^2)$, while at sufficiently small scales $s \ll r$, the balls $B_s(\gamma(t))$ are close to $B_s^{\R^2}(0^2)$ instead. We will account for this possibility and show that this is in fact the only other thing that can happen.

For $r > 0$, let $\mathcal D_r$ denote the collection of all (equivalence classes up to isometry of) closed balls of radius $r$ in $\R^2_+$. Notice that $\overline B_r^{\R^2}(0^2) \in \mathcal D_r$, since it is isometric to $\overline B^{\R^2_+}_{r}((0,r))$. We have the following replacement for Lemma~~\ref{Lem: R2 persistence}. 

\begin{Lem}\label{Lem: R2 persistence 2}
For every $\veps > 0$, there exist $\bar r(N, \veps), \delta(\veps) > 0$ such that the following holds: Let $(X, \dd_X, \mm_X)$ be an $\RCD(-(N-1),N)$ space of essential dimension $2$, and let $\gamma: [-1,1] \to X$ be a unit-speed geodesic. If for some $0 < r \leq \bar r$ and some $t \in [-9/10, 9/10]$, it holds that
\[\dd_{GH}(\overline B_{r}(\gamma(t)), \mathcal D_r) < \delta r,\]
then 
\[\dd_{GH}(\overline B_{s}(\gamma(t)), \mathcal D_s) < \veps s\]
for every $0 < s \leq r$. Moreover, 
\[
\lim_{s \to 0^+}\frac{\dd_{GH}(\overline B_s(\gamma(t)), \mathcal D_s)}{s} = 0.
\]
\end{Lem}

\begin{proof}
By the classification Theorem~\ref{Thm: RCD dim 1 class}, there exists $\bar\veps > 0$ such that if $(Y,y_0)$ is a pointed $\RCD(0,N)$ space of essential dimension $\le 1$ and 
\[\dd_{GH}(\overline B^{\R \times Y}_1((0,y_0)), \mathcal D_1) \le \bar \veps,
\]
then $B^{\R \times Y}_{1/10}((0,y_0)) \in \mathcal D_{1/10}$. The rest of the proof is as in Lemma~\ref{Lem: R2 persistence}.
\end{proof} 

\begin{Thm}\label{Thm: geodesic interior tangent}
Let $(X, \dd_X, \mm_X)$ be an $\RCD(K,N)$ space of essential dimension $2$. If $\gamma: [-1,1] \to X$ is a geodesic, then one of the following holds:
\begin{enumerate}
    \item for every $t \in (-1,1)$, every tangent space at $\gamma(t)$ is pointed isometric to $(\R^2, 0^2)$;
    \item for every $t \in (-1,1)$, every tangent space at $\gamma(t)$ is pointed isometric to $(\R^2_+, 0^2)$.
\end{enumerate}
\end{Thm}

\begin{proof}
It suffices to prove that for any $t \in (-1,1)$, the tangent spaces at $\gamma(t)$ are either all $(\R^2, 0^2)$ or all $(\R^2_+, 0^2)$. Indeed, once we know that at each $t \in (-1,1)$ all tangent spaces are of a single type (either $\R^2$ or $\R^2_+$), the continuity of tangent spaces from \cite{CN12, D25} implies the type cannot change with $t$ in the interior. Therefore, the tangent spaces are of the same type of all $t \in (-1,1)$, and are either all $\R^2$ or all $\R^2_+$, giving the desired conclusion.  

Fix $t \in (-1,1)$. Up to rescaling and restricting $\gamma$, we may assume that $(X, \dd, \mm)$ is an $\RCD(-(N-1), N)$ space, $\gamma$ is a unit-speed geodesic parameterized on $[-1, 1]$, and $t = 0$.

Let $(Z, \dd_Z, \mm_Z, z_0)$ be a tangent space of $X$ at $\gamma(0)$. It follows from Theorem~\ref{Thm: geod splitting thm} that $(Z, \dd_Z, \mm_Z, z_0)$ splits isometrically as a pointed metric measure product of $(\R, \dd_{\R}, \mathcal L^1, 0)$ and $(Y, \dd_Y, \mm_Y, y_0)$, where the latter is a pointed $\RCD(0,N-1)$ space of essential dimension $\leq 1$. By the classification result for such spaces (Theorem~\ref{Thm: RCD dim 1 class}), $Y$ must be isometric to $\R, \R_+$, a closed interval $I$, a circle, or a point. We will consider each possible case. 

First, assume that $Y = \R$. Then for any $\bar r, \delta > 0$, there exists $0 < r \leq \bar r$ such that
\[\dd_{GH}(\overline B_r(\gamma(0)), \overline B^{\R^2}_r(0^2)) < \delta r.\]
Together with Lemma~\ref{Lem: R2 persistence}, this implies that
\[
\lim_{s \to 0^+}\frac{\dd_{GH}(\overline B_s(\gamma(0)), \overline B_s(0^2))}{s} = 0.
\]
Therefore, for any tangent space $(Z', \dd_{Z'}, \mm_{Z'}, z_0')$ of $X$ at $\gamma(0)$, the ball $(B_1(z_0'), z_0')$ is pointed isometric to $(B^{\R^2}_1(0^2), 0^2)$. This implies that all tangent spaces at $\gamma(0)$ must be pointed isometric to $(\R^2, 0^2)$. 

Next, assume that we are in one of the following cases:
\begin{enumerate}
    \item $Y$ is isometric to a circle;
    \item $Y$ is isometric to $\R_+$ or an interval, and $y_0$ does not lie on the boundary of $Y$;
\end{enumerate}
In both cases, $(\R^2, \dd_{\R^2}, c\mathcal L^2, 0^2)$, where $c > 0$ is a normalization constant, is a tangent space of $Z \cong \R \times Y$ at $z_0 = (0, y_0)$. By Proposition~\ref{Pro: tangent of tangent}, $(\R^2, \dd_{\R^2}, c\mathcal L^2, 0^2)$ is also a tangent space of $X$ at $\gamma(0)$. The conclusion of the previous paragraph then gives that all tangent spaces at $\gamma(0)$ must then be $\R^2$, contradicting (1) or (2). 

Next, assume that $Y$ is isometric to $\R_+$ or an interval, and $y_0$ lies on the boundary of $Y$. This implies that for any $\bar r, \delta > 0$, there exists $0 < r \leq \bar r$ such that
\[\dd_{GH}(\overline B_r(\gamma(0)), \overline B^{\R_+^2}_r(0^2)) < \delta r.\]
Arguing as in the case $Y = \R$ and using Lemma~\ref{Lem: R2 persistence 2} in place of Lemma~\ref{Lem: R2 persistence}, we see that for any tangent space $(Z', \dd_{Z'}, \mm_{Z'}, z_0')$ of $X$ at $\gamma(0)$, $B_1(z_0') \in \mathcal D_1$, i.e., it is isometric to a closed ball in $\R_+^2$. If $B_1(z_0')$ is isometric to some ball in $\mathcal D_1$ that is not $B^{\R_+^2}_1(0^2)$, then $\R^2$ is a tangent space of $Z'$ at $z_0'$. This leads to a contradiction arguing as in the previous paragraph. Therefore, $(B_1(z_0'), z_0')$ is pointed isometric to $(B^{\R^2_+}_1(0^2), 0^2)$. This implies that all the tangent spaces at $\gamma(0)$ must be pointed isometric to $(\R^2_+, 0^2)$. In particular, the case where $Y$ is an interval cannot occur. 

Finally, assume that $Y$ is a singleton. Let $(Z' ,\dd_{Z'}, \mm_{Z'}, z_0')$ be any other tangent space at $\gamma(0)$. We have shown that
\begin{enumerate}
    \item $Z'$ cannot be isometric to $\R^2$, since then all tangent spaces are isometric $\R^2$;
    \item $Z'$ cannot be isometric to $\R^2_+$ with $z_0'$ on the boundary of $\R^2_+$, since then all tangent spaces are isometric $\R^2_+$;
    \item $Z'$ cannot be isometric to $\R^2_+$ with $z_0'$ not on the boundary of $\R^2_+$;
    \item $Z'$ cannot be isometric to the product of $\R$ with a closed interval or a circle. 
\end{enumerate}
This means that all the tangent spaces at $\gamma(0)$ must be isometric to $\R$, which implies $\gamma(0) \in \mathcal{R}_1$ (see Remark \ref{Rem: metric Rk rigidity}). It follows from \cite[Theorem 1.1]{KL16} that $\dim(X) = 1$, which contradicts our assumption that $\dim(X) = 2$. 
\end{proof}

\begin{Rem}
The above phenomenon is unique to spaces of low essential dimension. There are examples of noncollapsed Ricci limit spaces of dimension $5$ such that the tangent cones coming from the same sequence of rescalings along a geodesic are not constant \cite{CN12}. In the Alexandrov setting, it is known that tangent cones are constant along the interior of a geodesic in any dimension \cite{P98}. 
\end{Rem}

\section{Components of balls at the interior of geodesics}\label{sec:choice_sides_geo}

In this section, we consider geodesics with interiors supported in $\mathcal R_2$. The main result of this section is that any such geodesic disconnects a small ball centered at an interior point into two components. 

Let $X$ be a topological space. A subset $C \subset X$ is called a \emph{path-connected component} if $C$ is path-connected and $C$ is maximal with respect to inclusion, i.e., if $C \subset D \subset X$ and $D$ is path-connected, then $C = D$. Since all spaces considered in this paper are locally path-connected, any path-connected component is clopen and coincides with a connected components of the space. We will simply refer to a path-connected component as a \emph{component} for brevity. 

We wish to show that geodesics disconnect small balls centered at its interior into two components satisfying some additional properties. This motivates the following definition. To simplify notation, we will sometimes identify a curve $\gamma:[a,b]\to X$ with its support $\gamma([a,b])\subset X$. For instance, we write $B_r(\gamma(t)) \setminus \gamma$ to mean that $B_r(\gamma(t)) \setminus \gamma([a,b])$.

\begin{Def}[Good components]\label{Def: good comp}
Let $(X, \dd_X)$ be a geodesic space and let $\gamma:[a,b] \to X$ be a geodesic. For any $t \in [a,b]$ and $r > 0$, we say that $B_r(\gamma(t))$ has \emph{two good components with respect to $\gamma$} if there exist two disjoint components $A_\pm$ of $B_r(\gamma(t)) \setminus \gamma$ such that the following hold:
\begin{enumerate}
    \item $A_\pm$ are clopen in $B_r(\gamma(t)) \setminus \gamma$ and $B_r(\gamma(t)) \setminus \gamma = A_+ \sqcup A_-$. 
    \item The topological boundaries of $A_\pm$ in $B_r(\gamma(t))$ coincide with $\gamma \cap B_r(\gamma(t)).$
    \item Every continuous curve contained in $B_{10r}(\gamma(t))$ connecting a point $x \in A_+$ to a point $y \in A_-$ intersects $\gamma$.
\end{enumerate}
In this case, the two components $A_\pm$ are called \emph{good components} of $B_r(\gamma(t)) \setminus \gamma$.
\end{Def} 

%The following example shows that $B_r(\gamma(t))$ can have two good components with respect to $\gamma$, even when $\gamma(t)$ is an endpoint and $r$ is larger than the length of $\gamma$.
%\begin{Exa}
%Let $X = \R \times [0,1]$ and let $\gamma:[0,1] \to X$ be defined by $\gamma(t) := (0,t)$. Then $B_r(\gamma(0))$ has two good components with respect to $\gamma$ for every $r > 0$.
%\end{Exa}
In this section, we focus on proving that sufficiently small balls centered at an interior point of a geodesic $\gamma$ have two good components with respect to $\gamma$, leaving the case of balls centered at endpoints to Section~\ref{Sec: ext-int points}.

\begin{Thm}\label{Thm: good components existence}
Let $(X,\dd_X,\mm_X)$ be an $\RCD(-(N-1),N)$ space of essential dimension $2$, and let $\gamma:[-1,1]\to X$ be a unit-speed geodesic. There exist $\veps > 0$ and $0< \bar r(N) \leq 1/1000$ such that, if $0<r \le \bar r$ and
   \begin{equation}
       \dd_{GH}(\overline B_{1000 r}(\gamma(0)), \overline B_{1000r}(0^2)) < \veps\, r,
   \end{equation}
   then $B_r(\gamma(0))$ has two good components $A_\pm$ with respect to $\gamma$. Moreover, if $x$ and $y$ belong to the same set $A_\pm \cap B_{r/10}(\gamma(0))$, then every geodesic joining $x$ to $y$ is disjoint from $\gamma$.
\end{Thm}

We will prove that $B_r(\gamma(0)) \setminus \gamma$ decomposes as the disjoint union of two clopen subsets by an explicit construction in Subsection~\ref{Subsec: existence of comps}. We then verify that these two clopen subsets satisfy the additional properties in the definition of good components in Subsection~\ref{Subsec: properties of comps}. In Subsection~\ref{Subsec: geod neigh}, we will use good components to construct path-connected neighborhoods of a geodesic that is disconnected into exactly two components by the geodesic. 

We also record Theorem~\ref{Thm: good components existence} in the following non-quantitative form. 
\begin{Thm}\label{Thm: good components existence 2}
Let $(X, \dd_X, \mm_X)$ be an $\RCD(-(N-1),N)$ space of essential dimension $2$, and let $\gamma:[a,b] \to X$ be a regular geodesic. Then for any $a < c \leq d < b$, there exists some $\bar r$ such that for every $0 < r \leq \bar r$ and $t \in [c,d]$, the ball $B_r(\gamma(t))$ has two good components $A_\pm$ with respect to $\gamma$. Moreover, if $x$ and $y$ belong to the same set $A_\pm \cap B_{r/10}(\gamma(t))$, then every geodesic joining $x$ to $y$ is disjoint from $\gamma$.
\end{Thm}

\begin{proof}
Since $\gamma$ is regular, a compactness argument (using Proposition~\ref{Pro: GH distance balls center}) gives that for every $\veps > 0$, there exists $\tilde r  >0$ such that for every $0 < r \leq  \tilde r$ and $t \in [c,d]$,
\[
\dd_{GH}(\overline B_r(\gamma(t)), \overline B_r(0^2)) < \veps r.
\]
The corollary now follows by Theorem~\ref{Thm: good components existence} and a straightforward rescaling argument.
\end{proof}

\subsection{Existence of components}\label{Subsec: existence of comps}
The following is the main result of this subsection.

\begin{Pro}\label{Pro: comp exist}
Let $(X,\dd_X,\mm_X)$ be an $\RCD(-(N-1),N)$ space of essential dimension $2$, and let $\gamma:[-1,1]\to X$ be a unit-speed geodesic. Then there exist $\veps > 0$ and $0< \bar r(N) \leq 1/1000$ such that, if $0<r \le \bar r$ and 
\[\dd_{GH}(\overline B_{100r}(\gamma(0)), \overline B_{100r}(0^2)) < \veps r,\]
then there exist two disjoint non-empty clopen subsets $A_\pm$ of $B_r(\gamma(0)) \setminus \gamma$ such that $B_r(\gamma(0)) \setminus \gamma = A_+ \sqcup A_-$. Moreover, the topological boundaries of $A_\pm$ in $B_r(\gamma(0))$ coincide with $\gamma((-r,r)) = \gamma \cap B_r(\gamma(0))$.
\end{Pro}

The proposition will be proved by an explicit construction of the two clopen subsets. The idea of the construction is as follows. 

Lemma~\ref{Lem: R2 persistence} allows us to use the regularity of $\overline B_{10r}(\gamma(0))$ to obtain regularity for $\overline B_{s}(\gamma(t))$, for any $s \leq r$ and any $t \in [-r,r]$. Using the approximate half-ball construction in Subsection~\ref{Subsec: metric space} and the regularity of $B_{r}(\gamma(0))$, we can decompose $B_{r}(\gamma(0))$ into two disjoint subsets $H_\pm$ away from the $(r/100)$-neighborhood of $\gamma$. Similarly, using the regularity of $\overline B_{r/10}(\gamma(t))$ for each $t \in [-r,r]$, we can decompose $B_{r/10}(\gamma(t))$ into two disjoint subsets $H_\pm(t)$ away from the $(r/1000)$-tubular neighborhood of $\gamma$. 

The sets $H_\pm$ and $H_\pm(t)$ are determined up to exchanging $\pm$. Fix $H_\pm$. This can now be used to determine a labeling $H_\pm(t)$ for each $t \in [-r,r]$. More precisely, if $B_r(\gamma(0))$ is sufficiently regular, then it in fact determines approximate half-balls for $B_{r}(\gamma(0))$ away from a much smaller (for instance, $r/10^4$) tubular neighborhood of $\gamma$. Using this, it can be checked that for each $t \in [-r,r]$, exactly one of the sets $H_\pm(t)$ intersects $H_+$, and the other intersects $H_-$. Therefore, we can fix $H_\pm(t)$ so that $H_\pm(t) \subset H_\pm$. It is also possible to check that our choices for $H_\pm(t)$ is consistent, in the sense that if $x$ is contained in one of $H_\pm(t_1)$ and one of $H_\pm(t_2)$ for $t_1 \neq t_2$, then $x$ is contained in $H_+(t_1) \cap H_+(t_2)$ or $H_-(t_1) \cap H_-(t_2)$. 

We can now divide $B_r(\gamma(0))$ into two disjoint subsets away from the $(r/1000)$-neighborhood of $\gamma$, one subset being the points contained in $H_+$ or $H_+(t)$ for some $t \in [-r,r]$, and the other being the points contained in $H_-$ or $H_-(t)$ for some $t \in [-r,r]$. By repeating this procedure inductively, checking consistency at every step, we can divide $B_r(\gamma(0)) \setminus \gamma$ into the disjoint union of two clopen subsets. The key point is that the check for consistency at each scale only uses the regularity of balls at a slightly larger scale. 

We proceed with our plan. The actual construction differs slightly from the outline, but the underlying idea remains the same. We first give our construction assuming a purely metric condition, which we will later check is satisfied under the assumptions of Proposition~\ref{Pro: comp exist}. 
\begin{Lem}\label{Lem: comp exist metric}
Let $(X, \dd_X)$ be a locally compact geodesic space, and let $\gamma:[-10,10] \to X$ be a unit-speed geodesic. There exists $\veps > 0$ such that if 
\[\dd_{GH}(\overline B_{10}(\gamma(0)), \overline B_{10}(0^2)) < 10\veps,\]
and for every $t \in [-2,2]$ and $0 < r \leq 1$,
\[\dd_{GH}(\overline B_{r}(\gamma(t)), \overline B_{r}(0^2)) < \veps r,\]
then there exist two disjoint non-empty clopen subsets $A_\pm$ of $B_1(\gamma(0)) \setminus \gamma$ such that $B_1(\gamma(0)) \setminus \gamma = A_+ \sqcup A_-$. Moreover, the topological boundaries of $A_\pm$ in $B_1(\gamma(0))$ coincide with $\gamma((-1,1))$.
\end{Lem}

\begin{proof}
By Lemma~\ref{Lem: regular realization exist}, we may choose $\veps > 0$ so that if the assumptions of the lemma hold, then there exists an $(10^{-99})$-regular realization (see Definition~\ref{Def: regular realization}) of $(\overline B_{10}(\gamma(0)), \gamma)$, and, for each $t \in [-1,1]$ and $0 < r \leq 1$, an $(10^{-100}r)$-regular realization of $(\overline B_{r}(\gamma(t)), \gamma)$

Fix a $(10^{-5})$-regular realization $(Z, \iota_1, \iota_2)$ of $(\overline B_{10}(\gamma(0)), \gamma)$. Applying Lemma~\ref{Lem: approx half balls exist} to $Z$, the set 
\[\overline B_{10}(\gamma(0)) \setminus B_{10^{-3}}(\gamma([-10,10]))\] 
decomposes as the disjoint union of two clopen subsets $H^{Z}_{10^{-3}, \pm}(\gamma([-10,10]))$. Set
\begin{equation}\label{Eq: comp exist metric eq 0}
H_{\pm} := H^{Z}_{10^{-3}, \pm}(\gamma([-10,10])).
\end{equation}

Let $t \in [-2,2]$. From our assumption, for each $k \in \N$ there exists a $(10^{-(k+4)})$-regular realizations of $(\overline B_{10^{-k}}(\gamma(t)), \gamma)$. Applying Lemma~\ref{Lem: approx half balls exist} to any such regular realization, we obtain a decomposition of \[\overline B_{10^{-k}}(\gamma(t)) \setminus B_{10^{-(k+2)}}(\gamma([t-10^{-k}, t+10^{-k}]))\] into the disjoint union of two clopen subsets. 

By Lemma~\ref{Lem: approx half balls unique}, these subsets are independent of the $(10^{-(k+4)})$-regular realization used. We will choose a labeling $H^k_{\pm}(t)$ for these two subsets inductively on $k$ as follows. 

For the base case $k=0$, Lemma~\ref{Lem: approx half ball consistent} implies that the two subsets in the decomposition of 
\[\overline B_{1}(\gamma(t)) \setminus B_{10^{-2}}(\gamma([t-1, t+1]))\]
coincide with
\[H_\pm \cap \bigr(\overline B_{1}(\gamma(t)) \setminus B_{10^{-2}}(\gamma([t-1, t+1]))\bigr).\]
We denote these by $H^0_\pm(t)$, respectively. 

For the induction step, assume that $H^k_\pm(t)$ have been fixed for some $k \in \N$. By Lemma~\ref{Lem: approx half balls exist}, since there exists a $(10^{-(k+6)})$-regular realization of $(\overline B_{10^{-k}}(\gamma(t)), \gamma)$, we can decompose
\[\overline B_{10^{-k}}(\gamma(t)) \setminus B_{10^{-(k+4)}}(\gamma([t-10^{-k}, t+10^{-k}]))\]
into the union of two disjoint clopen subsets. By Lemma~\ref{Lem: approx half ball consistent}, one of these subsets contains $H_+^k(t)$ and the other contains $H_-^k(t)$. We denote them by $\tilde H^{k}_+(t)$ and $\tilde H^{k}_-(t)$, respectively.

Applying Lemma~\ref{Lem: approx half ball consistent} again, we have that one of the sets in the decomposition of \[\overline B_{10^{-(k+1)}}(\gamma(t)) \setminus B_{10^{-(k+3)}}(\gamma([t-10^{-(k+1)}, t+10^{-(k+1)}]))\] is contained in $\tilde H_+^k(t)$ and other in $\tilde H_-^k(t)$. We set $H^{k+1}_+(t)$ to be the former and $H^{k+1}_-(t)$ to be the latter. This finishes the inductive labeling procedure.

\begin{Cla}\label{Cla: comp exist metric Cla 1}
For any $k_1, k_2 \in \N$ and $t_1, t_2 \in [-2,2]$, if $|k_2 - k_1| \geq 3$, then the sets $H^{k_1}_+(t_1)$ and $H^{k_2}_-(t_2)$ are disjoint. 
\end{Cla}
\begin{proof}
By symmetry, we may assume that $k_2 \geq k_1+3$ and $t_2 \geq t_1$. It suffices to show that
\begin{equation}\label{Eq: comp exist metric eq1}
\overline B_{10^{-k_1}}(\gamma(t_1)) \setminus B_{10^{-(k_1+2)}}(\gamma([t_1-10^{-k_1}, t_1+10^{-k_1}])) \quad \text{and} \quad \overline B_{10^{-k_2}}(\gamma(t_2))
\end{equation}
are disjoint, since $H^{k_1}_+(t_1)$ is contained in the former and $H^{k_2}_-(t_2)$ is contained in the latter. 

There are two subcases to consider:
\begin{enumerate}
    \item If $t_2 \leq t_1 + 10^{-k_1}$: Then any $x \in \overline B_{10^{-k_2}}(\gamma(t_2))$ is contained in the $10^{-k_2}$-neighborhood of $\gamma([t_1-10^{-k_1}, t_1+10^{-k_1}])$, since $t_2 \in [t_1-10^{-k_1}, t_1+10^{-k_1}]$. In particular, if $k_2 > k_1+2$, then any such $x$ cannot belong in the first set of \eqref{Eq: comp exist metric eq1}. 
    \item If $t_2 > t_1 + 10^{-k_1}$: Let $x \in \overline B_{10^{-k_2}}(\gamma(t_2))$. If $x \notin \overline B_{10^{-k_1}}(\gamma(t_1))$ then there is nothing to check. Thus assume that $x \in \overline B_{10^{-k_1}}(\gamma(t_1))$. By the triangle inequality,
    \[t_2 - t_1 = \dd_X(\gamma(t_2), \gamma(t_1)) \leq \dd_X(\gamma(t_2), x) + \dd_X(x, \gamma(t_1)) < 10^{-k_2} + 10^{-k_1}.\] 
    Therefore, $t_2 \in (t_1 + 10^{-k_1}, t_1+10^{-k_1}+10^{-k_2}]$. By the traingle inequality,
    \begin{align*}
    \dd_X(x, \gamma([t_1-10^{-k_1}, t_1+10^{-k_1}])) & \leq \dd_X(x, \gamma(t_1+10^{-k_1}))\\
    &\leq \dd_X(x, \gamma(t_2)) + \dd_X(\gamma(t_2), \gamma(t_1+10^{-k_1}))\\
    &< 10^{-k_2} + 10^{-k_2} < 2(10)^{-k_2} < 10^{-(k_1+2)}. 
    \end{align*}
    In particular, $x$ does not belong in the first set of \eqref{Eq: comp exist metric eq1}. \qedhere
\end{enumerate}

\end{proof}

\begin{Cla}\label{Cla: comp exist metric Cla 2}
For any $k_1, k_2 \in \N$ and $t_1, t_2 \in [-2,2]$, the sets $H^{k_1}_+(t_1)$ and $H^{k_2}_-(t_2)$ are disjoint.
\end{Cla}

\begin{proof}
We prove that for any $k \in \N$, if $k_1, k_2 \le k$ and $t_1, t_2 \in [-2,2]$, then $H^{k_1}_+(t_1)$ and $H^{k_2}_-(t_2)$ are disjoint. The proof is by induction on $k$. 

For the base case $k=0$, we must have $k_1 = k_2 = 0$. By definition, $H^{0}_+(t_1) \subset H_+$ and $H^{0}_-(t_2) \subset H_-$, where $H_\pm$ are as in \eqref{Eq: comp exist metric eq 0}. Since $H_\pm$ are disjoint, it follows immediately that $H^{0}_+(t_1)$ and $H^{0}_+(t_2)$ are disjoint as well. 

For the induction step, assume that the claim holds for some $k \in \N$. We show that it also holds for $k+1$. Let $k_1,k_2\le k+1$ and $t_1,t_2\in[-1,1]$. It suffices to consider the case where at least one of $k_1, k_2 = k+1$, since otherwise the conclusion follows directly from the induction hypothesis. Without loss of generality, assume that $k_1 = k+1$. 

By Claim~\ref{Cla: comp exist metric Cla 1}, we only need to check that $H^{k_1}_+(t_1)$ and $H^{k_2}_-(t_2)$ are disjoint for $k_2 = k+1, k,$ and $k-1$. We consider only the case $k_2 = k+1$. The cases $k_2=k$ and $k_2=k-1$ follow by straightforward modifications of the same argument.

Assume for the sake of contradiction that there exists $x \in H^{k+1}_+(t_1) \cap H^{k+1}_-(t_2)$. By our assumptions, there exists a $(10^{-(k+7)})$-regular realization of $(\overline B_{10^{-(k-1)}}(\gamma(t_1)), \gamma)$. (In the case $k = 0$, we instead work with $\overline B_{10}(\gamma(0))$; the argument is similar with some minor adjustments, so we do not treat it separately.) Applying Lemma~\ref{Lem: approx half balls exist} to any such regular realization, we obtain a decomposition of
\[
\overline B_{10^{-(k-1)}}(\gamma(t_1)) \setminus B_{10^{-(k+5)}}(\gamma([t_1 - 10^{-(k-1)}, t_1 + 10^{-(k-1)}]))
\]
into the disjoint union of two clopen subsets.

By Lemma~\ref{Lem: approx half ball consistent}, exactly one of these clopen sets contains $\tilde H^k_+(t_1)$ and the other contains $\tilde H^k_-(t_1)$. We denote them by $\hat H^{k-1}_\pm(t_1)$, respectively. Recall that $\tilde H^k_\pm(t_1)$ are the clopen sets appearing in the inductive definition of $H^k_\pm(t_1)$, namely in the decomposition of
\[
\overline B_{10^{-k}}(\gamma(t_1)) \setminus B_{10^{-(k+4)}}(\gamma([t_1 - 10^{-k}, t_1 + 10^{-k}])).
\]
Moreover, as observed in the construction, 
\[H^{k+1}_+(t_1), H^{k}_+(t_1) \subset \tilde H^{k}_+(t_1) \quad \text{and} \quad H^{k+1}_-(t_1), H^{k}_-(t_1) \subset \tilde H^{k}_-(t_1),\] and therefore, 
\[H^{k+1}_+(t_1), H^{k}_+(t_1) \subset \hat H^{k-1}_+(t_1) \quad \text{and} \quad H^{k+1}_-(t_1), H^{k}_-(t_1) \subset \hat H^{k-1}_-(t_1).\] 
In particular, since $x\in H^{k+1}_+(t_1)$, we obtain $x \in \hat H^{k-1}_+(t_1)$.

On the other hand, by the triangle inequality, we have
\begin{equation}\label{Eq: comp exist metric 4}
|t_2 - t_1| = \dd_X(\gamma(t_2), \gamma(t_1)) \leq \dd_X(\gamma(t_2), x)+ \dd_X(x, \gamma(t_1)) < 2(10^{-(k+1)}).
\end{equation}
It follows that
\[
\overline B_{10^{-k}}(\gamma(t_2)) \subset \overline B_{10^{-(k-1)}}(\gamma(t_1)).
\]
Applying Lemma~\ref{Lem: approx half ball consistent} again, we conclude that exactly one of the sets 
$\tilde H^k_\pm(t_2)$ is contained in $\hat H^{k-1}_+(t_1)$ and the other in $\hat H^{k-1}_-(t_1)$. Since $x \in \hat H^{k-1}_+(t_1)$ and $x \in H^{k+1}_-(t_2) \subset \tilde H^k_-(t_2)$, the sets $\hat H^{k-1}_+(t_1)$ and $\tilde H^k_-(t_2)$ intersect. Therefore, $\tilde H^k_-(t_2) \subset \hat H^{k-1}_+(t_1)$ and $\tilde H^k_+(t_2) \subset \hat H^{k-1}_-(t_1)$. 

It suffices now to prove that the sets
\begin{equation}\label{Eq: comp exist metric 2}
\overline B_{10^{-k}}(\gamma(t_1)) \setminus B_{10^{-(k+2)}}(\gamma([t_1-10^{-k}, t_1+10^{-k}]))
\end{equation}
and
\begin{equation}\label{Eq: comp exist metric 3}
\overline B_{10^{-k}}(\gamma(t_2)) \setminus B_{10^{-(k+2)}}(\gamma([t_2-10^{-k}, t_2+10^{-k}]))
\end{equation}
have non-empty intersection to obtain a contradiction. Indeed, if this is the case then any point $x'$ in the intersection must belong to $H^k_+(t_1) \cap H^k_+(t_2)$ or $H^k_-(t_1) \cap H^k_-(t_2)$ by the induction hypothesis. However, the first intersection is empty, since $H^k_+(t_1) \subset \hat H^{k-1}_+(t_1)$ and $H^k_+(t_2) \subset \hat H^{k-1}_-(t_1)$. Similarly, the second intersection is empty. This leads to a contradiction. 

To prove that the two sets in \eqref{Eq: comp exist metric 2} and \eqref{Eq: comp exist metric 3} intersect, we use the regularity of $\overline B_{10^{-(k-1)}}(\gamma(t_1))$. Let $(Z, \iota_1, \iota_2)$ be any $(10^{-(k+7)})$-regular realization of $(\overline B_{10^{-(k-1)}}(\gamma(t_1)), \gamma)$. In this realization, $\gamma(t_1)$ is approximated by $0^2$ and $\gamma(t_2)$ is approximated by $(t_2-t_1, 0)$. Moreover, $0^2$ and $(t_2-t_1,0)$ are close to each other, with $|t_2 - t_1| < 2(10)^{-(k+1)}$ by \eqref{Eq: comp exist metric 4}. A direct computation shows that a point in $\overline B_{10^{-(k-1)}}(\gamma(t_1))$ that approximates $((t_2-t_1)/2, 2(10)^{-(k+2)})$ under this realization must be contained in both of the sets in \eqref{Eq: comp exist metric 2} and \eqref{Eq: comp exist metric 3}. 
\end{proof}

We define the sets
\[
\overline A_\pm := \Bigg(\Bigg(\bigcup_{t \in [-2,2], k \in \N} H^k_\pm(t)\Bigg) \cup \gamma([-1,1]) \Bigg) \cap \overline B_1(\gamma(0)).
\]
\begin{Cla}\label{Cla: comp exist metric Cla 3}
The following hold:
\begin{enumerate}
    \item $\overline A_+ \cap \overline A_- = \gamma([-1,1])$;
    \item $\overline A_+ \cup \overline A_- = \overline B_1(\gamma(0))$;
    \item $\overline A_\pm$ are closed.
\end{enumerate}
\end{Cla}
\begin{proof}
Claim~(1) follows directly from Claim~\ref{Cla: comp exist metric Cla 2} and the definition of $A_\pm$. 

Next, we verify claim~(2). Let $x \in \overline B_1(\gamma(0))$. If $x \in \gamma([-1,1])$, then $x \in \overline A_+ \cup \overline A_-$ by the definition of $\overline A_\pm$. Therefore, assume that $x \notin \gamma([-1,1])$. Let $t_0 \in [-10,10]$ be such that $\gamma(t_0)$ is minimal in distance to $x$ among the points in $\gamma([-10,10])$. In particular, $\dd_X(x, \gamma(t_0)) \leq \dd_X(x, \gamma(0)) \leq 1$. By the triangle inequality, we have 
\[|t_0| = \dd_X(\gamma(t_0), \gamma(0)) \leq \dd_X(\gamma(t_0), x) + \dd_X(x, \gamma(0)) \leq 2,\]
and so $t_0 \in [-2,2]$. Since $x \notin \gamma([-1,1])$ and $\gamma([-10,10]) \cap \overline B_1(\gamma(0)) = \gamma([-1,1])$, we have $x \neq \gamma(t_0)$. In particular, $0 < \dd_X(x, \gamma(t_0)) \leq 1$. Therefore, we may choose $k \in \N$ such that $10^{-(k+2)} < \dd_X(x, \gamma(t_0)) \leq 10^{-k}$. Since $\gamma(t_0)$ is a distance minimizer to $x$ in $\gamma([-10,10])$ and $[t_0- 10^{-k}, t_0+10^{-k}] \subset [-10,10]$, we have
\[
\dd_X(x, \gamma([t_0- 10^{-k}, t_0+10^{-k}]))= \dd_X(x, \gamma(t_0)) > 10^{-(k+2)}. 
\]
It follows that $x \in \overline B_{10^{-k}}(\gamma(t_0)) \setminus B_{10^{-(k+2)}}(\gamma([t_0- 10^{-k}, t_0+10^{-k}]))$. By our inductive construction of $H^k_\pm(t_0)$, we have 
\[
\overline B_{10^{-k}}(\gamma(t_0)) \setminus B_{10^{-(k+2)}}(\gamma([t_0- 10^{-k}, t_0+10^{-k}])) = H^k_+(t_0) \sqcup H^k_-(t_0). 
\]
Since $H^k_\pm(t_0) \subset \overline A_\pm$, we have $x \in \overline A_+ \cup \overline A_-$.

Finally, we verify claim~(3). Let $(x_i)_{i \in \N}$ be a sequence converging to $x$ in $\overline B_1(\gamma(0))$. Suppose that $x_i \in \overline A_+$ for all $i \in \N$. We claim that $x \in \overline A_+$. 

If $x \in \gamma([-1,1])$ then we are finished, since $\gamma([-1,1]) \subset \overline A_+$. Therefore, assume that this is not the case. As before, let $t_0 \in [-10,10]$ be such that $\gamma(t_0)$ is minimal in distance to $x$ among the points in $\gamma([-10,10])$. We have shown that
$0 < \dd_X(x, \gamma(t_0)) \leq 1$ and $t_0 \in [-2,2]$.

There are two cases, depending on whether $\dd_X(x, \gamma(t_0))=1$.  

If $\dd_{X}(x, \gamma(t_0)) = 1$, then $\dd_{X}(x, \gamma(0)) \geq \dd_{X}(x, \gamma(t_0)) = 1$. 
Since $x \in \overline B_1(\gamma(0))$, we must have $\dd_X(x, \gamma(0))=1$ and so $\gamma(0)$ is a distance minimizer between $x$ and $\gamma([-10,10])$. It follows that $\dd_X(x, \gamma([-1,1])) = 1$, which implies that $\dd_X(x_i, \gamma([-1,1])) \to 1$. In particular, $x_i \in \overline B_{1}(\gamma(0)) \setminus B_{1/100}(\gamma([-1,1]))$ for sufficiently large $i$. By our construction, 
\[
\overline B_{1}(\gamma(0)) \setminus B_{1/100}(\gamma([-1,1])) = H^0_+(0) \sqcup H^0_-(0).
\]
Since $x_i \in \overline A_+$, it follows that $x_i \in H^0_+(0)$. Since $H^0_+(0)$ is closed by Lemma~\ref{Lem: approx half balls exist}, we see that $x \in H^0_+(0)$ and therefore belongs to $\overline A_+$.

On the other hand, if $\dd_{X}(x, \gamma(t_0)) < 1$, then we can find $k \in \N$ such that
\begin{equation}\label{Eq: comp exist metric 5}
10^{-(k+2)} < \dd_X(x, \gamma(t_0)) < 10^{-k}. 
\end{equation} 
As before, $x \in \overline B_{10^{-k}}(\gamma(t_0)) \setminus B_{10^{-(k+2)}}(\gamma([t_0- 10^{-k}, t_0+10^{-k}]))$ and so must belong to either $H^k_+(t_0)$ or $H^k_-(t_0)$. Assume for the sake of contradiction that $x \in H^k_-(t_0)$. Since $H^k_-(t_0)$ is open in $\overline B_{10^{-k}}(\gamma(t_0)) \setminus B_{10^{-(k+2)}}(\gamma([t_0- 10^{-k}, t_0+10^{-k}]))$ by Lemma~\ref{Lem: approx half balls exist}, we have 
\[
B_r(x) \cap \big(\overline B_{10^{-k}}(\gamma(t_0)) \setminus B_{10^{-(k+2)}}(\gamma([t_0- 10^{-k}, t_0+10^{-k}]))\big) \subset H^k_-(t_0)
\]
for all small $r > 0$. 
Since the inequalities in \eqref{Eq: comp exist metric 5} are strict, we have
\[
B_r(x) \subset \overline B_{10^{-k}}(\gamma(t_0)) \setminus B_{10^{-(k+2)}}(\gamma([t_0- 10^{-k}, t_0+10^{-k}]))
\]
for all small $r > 0$. Therefore, there exists $r > 0$ such that $B_r(x) \subset H^k_-(t_0)$. Since $x_i$ converges to $x$, it follows that $x_i \in B_r(x) \subset H^k_-(t_0)$ for sufficiently large $i$. On the other hand, $x_i \in \overline A_i$ by our assumption, and $\overline A_+ \cap H^k_-(t_0) = \emptyset$. This is a contradiction.
\end{proof}

Let $A_\pm := \overline A_\pm \cap \big(B_1(\gamma(0)) \setminus \gamma\big)$. By Claim~\ref{Cla: comp exist metric Cla 3}, the sets $A_\pm$ are disjoint, closed in $B_r(\gamma(0)) \setminus \gamma$, and satisfy $B_1(\gamma(0)) \setminus \gamma = A_+ \sqcup A_-$. It follows that they are both open. 

Next, we prove that $\gamma((-1,1))$ is contained in the topological boundaries $A_\pm$ in $B_1(\gamma(0))$. Let $t \in (-1,1)$. For all sufficiently large $k \in \N$, we have $\overline B_{10^{-k}}(\gamma(t)) \subset B_{1}(\gamma(0))$. It follows that $H^k_\pm(t) \subset A_+$. In particular, $A_\pm$ both intersect $\overline B_{10^{-k}}(\gamma(t))$, and so $\gamma(t)$ is contained in the topological boundaries of $A_\pm$ in $B_1(\gamma(0))$. We note that this argument also shows that $A_\pm$ are non-empty. The reverse inclusion is immediate since $A_\pm$ are clopen in $B_1(\gamma(0)) \setminus \gamma$.
\end{proof}

\begin{proof}[Proof of Proposition~\ref{Pro: comp exist}]
Let $\veps_0$ be the universal constant $\veps$ appearing in Lemma~\ref{Lem: comp exist metric}. 

Let $r, \veps> 0$ be such that 
\[
\dd_{GH}(\overline B_{100r}(\gamma(0)), \overline B_{100r}(0^2)) < \veps r.
\]
By Lemma~\ref{Lem: regular realization exist}, there exists a $(\delta(\veps) r)$-regular realization of $(\overline B_{100r}(\gamma(0)), \gamma)$, where $\delta(\veps) \to 0$ as $\veps \to 0$. It follows from Lemma~\ref{Lem: realization scale change} that for any $t \in [-2r,2r]$,
\begin{equation}\label{Eq: comp exist thm 1}
\dd_{GH}(\overline B_{10r}(\gamma(t)), \overline B_{10r}(0^2)) < 3\delta(\veps) r. 
\end{equation}
Lemma~\ref{Lem: R2 persistence} implies that if $\delta(\veps)$ is sufficiently small (depending only on $\veps_0$) and $r$ is sufficiently small (depending only on $N$ and $\veps_0$), then 
\[
\dd_{GH}(\overline B_{s}(\gamma(t)), \overline B_{s}(0^2)) < \veps_0 s,
\]
for any $s \leq r$. Finally, by requiring that $\delta(\veps) \leq 10\veps_0/3$, \eqref{Eq: comp exist thm 1} gives that
\[
\dd_{GH}(\overline B_{10r}(\gamma(0)), \overline B_{10r}(0^2)) < 10\veps_0 r.
\]

After rescaling $X$ by $1/r$ and reparameterize $\gamma$ to have unit speed accordingly, we see that the assumptions of Lemma~\ref{Lem: comp exist metric} are satisfied, and obtain the desired sets $A_\pm$. 

Throughout the proof, we only required $\delta(\veps)$ to be smaller than constants depending on $\veps_0$, and $r$ to be smaller than a constant only depending on $N$ and $\veps_0$. Therefore, the argument goes through for a universal constant $\veps$ and any $r \leq \bar r(N)$. 
\end{proof}

\subsection{Properties of components}\label{Subsec: properties of comps} In this subsection, we will prove that the clopen subsets $A_\pm \subset B_r(\gamma(0)) \setminus \gamma$ given by Proposition~\ref{Pro: comp exist} satisfy the additional properties in the definition of good components (Definition~\ref{Def: good comp}). This will allow us to finish the proof of Theorem~\ref{Thm: good components existence}.

The following simple lemma is a direct consequence of the connectedness of the interval.
\begin{Lem}\label{Lem: clopen path-connected}
Let $X$ be a topological space and let $A \subset X$ be non-empty and clopen. Let $x, y \in X$. If $x \in A$ and there exists a continuous curve from $x$ to $y$ in $X$, then $y \in A$. In particular, $A$ is the union of a collection of path-connected components of $X$. 
\end{Lem}

For the remainder of this subsection, assume that the hypotheses of Proposition~\ref{Pro: comp exist} hold, and let $A_\pm$ be the clopen subsets of $B_r(\gamma(0)) \setminus \gamma$ given by the proposition. More precisely, we require $A_\pm$ to be the sets obtained via the construction in Lemma~\ref{Lem: comp exist metric}.

\begin{Lem}[Key lemma]\label{Lem: geod no glance}
Let $x,y\in A_+\cap B_{r/10}(\gamma(0))$ or $x,y\in A_-\cap B_{r/10}(\gamma(0))$. Then every geodesic connecting $x$ and $y$ is disjoint from $\gamma$.
\end{Lem}

\begin{proof}
Let $x, y \in A_+\cap B_{r/10}(\gamma(0))$ (the argument for $A_-$ is identical), and let $\rho:[0,1] \to X$ be a geodesic between $x$ and $y$. Then $\rho \subset B_r(\gamma(0))$ by the triangle inequality. 

First, we show that $\rho$ intersects $\gamma$ at most once. Since $A_+ \subset B_r(\gamma(0)) \setminus \gamma$, we have $x, y \notin \gamma$. Therefore, if $\rho$ intersects $\gamma$ more than once, then it follows from the non-branching property (Theorem~\ref{Thm: non-branching}) that $\gamma$ is a subsegment of $\rho$. Since $\rho \subset B_r(\gamma(0))$ and $\gamma$ is not entirely contained in $B_{r}(\gamma(0))$ by the triangle inequality, this gives a contradiction. 

Next, we show that in fact $\rho$ is disjoint from $\gamma$. Assume for the sake of contradiction that $\rho$ intersects $\gamma$ at $\rho(t_0)$ for some $t_0 \in (0,1)$. Since this is the unique intersection, the curve $\rho \lvert_{[0, t_0)}$ is contained in $B_{r}(\gamma(0)) \setminus \gamma$. Since $\rho(0) = x \in A_+$ and $A_+$ is clopen in $B_r(\gamma(0))\setminus \gamma$, 
Lemma~\ref{Lem: clopen path-connected} implies that $\rho([0,t_0))\subset A_+$. Similarly, $\rho((t_0,1])\subset A_+$. 

Let $0 < \delta < \min(1-t_0, t_0)$ be a small constant to be fixed later. Since $\rho(t_0 -\delta) \in B_r(\gamma(0)) \setminus \gamma$, there exists $s_0 > 0$ (depending on $\delta$) such that for all $0 < s \leq s_0$, we have $B_{s}(\rho(t_0-\delta)) \subset B_r(\gamma(0)) \setminus \gamma$. Let $k \in \N$ be sufficiently large (to be fixed later) so that $s:=10^{-k}r \leq s_0$. Since $B_{s}(\rho(t_0-\delta))$ is path-connected and $\rho(t_0 - \delta) \in A_+$, Lemma~\ref{Lem: clopen path-connected} implies that $B_{s}(\rho(t_0-\delta)) \subset A_+$.

We claim that any geodesic from a point $z \in B_{s}(\rho(t_0-\delta))$ to $y$ is disjoint from $A_-$. Indeed, if this is not the case, then it is straightforward to check by Lemma~\ref{Lem: clopen path-connected} and the fact that $z, y \in A_+$ that such a geodesic must intersect $\gamma([-1,1])$ at a minimum of two points. Arguing as above using the non-branching property, this yields a contradiction. 

Let $t_0' \in [-r,r]$ be such that $\rho(t_0) = \gamma(t_0')$. Recall from the proof of Proposition~\ref{Pro: comp exist} that $\overline B_s(\gamma(t_0'))$ is Gromov--Hausdorff close to $\overline B_s(0^2)$, and this was used to decompose \[\overline B_s(\gamma(t_0')) \setminus B_{s/100}(\gamma([t_0' -s,t_0'+s]))\] into two approximate half-balls $H^k_\pm(t_0')$ (recall that $s = 10^{-k}r$) in Lemma~\ref{Lem: comp exist metric}. By the construction of $A_\pm$ in that lemma, we have $H^k_+ \subset A_+$ and $H^k_- \subset A_-$.

Let $c > 0$ be the speed of the geodesic $\rho$. Choose $\delta > 0$ sufficiently small so that $\dd_X(\rho(t_0-\delta), \rho(t_0)) = c\delta \ll \dd_X(\rho(t_0 - \delta), y)$. For each point $z \in B_{s}(\rho(t_0-\delta))$, choose a unit-speed geodesic $\gamma_{z ,y}: [0, \dd_X(z,y)] \to X$ from $z$ to $y$. Notice that in the case $z = \rho(t_0-\delta)$, the unique geodesic from $z$ to $y$ is the corresponding subsegment of $\rho$ by the non-branching property.

Define the map $\Psi: B_{s}(\rho(t_0-\delta)) \to X$ by $\Psi(z) := \gamma_{z, y}(c \delta)$. This is well defined for sufficiently small $s > 0$, since in that case $\dd_X(z, y) > c\delta$ for all $z \in B_{s}(\rho(t_0-\delta))$. It is immediate from the construction that $\Psi(\rho(t_0-\delta))=\rho(t_0)$.

By the proof of \cite[Theorem 5.10]{D25} (see conclusions (I) - (IV) on page 1107), for all sufficiently small $s > 0$ (depending on $t_0$, the speed $c$ of $\rho$, and $N$), the set $\Psi(B_{s}(\rho(t_0-\delta))) \cap B_s(\rho(t_0))$ is $(\veps s)$-dense in $B_s(\rho(t_0))$, where $\veps \to 0$ as $\delta \to 0$ (independent of $s$). 

On the other hand, we have already observed that any geodesic from a point $z \in B_{s}(\rho(t_0-\delta))$ to $y$ does not intersect $A_-$. Hence $\Psi(B_{s}(\rho(t_0-\delta))) \cap A_- = \emptyset$. In particular, $\Psi(B_{s}(\rho(t_0-\delta))) \cap H^k_-(t_0') = \emptyset$. This contradicts the conclusion that $\Psi(B_{s}(\rho(t_0-\delta)))$ is arbitrarily dense in $B_{s}(\rho(t_0))$ as $\delta \to 0$. More precisely, $\overline B_{s}(\rho(t_0))=\overline B_{s}(\gamma(t_0'))$ is Gromov--Hausdorff close to $\overline B_s(0^2)$ and $H^k_-(t_0')$ is an approximate half-ball. Therefore, it is impossible for any set that does not intersect $H^k_-(t_0')$ to be arbitrarily dense in $B_{s}(\rho(t_0))$.
\end{proof}

\begin{Pro}\label{Pro: comp path-connected}
The sets $A_\pm$ are the two components of $B_r(\gamma(0)) \setminus \gamma$. 
\end{Pro}

\begin{proof}
We first prove that $A_+$ is path-connected; the argument for $A_-$ is the same. Let $x_1, x_2 \in A_+$. We explicitly construct a continuous curve from $x_1$ to $x_2$ contained in $A_+$.

For $i= 1,2$, let $c_i: [0,1] \to X$ be a geodesic from $x_i$ to $\gamma(0)$. Since $x_i \in B_r(\gamma(0))$, the geodesic $c_i$ is contained in $B_r(\gamma(0))$.

We show that $c_i$ intersects $\gamma$ only at $\gamma(0)$. Indeed, since $x_i \notin \gamma$, it follows from the non-branching property that if $c_i$ intersects $\gamma$ more than once then one of $\gamma \lvert_{[0,1]}$ or $\gamma \lvert_{[-1,0]}$ is a subsegment of $c_i$. This yields a contradiction since $c_i \subset B_r(\gamma(0))$ and $\gamma(-1), \gamma(1) \notin B_{r}(\gamma(0))$.

Therefore, $c_i([0,9/10]) \subset B_r(\gamma(0)) \setminus \gamma$. Since $A_+$ is clopen in $B_r(\gamma(0)) \setminus \gamma$ and $c_i(0) = x_i \in A_+$, Lemma~\ref{Lem: clopen path-connected} implies that $c_i([0,9/10]) \subset A_+.$

Let $c_3:[0,1] \to X$ be a geodesic from $c_1(9/10)$ to $c_2(9/10)$. Since $c_1(1/2), c_2(1/2) \in A_+ \cap B_{r/10}(\gamma(0))$, it follows from Lemma~\ref{Lem: geod no glance} and the triangle inequality that $c_3 \subset B_r(\gamma(0)) \setminus \gamma$. The endpoints of $c_3$ are contained in $A_+$, so Lemma~\ref{Lem: clopen path-connected} implies $c_3 \subset A_+.$

Finally, the continuous curve in $A_+$ from $x_1$ to $x_2$ can be obtained by concatenating $c_1$ from $x_1$ to $c_1(9/10)$, $c_3$ from $c_1(9/10)$ to $c_2(9/10)$, and $c_2$ (reversed) from $c_2(9/10)$ to $x_2$. 

By Lemma~\ref{Lem: clopen path-connected}, each of $A_\pm$ is a union of components of $B_r(\gamma(0))\setminus \gamma$. Since $A_\pm$ are both path-connected, they must each be a single component of $B_r(\gamma(0))\setminus \gamma$. Moreover, $B_r(\gamma(0))\setminus \gamma = A_+ \sqcup A_-$, so it has exactly two components.
\end{proof}

\begin{proof}[Proof of Theorem~\ref{Thm: good components existence}]
Let $\veps_0, \bar r_0$ be the constants $\veps, \bar r(N) > 0$ appearing in Proposition~\ref{Pro: comp exist}. 

Let $\veps, r > 0$ be such that 
\[
\dd_{GH}(\overline B_{1000r}(\gamma(0)), \overline B_{1000 r}(0^2)) < \veps r. 
\]
We will bound $\veps$ by universal constants and $r$ by constants depending only on $N$.

If $\veps$ is sufficiently small (depending only on $\veps_0$), then
\[
\dd_{GH}(\overline B_{100r}(\gamma(0)), \overline B_{100r}(0^2)) < \veps_0 r. 
\]
If in addition $r \leq \bar r_0$, then Proposition~\ref{Pro: comp exist} gives two clopen subsets $A_\pm$ of $B_r(\gamma(0)) \setminus \gamma$.

By further requiring that $\veps < 10\veps_0$ and $r \leq \bar r_0/10$, we can apply Proposition~\ref{Pro: comp exist} at scale $10r$ to obtain two (auxiliary) clopen subsets $B_\pm$ of $B_{10r}(\gamma(0)) \setminus \gamma$.

By Proposition~\ref{Pro: comp exist}, Lemma~\ref{Lem: geod no glance}, and Proposition~\ref{Pro: comp path-connected}, the components $A_\pm$ satisfy all the claimed properties in Theorem~\ref{Thm: good components existence}, other than property~(3) in Definition~\ref{Def: good comp}.

By Proposition~\ref{Pro: comp path-connected}, $B_\pm$ are the components of $B_{10r}(\gamma(0)) \setminus \gamma$. Therefore, any curve in $B_{10r}(\gamma(0))$ joining two points $x_\pm \in B_\pm$ must intersect $\gamma$. Thus, in order to verify property~(3), it suffices to show that, up to permuting $\pm$, $A_\pm \subset B_\pm$. By Proposition~\ref{Pro: comp exist}, $\gamma(0)$ lies in the topological boundaries of $B_\pm$, hence $B_\pm$ both intersect $B_{r}(\gamma(0))$. Lemma~\ref{Lem: big set small set comp} below with $U = B_{10r}(\gamma(0)), V = B_{r}(\gamma(0))$, and $C = \gamma \cap B_{10r}(\gamma(0))$ thus implies the claim. 
\end{proof}

\begin{Lem}\label{Lem: big set small set comp}
Let $U$ be a topological space. Let $V \subset U$ be open and $C \subset U$ be closed. Assume that $U \setminus C$ and $V \setminus C$ have exactly two path-connected components $U_\pm$ and $V_\pm$, respectively. If $U_\pm$ both intersect $V$, then $V_\pm = U_\pm \cap V$, up to a permutation of $\pm$.
\end{Lem}
\begin{proof}
Since $V_\pm$ are path-connected and contained in $U \setminus C$, they must each be contained entirely in one of the components $U_\pm$. Since $U_+$ intersects $V$ and is disjoint from $C$, it must intersect at least one of $V_\pm$. Up to relabeling $V_\pm$, we can assume that $U_+$ intersects $V_+$. Then 
$V_+ \subset U_+$ and $V_+ \cap U_- = \emptyset$.  Similarly, $U_-$ must intersect at least one of $V_\pm$. Since $V_+$ is disjoint from $U_-$, it follows that $U_-$ intersects $V_-$. Therefore, $V_- \subset U_-$ and $V_- \cap U_+ = \emptyset$. 
\end{proof}

\subsection{Neighborhoods of geodesics}\label{Subsec: geod neigh} In this subsection, we use the good components obtained in Theorem~\ref{Thm: good components existence 2} to construct path-connected neighborhoods of geodesics that are disconnected into exactly two components by the geodesic. 

\begin{Lem}\label{Lem: good comp side induct}
Let $(X, \dd_X)$ be a geodesic space. Let $\gamma:[a,b] \to X$ be a geodesic, and let $t, t' \in [a,b]$ and $r, r' > 0$. Assume that the balls $B_r(\gamma(t))$ and $B_{r'}(\gamma(t'))$ each have two good components with respect to $\gamma$, denoted by $A_\pm$ and $A'_\pm$, respectively. If $B_r(\gamma(t)) \cap B_{r'}(\gamma(t')) \neq \emptyset$, then, up to relabeling $A_\pm'$, we have
\[
A_+ \cap A_+' \neq \emptyset, \quad A_- \cap A_-' \neq \emptyset, \quad A_+ \cap A_-' = \emptyset, \quad A_- \cap A_+' = \emptyset,
\]
\end{Lem}
\begin{proof}
Assume without loss of generality that $t \leq t'$ and $r \leq r'$. Since $B_r(\gamma(t)) \cap B_{r'}(\gamma(t')) \neq \emptyset$, the triangle inequality implies that $\dd_X(\gamma(t), \gamma(t')) < r + r'$. In particular, there exist $t_0 \in [t, t']$ and $r_0 > 0$ such that $B_{r_0}(\gamma(t_0)) \subset B_r(\gamma(t)) \cap B_{r'}(\gamma(t'))$.

Since $A_+$ is a good component of $B_r(\gamma(t)) \setminus \gamma$, the point $\gamma(t_0)$ lies in the topological boundary of $A_+$. In particular, $B_{r_0}(\gamma(t_0))$ intersects $A_+$. Fix a point 
\[x \in B_{r_0}(\gamma(t_0)) \cap A_+ \subset B_{r_0}(\gamma(t_0)) \setminus \gamma.\] 
Since $B_{r_0}(\gamma(t_0)) \setminus \gamma \subset B_{r'}(\gamma(t')) \setminus \gamma$,
the point $x$ lies in one of $A_\pm'$. Up to relabeling $A_\pm'$, we may assume that $x \in A_+'$. In particular, $A_+ \cap A_+' \neq \emptyset$

Next, we show that $A_+ \cap A_-' = \emptyset$. Suppose for the sake of contradiction that there exists a point $y \in A_+ \cap A_-'$. Since $x, y \in A_+$, they can be connected by a curve $\alpha$ contained in $A_+$. On the other hand, $\dd_X(\gamma(t), \gamma(t')) < r + r' \leq 2r'$, and so
\[
A_+ \subset B_r(\gamma(t))\setminus \gamma \subset B_{10r'}(\gamma(t')) \setminus \gamma.
\]
In particular, $\alpha$ is a curve connecting $x \in A_+'$ and $y \in A_-'$ that is contained entirely in $B_{10r'}(\gamma(t'))$ and does not intersect $\gamma$, contradicting property~(3) in Definition~\ref{Def: good comp}. 

Next, we show that $A_- \cap A_-' \neq \emptyset$. Arguing as above, we see that $B_{r_0}(\gamma(t_0))$ intersects $A_-$. Fix a point $z \in B_{r_0}(\gamma(t_0)) \cap A_-$. Then $z$ lies in one of $A_\pm'$. If $z \in A_+'$, then by applying the previous argument, with $z$ in place of $x$, we would obtain $A_- \cap A_-' = \emptyset$. Combined with $A_+ \cap A_-' = \emptyset$, this would imply that $A_-'$ and $B_{r_0}(\gamma(t_0))$ are disjoint, contradicting the fact that $\gamma(t_0)$ lies in the topological boundary of $A_-'$. Therefore, $z \in A_-'$ and $A_- \cap A_-' \neq \emptyset$.

Finally, using that $z \in A_- \cap A_-'$ and applying again the previous argument with $z$ in place of $x$, we obtain $A_- \cap A_+' = \emptyset$. This completes the proof. 
\end{proof}

The previous lemma shows that, once a labeling of the two good components of $B_r(\gamma(t))\setminus \gamma$ is fixed, there is a natural way to uniquely determine the labeling of the good components of any overlapping ball $B_{r'}(\gamma(t'))\setminus \gamma$. In particular, choosing which component corresponds to the ``$\pm$'' side for one ball induces a consistent choice of sides for all balls intersecting it. 

\begin{Lem}\label{Lem: geod neigh side choice}
Let $(X, \dd_X)$ be a geodesic space, and let $\gamma:[a,b] \to X$ be a geodesic. Let $I = (a,b), [a,b), (a,b]$, or $[a,b]$. Suppose that for every $t \in I$, there exists $r(t) > 0$ such that $B_{r(t)}(\gamma(t))$ has two good components with respect to $\gamma$. Then there exists a labeling of the good components of $B_{r(t)}(\gamma(t)) \setminus \gamma$ by $A_\pm(t)$ for all $t \in I$ such that the following hold:
\begin{enumerate}
    \item For all $t,t' \in I$, we have $A_+(t) \cap A_-(t') = \emptyset$ and $A_-(t) \cap A_+(t') = \emptyset$. 
    \item If $t, t' \in I$ are such that $B_{r(t)}(\gamma(t)) \cap B_{r(t')}(\gamma(t')) \neq \emptyset$, then $A_+(t) \cap A_+(t') \neq \emptyset$ and $A_-(t) \cap A_-(t') \neq \emptyset$.
\end{enumerate}
Moreover, this labeling is unique up to exchanging $A_\pm(t)$ for all $t \in I$.
\end{Lem}

\begin{proof}
We prove only the case $I = (a,b)$; the other cases are similar. For simplicity, assume that $a = -1$ and $b = 1$. Fix a labeling $A_\pm(0)$ for the two components of $B_{r(0)}(\gamma(0)) \setminus \gamma$.

Let $t \in (0,1)$. By the compactness of $[0,t]$, there exists a finite sequence $0 = t_0 < \dots < t_k = t$ such that $B_{r(t_i)}(\gamma(t_i)) \cap B_{r(t_{i+1})}(\gamma(t_{i+1})) \neq \emptyset$
for any $0 \leq i \leq k-1$. By passing to a subsequence, we may in addition assume that each ball $B_{r(t_i)}(\gamma(t_i))$ only intersects its predecessor $B_{r(t_{i-1})}(\gamma(t_{i-1}))$ and its successor $B_{r(t_{i+i})}(\gamma(t_{i+1}))$. We call any such sequence $(t_i)$ a \emph{good sequence} from $0$ to $t$. The notion is defined analogously for $t\in(-1,0)$.

Fix $t\in(-1,0)\cup(0,1)$ and let $(t_i)$ be a good sequence from $0$ to $t$. Using Lemma~\ref{Lem: good comp side induct} and the fixed labeling $A_\pm(0)$ at $t_0=0$ as the base step, we label the good components of each
$B_{r(t_i)}(\gamma(t_i))\setminus \gamma([-1,1])$ inductively as follows: we define $A_\pm(t_i)$ to be the unique good component intersecting $A_\pm(t_{i-1})$. In this way, we obtain a labeling $A_{\pm}(t)$ of the good components of $B_{r(t)}(\gamma(t))$. It is not difficult to check, using Lemma~\ref{Lem: good comp side induct} and an inductive argument, that this labeling is independent of the good sequence from $0$ to $t$. In this way, we obtain a labeling $A_\pm(t)$ for every $t\in(-1,1)$.  Another induction argument using Lemma~\ref{Lem: good comp side induct} shows that this labeling satisfies the conclusions of the lemma.

Finally, we prove uniqueness. It suffices to show that any labeling $A_\pm(\cdot)$ satisfying the conclusions of the lemma is completely determined by the labeling $A_\pm(0)$. Let $A'_\pm(\cdot)$ be another labeling satisfying the conclusions of the lemma, and assume that $A'_\pm(0)=A_\pm(0)$. Let $t\in(-1,0)\cup(0,1)$ be arbitrary and let $(t_i)$ be a good sequence from $0$ to $t$. By induction on $i$ and the fact that both labelings satisfy the conclusions of the lemma, we obtain $A'_\pm(t_i)=A_\pm(t_i)$ for all $i$. In particular, $A'_\pm(t)=A_\pm(t)$. 
\end{proof}

%Before giving the proof, let us discuss the key difference compared to the previous proposition. As before, we would like to cover $\gamma((0,1))$ with small balls, each of which is contained in $V$ and has two good components with respect to $\gamma$. There are two issues:
%\begin{enumerate}
   % \item $\gamma(0)$ and $\gamma(1)$ are not regular, so we cannot use a compactness argument to obtain a uniform radius $r$ so that every ball of radius $r$ centered at $\gamma(t)$, $t \in (0,1)$ is $(\veps r)$-regular. In particular, we cannot obtain a uniform radius $r$ so that every ball of radius $r$ along $\gamma$ has two good components with respect to $\gamma$.  
   % \item $V$ is a neighborhood of $\gamma((0,1))$ instead of $\gamma([0,1])$. This means that for the balls in our cover whose center is close to one of the endpoints $\gamma(0)$ and $\gamma(1)$, the radius may have to be very small so that the ball is contained in $V$. 
%\end{enumerate}
%On the other hand, for any $0 < a < b < 1$, the set $\gamma((a,b))$ can still by covered by finitely many balls as in the proof of the previous proposition. Therefore, we can deal with both issues by first covering $\gamma((1/4,3/4))$ with balls of uniform radius with the desired properties, and then adding more (possibly smaller) balls to this collection to cover $\gamma((1/8, 7/8))$ and so on. This yields a countable collection of balls whose union gives the desired neighborhood $U$ of $\gamma((0,1))$. 

\begin{Pro}\label{Pro: geod neigh}
Let $\gamma:[0,1] \to X$ be a regular geodesic. Then for any open neighborhood $V$ of $\gamma((0,1))$, there exists an open path-connected neighborhood $U\subset V$ of $\gamma((0,1))$ such that $U \setminus \gamma$ has exactly two components $U_\pm$. Moreover, the topological boundaries of $U_\pm$ in $U$ coincide with $\gamma((0,1))$.
\end{Pro}

\begin{proof}
By Theorem~\ref{Thm: good components existence 2}, for each $t \in (0,1)$, there exists $r(t) > 0$ such that $B_{r(t)}(\gamma(t))$ has two good components with respect to $\gamma$. Since $r(t)$ can be chosen arbitrarily small for each $t$, we may further assume that $B_{r(t)}(\gamma(t)) \subset V$.

Label the components of $B_{r(t)}(\gamma(t)) \setminus \gamma$ for $t \in (0,1)$ according to Lemma~\ref{Lem: geod neigh side choice} and let
\[
U := \bigcup_{t \in (0,1)} B_{r(t)}(\gamma(t)) \subset V \quad \text{and} \quad U_\pm := \bigcup_{t \in (0,1)} A_\pm(t).
\]
Clearly, $U$ is path-connected, since any $x \in U$ can be connected to $\gamma((0,1))$ by a curve in $U$. By property (1) of Lemma~\ref{Lem: geod neigh side choice}, the sets $U_\pm$ are disjoint. Since
$B_{r(t)}(\gamma(t)) \setminus \gamma = A_+(t) \sqcup A_-(t)$ for all $t \in (0,1)$, it follows that $U \setminus \gamma = U_+ \sqcup U_-$. The sets $U_\pm$ are both open and hence they must also closed in $U \setminus \gamma$, being complements of each other. By Lemma~\ref{Lem: clopen path-connected}, it follows that $U_\pm$ are each a union of components of $U$. On the other hand, since the union of two path-connected sets with non-empty intersection is path-connected, it is not difficult to use property (2) of Lemma~\ref{Lem: geod neigh side choice} and an inductive argument to show that both $U_\pm$ are path-connected. Therefore, $U_\pm$ are precisely the two components of $U$.

It remains to show that the topological boundaries of $U_\pm$ in $U$ coincide with $\gamma((0,1))$. It follows directly from the definition of good components and the construction of $U$ that $\gamma((0,1))$ is contained in the topological boundaries of $U_\pm$ in $U$. Since $U_\pm$ are clopen in $U \setminus \gamma$, the reverse inclusion is immediate.  
\end{proof}

\section{Corners in spaces of essential dimension two}\label{Sec: corner}
In this section, we introduce the notion of a corner, which is a curve obtained by concatenating two geodesics. We will show that any corner formed by concatenating two extendible geodesics disconnects all small balls centered at the point of intersection. This is a step towards proving a Jordan-type theorem for a class of simple loops obtained by concatenating geodesics, asserting that such loops disconnect the space $X$ into exactly two components. 

\begin{Def}\label{Def: corner}
Let $(X, \dd_X, \mm_X)$ be an $\RCD(K,N)$ space of essential dimension $2$. A continuous curve $\alpha:[a,b] \to X$ is called a \emph{corner} if there exists $c \in (a,b)$ such that
\begin{enumerate}
    \item the restrictions $\alpha \lvert_{[a,c]}$ and $\alpha \lvert_{[c,b]}$ are regular geodesics;
    \item for every $\veps > 0$, the restriction $\alpha \lvert_{[c-\veps, c+\veps]}$ is not a geodesic;
    \item $\alpha([a,c]) \cap \alpha([c,b]) = \{\alpha(c)\}$;
\end{enumerate}
The point $p := \alpha(c)$ is called the \emph{vertex} of $\alpha$. We call the two geodesics $\alpha \lvert_{[a,c]}$ and $\alpha \lvert_{[c,b]}$ (as well as their images) the \emph{edges} of $\alpha$. If the two edges of $\alpha$ are extendible past $p$, then we call $\alpha$ an \emph{extendible corner}.
\end{Def}
Condition~(2) ensures that the decomposition of a corner $\alpha$ into two edges is unique. As a convention, we always parameterize a corner so that its two edges have the same speed $c$. In this case, we call $c$ the \emph{speed} of $\alpha$. 

\subsection{Components of balls at extendible corners} In this subsection, we extend the notion of good components to corners and show that every extendible corner disconnects sufficiently small balls centered at its vertex into two good components. 

\begin{Def}[Good components at corners]\label{Def: good comp corner}
Let $(X, \dd_X, \mm_X)$ be an $\RCD(K,N)$ space of essential dimension $2$, and let $\alpha:[a,b] \to X$ be a corner. Let $p := \alpha(c)$ be the vertex of $\alpha$. For any $r > 0$, we say that $B_r(p)$ has \emph{two good components with respect to $\alpha$} if there exist two disjoint components $A_\pm$ of $B_r(p) \setminus \alpha$ such that the following hold:
\begin{enumerate}
    \item $A_\pm$ are clopen in $B_r(p) \setminus \alpha$ and $B_r(p) \setminus \alpha = A_+ \sqcup A_-$. 
    \item The topological boundaries of $A_\pm$ in $B_r(p)$ coincide with $\alpha \cap B_r(p).$
    \item Every continuous curve contained in $B_{10r}(p)$ connecting a point $x \in A_+$ to a point $y \in A_-$ intersects $\alpha$.
\end{enumerate}
In this case, the two components $A_\pm$ are called \emph{good components} of $B_r(p) \setminus \alpha$.
\end{Def}
 
\begin{Thm}[cf. Theorem~\ref{Thm: good components existence 2}]\label{Thm: good components existence corner}
Let $(X, \dd_X, \mm_X)$ be an $\RCD(K,N)$ space of essential dimension $2$, and let $\alpha:[-1,1] \to X$ be an extendible corner with vertex $p := \alpha(c)$. Then there exists $\bar r > 0$ such that for every $0 < r \leq \bar r$, the ball $B_r(p)$ has two good components with respect to $\alpha$. 
\end{Thm}

\begin{proof}
Up to rescaling, reparameterizing, and shortening the edges of $\alpha$, we may assume that we are in the following situation:
\begin{enumerate}
    \item $\alpha:[-1,1] \to X$ is a unit-speed extendible corner with vertex $p := \alpha(0)$.
    \item There exist two unit-speed regular geodesics $\gamma_1, \gamma_2:[-1,1] \to X$ with $\gamma_1 \neq \gamma_2$ and
     \[
    \alpha(t)=
    \begin{cases}
        \gamma_1(t), & t\in[-1,0],\\
        \gamma_2(t), & t\in[0,1].
    \end{cases}
    \]
\end{enumerate}
Clearly, this does not change the set $B_r(p) \setminus \alpha$ for $r$ sufficiently small.

By Theorem~\ref{Thm: good components existence 2}, there exists $\bar r>0$ such that for every $0<r\le 10\bar r$ and each $i\in\{1,2\}$, the ball $B_r(p)$ has two good components $A_\pm$ with respect to $\gamma_i$. Fix $0<r\le \bar r$, and for each $i\in\{1,2\}$, let $A_i^\pm$ denote the two good components of $B_r(p)\setminus \gamma_i$ as above.

Since $\gamma_1$ and $\gamma_2$ are unit-speed, extendible, intersect at $p$, and satisfy $\gamma_1 \neq \gamma_2$, the non-branching property (Theorem~\ref{Thm: non-branching}) implies that they intersect only at $p$. In particular, $\gamma_2((0,r)) \subset B_r(p) \setminus \gamma_1.$ This implies that $\gamma_2((0,r))$ is contained entirely in one of the components $A_1^\pm$. We fix the labeling of $A_1^\pm$ so that 
$\gamma_2((0,r)) \subset A_1^+$. Similarly, we fix the labeling of $A_2^\pm$ so that $\gamma_1((-r,0)) \subset A_2^+$.

We claim that $\gamma_1((0,r)) \subset A_2^-$ and $\gamma_2((-r,0)) \subset A_1^-$. We prove the first inclusion; the second follows by symmetry. Suppose for the sake of contradiction that $\gamma_1((0,r)) \subset A_2^-$ does not hold. Since $\gamma_1((0,r))$ is path-connected and contained in $B_r(p) \setminus \gamma_2$, it must be contained entirely in the other component. In other words, $\gamma_1((0,r)) \subset A_2^+$.
In particular, both $\gamma_1(-r/10)$ and $\gamma_1(r/10)$ lie in $A_2^+$. Therefore, the geodesic $\gamma_1 \lvert_{[-r/10, r/10]}$ joins two points in $A_2^+ \cap B_{r/10}(\gamma_2(0))$ and intersects $\gamma_2$ at $p = \gamma_2(0)$. This contradicts Theorem~\ref{Thm: good components existence 2}.

Let 
\[
A_+ := A^+_1 \cap A^+_2 \quad \text{and} \quad A_- := A^-_1 \cup A^-_2.
\]
We will show that $A_\pm$ are two good components of $B_r(p)$ with respect to $\alpha$.

First, we verify property~(1) in Definition~\ref{Def: good comp corner}. Clearly, $A_\pm$ are both open. Since $B_r(p) \setminus \alpha = A_+ \sqcup A_-$, the sets $A_\pm$ are also closed in $B_r(p) \setminus \alpha$.

Next, we check that $A_\pm$ are non-empty. The set $A_-$ is clearly non-empty so we focus on $A_+$. Since $\gamma_2(r/2) \in A_1^+$ and $A_1^+$ is open, there exists $s > 0$ such that $B_s(\gamma_2(r/2)) \subset A_1^+$. On the other hand, $\gamma_2(r/2)$ lies in the topological boundary of $A_2^+$. In particular, $B_s(\gamma_2(r/2))$ intersects $A^2_+$. Therefore, $A_+ = A^+_1 \cap A^+_2$ is non-empty.

Next, we check that $A_\pm$ are path-connected. We begin with $A_+$. Let $x, y \in A_+ = A_1^+ \cap A_2^+$. The curves from $x$ to $y$ constructed in Proposition~\ref{Pro: comp path-connected} to show the path-connectedness of $A_1^+$ and $A_2^+$ are the same, and is thus contained in $A_+$. Therefore, $A_+$ is path-connected. 

We now turn to $A_-$. Since the union of two path-connected sets with non-empty intersection is path-connected, it suffices to check that $A_1^- \cap A_2^- \neq \emptyset.$ This follows from the same argument used to show that $A_+ = A_1^+ \cap A_2^+$ is non-empty. 

It follows now from Lemma~\ref{Lem: clopen path-connected} that $A_\pm$ are exactly the two components of $B_r(p) \setminus \alpha$.  

Next, we verify property (2) in Definition~\ref{Def: good comp corner}. It is immediate that $\alpha((-r,r))$ is contained in the topological boundary of $A_-$ so we focus on $A_+$. Let $t \in (0,r)$, so that $\alpha(t) = \gamma_2(t)$. By repeating the same argument as when we checked that $A_+$ is non-empty, we obtain
\[
A_+ \cap B_s(\gamma_2(t)) = (A^+_1 \cap A^+_2) \cap B_s(\gamma_2(t)) \neq \emptyset
\]
for all small $s>0$. This shows that $\gamma_2(t)$ lies in the topological boundary of $A_+$. The same holds for $t \in (-r,0)$. Therefore, $\alpha((-r,r))$ is contained in the topological boundary of $A_+$.

We have shown that $\alpha((-r,r))$ is contained in the topological boundaries of $A_\pm$ in $B_r(p)$. The reverse inclusion is immediate since $A_\pm$ are clopen in $B_r(p) \setminus \alpha$.

Finally, by our choice of $r$ and applying our previous arguments at the scale $10r$, we see that $B_{10r}(p) \setminus \alpha$ also has two components $B_\pm$, verifying all the properties of good components except property~(3). Arguing as in the proof of Theorem~\ref{Thm: good components existence} using $B_\pm$ as auxiliary sets, we obtain property~(3) in Definition~\ref{Def: good comp corner} for $A_\pm$.  
\end{proof}

\subsection{Neighborhoods of extendible corners}\label{Subsec: corner neigh}
In this subsection, we use the good components obtained in Theorems~\ref{Thm: good components existence 2} and \ref{Thm: good components existence corner} to construct path-connected neighborhoods of extendible corners that are disconnected into exactly two components by the corner. The construction follows the same ideas as that of Subsection~\ref{Subsec: geod neigh}.

For the remainder of this subsection, let $(X, \dd_X, \mm_X)$ be an $\RCD(K,N)$ space of essential dimension $2$.

We first outline two lemmas, which are analogous to Lemmas~\ref{Lem: good comp side induct} and \ref{Lem: geod neigh side choice} in Subsection~\ref{Subsec: geod neigh}. The proofs are essentially the same, so we will omit them.

\begin{Lem}\label{Lem: good comp corner side induct}
Let $\alpha:[-1,1] \to X$ be a corner with vertex $\alpha(c)$. Let $r > 0$ and assume that $B_r(\alpha(c))$ has two good components with respect to $\alpha$, denoted by $A_\pm$. Moreover, let $t \in (c,1]$ and $s > 0$, and assume that the ball $B_{10s}(\alpha(t))$ does not intersect the geodesic $\alpha \lvert_{[-1,c]}$ and that $B_{s}(\alpha(t))$ has two good components with respect to the geodesic $\alpha \lvert_{[c,1]}$, denoted by $A'_\pm$. If $B_r(\alpha(c)) \cap B_{s}(\alpha(t)) \neq \emptyset$, then, up to relabeling $A_\pm'$, 
\[
A_+ \cap A_+' \neq \emptyset, \quad A_- \cap A_-' \neq \emptyset, \quad A_+ \cap A_-' = \emptyset, \quad A_- \cap A_+' = \emptyset,
\]
\end{Lem}

\begin{Lem}\label{Lem: corner neigh side choice}
Let $\alpha:[-1,1] \to X$ be a corner with vertex $\alpha(c)$. Suppose that for every $t \in (-1,1)$, there exists $r(t) > 0$ such that the following hold:
\begin{enumerate}
    \item[(i)] If $t = c$, then $B_{r(t)}(\alpha(t))$ has two good components with respect to $\alpha$.
    \item[(ii)] If $t \in (-1,c)$, then $B_{10r(t)}(\alpha(t))$ does not intersect the geodesic $\alpha \lvert_{[c,1]}$ and $B_{r(t)}(\alpha(t))$ has two good components with respect to the geodesic $\alpha \lvert_{[-1,c]}$.
    \item[(iii)] If $t \in (c,1)$, then $B_{10r(t)}(\alpha(t))$ does not intersect the geodesic $\alpha \lvert_{[-1,c]}$ and $B_{r(t)}(\alpha(t))$ has two good components with respect to the geodesic $\alpha \lvert_{[c,1]}$.
\end{enumerate}
Then there exists a labeling of the good components of $B_{r(t)}(\alpha(t))$ (with respect to $\alpha$, $\alpha \lvert_{[-1,c]}$, or $\alpha \lvert_{[c,1]}$, depending on whether $t = c$, $t \in (-1,c)$, or $t \in (c,1)$, respectively) by $A_\pm(t)$ for all $t \in (-1,1)$ such that the following hold:
\begin{enumerate}
    \item For all $t,t' \in (-1,1)$, we have $A_+(t) \cap A_-(t') = \emptyset$ and $A_-(t) \cap A_+(t') = \emptyset$.
    \item If $t, t' \in (-1,1)$ are such that $B_{r(t)}(\alpha(t)) \cap B_{r(t')}(\alpha(t')) \neq \emptyset$, then $ A_+(t) \cap A_+(t') \neq \emptyset$ and $A_-(t) \cap A_-(t') \neq \emptyset$.
\end{enumerate}
Moreover, this labeling is unique up to exchanging $A_\pm(t)$ for all $t \in (-1,1)$.
\end{Lem}
%We note that the first part (ii) and (iii) above ensures that if $t \in (a,c)$ and $t' \in (b,c)$, then
%\[
%B_{r(t)}(\alpha(t)) \cap B_{r(t')}(\alpha(t')) = \emptyset
%\]

We can now construct a path-connected neighborhood of the interior of an extendible corner which is disconnected into two components by the corner, using the same ideas as in the proof of Proposition~\ref{Pro: geod neigh}.
\begin{Pro}\label{Pro: corner neigh}
Let $\alpha:[-1,1] \to X$ be an extendible corner. Then for any open neighborhood $V$ of $\alpha((-1,1))$, there exists an open path-connected neighborhood $U \subset V$ of $\alpha((-1,1))$ such that $U \setminus \alpha$ has exactly two components $U_\pm$. Moreover, the topological boundaries of $U_\pm$ in $U$ coincide with $\alpha((-1,1))$.
\end{Pro}
\begin{proof}
Let $\alpha(c)$ be the vertex of $\alpha$. Since $\alpha$ is extendible, there exists $r(c) > 0$ such that $B_{r(c)}(\alpha(c))$ has two good components with respect to $\alpha$ by Theorem~\ref{Thm: good components existence corner}.

Let $t \in (c,1)$. By Theorem~\ref{Thm: good components existence 2}, there exists $r(t) > 0$ such that $B_{r(t)}(\alpha(t))$ has two good components with respect to the regular geodesic $\alpha \lvert_{[c,1]}$. Since $r(t)$ can be chosen arbitrarily small for each $t$, we may further assume that 
\[
B_{r(t)}(\alpha(t)) \subset V \quad \text{and} \quad
B_{10r(t)}(\alpha(t)) \cap \alpha([-1,c]) = \emptyset. 
\]
We make an analogous choice of $r(t)$ for each $t \in (-1,c)$.

Let $A_\pm(t)$ be a labelling of the good components given by Lemma~\ref{Lem: corner neigh side choice} and set
\[
U := \bigcup_{t \in (-1,1)} B_{r(t)}(\alpha(t)) \subset V \quad \text{and} \quad U_\pm := \bigcup_{t \in (-1,1)} A_\pm(t).
\]
The same argument as in the proof of Proposition~\ref{Pro: geod neigh} shows that $U$ and $U_\pm$ satisfy all the conclusions of the present proposition. 
\end{proof}

\subsection{Mayer--Vietoris argument}\label{Subsec: MV extend corner}
In Theorem~\ref{Thm: good components existence corner}, we proved that an extendible corner $\alpha$ disconnects all sufficiently small balls of radius $r \leq \bar r$ centered at its vertex $p$ into two components, for some $\bar r > 0$. One can trace through the proof to see that $\bar r$ can be chosen depending on the regularity at $p$ and on the quantitative extendability of the corner. In particular, if the edges of $\alpha$ can only be extended a small distance, then $\bar r$ must be correspondingly small. In this subsection, we improve our control on $\bar r$ and relate it to the semi-locally simply connected property (Theorem~\ref{Thm: semi-locally simply connected}) of $X$ at $p$. This will be achieved by a Mayer--Vietoris argument using the path-connected neighborhoods of $\alpha$ obtained in Subsection~\ref{Subsec: corner neigh}. Since such a quantity is more topologically stable, this will allow us to approximate arbitrary corners by extendible corners in Section~\ref{Sec: ext-int points} to show that they must also disconnect sufficiently small ball at their vertices.

\begin{Pro}\label{Pro: comp exist extend corner quant}
Let $(X, \dd_X, \mm_X)$ be an $\RCD(K,N)$ space of essential dimension $2$, and let $\alpha:[-1,1] \to X$ be an extendible corner with vertex $p$. Let $R > 0$ and assume that the two edges of $\alpha$ have length at least $R$. Let $0 < r \leq R$ be such that the inclusion $B_r(p) \hookrightarrow B_R(p)$ induces the trivial map on $\pi_1$. Then $B_r(p) \setminus \alpha$ has exactly two components $A_\pm$. Moreover, the topological boundaries of $A_\pm$ in $B_r(p)$ coincide with $\alpha \cap B_r(p)$.
\end{Pro}

\begin{proof}
By shortening the edges of $\alpha$ and reparameterizing, we may assume that both edges have length exactly $R$ (so $\alpha(0) = p$); this does not change the set $B_r(p) \setminus \alpha.$ By Proposition~\ref{Pro: corner neigh}, there exists a path-connected neighborhood $U \subset B_R(p)$ of $\alpha((-1,1))$ such that the set $U \setminus \alpha$ has exactly two components, denoted by $U_\pm$.
 
Let $\tilde \alpha$ denote the constant-speed reparameterization of $\alpha \lvert_{[-r/R, r/R]}$ over $[-1,1]$. Then $\tilde \alpha((-1,1)) = \alpha \cap B_r(p)$, and $U \cap B_r(p)$ is an open neighborhood of $\tilde \alpha((-1,1))$. Applying Proposition~\ref{Pro: corner neigh} to $\tilde \alpha$, we obtain a path-connected neighborhood $\tilde U \subset U \cap B_r(p)$ of $\tilde \alpha((-1,1))$ such that the set $\tilde U \setminus \tilde \alpha = \tilde U \setminus \alpha$ has exactly two components, denoted by $\tilde U_\pm$.

Since $\tilde \alpha(0) = p$ lies in the topological boundaries of $U_\pm$, both sets must intersect $\tilde U$. By Lemma~\ref{Lem: big set small set comp}, up to permuting the labels $\tilde U_\pm$, we have $\tilde U_\pm \subset U_\pm$.

Fix $x_\pm \in \tilde U_\pm$. We claim that there does not exist a continuous curve from $x_+$ to $x_-$ in $B_r(p) \setminus \alpha$. Suppose for the sake of contradiction that such a curve $\beta_1$ exists. Since $\tilde U$ is path-connected, there exists a continuous curve $\beta_2$ from $x_-$ to $x_+$ contained in $\tilde U$. Since $\tilde U_\pm$ are the components of $\tilde U \setminus \alpha$, the curve $\beta_2$ must intersect $\alpha$. 

Let $\beta$ be the loop obtained by concatenating $\beta_1$ and $\beta_2$. Then $\beta$ is a loop contained in $B_r(p)$, and is therefore contractible in $B_R(p)$ by our assumption. In particular, $\beta$ represents the trivial class in $H_1(B_R(p);\Z)$ via the Hurewicz homomorphism.

Let $V := B_{R}(p) \setminus \alpha$. Then
\[
U \cup V = B_{R}(p) \quad \text{and} \quad U\cap V = U \setminus \alpha([-1,1]) = U_+ \sqcup U_-.
\]
Consider Mayer--Vietoris sequence in singular homology associated to the cover $B_{R}(p) = U \cup V$:
\[
  \dots \longrightarrow H_1(B_{R}(p); \Z) {\longrightarrow} H_0(U \cap V; \Z) {\longrightarrow} H_0(U; \Z)\oplus H_0(V; \Z)\longrightarrow H_0(B_{R}(p); \Z)\longrightarrow 0
\]
with associated boundary map \[\partial_*: H_1(B_{R}(p); \Z) \to H_0(U \cap V; \Z).\] Clearly, since $[\beta]=0$ in $H_1(B_R(p);\Z)$, we have $\partial_*([\beta]) = 0$. On the other hand, writing $\beta = \beta_1 + \beta_2$ as $1$-chains with $\beta_1 \in C_1(V)$ and $\beta_2 \in C_1(U)$, \[\partial_*([\beta]) = [\partial \beta_1] = [x_+] - [x_-],\]
up to a sign convention. This implies that $[x_+] = [x_-]$ in $H_0(U \cap V; \Z)$, i.e., $x_\pm$ lie in the same component of the set $U \cap V = U \setminus \alpha$, contradicting our assumption that $x_\pm \in \tilde U_\pm \subset U_\pm$. Therefore, $x_\pm$ cannot be connected by a continuous curve in $B_r(p) \setminus \alpha$, and we conclude that $B_r(p) \setminus \alpha$ has at least two components. 

Next, we show that $B_{r}(p) \setminus \alpha$ has at most two components. Let $x \in B_r(p) \setminus \alpha$ and let $\rho:[0,1] \to X$ be a geodesic from $x$ to $p$.

We claim that $\rho$ intersects $\alpha$ exactly once, namely at $p$. Suppose that this is not the case. Since $x \notin \alpha$, it then follows from the non-branching property (Theorem~\ref{Thm: non-branching}) that one of the edges of $\alpha$ must be a subsegment of $\rho$. This yields a contradiction, since $\rho \subset B_r(p)$ but $\alpha(1), \alpha(-1) \notin B_{r}(p)$.

Therefore, $\rho([0,1)) \subset B_r(p) \setminus \alpha.$ By continuity, there exists $\veps > 0$ such that $\rho(1-\veps)$ lies in $\tilde U$, and hence in one of $\tilde U_\pm$. The restriction $\rho \lvert_{[0,1-\veps]}$ is then a continuous curve from $x$ to one of $\tilde U_\pm$ contained in $B_{r}(p) \setminus \alpha$. Since $\tilde U_\pm$ are both path-connected and $x$ is arbitrary, it follows that $B_{r}(p) \setminus \alpha$ has at most two components. 

Let $A_\pm$ be the two components of $B_{r}(p) \setminus \alpha$. Up to permuting the labeling $A_\pm$, it follows from the arguments above that $\tilde U_\pm \subset A_\pm$. Since the topological boundaries of $\tilde U_\pm$ in $B_r(p)$ both contain $\tilde \alpha((-1,1))$, the topological boundaries of $A_\pm$ in $B_r(p)$ must also contain $\tilde \alpha((-1,1))$. The reverse inclusions follow since $A_\pm$ are clopen in $B_r(p) \setminus \alpha$. 
\end{proof}

\section{Interior and extremal points}\label{Sec: ext-int points}

In this section, we define interior and extremal points. When we prove our manifold theorem in Section~\ref{Sec: man struct}, we will see that the set of interior points correspond to the interior of the manifold, and the set of extremal points correspond to the boundary of the manifold. 

Let $(X, \dd_X, \mm_X)$ be an $\RCD(K,N)$ space of essential dimension $2$.

\begin{Def}[Extremal and interior points]
\label{Def: ext-int points}
A point $p \in X$ is an \emph{extremal point} if there exists a regular geodesic $\gamma:[0,1] \to X$ such that $\gamma(0) = p$ and $B_r(p) \setminus \gamma$
is not path-connected for some $0< r \leq R$, where $R$ is the length of $\gamma$. Any point $p \in X$ that is not extremal is an \emph{interior point}. The set of extremal points is called the \emph{extremal set} and is denoted by $\Ext(X)$, and the set of interior points is called the \emph{interior set} and is denoted by $\Int(X)$. 
\end{Def}

The following is the main result of this section, and will be proved in Subsection~\ref{Subsec: ext set closed}
\begin{Thm}\label{Thm: int set open}
$\Int(X)$ is open and dense in $X$, and $\Ext(X)$ is closed in $X$. 
\end{Thm}

For the remainder of this subsection, let $p, q \in X$ with $R := \dd_X(p,q)$, and let $\gamma \colon [0,1] \to X$ be a geodesic from $p$ to $q$. We establish several lemmas concerning $B_R(p) \setminus \gamma$ that will be used to study interior and extremal sets later in the section.

\begin{Lem}\label{Lem: can always connect to point on geod}
For every $x \in B_{R}(p) \setminus \gamma$ and $t \in [0,1)$, there exists a continuous curve $\alpha:[0,1] \to X$ from $x$ to $\gamma(t)$ such that $\alpha([0,1)) \subset B_{R}(p) \setminus \gamma$.
\end{Lem}

\begin{proof}
We prove the lemma only for $t = 1/2$; the proof can be adapted to arbitrary $t \in [0,1)$.

Let $x \in B_{R}(p) \setminus \gamma$, and let $\rho:[0,1] \to X$ be a geodesic from $x$ to $p$. We claim that $\rho$ intersects $\gamma$ only at $\rho(1) = \gamma(0) = p$. The two geodesics $\rho$ and $\gamma$ share a common endpoint in $p$. Moreover, the other endpoint of $\gamma$, which is $q$, is further away from $p$ than the other endpoint of $\rho$, which is $x$. Thus if $\rho$ intersects $\gamma$ at a point other than $p$, it would follow from the non-branching property (Theorem~\ref{Thm: non-branching}) that $\rho$ is a subsegment of $\gamma$. This contradicts the assumption that $x \notin \gamma$. Therefore, $\rho([0,1)) \subset B_R(p) \setminus \gamma.$ 

Set $y := \rho(1/2)$. In particular, $y \notin \gamma$. Let $\sigma:[0,1] \to X$ be a geodesic from $y$ to $\gamma(1/2)$. By the triangle inequality, $\sigma \subset B_R(p)$. 

\noindent \textbf{Case 1:} Suppose that $\sigma$ does not intersect $\gamma$ except at $\sigma(1) = \gamma(1/2)$. Then $\sigma([0,1)) \subset B_R(p) \setminus \gamma.$ In this case, the desired curve $\alpha$ can be taken as the concatenation of $\rho$ (from $x = \rho(0)$ to $y = \rho(1/2)$) and $\sigma$ (from $y = \sigma(0)$ to $\gamma(1/2) = \sigma(1)$).

\noindent \textbf{Case 2:} Suppose that $\sigma$ intersects $\gamma$ at some point other than $\sigma(1) = \gamma(1/2)$. 

\noindent \textbf{Subcase 1:} Suppose that $\sigma$ and $\gamma$ intersect at $\gamma(t)$ for some $t > 1/2$. Then $\sigma$ and $\gamma \lvert_{[1/2,1]}$ intersect at least twice and share a common endpoint in $\gamma(1/2) = \sigma(1)$. Moreover, the other endpoint of $\sigma$, which is $y$, does not lie on $\gamma$.  It follows from the non-branching property that $\gamma \lvert_{[1/2,1]}$ is a subsegment of $\sigma$. However, $\sigma \subset B_R(p)$ but $\gamma(1) \notin B_R(p)$, yielding a contradiction. 

\noindent \textbf{Subcase 2:} Suppose that $\sigma$ and $\gamma$ intersect at $\gamma(t)$ for some $t < 1/2$. Arguing as in Subcase~1, we see that $\gamma \lvert_{[0,1/2]}$ must be a subsegment of $\sigma$. In particular, $p = \gamma(0) \in \sigma$. 

Consequently, $\sigma$ and $\rho$ intersect at two points, namely $y$ and $p$. Applying the non-branching property again, it follows that the subsegment of $\sigma$ from $y$ to $p$ must be a subsegment of $\rho$.

Let $0 < r < R/2-\dd_X(y,p)$. We will impose further restrictions on $r$ later. It follows from the triangle inequality that $\dd_X(z,p) < R/2$ for any $z \in B_r(y)$. In particular, any geodesic from $z$ to $\gamma(1/2)$ is contained in $B_{R}(p)$. Since $y \in B_R(p) \setminus \gamma$, we can also choose $r$ sufficiently small so that $B_r(y) \subset  B_R(p) \setminus \gamma$. 

We claim that $B_r(y) \setminus \rho$ is non-empty. Indeed, if $B_r(y) \subset \rho$, then $X$ is locally isometric to an interval near $y$. This implies $y \in \mathcal R_1$, and hence $\dim(X) = 1$ by \cite[Theorem 1.1]{KL16}, contradicting our assumption that $\dim(X) = 2$.

Fix $z \in B_r(y) \setminus \rho$, and let $\eta:[0,1] \to X$ be a geodesic from $z$ to $\gamma(1/2)$. We claim that $\eta$ does not intersect $\gamma$ except at $\eta(1) = \gamma(1/2)$. Suppose for the sake of contradiction that $\eta$ intersects $\gamma$ at another point. Then there are two cases. 

In the case where $\eta$ intersects $\gamma$ at $\gamma(t)$ for some $t > 1/2$, we can argue exactly as we did with $\sigma$ in Subcase~1 to obtain a contradiction.

In the case where $\eta$ intersects $\gamma$ at $\gamma(t)$ for some $t < 1/2$, arguing as in Subcase~1, we conclude that $\gamma \lvert_{[0,1/2]}$ is a subsegment of $\eta$. Since $\eta$ and $\sigma$ share a common subsegment in $\gamma \lvert_{[0,1/2]}$ and a common endpoint in $\gamma(1/2)$, one must be a subsegment of the other by the non-branching property. 

If $\eta$ is a subsegment of $\sigma$, then $z = \eta(0)$ lies on the subsegment of $\sigma$ from $y$ to $p$. Since this subsegment of $\sigma$ is contained in $\rho$, this contradicts the assumption that $z \in B_r(y) \setminus \rho$. 

On the other hand, if $\sigma$ is a subsegment of $\eta$, then $\rho$ and $\eta$ share a common subsegment, namely the subsegment of $\sigma$ from $y$ to $p$. By the non-branching property, one of $\rho$ and the subsegment of $\eta$ from $z$ to $p$ must be a subsegment of the other. Since $\dd_X(y, p) = \dd_X(x, p)/2$, we may choose $r$ sufficiently small so that $\dd_X(z,p) < \dd_X(x,p)$. It follows that the subsegment of $\eta$ from $z$ to $p$ must be a subsegment of $\rho$, contradicting our assumption that $z \in B_r(y) \setminus \rho$.

We conclude that $\eta$ does not intersect $\gamma$ except at $\eta(1) = \gamma(1/2)$. In particular, $\eta([0,1)) \subset B_R(p) \setminus \gamma$. The desired curve $\alpha$ can be obtained by concatenating $\rho$ from $x$ to $y$, a geodesic from $y$ to $z$, and $\eta$ from $z$ to $\gamma(1/2)$. This proves the lemma in Subcase~2. 
\end{proof}

\begin{Lem}\label{Lem: ext-pt comp boundary}
If $B_{R}(p) \setminus \gamma$ has exactly two components $B_\pm$, then the topological boundaries of both $B_\pm$ in $B_R(p)$ coincide with $\gamma([0,1))$.
\end{Lem}

\begin{proof}
Let $x \in B_\pm$. By Lemma~\ref{Lem: can always connect to point on geod}, there exists a continuous curve $\alpha:[0,1] \to X$ connecting $x$ to $\gamma(t)$ for any $t \in [0,1)$, such that $\alpha([0,1)) \subset B_R(p) \setminus \gamma.$ Therefore, $\gamma([0,1))$ is contained in the topological boundaries of $B_\pm$ in $B_R(p)$. The reverse inclusion follows since $B_\pm$ are clopen in $B_{R}(p) \setminus \gamma$. 
\end{proof}
 
\begin{Lem}\label{Lem: ball minus geod at most two comp}
If $\gamma$ is a regular geodesic, then $B_{R}(p) \setminus \gamma$ has at most two components.
\end{Lem}
\begin{proof}
Since $\gamma$ is regular, Theorem~\ref{Thm: good components existence} ensures that for sufficiently small $r > 0$, the set $B_{r}(\gamma(1/2)) \setminus \gamma$ has two components $A_\pm$. Letting $r \leq R/2$, we also have $B_{r}(\gamma(1/2)) \subset B_{R}(p)$. 

We claim that any point $x \in B_R(p) \setminus \gamma$ can be connected to at least one of $A_\pm$ by a continuous curve contained entirely in $B_R(p) \setminus \gamma$. This immediately implies the proposition, since $A_\pm$ are both path-connected and contained in $B_R(p) \setminus \gamma$.

Let $x \in B_R(p) \setminus \gamma$. By Lemma~\ref{Lem: can always connect to point on geod}, there exists a continuous curve $\alpha:[0,1] \to X$ connecting $x$ to $\gamma(1/2)$, such that $\alpha([0,1)) \subset B_R(p) \setminus \gamma.$ By continuity, there exists $0 \leq t < 1$ such that $\alpha(t)$ lies in $B_{r}(\gamma(1/2))$, and hence in one of $A_\pm$. The curve $\alpha \lvert_{[0,t]}$ is contained in $B_{R}(p) \setminus \gamma$, and connects $x$ to one of $A_\pm$, as claimed. 
\end{proof}

Clearly, by considering subsegments of $\gamma$, the previous lemma generalizes to any radius $r \leq R$. Thus if $p$ is an extremal point and $\gamma$ is as in Definition~\ref{Def: ext-int points}, then $\gamma$ disconnects $B_r(p)$ into exactly two components. 

\begin{Lem}\label{Lem: ext point disconnect all small ball}
If $\gamma$ is a regular geodesic and $B_{R}(p) \setminus \gamma$ has exactly two components $B_\pm$. Then for any $0 < r \leq R$, the set $B_{r}(p) \setminus \gamma$ has exactly two components, namely $B_\pm' :=B_\pm \cap B_r(p).$
\end{Lem}

\begin{proof}
By Lemma~\ref{Lem: ext-pt comp boundary}, both $B_\pm$ intersect any ball centered at $\gamma(0)$. Therefore, both $B'_\pm$ are non-empty. Since $B_\pm$ are distinct components of $B_{R}(p) \setminus \gamma$, it is immediate that $B_\pm'$ cannot be connected by a continuous path in $B_r(p) \setminus \gamma$. Therefore, $B_{r}(p) \setminus \gamma$ has at least two components. On the other hand, Lemma~\ref{Lem: ball minus geod at most two comp} implies that $B_{r}(p) \setminus \gamma$ has at most two components. Therefore, it has exactly two. Since $B_r(p) \setminus \gamma = B_+' \sqcup B_-'$, and $B_\pm$ are path-disconnected from each other, they must be precisely the two components.
\end{proof}

\begin{Lem}\label{Lem: geod at ext-pt not extend}
If $\gamma$ is a regular geodesic and $B_R(p) \setminus \gamma$ is not path-connected, then $\gamma$ is not extendible past $\gamma(0)$. Moreover, any subsegment of $\gamma$ that contains $\gamma(0)$ is not extendible past $\gamma(0)$. 
\end{Lem}

\begin{proof}
Suppose for contradiction that $\gamma$ is extendible past $\gamma(0)$. Then there exists $0 < r \leq R$ and a unit-speed geodesic $\tilde \gamma:[-r,r] \to X$ such that $\tilde \gamma(0) = p$ and $\tilde \gamma \lvert_{[0,r]}$ coincides with a subsegment of $\gamma$. By Lemma~\ref{Lem: ball minus geod at most two comp}, the set $B_R(p) \setminus \gamma$ has exactly two components. Therefore, the set $B_r(p) \setminus \tilde \gamma([0,r]) = B_r(p) \setminus \gamma$ must also have exactly two components by Lemma~\ref{Lem: ext point disconnect all small ball}. We will show that $B_r(p) \setminus \tilde \gamma([0,r])$ is path-connected, thereby giving a contradiction. 

Let $x  \in B_r(p) \setminus \tilde \gamma([0,r])$, it suffices to show that $x$ can be connected to $\tilde \gamma(-r/2)$ by a continuous curve contained in $B_r(p) \setminus \tilde \gamma([0,r])$. Let $\rho:[0,1] \to X$ be a geodesic from $x$ to $p$. By the non-branching property, if $\rho$ intersects $\tilde \gamma([0,r])$ at some point other than $p = \rho(1) = \tilde \gamma(0)$, then $\rho$ must coincide with a subsegment of $\tilde \gamma \lvert_{[0,r]}$. In particular, this would imply $x \in \tilde \gamma([0,r])$, contradicting our assumption on $x$. Therefore, $\rho([0,1)) \subset B_r(p) \setminus \tilde \gamma([0,r]).$ 

Set $y := \rho(1/2)$. In particular $y \notin \tilde \gamma([0,r])$. Let $\sigma:[0,1] \to X$ be a geodesic from $y$ to $\tilde \gamma(-r/2)$. By the triangle inequality, $\sigma \subset B_{r}(p)$. 

We claim that $\sigma$ does not intersect $\tilde \gamma([0,r])$. Assume for contradiction that an intersection occurs at $z \in \tilde \gamma([0,r])$. Then $\sigma$ intersects $\tilde \gamma$ at two points, namely $z$ and $\tilde \gamma(-r/2)$. It follows from the non-branching property that the subsegment of $\tilde \gamma$ from $z$ to $\tilde \gamma(-r/2)$ is a subsegment of $\sigma$. This implies that $\sigma$ and $\tilde \gamma \lvert_{[-r/2,r]}$ share a common subsegment, as well as a common endpoint in $\tilde \gamma(-r/2)$. Since the other endpoint of $\sigma$, namely $y$, does not lie on $\tilde \gamma([0,r])$, it follows from the non-branching property that $\tilde \gamma \lvert_{[-r/2,r]}$ must be a subsegment of $\sigma$. This yields a contradiction since $\sigma \subset B_r(p)$ and $\tilde \gamma(r) \notin B_r(p)$. 

Therefore, $\sigma([0,1]) \subset B_r(p) \setminus \tilde \gamma([0,r]).$ It follows that the concatenation of $\rho$ (from $x = \rho(0)$ to $y = \rho(1/2)$) with $\sigma$ (from $y = \sigma(0)$ to $\tilde\gamma(-r/2) = \sigma(1)$) yields a continuous curve contained in $B_r(p)\setminus \tilde\gamma([0,r])$ connecting $x$ to $\tilde\gamma(-r/2)$. 

We have shown that $\gamma$ is not extendible past $\gamma(0)$. The argument applies verbatim to show the more general statement in the proposition. 
\end{proof}

\subsection{Components of balls at corners}  In this subsection, we show that any corner disconnects all sufficiently small balls of radius $r \leq \bar r$ centered at its vertex $p$, for some $\bar r > 0$. This will be achieved by approximating an arbitrary corner by extendible corners and applying Proposition~\ref{Pro: comp exist extend corner quant}. We will then relate $\bar r$ to the semi-locally simply connected property (see Theorem~\ref{Thm: semi-locally simply connected}) of $X$ at $p$. 

For the remainder of this subsection, let $(X, \dd_X, \mm_X)$ be an $\RCD(K,N)$ space of essential dimension $2$. We establish our approximation scheme for corners. Let $\alpha:[-1,1] \to X$ be a corner with vertex $p$, and set $q := \alpha(1)$ and $\tilde q := \alpha(-1)$. Let $\beta: [-1,1] \to X$ be an extendible corner with vertex $p'$. Set $q' := \beta(1)$ and $\tilde q' := \beta(-1)$. For any $\veps > 0$, we say that $\beta$ is an \emph{$\veps$-extendible approximation} of $\alpha$ if 
\[
\dd_X(p, p') < \veps, \quad \dd_X(q, q') < \veps, \quad \text{and} \quad \dd_X(\tilde q, \tilde q') < \veps. 
\]

\begin{Lem}\label{Lem: eps approx corner exist}
There exists an $\veps$-extendible approximation of $\alpha$ for every $\veps > 0$.
\end{Lem}

\begin{proof}
It suffices to prove the lemma for all $\veps$ sufficiently small so that the pairwise intersections of the $\veps$-balls around $p$, $q$ and $\tilde q$ are empty. By Proposition~\ref{Pro: geodesic a.e. extend}, we can choose $p' \in B_\veps(p)$ such that for $\mm_X$-a.e. $x \in X$, there is a unique geodesic from $p'$ to $x$ that is extendible in both directions and contained in $\mathcal R_2$. Therefore, we can choose $q' \in B_\veps(q)$ (resp. $\tilde q' \in B_\veps(\tilde q)$) such that there is a unique geodesic from $q'$ (resp. $\tilde q'$) to $p'$ that is extendible in both directions and contained in $\mathcal R_2$. Let $\gamma$ be the geodesic from $q'$ to $p'$ and let $\bar \gamma$ be its maximal extension from $q'$. We claim that we can choose $\tilde q'$ as above so that, in addition, the geodesic from $\tilde q'$ to $p$ intersects $\bar \gamma$ exactly once, namely at $p$. Indeed, if this is not the case then it follows from the non-branching property (Theorem~\ref{Thm: non-branching}) that for $\mm_X$-a.e. $x \in B_\veps(\tilde q)$, any geodesic from $x$ to $p'$ contains one of the subsegments obtained by subdividing $\bar \gamma$ at $p'$ as a proper subsegment. Therefore, the union of the maximal extensions of these two subsegments from $p'$ has full measure in $B_\veps(\tilde q)$, contradicting Lemma~\ref{Lem: geod meas 0}. Choosing such a $\tilde q'$, the curve $\beta$ obtained by concatenating the geodesic from $\tilde q'$ to $p'$ with the geodesic from $p'$ to $q'$ is an $\veps$-extendible approximation of $\alpha.$
\end{proof}
 
\begin{Pro}\label{Pro: comp exist corner quant 2}
Let $\alpha:[-1,1] \to X$ be a corner with vertex $p$. Let $R > 0$ and assume that the two edges of $\alpha$ both have length strictly greater than $R$. Let $0 < r \leq R$ be such that the inclusion $B_r(p) \hookrightarrow B_R(p)$ induces the trivial map on $\pi_1$. Then $B_r(p) \setminus \alpha$ is not path-connected.
\end{Pro}

\begin{proof}
Let $\bar R > R$ be strictly less than the lengths of the two edges of $\alpha.$ Let $\bar \alpha:[-1,1] \to X$ be the corner obtained by shortening both edges of $\alpha$ to length $\bar R$, parameterized to have constant speed.
Then $p = \bar \alpha(0)$ is the vertex of $\bar \alpha$. Let $\gamma, \tilde \gamma:[0,1] \to X$ denote the two edges of $\bar \alpha$, parameterized so that $\gamma(t) = \bar \alpha(t)$ and $\tilde \gamma(t) = \bar \alpha(-t)$ for all $t \in [0,1]$.

Since $\gamma$ and $\tilde \gamma$ are obtained by shortening the edges of $\alpha$, both geodesics must be extendible past their respective endpoints $q := \gamma(1)$ and $\tilde q := \tilde \gamma(1)$.

Since $\bar \alpha \cap B_r(p) = \alpha \cap B_r(p)$, it suffices to show that $B_r(p) \setminus \bar \alpha$ is not path-connected.

Fix a sequence $(\veps_i)_{i \in \N}$ such that $\veps_i \downarrow 0$. By Lemma~\ref{Lem: eps approx corner exist}, for each $i$ there exists an $\veps_i$-extendible approximation $\bar \alpha_i:[-1,1] \to X$ of $\bar \alpha$. Let $p_i$ denote the vertex of $\bar \alpha_i$, and let $\gamma_i, \tilde \gamma_i:[0,1] \to X$ be its edges, reparameterized with constant speed on $[0,1]$ so that $p_i = \gamma_i(0) = \tilde \gamma_i(0)$, $\gamma_i(1) = \alpha_i(1)$, and $\tilde \gamma_i(1) = \alpha_i(-1)$. Set $q_i := \gamma_i(1)$ and $\tilde q_i := \tilde \gamma_i(1)$.

The edges $\gamma$ and $\tilde \gamma$ both have length $\bar R$, which implies that $\gamma_i$ and $\tilde \gamma_i$ have lengths at least $\bar R - 2\veps_i$. Therefore, by choosing $\veps_0$ sufficiently small, we can ensure that the lengths of $\gamma_i$ and $\tilde \gamma_i$ are at least $\bar R' := (\bar R + R)/2 > R$ for all $i$. 

Since $\gamma$ is extendible past $q = \gamma(1)$, it must be the unique geodesic from $p$ to $q$ by the non-branching property. Therefore, $\gamma_i$ converges to $\gamma$ uniformly. Similarly, $\tilde \gamma_i$ converges to $\tilde \gamma$ uniformly. 

We claim that, for sufficiently large $i$, the geodesics $\gamma_i$ and $\tilde \gamma_i$ intersect only at $p_i$. If this were not the case, there would exist arbitrarily large $i$ for which $\gamma_i$ and $\tilde \gamma_i$ intersect at some point other than $p_i$. Since $q_i \to q$ and $\tilde q_i \to \tilde q$, we see that $q_i \neq \tilde q_i$ for sufficiently large $i$. In particular, the other point of intersection between $\gamma_i$ and $\tilde \gamma_i$ cannot coincide simultaneously with both endpoints $q_i$ and $\tilde q_i$. The non-branching property would then imply that one of $\gamma_i$ or $\tilde \gamma_i$ is a subsegment of the other. Uniform convergence then forces the same relation for $\gamma$ and $\tilde \gamma$, contradicting the definition of a corner.

Consequently, up to removing finitely many indices, we may assume that $\gamma_i$ and $\tilde \gamma_i$ intersect only at $p_i$ for all $i$. To avoid repeated qualifiers, we adopt the convention that all subsequent statements hold for sufficiently large $i$, even if not explicitly mentioned.

Let $0 < r \leq R$ be as in the proposition. For each $i$, define 
\[r_i := r - \dd_X(p_i,p) \quad \text{and} \quad R_i := R + \dd_X(p_i, p).\] Since $\dd_X(p_i, p) < \veps_i \to 0$, we have $r_i > 0$ and $R_i < \bar R'$ for sufficiently large $i$. The triangle inequality implies $B_{r_i}(p_i) \subset B_r(p)$ and $B_{R}(p) \subset B_{R_i}(p_i)$. Therefore, the inclusion $B_{r_i}(p_i) \rightarrow B_{R_i}(p_i)$ also induces a trivial map on $\pi_1$. By Proposition~\ref{Pro: comp exist extend corner quant}, $B_{r_i}(p_i) \setminus \bar \alpha_i$ has two components. 

Let $t_0 \in (0,1]$ be such that $\dd_X(\gamma(t_0), p) = r/2.$ For each $i$, the length of $\gamma_i$ exceeds $R$. Therefore, by Theorem~\ref{Thm: good components existence}, there exist $\delta, s_0(K, N, R, t_0) > 0$, independent of $i$, with the following property: if for some $0 < s \leq s_0$ there exists a $(\delta s)$-regular realization $(Z, \iota_1, \iota_2)$ of $(\overline B_{1000s}(\gamma_i(t_0)), \gamma_i)$, then $B_s(\gamma_i(t_0))\setminus \gamma_i$ has exactly two components. 

Tracing the construction of these components in Lemma~\ref{Lem: comp exist metric}, we see that the following holds: Let $x_\pm$ be contained in $B_s(\gamma_i(t_0))$ but outside the closed $s/100$ tubular neighborhood of the segment $\gamma_i \cap \overline B_s(\gamma_i(t_0))$. If there exist points $x_\pm' \in \bar B_{1000s}(0^2)$ such that $x_+'$ is contained in the upper half-plane, $x_-'$ is contained in the lower half-plane, and $\dd_Z(\iota_1(x_\pm), \iota_2(x_\pm')) < \delta s$, then $x_\pm$ lie in different components of $B_s(\gamma_i(t_0))\setminus \gamma_i$. 

Since $\gamma$ has length greater than $R$, the previous discussion also applies to $\gamma$. Moreover, because the edges of corners are regular, we have $\gamma(t_0) \in \mathcal R_2$. Therefore, there exists a $(\delta s/2)$-regular realization of $(\overline B_{1000s}(\gamma(t_0)), \gamma)$ for all sufficiently small $s > 0$. Fix such an $s > 0$ so that $s \leq s_0$; we may decrease $s$ as the proof continues. Let $(Z, \iota_1, \iota_2)$ be the corresponding $(\delta s/2)$-regular realization of $(\overline B_{1000s}(\gamma(t_0)), \gamma)$.

Using the realization $Z$, fix points $x_\pm$ in $B_s(\gamma(t_0))$ but outside the closed $s/100$ tubular neighborhood of the segment $\gamma \cap \overline B_s(\gamma(t_0))$, so that they can be approximated in $Z$ by points in the upper and lower half-plane, respectively. By the previous discussion, $B_s(\gamma(t_0)) \setminus \gamma$ has two components with $x_\pm$ belonging to different ones. 

Since $X$ is a geodesic space and $\gamma_i$ converges uniformly to $\gamma$, by gluing $Z$ and $X$ along $\overline B_{1000s}(\gamma(t_0))$, we obtain a $(\delta s)$-regular realization of $(\overline B_{1000s}(\gamma_i(t_0)), \gamma_i)$ for sufficiently large $i$. By the uniform convergence of $\gamma_i$, the points $x_\pm$ are contained in $B_s(\gamma_i(t_0))$ and outside the closed $s/100$ tubular neighborhood of the segment of $\gamma_i \cap \overline B_s(\gamma_i(t_0))$, for sufficiently large $i$. Moreover, they are still approximated in $Z'$ by points in the upper and lower half-plane. Therefore, $x$ and $y$ must lie in different components of $B_s(\gamma_i(t_0)) \setminus \gamma_i$.

%Let $Z'$ be the metric space obtained by gluing $X$ and $Z$ along \[\overline B_{1000s}(\gamma(t_0)) \overset{\mathrm{isom}}{\cong} \iota_1(\overline B_{1000s}(\gamma(t_0))).\]
%Let $\iota_1': \overline B_{1000s}(\gamma_i(t_0)) \to Z'$ denote the isometric embedding obtained by composing the natural inclusion $\overline B_{1000s}(\gamma_i(t_0)) \hookrightarrow X$
%with the canonical inclusion of $X$ inside $Z'$ under the gluing; for notational simplicity, we suppress the dependence on $i$. Similarly, let $\iota_2': \overline B_{1000s}(0^2) \to Z'$ be the isometric embedding obtained by composing $\iota_2$ with the canonical inclusion of $Z$ into $Z'$.

%Since $X$ is a geodesic space and $\gamma_i$ converges to $\gamma$ uniformly, we can argue as in the proof of Proposition~\ref{Pro: GH distance balls center} to show that for sufficiently large $i$, the triple $(Z', \iota_1', \iota_2')$ is a $(\delta s)$-regular realization of $(\overline B_{1000s}(\gamma_i(t_0)), \gamma_i)$.

%By the uniform convergence of $\gamma_i$, the points $x_\pm$ are contained in $B_s(\gamma_i(t_0))$ and outside the closed $s/100$ tubular neighborhood of the segment of $\gamma_i \cap \overline B_s(\gamma_i(t_0))$, for sufficiently large $i$. Moreover, they are still approximated in $Z'$ by points in the upper and lower half-plane. Therefore, $x$ and $y$ must lie in different components of $B_s(\gamma_i(t_0)) \setminus \gamma_i$. 

By possibly decreasing $s$, we may assume that $B_s(\gamma(t_0)) \subset B_r(p) \setminus \tilde \gamma.$
By the uniform convergence of $\gamma_i$ to $\gamma$ and of $\tilde \gamma_i$ to $\tilde \gamma$, it follows that $B_s(\gamma_i(t_0)) \subset B_r(p) \setminus \tilde \gamma_i$ for sufficiently large $i$. In particular, $\bar\alpha_i \cap B_s(\gamma_i(t_0)) = \gamma_i \cap B_s(\gamma_i(t_0))$, and so the set $B_s(\gamma_i(t_0)) \setminus \bar \alpha_i = B_s(\gamma_i(t_0)) \setminus \gamma_i$ has exactly two components, with $x_\pm$ belonging to different ones.

We now show that $x_\pm$ cannot be connected by a continuous curve contained in $B_r(p) \setminus \bar \alpha$. Assume for the sake of contradiction that such a curve $\rho:[0,1] \to X$ exists. Since $\rho([0,1])$ is compact, and $r_i \to r$ and $p_i \to p$, it follows that $\rho \subset B_{r_i}(p_i)$ for sufficiently large $i$. Moreover, since $\rho$ disjoint from $\bar \alpha$, it follows from the uniform convergence of $\bar \alpha_i \to \bar \alpha$ that $\rho$ is also disjoint from $\bar \alpha_i$ for sufficiently large $i$. Therefore, $\rho$ is a continuous curve from $x$ to $y$ contained in $B_{r_i}(p_i) \setminus \alpha_i$. 

Denote by $A_\pm$ be the two components of $B_s(\gamma_i(t_0)) \setminus \bar \alpha_i$ so that $x_\pm \in A_\pm$. Denote by $B_\pm$ the two components of $B_{r_i}(p) \setminus \bar \alpha_i$. By Proposition~\ref{Pro: comp exist extend corner quant}, the point $\gamma_i(t_0)$ lies in the topological boundaries of both $B_\pm$. In particular, $B_\pm \cap B_s(\gamma_i(t_0)) \neq \emptyset.$
Therefore, Lemma~\ref{Lem: big set small set comp} implies that $A_\pm = B_\pm \cap B_s(\gamma_i(t_0))$, up to permuting the labels $B_\pm$. In particular, $x_\pm \in B_\pm$. This is a contradiction, as we have shown that $\rho$ is a curve in $B_{r_i}(p_i) \setminus \bar \alpha_i$ from $x_+$ to $y_+$. Therefore, $B_r(p) \setminus \alpha = B_r(p) \setminus \bar \alpha$ is not path-connected.
\end{proof}

\begin{Exa}
Let $X = \R^2_+$, and let $\alpha$ be a corner whose intersection with $\partial X$ is exactly its vertex. Then $\alpha$ disconnects all small balls centered its vertex into three components. 
\end{Exa}
The previous example shows that, although we have already proved that a corner always disconnects sufficiently small balls around its vertex into at least two components, one cannot in general expect exactly two. As we will see, the number of components is always either two or three if the ball is sufficiently small, with this example illustrating the other case. Later, it will become clear that the case of two components occurs when the vertex is an interior point, whereas the case of three components occurs when the vertex is an extremal point.

\begin{Pro}\label{Pro: corner at most three comp}
Let $\alpha:[-1,1] \to X$ be a corner with vertex $p$. Let $R > 0$ and assume that the two edges of $\alpha$ have length at least $R$. Then $B_R(p) \setminus \alpha$ has at most three components. Moreover, there exists at least one component $Q$ such that its topological boundary in $B_R(p)$ coincides with $\alpha \cap B_R(p)$. 
\end{Pro}
\begin{proof}
Up to shortening the edges of $\alpha$ and reparameterizing, we may assume that both edges of $\alpha$ have length exactly $R$ and that $p = \alpha(0)$. Clearly, this does not change the set $B_R(p) \setminus \alpha$. Let $\gamma, \tilde \gamma:[0,1] \to X$ denote the two edges of $\alpha$, parameterized so that $\gamma(t) = \alpha(t)$ and $\tilde \gamma(t) = \alpha(-t)$ for all $t \in [0,1]$.

By Proposition~\ref{Pro: geod neigh}, there exists a path-connected neighborhood $U \subset B_R(p) \setminus \tilde \gamma$ of $\gamma((0,1))$ such that $U \setminus \gamma$ has exactly two components $U_\pm$. Similarly, there exists a path-connected neighborhood $\tilde U \subset B_R(p) \setminus \gamma$ of $\tilde \gamma((0,1))$ such that $\tilde U \setminus \tilde \gamma$ has exactly two components $\tilde U_\pm$.

Let $x \in B_R(p) \setminus \alpha \subset B_R(p) \setminus \gamma$. By Lemma~\ref{Lem: can always connect to point on geod}, there exists a continuous curve $\beta:[0,1] \to X$ from $x$ to $\gamma(1/2)$ such that $\beta([0,1)) \subset B_R(p) \setminus \gamma.$ In particular, $p$ does not lie on $\beta$. 

Since $\beta(0) \notin \alpha$ and $\beta(1) \in \alpha$, there is a smallest $s \in (0,1]$ such that $\beta(s) \in \alpha$. It follows that $\beta([0,s)) \subset B_R(p) \setminus \alpha.$ Since $\beta(s) \neq p$, it must be that $\beta(s) \in \gamma((0,1))$ or $\beta(s) \in \tilde \gamma((0,1))$.

If $\beta(s) \in \gamma((0,1))$, then there exists some $\veps > 0$ such that $\beta(s-\veps) \in U_+$ or $U_-$, and so $\beta \lvert_{[0,s-\veps]}$ is a continuous curve in $B_R(p) \setminus \alpha$ from $x$ to one of $U_\pm$. Similarly, if $\beta(s) \in \tilde \gamma((0,1))$, the same argument applies with $\tilde U_\pm$. 

Thus, any $x \in B_R(p) \setminus \alpha$ can be connected by a continuous curve in $B_R(p) \setminus \alpha$ to at least one of $U_\pm$ and $\tilde U_\pm$. This shows that $B_R(p) \setminus \alpha$ has at most four components. 

To see that $B_R(p) \setminus \alpha$ has at most three components, it suffices to show that at least one of $U_\pm$ can be connected to at least one of $\tilde U_\pm$ by a curve in $B_R(p) \setminus \alpha$. The argument is similar to the previous construction so we will only give a sketch: we first use Lemma~\ref{Lem: can always connect to point on geod} to obtain a continuous curve from a point on $\tilde \gamma((0,1))$ to a point on $\gamma((0,1))$ that is contained in $B_R(p) \setminus \{p\}$. Then, by restricting to an appropriate subsegment of this curve, we can ensure that it remains entirely in $B_R(p) \setminus \alpha$ except at the endpoints, which lie on $\tilde \gamma((0,1))$ and $\gamma((0,1))$ respectively. By further shortening this subsegment a little on both ends, we obtain the desired curve from one of $U_\pm$ to one of $\tilde U_\pm$ contained in $B_R(p) \setminus \alpha$. 

%To prove that $B_R(p) \setminus \alpha([-1,1])$ has at most three components, it suffices to show that at least one of $U_\pm$ can be connected to at least one of $\tilde U_\pm$ by a curve contained in $B_R(p) \setminus \alpha([0,1])$. 

%Applying Lemma~\ref{Lem: can always connect to point on geod} again, there exists a continuous curve $\beta:[0,1] \to X$ (where we recycle the notation $\beta$) from $\tilde \gamma(1/2)$ to $\gamma(1/2)$, such that 
%\[\beta([0,1)) \subset B_R(p) \setminus \gamma([0,1]).\] In particular,
%\[
%\beta([0,1]) \subset B_R(p) \setminus \{p\}.
%\]

%Since $\beta(0) \in \tilde \gamma([0,1])$ and $\beta(1) \notin \tilde \gamma([0,1])$, there must be a largest time $s \in [0,1)$ such that $\beta(s) \in \tilde \gamma([0,1])$. It follows that
%\[
%\beta((s,1]) \subset B_R(p) \setminus \tilde \gamma([0,1]).
%\]

%Since $\beta(s) \notin \tilde \gamma([0,1])$ (being in $\gamma([0,1])$ and not equal to $p$) and $\beta(1) \in \tilde \gamma([0,1])$, there must be a smallest time $s' \in (s, 1]$ such that $\beta(s') \in \tilde \gamma([0,1])$. It follows that 
%\[
%\beta((s,t)) \subset B_R(p) \setminus (\tilde \gamma ([0,1]) \cup \gamma ([0,1])) = B_R(p) \setminus \alpha([-1,1]).
%\]
%As before, by taking $\veps > 0$ sufficiently small it follows that $\beta \lvert_{[s+\veps, s'-\veps]}$ is a curve from one of $\tilde U_\pm$ to one of $U_\pm$, and is contained in $B_R(p) \setminus \alpha([0,1])$. 

Therefore, up to relabeling, we may assume that $U_+$ and $\tilde U_+$ lie in the same component of $B_R(p) \setminus \alpha$, denoted by $Q$. Then $\alpha \cap B_R(p)$ is contained in the topological boundary of $Q$ in $B_R(p)$, since $\gamma([0,1))$ and $\tilde \gamma([0,1))$ are contained in the topological boundaries of $U_+$ and $\tilde U_+$, respectively, by Proposition~\ref{Pro: geod neigh}. Since $B_R(p) \setminus \alpha$ is locally path-connected and $Q$ is one of its components, the reverse inclusion of the boundary follows immediately.
\end{proof}

\subsection{Corners with interior vertices} In this subsection, we study corners whose vertices are interior points. We will prove that such corners disconnect small balls centered at their vertices into exactly two components. 

For the remainder of this subsection, let $(X, \dd_X, \mm_X)$ be an $\RCD(K,N)$ space of essential dimension $2$. We begin by establishing a stronger version of Proposition~\ref{Pro: corner at most three comp} under the assumption that the vertex lies in the interior set.

\begin{Pro}\label{Pro: corner int at most two comp}
Let $\alpha:[-1,1] \to X$ be a corner with vertex $p$. Let $R > 0$ and assume that the two edges of $\alpha$ have length at least $R$. If $p \in \Int(X)$, then $B_R(p) \setminus \alpha$ has at most two components. Moreover, the topological boundary of any such component in $B_R(p)$ coincides with $\alpha\cap B_R(p)$. 
\end{Pro}

\begin{proof}
Let $\gamma$, $\tilde \gamma$, $U$, $\tilde U$, $U_\pm$, and $\tilde U_\pm$ be as in the proof of Proposition~\ref{Pro: corner at most three comp}. As in that proof, we may assume that the edges of $\alpha$ have length exactly $R$ and that $p = \alpha(0)$. To ease exposition, we say that two sets in $\{U_+, U_-, \tilde U_+, \tilde U_-\}$ are \emph{joined} if they are contained in the same component of $B_R(p) \setminus \alpha$. It follows from the proof of Proposition~\ref{Pro: corner at most three comp} that $U_+$ and $\tilde U_+$ are joined.

Fix $x_\pm \in U_+$. Since $p \in \Int(X)$, the set $B_R(p) \setminus \gamma$ is path-connected. Therefore, there exists a continuous curve $\beta:[0,1] \to X$ from $x_+$ to $x_-$ contained in $B_R(p) \setminus \gamma$. If $\beta$ does not intersect $\tilde \gamma$, then $U_+$ and $U_-$ are joined. On the other hand, if $\beta$ intersects $\tilde \gamma$, then there exists a smallest $t \in (0,1]$ such that $\beta(t) \in \tilde \gamma$. By continuity, there exists $\veps > 0$ such that $\beta(t-\veps)$ lies in one of $\tilde U_\pm$. Since $\beta \lvert_{[0, t-\veps]} \subset B_R(p) \setminus \alpha$, it follows that $U_-$ is joined to at least one of $\tilde U_\pm$. A symmetric argument yields the same conclusion for $\tilde U_-$.

%In all, we have shown that the following hold:
%\begin{enumerate}
  %  \item $U_+$ is joined to $U_-$;
  %  \item $U_-$ is joined to at least one of $U_+$, $\tilde U_+$, $\tilde U_-$;
  %  \item $\tilde U_-$ is joined to at least one of $\tilde U_+$, $U_+$, $U_-$.
%\end{enumerate}
By the proof of Proposition~\ref{Pro: corner at most three comp}, any component of $B_R(p) \setminus \alpha$ contains at least one of the four sets $U_\pm$ and $\tilde U_\pm$. A simple case check using that $U_+$ and $\tilde U_+$ are joined along with the conclusions of the previous paragraph shows that $B_R(p) \setminus \alpha$ has either one or two components. 

In the case where there are two components, it is not difficult to see that one component contains $U_+$ and $\tilde U_+$, and the other contains $U_-$ and $\tilde U_-$. The claim concerning the boundary of the components follows exactly as in the proof of Proposition~\ref{Pro: corner at most three comp}. 
\end{proof}

The following proposition follows immediately from Propositions~\ref{Pro: comp exist corner quant 2}~and~\ref{Pro: corner int at most two comp}.
\begin{Pro}\label{Pro: comp exist corner int quant}
Let $\alpha:[-1,1] \to X$ be a corner with vertex $p$. Let $R > 0$ and assume that the two edges of $\alpha$ have length strictly greater than $R$. Let $0 < r \leq R$ be such that the inclusion $B_r(p) \hookrightarrow B_R(p)$ induces the trivial map on $\pi_1$. If $p$ is an interior point, then $B_r(p) \setminus \alpha$ has exactly two components $A_\pm$. Moreover, the topological boundaries of $A_\pm$ in $B_r(p)$ coincide with $\alpha \cap B_r(p)$.
\end{Pro}

By Theorem~\ref{Thm: semi-locally simply connected}, we obtain the following proposition from the previous one. 
\begin{Pro}\label{Pro: good comp exist corner int}
Let $\alpha:[-1,1] \to X$ be a corner with vertex $p$. If $p \in \Int(X)$ then there exists $\bar r > 0$ such that for every $0 < r \leq \bar r$, the ball $B_r(p)$ has two good components with respect to $\alpha$.
\end{Pro}

\begin{proof}
Let $R$ be the minimum of edge lengths of $\alpha$. By Theorem~\ref{Thm: semi-locally simply connected}, there exists $0 < \bar r \leq R/2$ such that the inclusion $B_{10\bar r}(p) \hookrightarrow B_{R/2}(p)$ induces the trivial map on $\pi_1$. Proposition~\ref{Pro: comp exist corner int quant} implies that, for every $0 < r \leq 10\bar r$, the set $B_r(p) \setminus \alpha$ has two components. In particular, for every $0<r\le \bar r$, the ball $B_r(p)$ has two good components with respect to $\alpha$. This follows by considering the components of $B_{10r}(p)\setminus \alpha$; see the proof of Theorem~\ref{Thm: good components existence}.
\end{proof}

We can now remove the assumption that the edges of $\alpha$ have length \emph{strictly} greater than $R$ in Proposition~\ref{Pro: comp exist corner int quant}.
\begin{Pro}\label{Pro: comp exist corner int quant 2}
Let $\alpha:[-1,1] \to X$ be a corner with vertex $p$. Let $R > 0$ and assume that the two edges of $\alpha$ have length at least $R$. Let $0 < r \leq R$ be such that the inclusion $B_r(p) \hookrightarrow B_R(p)$ induces the trivial map on $\pi_1$. If $p \in \Int(X)$, then $B_r(p) \setminus \alpha$ has exactly two components $A_\pm$. Moreover, the topological boundaries of $A_\pm$ in $B_r(p)$ coincide with $\alpha \cap B_r(p)$.
\end{Pro}

\begin{proof}
Up to shortening the edges of $\alpha$ and reparameterizing, we may assume that both edges of $\alpha$ have length exactly $R$ and that $p = \alpha(0)$; this does not change the set $B_r(p) \setminus \alpha$. By Proposition~\ref{Pro: corner neigh}, Theorem~\ref{Thm: good components existence 2}, and Proposition~\ref{Pro: good comp exist corner int}, there exists an open path-connected neighborhood $U \subset B_R(p)$ of $\alpha((-1,1))$ such that $U \setminus \alpha$ has exactly two components $U_\pm$. Set $V := B_R(p) \setminus \alpha.$ The proposition follows by considering the Mayer--Vietoris sequence for the cover $B_R(p) = U \cup V$ and arguing as in the proof of Proposition~\ref{Pro: comp exist extend corner quant}.
\end{proof}

\subsection{Equivalent definition of extremal points}
In this subsection, we show that the definition of extremal point is equivalent to the a priori stronger condition that \emph{every} regular geodesic $\gamma$ with an endpoint at $p$ disconnects $B_r(p)$ for all $0 < r \leq \bar r$, where $\bar r$ depends only on the length of $\gamma$ and the semi-locally simply connected property at $p$. 

For the remainder of this subsection, let $(X, \dd_X, \mm_X)$ be an $\RCD(K,N)$ space of essential dimension $2$.

\begin{Pro}\label{Pro: ext-pt strong def}
Let $p \in X$. Then $p \in \Ext(X)$ if and only if for every regular geodesic $\gamma:[0,1] \to X$ with $\gamma(0) = p$, there exists $0 < r \leq R$, where $R$ is the length of $\gamma$, such that $B_r(p)\setminus \gamma$ has exactly two components.
\end{Pro}

\begin{proof}
The ``if'' direction is immediate. Indeed, it suffices to show that there exists $x \in X$ such that a geodesic from $p$ to $x$ is regular, and this follows from Proposition~\ref{Pro: every pt have extend reg geod}.

We now prove the ``only if'' direction. Assume that $p$ is extremal. Then there exists a regular geodesic $\bar \gamma_0:[0,1] \to X$ with $\bar \gamma_0(0) = p$ and length $\bar R_0 > 0$ such that $\bar \gamma_0$ disconnects $B_{\bar r_0}(p)$ for some $\bar r_0 \leq \bar R_0$. By Lemma~\ref{Lem: ball minus geod at most two comp}, $B_{\bar r_0}(p) \setminus \bar \gamma_0$ has exactly two components. It then follows from Lemma~\ref{Lem: ext point disconnect all small ball} that $B_{r}(p) \setminus \bar \gamma_0$ has exactly two components for all $0 < r \leq \bar r_0.$

Let $\bar \gamma_1:[0,1] \to X$ be a regular geodesic with $\bar \gamma_1(0) = p$. We must show that $\bar \gamma_1$ disconnects a small ball around $p$ into exactly two components. Up to shortening $\bar \gamma_1$, we may assume that $\bar \gamma_1(0)$ has length $R < \bar R_0$. If $\bar \gamma_1$ intersects $\bar \gamma_0$ at some point other than $p$, then the non-branching property (Theorem~\ref{Thm: non-branching}) implies that $\bar \gamma_1$ is a subsegment of $\bar \gamma_0$. In particular, $\bar \gamma_1$ disconnects all small balls around $p$. Thus assume that $\bar \gamma_1$ and $\bar \gamma_0$ intersect only at $p$. 

Let $\bar \alpha$ be the curve obtained by concatenating $\bar \gamma_0$ (reversed) with $\bar \gamma_1$. Since $\bar\gamma_0$ disconnects all sufficiently small balls centered at $p$, Lemma~\ref{Lem: geod at ext-pt not extend} implies that any subsegment of the geodesic $\bar\gamma_0$ containing $p$ is not extendible past $p$. It follows that any restriction of $\bar \alpha$ to a neighborhood of $p$ cannot be a geodesic. Therefore, $\bar\alpha$ is a corner. 

By Theorem~\ref{Thm: semi-locally simply connected}, for all sufficiently small $0 < r \leq R/2$, the inclusion $B_r(p) \hookrightarrow B_{R/2}(p)$ induces the trivial map on $\pi_1$. Proposition~\ref{Pro: comp exist corner quant 2} then implies that $B_r(p) \setminus \bar \alpha$ has at least two components. By further impose that $r \leq \bar r_0$, we may assume that $B_r(p) \setminus \bar \gamma_0$ has exactly two components as well. Since $\bar \gamma_1((0,1]) \cap B_r(p)$ is disjoint from $\bar \gamma_0$, it must be contained in one of the components of $B_r(p) \setminus \bar \gamma_0$. We denote this component by $A_+$, and the other by $A_-$. 

Shorten $\bar \gamma_0$ and $\bar \gamma_1$ to geodesics $\gamma_0, \gamma_1:[0,1] \to X$ such that $p = \gamma_0(0) = \gamma_1(0)$, and $\gamma_0$ and $\gamma_1$ have the same length $r$. Consider the curve $\alpha:[-1,1] \to X$ 
define by
\[
\alpha(t) =
\begin{cases}
\gamma_0(-t) & \text{if $t \in [-1,0]$} \\
\gamma_1(t) & \text{if $t \in [0,1]$}. 
\end{cases}
\] 
Since $B_r(p) \setminus \bar \gamma_1 = B_r(p) \setminus \gamma_1$, it suffices to show that the latter has exactly two components. 

By Proposition~\ref{Pro: geod neigh}, there exists a path-connected neighborhood $U_1 \subset B_r(p) \setminus \gamma_0$ of $\gamma_1((0,1))$ such that $U_1 \setminus \gamma_1$ has exactly two components $U^1_\pm$. Define $U_0$ analogously for $\gamma_0$ and denote by $U^0_\pm$ the components of $U_0 \setminus \gamma_0$. 

By the proof of Proposition~\ref{Pro: corner at most three comp} and up to relabeling, we may assume that $U_+^0$ and $U_+^1$ lie in the same component of $B_r(p) \setminus \alpha$. Since $\gamma_1((0,1)) \subset A_+$ and $U_1$ is path-connected, we have $U_1 \subset A_+$. This implies that $U_-^1, U_+^1, U_+^0 \subset A_+$. Since $\gamma_0((0,1))$ lies in the topological boundaries of $A_\pm$, both $A_\pm$ intersect $U_0$. Thus it follows from Lemma~\ref{Lem: big set small set comp} that $U_-^0 \subset A_-$.

Next, we claim that the following hold for all sufficiently small $s > 0$:
\begin{enumerate}
    \item $B_s(\gamma_1(1/2)) \subset U_1 \subset  B_r(p) \setminus \gamma_0$;
    \item $B_s(\gamma_1(1/2)) \setminus \gamma_1$ has exactly two components $B_\pm$;
    \item $B_\pm$ are contained in different components of $B_r(p) \setminus \alpha$. 
\end{enumerate}
Indeed, (1) follows from triangle inequality, (2) follows from Theorem~\ref{Thm: good components existence 2}, and (3) follows from the proof of Proposition~\ref{Pro: comp exist corner int quant 2}. 

By Proposition~\ref{Pro: geod neigh}, $\gamma_1(1/2)$ lies in the topological boundaries of $U^1_\pm$. Therefore, $U^1_\pm$ intersect $B_s(\gamma_1(1/2))$. By Lemma~\ref{Lem: big set small set comp} and (1), up to relabeling $B_\pm$, we may assume that $B_\pm \subset U^1_\pm$. Fix $x_\pm \in B_\pm$, we claim that there does not exist a curve in $B_r(p) \setminus \gamma_1$ joining $x_\pm$.

Suppose for the sake of contradiction that such a curve $\beta:[0,1] \to X$ exists. By (3), $\beta$ must intersect $\alpha$. Since $\beta$ does not intersect $\gamma_1$, it must intersect $\gamma_0$. Thus there exists a smallest $t \in (0,1)$ such that $\beta(t) \in \gamma_2$. By continuity, there exists $0 < t' < t$ such that $\beta(t')$ lies in $U_0$, and therefore in one of $U^0_\pm$. If $\beta(t') \in U^0_+$, then $\beta \lvert_{[0,t']}$ is a curve in $B_r(p) \setminus \alpha$ from $U^1_-$ to $U^0_+$. Since $U^0_+$ and $U^1_+$ are contained in the same component of $B_r(p) \setminus \alpha$, this contradicts (3). On the other hand, if $\beta(t') \in U^0_-$, then $\beta \lvert_{[0,t']}$ is a curve in $B_r(p) \setminus \alpha$ (and thus in $B_r(p) \setminus \gamma_0$) from $U^1_-$ to $U^0_-$. However, we have already shown that $U_-^1 \subset A_+$ and $U^0_- \subset A_-$, yielding a contradiction.  

Therefore, $x_\pm$ lie in different components of $B_r(p) \setminus \gamma_1$. In particular, $B_r(p) \setminus \gamma_1$ has at least two components. By Lemma~\ref{Lem: ball minus geod at most two comp}, it has exactly two. 
\end{proof}

The following corollary is immediate from Lemma~\ref{Lem: geod at ext-pt not extend} and Proposition~\ref{Pro: ext-pt strong def}.
\begin{Cor}\label{Cor: int geod no ext}
If $x$ lies in the interior of a regular geodesic $\gamma$, then $x \in \Int(X)$. 
\end{Cor}

It follows from the Proposition~\ref{Pro: ext-pt strong def} that any regular geodesic $\gamma$ with an endpoint $p \in \Ext(X)$ disconnects $B_r(p)$ into two components for some $r > 0$ depending on $\gamma$. At the moment, we do not have any quantitative control over $r$. This will present issues for our eventual goal of proving that the $\Ext(X)$ is closed. For example, it might be the case that we have a sequence of points $p_i \in \Ext(X)$ converging to some point $p \in X$ and a sequence $r_i \to 0$, such that $B_{r_i}(p_i)$ is never disconnected by any regular geodesic with an end point in $p_i$. This would make it impossible to pass any disconnection properties to the limit point $p$.

In the next part of this subsection, we remedy this by showing that if $\gamma$ has length $R$ then it must disconnect $B_r(p)$ for any $0 < r \leq R$ such that the inclusion $B_r(p) \hookrightarrow B_R(p)$ induces the trivial map on $\pi_1$. The strategy is the same as that of Subsection~\ref{Subsec: MV extend corner}.

\begin{Lem}\label{Lem: geod endpoint ext good comp}
Let $\gamma:[0,1] \to X$ be a regular geodesic such that $p := \gamma(0)$ is extremal. Then there exists $\bar r>0$ such that, for every $0<r<\bar r$, the ball $B_r(p)$ has two good components with respect to $\gamma$.
\end{Lem}

\begin{proof}
Let $R$ be the length of $\gamma$. By Proposition~\ref{Pro: ext-pt strong def}, there exists $0 < \bar r \leq R/10$ such that $B_{10 \bar r}(p) \setminus \gamma$ has exactly two components. It follows from Lemmas~\ref{Lem: ext-pt comp boundary}~and~\ref{Lem: ext point disconnect all small ball} that for every $0<r\le 10\bar r$, the set $B_r(p)\setminus \gamma$ has exactly two components, and the topological boundary of each component in $B_r(p)$ coincides with $\gamma\cap B_r(p)$. Therefore, for every $0<r\le \bar r$, the ball $B_r(p)$ has two good components with respect to $\gamma$, where property~(3) in the definition of good components (Definition~\ref{Def: good comp}) follows by applying the same argument as in the proof of Theorem~\ref{Thm: good components existence}, using the components of $B_{10r}(p)\setminus \gamma$ as auxiliary sets. 
\end{proof}

Applying Lemma~\ref{Lem: geod endpoint ext good comp} and Lemma~\ref{Lem: geod neigh side choice} (with $I = [0,1)$) and arguing as in the proof of Proposition~\ref{Pro: geod neigh}, we obtain the following proposition.
\begin{Pro}\label{Pro: geod neigh endpoint ext}
Let $\gamma:[0,1] \to X$ be a regular geodesic such that $p := \gamma(0)$ is extremal. Then for any open neighborhood $V$ of $\gamma([0,1))$, there exists an open path-connected neighborhood $U \subset V$ of $\gamma([0,1))$ such that $U \setminus \gamma$ has exactly two components $U_\pm$. Moreover, the topological boundaries of $U_\pm$ in $U$ coincide with $\gamma([0,1))$.
\end{Pro}

We can now prove the following proposition (cf. Proposition~\ref{Pro: comp exist extend corner quant}). 
\begin{Pro}\label{Pro: comp exist geod ep ext quant}
Let $\gamma:[0,1] \to X$ be a regular geodesic such that $p := \gamma(0)$ is extremal.. Let $R > 0$ and assume that $\gamma$ has length at least $R$. Let $0 < r \leq R$ be such that the inclusion $B_r(p) \hookrightarrow B_R(p)$ induces the trivial map on $\pi_1$. Then $B_r(p) \setminus \gamma$ has exactly two components $A_\pm$. Moreover, the topological boundaries of $A_\pm$ in $B_r(p)$ coincide with $\gamma \cap B_r(p)$.
\end{Pro}

\begin{proof}
It suffices to prove the proposition for all $\gamma$ of length $R$. By Proposition~\ref{Pro: geod neigh endpoint ext}, there exists a path-connected neighborhood $U \subset B_R(p)$ of $\gamma([0,1))$ such that $U \setminus \gamma$ has exactly two components. Let $V := B_R(p) \setminus \gamma.$ The proposition follows by considering the Mayer--Vietoris sequence for the cover $B_R(p) = U \cup V$ and arguing as in the proof of Proposition~\ref{Pro: comp exist extend corner quant}.
\end{proof}

\subsection{Proof of Theorem~\ref{Thm: int set open}}\label{Subsec: ext set closed}

We first show that $\Ext(X)$ is sequentially closed. Let $p_i \in \Ext(X)$ be a sequence converging to some $p \in X$.  By Proposition~\ref{Pro: every pt have extend reg geod}, there exists $q \in \mathcal{R}_2$ such that the geodesic $\gamma:[0,1] \to X$ from $p$ to $q$ is unique and regular. Let $L := \dd_X(p,q)$. It follows from the triangle inequality that $L_i := \dd_X(p_i, q) > L/2$ for all sufficiently large $i$. To avoid repeated qualifiers, we adopt the convention that all subsequent statements hold for sufficiently large $i$, even if not explicitly mentioned.

For each $i \in \N$, fix a geodesic $\gamma_i:[0,1] \to X$ from $p_i$ to $q$. Since $\gamma$ is the unique geodesic from $p$ to $q$, the sequence $\gamma_i$ converges uniformly to $\gamma$. Since the length of each $\gamma_i$ exceeds $L/2$, it follows from Lemma~\ref{Lem: R2 persistence} that there exist $\delta, s_0(K, N, L) > 0$, independent of $i$, with the following property: if there exists $0 < s \leq s_0$ such that
\begin{equation}\label{Eq: ext set closed 1}
\dd_{GH}(\overline B_s(\gamma_i(1/2)), \overline B_s(0^2)) < \delta s,
\end{equation}
then 
\[
\frac{\dd_{GH}(\overline B_{s'}(\gamma_i(1/2)), \overline B_{s'}(0^2))}{s'} \to 0 \quad \text{as } s' \to 0.
\]
In particular, this would imply that $\gamma_i(1/2) \in \mathcal R_2$. 

Since $\gamma$ is regular, there exists $0<s \leq s_0$ such that 
\[
\dd_{GH}(\overline B_s(\gamma(1/2)), \overline B_s(0^2)) < \frac{\delta s}{2}.
\]
Since $\gamma_i(1/2)$ converges to $\gamma(1/2)$, Proposition~\ref{Pro: GH distance balls center} implies that, for sufficiently large $i$,
\[
\dd_{GH}(\overline B_s(\gamma(1/2)), \overline B_s(\gamma_i(1/2))) < \frac{\delta s}{2}.
\]
By triangle inequality for the Gromov--Hausdorff distance, we obtain \eqref{Eq: ext set closed 1}, and so $\gamma_i(1/2) \in \mathcal{R}_2$. By Theorem~\ref{Thm: geodesic interior tangent}, the geodesics $\gamma_i$ must all be regular for sufficiently large $i$. 

Let $R := L/4$. By Theorem~\ref{Thm: semi-locally simply connected}, there exists $0< r \leq R$ such the inclusion $B_r(p) \hookrightarrow B_R(p)$ induces the trivial map on $\pi_1$. For each $i$, define
\[
r_i := r-\dd_X(p_i, p) \quad \text{and} \quad R_i := R + \dd_X(p_i, p).
\]
Since $\dd_X(p_i ,p) \to 0$, we have $r_i > 0$ and $R_i < L/2$ for sufficiently large $i$. The triangle inequality implies $B_{r_i}(p_i) \subset B_r(p)$ and $B_R(p) \subset B_{R_i}(p_i)$.
Therefore, the inclusion $B_{r_i}(p_i) \hookrightarrow B_{R_i}(p_i)$ must also induce a trivial map on $\pi_1$. The point $p_i$ is extremal, and each $\gamma_i$ is regular with length $L_i > L/2 > R_i.$ Thus it follows from Proposition~\ref{Pro: comp exist geod ep ext quant} that the set $B_{r_i}(p_i) \setminus \gamma_i$ has exactly two components.

Let $t_0 \in (0,1]$ be such that $\dd_X(\gamma(t_0), p) = r/2$. Arguing as in the proof of Proposition~\ref{Pro: comp exist corner quant 2}, there exists $s > 0$ such that the following hold for all sufficiently large $i$:
\begin{enumerate}
    \item $B_s(\gamma(t_0)) \subset B_r(p)$ and $B_s(\gamma_i(t_0)) \subset B_{r_i}(p_i)$.
    \item $B_s(\gamma(t_0)) \setminus \gamma$ and $B_s(\gamma_i(t_0)) \setminus \gamma_i$ have exactly two components.
    \item There exist points $x_\pm \in X$ such that
    \[
    x_\pm \in B_s(\gamma(t_0)) \setminus \gamma \subset B_r(p) \setminus \gamma, \quad x_\pm \in B_s(\gamma_i(t_0)) \setminus \gamma_i \subset B_{r_i}(p_i)\setminus \gamma_i,
    \]
    and $x_\pm$ lie in different components of $B_s(\gamma(t_0)) \setminus \gamma$ and $B_s(\gamma_i(t_0)) \setminus \gamma_i$.
\end{enumerate}
Continuing the argument as in the proof of Proposition~\ref{Pro: comp exist corner quant 2}, we conclude that $x_+$ and $x_-$ cannot be connected by a continuous curve in $B_r(p) \setminus \gamma$. Therefore, $B_r(p) \setminus \gamma$ is not path-connected, and it follows by definition that $p \in \Ext(X)$. Therefore, $\Ext(X)$ is closed.

We now show that $\Int(X)$ is dense in $X$. Let $x \in X$ and $r > 0$. By Proposition~\ref{Pro: every pt have extend reg geod}, there exists $y \in B_r(x)$ such that the geodesic $\gamma:[0,1] \to X$ from $x$ to $y$ is unique, regular, and extendible past $y$. Any extension $\tilde \gamma$ of $\gamma$ past $y$ is also a regular geodesic by Theorem~\ref{Thm: geodesic interior tangent}. Since $y$ lies in the interior of $\tilde \gamma$, it follows from Corollary~\ref{Cor: int geod no ext} that $y \in \Int(X)$.  \qed

\section{Jordan Theorem for polygonal loops}\label{Sec: poly loop}

Let $(X,\dd_X,\mm_X)$ be an $\RCD(K,N)$ space of essential dimension $2$. In
this section, we introduce the notion of polygonal loops and prove a Jordan curve theorem for such loops.

\begin{Def}[Polygonal curve]\label{Def: poly curve}
A continuous curve $\alpha:[a,b]\to X$ is called \emph{polygonal} if there exist
$a=t_0<t_1<\cdots<t_m=b$ such that, for each $i=1,\dots,m$, the restriction
$\alpha|_{[t_{i-1},t_i]}$ is a regular geodesic. A polygonal curve is \emph{simple} if it is injective and \emph{interior} if $\alpha \subset \Int(X)$. 
\end{Def}

For any polygonal curve $\alpha:[a,b] \to X$, there are many different partitions $\{t_i\}_{i=0}^{m}$ of $[a,b]$ for which the condition in Definition~\ref{Def: poly curve} holds. To avoid ambiguity, we always fix an associated partition for a given polygonal curve. Several notions and constructions introduced later will depend implicitly on this choice, and we suppress this dependence in the notation.

The points $\alpha(t_0)$ and $\alpha(t_m)$ are called the endpoints of $\alpha$, and the points $\alpha(t_0),\dots,\alpha(t_{m})$ are called the
\emph{vertices} of $\alpha$. For each $i = 1, \dots, m$, the geodesic $\alpha \lvert_{[t_{i-1} ,t_i]}$ is called an \emph{edge} of $\alpha$, and its image $\alpha([t_{i-1}, t_i])$ is called an \emph{edge}. In particular, a corner is a polygonal curve, and the notions of vertices, edges, and edges agree with the corresponding terminology for corners (see Definition~\ref{Def: corner}).

A vertex is said to be \emph{incident} to an edge if it is one of the two endpoints of the edge. Clearly, each edge has two incident vertices. Each vertex of the polygonal curve that is not an endpoint has two incident edges, while each endpoint has exactly one incident edge. Two distinct vertices are \emph{adjacent} if they are the endpoints of an edge, and two distinct edges are \emph{adjacent} if they share a common endpoint.

We say that a polygonal curve is \emph{constant-speed} if there exists some $c > 0$ such that every edge of $\alpha$ has speed $c$. Any polygonal curve can be reparameterized to be constant-speed, and hence in the sequel we always take polygonal curves to be constant-speed. 

\begin{Def}[Polygonal loop]\label{Def: poly loop}
A polygonal curve $\alpha:[a,b]\to X$ is a \emph{polygonal loop} if $\alpha(a) = \alpha(b)$. A polygonal loop is \emph{simple} if $\alpha \lvert_{[a,b)}$ is injective and \emph{interior} if $\alpha \subset \Int(X)$. 
\end{Def}

We note that, for a polygonal loop, any partition $\{t_i\}_{i=0}^m$ of $[a,b]$ for which Definition~\ref{Def: poly curve} holds must have $m \geq 2$. The notions of incidence and adjacency are defined accordingly so that every vertex of a polygonal loop has exactly two incident edges and two adjacent vertices, and every edge has two incident vertices. 

\begin{Rem}\label{Rem: vertex corner or geod}
Any simple polygonal loop is locally either a corner or a geodesic at a vertex $\alpha(t_i)$. More precisely, for each $i = 0, \dots, m$, there exists $\bar r > 0$ such that for any $0 < r \leq \bar r$,
\begin{enumerate}
    \item $\overline B_r(\alpha(t_i))$ intersects $\alpha$ only along the two incident edges of $\alpha(t_i)$, so that
    \[\alpha \cap \overline B_{r}(\alpha(t_i)) = \alpha([t_i - r/c, t_i + r/c]);\]
    \item $\alpha\lvert_{[t_i - r/c, t_i + r/c]}$ is either a geodesic, or a corner with vertex $\alpha(t_i)$ and edges $\alpha \lvert_{[t_i - r/c, t_i]}$ and $\alpha \lvert_{[t_i, t_i + r/c]}$,
\end{enumerate}
where $c$ is the speed of $\alpha$. Here we extend $\alpha$ periodically to $\R$ so that neighborhoods of the vertices $\alpha(t_0) = \alpha(a)$ and $\alpha(t_m) = \alpha(b)$ can be described consistently without worrying about the parameter exceeding the interval $[a,b]$.
\end{Rem}

\begin{Thm}[Jordan theorem]\label{Thm: Jordan theorem}
Let $(X,\dd_X,\mm_X)$ be an $\RCD(K,N)$ space of essential dimension $2$. Let $\alpha \subset B_R(p)\cap \Int(X)$ be a simple polygonal loop. If the inclusion $B_R(p) \hookrightarrow X$ induces the trivial map on $\pi_1$, then $B_R(p) \setminus \alpha$ has exactly two components, and their topological boundaries in $B_R(p)$ both coincide with $\alpha$.

If, in addition, there exist $0 < s \le r < R$ such that $\alpha \subset B_s(p)$ and the inclusion $B_s(p) \hookrightarrow B_r(p)$ induces the trivial map on $\pi_1$, then one of the components of $B_R(p) \setminus \alpha$ is contained in $\overline B_r(p)$.
In this case, $X \setminus \alpha$ also has exactly two components, and at least one of them coincides with a component of $B_R(p) \setminus \alpha$ contained in $\overline B_r(p)$. Moreover, the topological boundaries of both components coincide with $\alpha$.
% {\color{red}Let $(X,\dd_X,\mm_X)$ be an $\RCD(K,N)$ space of essential
% dimension $2$, and let $\alpha:[0,1] \to X$ be a simple interior polygonal loop. Assume that $\alpha([0,1]) \subset B_R(p)$ for some $p \in X$ and $R > 0$. If every loop in $B_R(p)$ is contractible in $X$, then $B_R(p) \setminus \alpha([0,1])$ has exactly two components, and the topological boundary in $B_R(p)$ of each component coincides with $\alpha([0,1])$.

% If in addition there exist $0 < s \leq r < R$ such that
% $\alpha([0,1]) \subset B_s(p)$ and every loop in $B_s(p)$ is contractible in
% $B_r(p)$, then at least one of the components of
% $B_R(p) \setminus \alpha([0,1])$ is contained in $\overline B_r(p)$.
% In this case, $X \setminus \alpha([0,1])$ has exactly two path-connected
% components and at least one of these components coincides with a component of
% $B_R(p) \setminus \alpha([0,1])$ contained in $\overline B_r(p)$. Moreover, the topological boundaries of both components coincide with $\alpha([0,1])$.
% }
\end{Thm}

\begin{Rem}\label{Rem:simply-connectedness}
The assumption that every loop in $B_R(p)$ is contractible in $X$ is necessary. For instance, let
$X=\mathbb T^2$ be the flat two--torus, and let $p \in X$ and $R > 0$ be such that $X = B_R(p)$. There exist simple interior polygonal loops
$\alpha:[0,1] \to X$ which are not null-homotopic (for example, a closed geodesic loop in
a primitive homology class). Such a loop does not disconnect $X$; indeed,
$X\setminus \alpha$ is path-connected.
\end{Rem}

%\begin{Rem}\label{Rem: canon inside}
%In the second part of Theorem~\ref{Thm: Jordan theorem}, if there exists a point $x \in X \setminus \overline B_r(p)$, then there is a unique component of $B_R(p) \setminus \alpha([0,1])$ that is contained in $\overline B_{r}(p)$, and it coincides with the unique component of $X \setminus \alpha([0,1])$ contained in $\overline B_{r}(p)$. In particular,
%this allows one to canonically think of that component of
%$X\setminus \alpha([0,1])$ as the ``inside'' component, and the other as the ``outside'' component.
%\end{Rem}

Theorems~\ref{Thm: Jordan theorem}~and~\ref{Thm: semi-locally simply connected} immediately imply the following corollary.

\begin{Cor}\label{Cor: Jordan theorem}
Let $(X,\dd_X,\mm_X)$ be an $\RCD(K,N)$ space of essential dimension $2$. For every $p\in X$ and $R>0$, there exists $0<r(p,R)<R$ such that the following holds for every $0<r<r(p,R)$: If $\alpha\subset B_r(p)\cap \Int(X)$ is a simple polygonal loop, then $X \setminus \alpha$ has two components, and at least one of them is contained in $B_R(p)$. Moreover, the topological boundaries of both components coincide with $\alpha$.
\end{Cor}

The strategy is as follows. First, using the results developed in Sections~\ref{sec:choice_sides_geo} and \ref{Sec: corner}, we will construct a path-connected tubular neighborhood $U$ of $\alpha$, as a union of small balls centered along $\alpha$, such that $U \setminus \alpha$ has exactly two components. To do so, we will need to control the topology of small balls centered at points along $\alpha$. Proposition~\ref{Pro: good comp exist corner int} gives us control for small balls centered at the vertices that are locally corners, while Theorem~\ref{Thm: good components existence 2} gives us control for small balls centered at any other point along $\alpha$. 

By arguing as in Subsections~\ref{Subsec: geod neigh}~and~\ref{Subsec: corner neigh}, starting from the base point $\alpha(0)$ of $\alpha$ we can keep track of the local sides of each ball with respect to $\alpha$ that makes up the tubular neighborhood $U$ as we move along the loop until we return to the starting point. A priori, it could happen that the two local sides are exchanged after one turn, producing a M\"obius-type
twist. In that case, the set $U \setminus \alpha$ could be connected. We exclude this scenario by using the assumption that every loop in $B_R(p)$ is null-homotopic in $X$. This reduction amounts to a homological obstruction, carried out in
Section~\ref{subsec:Mobius}.

Once the local disconnection of $U\setminus \alpha$ is established, a Mayer--Vietoris argument as in Subsection~\ref{Subsec: MV extend corner} upgrades it to a disconnection statement for $B_R(p)\setminus \alpha$, thus proving the first part of Theorem~\ref{Thm: Jordan theorem}. Here again the assumption on the contractibility in $X$ of loops in $B_R(p)$ is used. This argument will be carried out in Subsection~\ref{Subsec: Jordan proof}. Finally, in Subsection~\ref{Subsec: Jordan bounded comp}, we prove the second part of Theorem~\ref{Thm: Jordan theorem}. 

\subsection{M\"obius-type twist}\label{subsec:Mobius}

Let $U$ be a path-connected topological space. Consider an open cover
$U = A\cup B$, where $A$ and $B$ are path-connected. Assume that $A\cap B$ has
exactly two components, denoted by $C_1$ and $C_2$.

\begin{Exa}
A basic example is $U=S^1$, where $A$ and $B$ are arcs whose intersection is the
union of two disjoint open intervals.
\end{Exa}

This configuration forces a nontrivial class in $H_1(U;\Z)$. Indeed, the
Mayer--Vietoris sequence for $U=A\cup B$ contains
\[
H_1(U) \xrightarrow{\;\partial_*\;} H_0(A\cap B)
\longrightarrow H_0(A)\oplus H_0(B)\longrightarrow H_0(U)\to 0.
\]
Since $A,B,U$ are path connected, $H_0(A)\simeq H_0(B)\simeq H_0(U)\simeq \Z$.
Moreover, $A\cap B=C_1\sqcup C_2$, hence $H_0(A\cap B)\simeq \Z^2$. The map
$H_0(A\cap B)\to H_0(A)\oplus H_0(B)$ sends both generators to $(1,1)$, so its
kernel is generated by $[p]-[q]$ for any $p\in C_1$, $q\in C_2$. By exactness,
$[p]-[q]$ lies in the image of $\partial_*$, hence there exists $[\gamma]\in
H_1(U;\Z)$ with $\partial_*([\gamma])=[p]-[q]$.

\begin{Lem}\label{Lem: mobius orientability}
Let $U=A\cup B$ and let $[\gamma]\in H_1(U;\Z)$ be as above. Let $V\subset U$ be
an open path-connected subset such that:
\begin{enumerate}[label=(\arabic*),leftmargin=2.2em]
\item $A':=A\cap V$ and $B':=B\cap V$ each have exactly two components,
say $A'=A_1'\sqcup A_2'$ and $B'=B_1'\sqcup B_2'$;
\item $A'\cap B'$ has exactly four components, namely
$A_i'\cap C_j$ for $i=1,2$ and $j=1,2$ (equivalently, the same four components can be
written as $B_i'\cap C_j$ in some order).
\end{enumerate}
Then the inclusion-induced map $\iota_*:H_1(V;\Z)\to H_1(U;\Z)$ is not surjective.
More precisely, $[\gamma]\notin \iota_*(H_1(V;\Z))$.
\end{Lem}

\begin{Exa}\label{Rem: mobius}
The motivating example for Lemma~\ref{Lem: mobius orientability} is the open
M\"obius strip. Let
\[
U:=\bigl([-1,1]\times(-1,1)\bigr)\big/\!\sim,
\qquad\text{where}\qquad
(-1,y)\sim(1,-y)\ \ \text{for all }y\in(-1,1).
\]
Set $A:=\bigl([-1,1]\setminus\{3/4\}\bigr)\times(-1,1)\big/\!\sim$ and
$B:=\bigl([-1,1]\setminus\{1/4\}\bigr)\times(-1,1)\big/\!\sim$. Then $U=A\cup B$,
and $A\cap B$ has exactly two components.

Let $\gamma$ be the center loop of $U$, and let $V:=U\setminus \gamma$.
One checks that $A':=A\cap V$ and $B':=B\cap V$ each have two path-connected
components, while $A'\cap B'$ has four components, as in
Lemma~\ref{Lem: mobius orientability}. Moreover, $[\gamma]\in H_1(U;\Z)$ is the generator, and $[\gamma]\notin \iota_*\bigl(H_1(V;\Z)\bigr)$.
\end{Exa}

\begin{proof}[Proof of Lemma~\ref{Lem: mobius orientability}]
Up to permuting $B_1'$ and $B_2'$, we may assume that $A_1'\cap C_1 \;=\; B_1'\cap C_1$.
Next, $B_1'\cap C_2$ is a component of $(A'\cap B')\cap C_2$.
Since $(A'\cap B')\cap C_2$ has exactly two components, namely $A_1'\cap C_2$
and $A_2'\cap C_2$, we must have
\[
B_1'\cap C_2 = A_1'\cap C_2 \qquad \text{or} \qquad B_1'\cap C_2 = A_2'\cap C_2.
\]

Suppose first that $B_1'\cap C_2=A_1'\cap C_2$. Then $B_1'$ meets $A'$ only
through $A_1'$, and similarly $B_2'$ meets $A'$ only through $A_2'$. Consequently,
\[
(A_1'\cup B_1')\cap (A_2'\cup B_2')=\emptyset
\quad\text{and}\quad
V=(A_1'\cup B_1')\ \cup\ (A_2'\cup B_2').
\]
In particular, $V$ is disconnected which contradicts our assumptions.
Therefore, we must be in the remaining alternative
$B_1'\cap C_2=A_2'\cap C_2$, and hence the incidences are necessarily
\begin{enumerate}[label=(\arabic*),leftmargin=2.2em]
\item $A_1' \cap C_1 = B_1' \cap C_1$;
\item $B_1' \cap C_2 = A_2' \cap C_2$;
\item $A_2' \cap C_1 = B_2' \cap C_1$;
\item $B_2' \cap C_2 = A_1' \cap C_2$.
\end{enumerate}
Now, consider the Mayer--Vietoris sequence for $V = A' \cup B'$:
\begin{equation*}
  \dots \longrightarrow H_1(V) {\longrightarrow} H_0(A' \cap B') {\longrightarrow} H_0(A')\oplus H_0(B')\longrightarrow H_0(V)\longrightarrow 0 .
\end{equation*}
Choose points
\[
p_1\in A_1'\cap C_1,\qquad p_2\in A_2'\cap C_1,\qquad
q_1\in A_1'\cap C_2,\qquad q_2\in A_2'\cap C_2,
\]
one in each of the four components of $A'\cap B'$. Suppose, toward
a contradiction, that there exists $[\tilde\gamma]\in H_1(V;\Z)$ such that
$\iota_*([\tilde\gamma])=[\gamma]\in H_1(U;\Z)$.
Consider the Mayer--Vietoris boundary maps $\partial_*^V:H_1(V;\Z)\to H_0(A'\cap
B';\Z)$ and $\partial_*^U:H_1(U;\Z)\to H_0(A\cap B;\Z)$. By naturality of the
Mayer--Vietoris sequence with respect to the inclusion $V\hookrightarrow U$, the
diagram with boundary maps commutes, hence
\[
j_*\bigl(\partial_*^V([\tilde\gamma])\bigr)=\partial_*^U\bigl(\iota_*([\tilde\gamma])\bigr)
=\partial_*^U([\gamma])=[p]-[q],
\]
where $j_*:H_0(A'\cap B';\Z)\to H_0(A\cap B;\Z)$ is induced by inclusion and
$p\in C_1$, $q\in C_2$ are fixed basepoints as above.

Since $A'\cap B'$ has four components, the group $H_0(A'\cap B';\Z)$
is free abelian on the classes $[p_1],[p_2],[q_1],[q_2]$. Moreover,
$j_*([p_1])=j_*([p_2])=[p]$ and $j_*([q_1])=j_*([q_2])=[q]$. Therefore
$\partial_*^V([\tilde\gamma])$ must have total coefficient $+1$ on $C_1$ and total
coefficient $-1$ on $C_2$, i.e.\ it must be of the form
\[
\partial_*^V([\tilde\gamma])
=(m+1)[p_1]-m[p_2]-(n+1)[q_1]+n[q_2]
\]
for some $m,n\in\Z$.

Let $j_1':A'\cap B'\hookrightarrow A'$ and $j_2':A'\cap B'\hookrightarrow B'$ be
the inclusion maps. Using the incidence relations \emph{(1)--(4)} above, we have:
\begin{enumerate}[label=(\arabic*),leftmargin=2.2em]
\item $(j_1')_*([p_1])=(j_1')_*([q_1])$ and $(j_1')_*([p_2])=(j_1')_*([q_2])$,
since $p_1,q_1\in A_1'$ and $p_2,q_2\in A_2'$;
\item $(j_2')_*([p_1])=(j_2')_*([q_2])$ and $(j_2')_*([p_2])=(j_2')_*([q_1])$,
since $p_1,q_2\in B_1'$ and $p_2,q_1\in B_2'$.
\end{enumerate}
Therefore,
\begin{align*}
(j_1')_*\bigl(\partial_*([\tilde\gamma])\bigr)
&=(m+1)(j_1')_*([p_1]) - m(j_1')_*([p_2]) -(n+1)(j_1')_*([q_1]) + n(j_1')_*([q_2])\\
&=(m-n)\Bigl((j_1')_*([p_1])-(j_1')_*([p_2])\Bigr),
\end{align*}
and
\begin{align*}
(j_2')_*\bigl(\partial_*([\tilde\gamma])\bigr)
&=(m+1)(j_2')_*([p_1]) - m(j_2')_*([p_2]) -(n+1)(j_2')_*([q_1]) + n(j_2')_*([q_2])\\
&=(m+n+1)(j_2')_*([p_1])-(m+n+1)(j_2')_*([p_2])\\
&=(m+n+1)\Bigl((j_2')_*([p_1])-(j_2')_*([p_2])\Bigr).
\end{align*}

On the other hand, by exactness of the Mayer--Vietoris sequence for $V=A'\cup B'$,
we have $(j_1')_*\!\circ\partial_* = 0$ and $(j_2')_*\!\circ\partial_* = 0$.
Since $A'$ has two components, $(j_1')_*([p_1])\neq (j_1')_*([p_2])$,
and similarly $(j_2')_*([p_1])\neq (j_2')_*([p_2])$. Hence the two identities
above force $m-n=0$ and $m+n+1=0$, which is impossible for integers $m,n$. This completes the argument. 
\end{proof}

It is well known that the open M\"obius strip cannot be homeomorphic to an open
subset of a simply connected $2$-manifold, by orientability. The next
proposition generalizes this observation: any space $U$ which ``looks like a
M\"obius strip'' in the sense of Lemma~\ref{Lem: mobius orientability} (see
Remark~\ref{Rem: mobius}) cannot be homeomorphic to open subset of a simply connected
space. More precisely, we will show that if $U$ is homeomorphic to an open subset of some space $X$, then there exists a loop in $U$ that is not contractible in $X$. 

In the sequel, we will apply this proposition to a neighborhood $U$ of a polygonal loop $\alpha$ and $V=U\setminus \alpha$, thereby excluding a M\"obius-type twist under the assumption that every loop in $U$ is contractible in $X$. 

\begin{Pro}\label{Pro: mobius nbd no embed}
Let $X$ be a path-connected topological space and let $V\subset U\subset X$ be
open subsets. Assume that $U\setminus V$ is closed in $X$, and that $U$ and $V$
satisfy the assumptions of Lemma~\ref{Lem: mobius orientability}. Then there exists a loop in $U$ that is not contractible in $X$.  
\end{Pro}

\begin{proof}
Set $F:=U\setminus V$. By assumption, $F$ is closed in $X$, hence $X\setminus F$
is open. Since $V=U\cap (X\setminus F)$, we obtain an open cover
\[
X=(X\setminus F)\ \cup\ U,
\qquad\text{with}\qquad
(X\setminus F)\cap U = V.
\]
Applying Mayer--Vietoris to this cover yields the segment
\[
H_1(V;\Z)\xrightarrow{\ \Phi\ } H_1(X\setminus F;\Z)\oplus H_1(U;\Z)
\longrightarrow H_1(X;\Z),
\]
where $\Phi$ is induced by the inclusions $V\hookrightarrow X\setminus F$ and
$V\hookrightarrow U$.

By Lemma~\ref{Lem: mobius orientability}, the inclusion-induced map
$\iota_*:H_1(V;\Z)\to H_1(U;\Z)$ is not surjective. Choose $[\gamma]\in H_1(U;\Z)$
with $[\gamma]\notin \iota_*(H_1(V;\Z))$, and consider the element $(0,[\gamma])$
in $H_1(X\setminus F;\Z)\oplus H_1(U;\Z)$. Then $(0,[\gamma])\notin \mathrm{Im}(\Phi)$,
because its second component is not in $\iota_*(H_1(V;\Z))$. Therefore, the map
$H_1(X\setminus F;\Z)\oplus H_1(U;\Z)\to H_1(X;\Z)$ cannot be the zero map, and
in particular $H_1(X;\Z)\neq 0$. 

More precisely, denoting by $j: U \hookrightarrow X$ the inclusion map, the above argument shows that $j_*([\gamma]) \neq 0 \in H_1(X; \Z)$, where $j_*: H_1(U; \Z) \to H_1(X; \Z)$ is map induced by $j$. Since $H_1(X;\Z)$ is the
abelianization of $\pi_1(X)$ (by the Hurewicz theorem), it follows that $\gamma$
represents a nontrivial element of $\pi_1(X)$. In particular, $\gamma$ is not
contractible in $X$.
\end{proof}

\subsection{Proof of the Jordan Theorem \ref{Thm: Jordan theorem}: local separation}\label{Subsec: Jordan proof}
In this section, we prove the following theorem, which corresponds to the first part of Theorem~\ref{Thm: Jordan theorem}. 

\begin{Thm}[Jordan theorem I.]\label{Thm: Jordan theorem 1}
Let $(X,\dd_X,\mm_X)$ be an $\RCD(K,N)$ space of essential dimension $2$. Let $\alpha \subset B_R(p)\cap \Int(X)$ be a simple polygonal loop. If the inclusion $B_R(p) \hookrightarrow X$ induces the trivial map on $\pi_1$, then $B_R(p) \setminus \alpha$ has two components, whose topological boundaries in $B_R(p)$ both coincide with $\alpha$.
\end{Thm}

As anticipated, we will prove Theorem~\ref{Thm: Jordan theorem 1} by first constructing a path-connected neighborhood of $\alpha$ that is disconnected into exactly two components by $\alpha$. 
\begin{Pro}\label{Pro: poly loop neigh}
Let $(X,\dd_X,\mm_X)$ be an $\RCD(K,N)$ space of essential
dimension $2$. Let $\alpha \subset B_R(p)\cap \Int(X)$ be a simple polygonal loop. If the inclusion $B_R(p) \hookrightarrow X$ induces the trivial map on $\pi_1$, then there exists an open path-connected neighborhood $U \subset B_R(p)$ of $\alpha$ such that $U \setminus \alpha$ has exactly two components $U_\pm$. Moreover, the topological boundaries of $U_\pm$ in $U$ coincide with $\alpha$.
\end{Pro}

\begin{proof}
By rescaling, we may assume that $\alpha$ is parameterized on $[0,1]$ and unit-speed. Let $0 = t_0 < \dots < t_m = 1$ be a partition of $[0,1]$ such that $\alpha \lvert_{[t_{i-1},t_{i}]}$ is a regular geodesic for each $i = 1, \dots, m$. In what follows, we extend $\alpha$ periodically to $\R$ as needed. Similarly, we set $t_{-1} := t_{m-1}-1$ and $t_{m+1} := t_{1}+1$, so that $\alpha(t_{-1})$ and $\alpha(t_{m+1})$ are adjacent to $\alpha(t_0) = \alpha(t_m)$. 

For each vertex $\alpha(t_i)$, either the concatenation of its two incident edges $\alpha \lvert_{[t_{i-1}, t_{i+1}]}$ is a corner, or $\alpha \lvert_{[t_i-\veps, t_i+\veps]}$ is a geodesic for some $\veps > 0$. By adding the points $t_i \pm \veps$ to our partition $\{t_i\}_{i=0}^m$ of $[0,1]$ for each $i$ in the second case, we may assume that, for each vertex $\alpha(t_i)$, the concatenation of its two incident edges $\alpha \lvert_{[t_{i-1}, t_{i+1}]}$ is either a corner or a geodesic.  

For each $t \in [0,1]$, we choose $r(t) > 0$ as follows:
\begin{enumerate}
    \item If $\alpha(t)$ is in the interior of an edge then $B_{10r(t)}(\alpha(t))$ intersects $\alpha$ only along that edge, and the intersection is a geodesic of length $20r(t)$. If $\alpha(t)$ is a vertex then $B_{10r(t)}(\alpha(t))$ intersects $\alpha$ along the two edges incident to $\alpha(t)$, and the intersection is either a geodesic of length $20r(t)$ or a corner with vertex $\alpha(t)$ and two edges of length $10r(t)$.
    \item $B_{r(t)}(\alpha(t))$ has two good components with respect to the edge containing $\alpha(t)$, in the case where $\alpha(t)$ is in the interior of an edge, or to the geodesic or corner formed by the two edges incident to $\alpha(t)$, in the case where $\alpha(t)$ is a vertex. 
    \item  $B_{r(t)}(\alpha(t)) \subset B_R(p)$.
\end{enumerate}
It follows from the simpleness of $\alpha$, Theorem~\ref{Thm: good components existence 2}, Proposition~\ref{Pro: good comp exist corner int} that such an $r(t)$ exists for each $t$. We take $r(0) = r(1)$.

Combining (1) and (2), we see that each $B_{r(t)}(\alpha(t)) \setminus \alpha$ has exactly two components. Arguing as in the proof of Lemma~\ref{Lem: geod neigh side choice}, using Lemmas~\ref{Lem: good comp side induct} and \ref{Lem: good comp corner side induct}, we can label the two components of $B_{r(t)}(\alpha(t)) \setminus \alpha$ by $G_\pm(t)$ for each $t \in [0, t_1]$ so that the following hold:
\begin{enumerate}
    \item[(a)] For all $t,t' \in [0, t_1]$, \[G_+(t) \cap G_-(t') = \emptyset \quad \text{and} \quad G_-(t) \cap G_+(t') = \emptyset.\]
    \item[(b)] If $t, t' \in [0,t_1]$ are such that $B_{r(t)}(\alpha(t)) \cap B_{r(t')}(\alpha(t')) \neq \emptyset$,
    then
    \[
    G_+(t) \cap G_+(t') \neq \emptyset \quad \text{and} \quad G_-(t) \cap G_-(t') \neq \emptyset.
    \]
\end{enumerate}
Set 
\[
A := \bigcup_{t \in [0, t_1]} B_{r(t)}(\alpha(t)) \quad \text{and} \quad A_\pm := \bigcup_{t \in [0,t_1]} G_\pm(t).
\]
Arguing as in the proof of Proposition~\ref{Pro: geod neigh}, we see that $A$ is path-connected. Moreover, $A \setminus \alpha$ has exactly two components, namely $A_\pm$.

Similarly label the two components of $B_{r(t)}(\alpha(t)) \setminus \alpha([0,1])$ by $H_\pm(t)$ for $t \in [t_1,1]$ so that properties analogous to (a) and (b) hold. Up to permuting $H_\pm(t)$ for all $t \in [t_1,1]$, we may in addition assume that $H_\pm(t_1) = G_\pm(t_1)$. Set
\[
B := \bigcup_{t \in [t_1,1]} B_{r(t)}(\alpha(t)) \quad \text{and} \quad B_\pm := \bigcup_{t \in [t_1,1]} H_\pm(t).
\]
Then $B$ is path-connected and $B \setminus \alpha$ has exactly two components, namely $B_\pm$.

By property~(1) in the definition of $r(t)$, the balls $B_{r(0)}(\alpha(0))$ and $B_{r(t_1)}(\alpha(t_1))$ are disjoint, and for every $t \in (0, t_1)$ and $t' \in (t_1, 1)$, the balls $B_{r(t)}(\alpha(t))$ and $B_{r(t')}(\alpha(t'))$ are disjoint. Therefore,
\[
A \cap B = B_{r(0)}(\alpha(0)) \sqcup B_{r(t_1)}(\alpha(t_1)).
\]
It follows that
\[
\big(A \cap B\big) \setminus \alpha = \big(B_{r(0)}(\alpha(0)) \sqcup B_{r(t_1)}(\alpha(t_1))\big) \setminus \alpha
\]
has exactly four components, namely $G_\pm(0)$ and $G_\pm(t_1)$.

We claim that $G_\pm(0) = H_\pm(1)$. Assume that this is not the case. Then $G_+(0) = H_-(1)$ and $G_-(0) = H_+(1)$. It follows that
\begin{align}\label{Eq: poly loop neigh 1}
    \begin{split}
    &A_+ \cap B_+ = G_+(t_1) \neq \emptyset, \quad B_+ \cap A_- = G_-(0) \neq \emptyset,\\
    &A_- \cap B_- = G_-(t_1) \neq \emptyset, \quad B_- \cap A_+ = G_+(0) \neq \emptyset.
    \end{split}
\end{align}

Set 
\[
U := \bigcup_{t \in [0, 1]} B_{r(t)}(\alpha(t)).
\]
Then $U \subset B_R(p)$ by our choice of $r(t)$, and $U = A \cup B$
is path-connected, being the union of the two path-connected sets with non-empty intersection. Moreover,
\[U \setminus \alpha = \big(A \setminus \alpha\big)  \cup \big(B \setminus \alpha\big) = A_+ \cup A_- \cup B_+ \cup B_-\]
is also path-connected, since $A_\pm$ and $B_\pm$ are path-connected and the relations \eqref{Eq: poly loop neigh 1} hold. 

It follows that $U$, $A$, and $B$ satisfy the conditions of Lemma~\ref{Lem: mobius orientability}, with $U \setminus \alpha$ playing the role of $V$ in the lemma. By Proposition~\ref{Pro: mobius nbd no embed}, there exists a loop in $U$ that is not contractible in $X$. Since $U \subset B_R(p)$, this contradicts our assumption. 

Therefore, $G_\pm(0) = H_\pm(1)$. This implies that
\begin{align}\label{Eq: poly loop neigh 2}
    \begin{split}
        &A_+ \cap B_+ = G_+(0) \sqcup G_+(t_{1}), \quad A_+ \cap B_- = \emptyset\\
        &A_- \cap B_- = G_-(0) \sqcup G_-(t_{1}), \quad A_- \cap B_+ = \emptyset.
    \end{split}
\end{align}
Set $U_\pm:= A_\pm \cup B_\pm$. Since $A_\pm$ and $B_\pm$ are open and path-connected, and have non-empty intersection by \eqref{Eq: poly loop neigh 2}, the sets $U_\pm$ are open and path-connected. Moreover, \eqref{Eq: poly loop neigh 2} implies that $U \setminus \alpha = U_+ \sqcup U_-$. Therefore, $U_\pm$ are also closed in $U \setminus \alpha$. By Lemma~\ref{Lem: clopen path-connected}, the set $U \setminus \alpha$ has exactly two components, namely $U_\pm$.

Finally, we check the claim on the topological boundaries of $U_\pm$ in $U$. It is immediate from our choice of $r(t)$ and the construction that $\alpha$ is contained in the topological boundaries of $U_\pm$ in $U$. Since $U_\pm$ are clopen in $U \setminus \alpha$, the reverse inclusion follows. 
\end{proof}

%\begin{Rem}\label{Rem: poly loop neigh}
%Under the hypotheses of Proposition~\ref{Pro: poly loop neigh}, suppose that we are given an open set $V$ with $\alpha \subset V \subset B_R(p).$ Then we can construct $U$ satisfying the conclusions of the proposition such that, in addition, $U \subset V$. This can be achieved by choosing $r(t)$ sufficiently small so that $B_{r(t)}(\alpha(t)) \subset V$ for all $t \in [0,1]$.
%\end{Rem}

\begin{proof}[Proof of Theorem~\ref{Thm: Jordan theorem 1}]
Let $U$ be a neighborhood of $\alpha$ as in Proposition~\ref{Pro: poly loop neigh} and set
$V := X \setminus \alpha$. The theorem follows by considering the Mayer--Vietoris sequence for the cover $X = U \cup V$ and arguing as in the proof of Proposition~\ref{Pro: comp exist extend corner quant}.
\end{proof}

\subsection{Proof of the Jordan Theorem \ref{Thm: Jordan theorem}: global separation}\label{Subsec: Jordan bounded comp} 
Let us consider the following configuration: Let $(X, \dd_X, \mm_X)$ be an $\RCD(K,N)$ space of essential dimension $2$. Let $p \in \Int(X)$ and let $r > 0$. Let $\alpha, \beta: [-1,1] \to X$ be curves with $\alpha(0) = \beta(0) = p$, each of which is either a geodesic of length $2r$ or a corner with vertex $p$ and two edges of length $r$. 
Suppose further that $\alpha$ and $\beta$ intersect only at $p$, and that the following hold:
\begin{enumerate}
    \item[(i)] $B_r(p) \setminus \alpha$ has exactly two components $A_\pm$, and the topological boundaries of $A_\pm$ in $B_r(p)$ coincide with $\alpha((-1,1))$.
    \item[(ii)] $B_r(p) \setminus \beta$ has exactly two components $B_\pm$, and the topological boundaries of $B_\pm$ in $B_r(p)$ coincide with $\beta((-1,1))$.  
\end{enumerate}
Since $\alpha((-1,0))$ and $\alpha((0,1))$ do not intersect $\beta$, each of them must be contained in one of $B_\pm$. Similarly, $\beta((-1,0))$ and $\beta((0,1))$ must each be contained in one of $A_\pm$. 

\begin{Lem}\label{Lem: local config}
The sets $\alpha((-1,0))$ and $\alpha((0,1))$ lie in the same component of $B_r(p) \setminus \beta$ if and only if the sets $\beta((-1,0))$ and $\beta((0,1))$ lie in the same component of $B_r(p) \setminus \alpha$.
\end{Lem}

\begin{proof}
Assume that $\alpha((-1,0))$ and $\alpha((0,1))$ are both contained in the same component of $B_r(p) \setminus \beta$, say $B_+$. Since $B_-$ is path-connected and does not intersect $\alpha$, it must contained in one of the components $A_\pm$. Up to relabeling $A_\pm$, we may assume that $B_- \subset A_-$.

By (ii), $\beta((-1,1))$ is contained in the closure of $B_-$ in $B_r(p)$. Therefore, the set
\[\beta((-1,0) \cup (0,1)) = \beta((-1,1)) \cap \big(B_r(p) \setminus \alpha\big)\]
is contained in the closure of $B_-$ in $B_r(p) \setminus \alpha$. Since $A_-$ is closed in $B_r(p) \setminus \alpha$ and $B_- \subset A_-$, we have $\beta((-1,0) \cup (0,1)) \subset A_-.$ The same argument works with $\alpha$ and $\beta$ reversed. 
\end{proof}

%\begin{Rem}\label{Rem: geod geod config}
%Using Lemma~\ref{Lem: geod no glance}, it is not difficult to check that in the case where $\alpha$ and $\beta$ are both geodesics, then $\alpha((-1,0))$ and $\alpha((0,1))$ \emph{must} lie in different components of $B_r(p) \setminus \beta$ and $\beta((-1,0))$ and $\beta((0,1))$ \emph{must} lie in different components of $B_r(p) \setminus \alpha$. Since we will not need this stronger conclusion in the upcoming application of Lemma~\ref{Lem: local config}, we do not emphasize this point in the lemma. 
%\end{Rem}

The following theorem corresponds to the second part of Theorem~\ref{Thm: Jordan theorem}.

\begin{Thm}[Jordan theorem II.]\label{Thm: Jordan theorem 2}
Under the assumptions of Theorem~\ref{Thm: Jordan theorem 1}, if in addition there exist $0 < s \leq r < R$ such that $\alpha \subset B_s(p)$ and the inclusion $B_s(p) \hookrightarrow B_r(p)$ induces the trivial map on $\pi_1$, then at least one of the components of $B_R(p) \setminus \alpha$ is contained in $\overline B_r(p)$. In this case, $X \setminus \alpha$ has exactly two components and at least one of these components coincides with a component of $B_R(p) \setminus \alpha$ contained in $\overline B_r(p)$. Moreover, the topological boundaries of both components coincide with $\alpha$.
\end{Thm}

\begin{proof}
We prove the theorem under the additional assumption that $p \in \Int(X)$. The general case follows by an approximation argument, since $\Int(X)$ is dense in $X$ by Theorem~\ref{Thm: int set open}.

Let $\tilde A_\pm$ the two components of $B_R(p) \setminus \alpha$, as given by Theorem~\ref{Thm: Jordan theorem 1}. Since $\alpha \subset B_s(p)$ and the inclusion $B_s(p) \hookrightarrow X$ also induces the trivial map on $\pi_1$, Theorem~\ref{Thm: Jordan theorem 1} implies that $B_{s}(p) \setminus \alpha$ has two components $A_\pm$. Moreover, $\alpha(0)$ lies in the topological boundary in $B_R(p)$ of $\tilde A_\pm$, and so $\tilde A_\pm$ intersect $B_s(p)$. By Lemma~\ref{Lem: big set small set comp}, up to relabeling $A_\pm$, we have $A_\pm \subset \tilde A_\pm$. 

Suppose for the sake of contradiction that there exist $x_\pm \in A_\pm$ such that $\dd_X(x_\pm, p) > r$. By perturbing $x_\pm$ using Lemma~\ref{Lem: perturb generic int}, we may assume that the geodesics from $x_\pm$ to $p$ are unique, regular, and intersect $\alpha$ only at finitely many points. 

Let $x_+'$ be the point on the geodesic from $x_+$ to $p$ such that $\dd_X(x_+', p) = r.$ Since the segment from $x_+$ to $x_+'$ is outside of $B_r(p)$, it does not intersect $\alpha$. Therefore, $x_+$ and $x_+'$ are in the same component of $B_R(p) \setminus \alpha$, namely $\tilde A_+$. Similarly, let $x_-'$ be the point on the geodesic from $x_-$ to $p$ such that $\dd_X(x_-', p) = r$. Then $x_-' \in \tilde A_-$. We replace $x_\pm$ by $x_\pm'$.  

Since $x_\pm \in \tilde A_\pm$, we have $x_+ \neq x_-$. Since $\dd_X(x_\pm, p) = r$, it follows by the non-branching property (Theorem~\ref{Thm: non-branching}) that the geodesics from $x_\pm$ to $p$ intersect only at $p$. 

Let $\beta:[-1,1] \to X$ be the polygonal curve with vertex $\beta(0) = p$ obtained by concatenating the geodesic from $x_+$ to $p$ with the geodesic from $p$ to $x_-$. Then either $\beta$ is a corner, or there exists $\veps > 0$ such that $\beta \lvert_{[-\veps, \veps]}$ is a geodesic. In the first case, Proposition~\ref{Pro: comp exist corner int quant} implies that $B_s(p) \setminus \beta$ has exactly two components. In the second case, the same conclusion follows by minor variants of Propositions~\ref{Pro: corner neigh}~and~\ref{Pro: comp exist extend corner quant}. In either case, we denote the two components by $B_\pm$, and note that the topological boundaries of $B_\pm$ in $B_s(p)$ coincide with $\beta \cap B_s(p)$. 

To facilitate the next part of the argument, we introduce the following terminology. We say that \emph{$\beta$ changes sides with respect to $\alpha$ (in $B_s(p)$) at $\beta(t)$}, for $t \in [-1,1]$, if there exists $\veps > 0$ such that $\beta((t-\veps, t))$ is contained in one of $A_\pm$ and $\beta((t, t+\veps))$ in the other. We also define the analogous notion for $\alpha$ with respect to $\beta$. It is evident that $\alpha$ can only change sides with respect to $\beta$, and vice versa, at an intersection point of $\alpha$ and $\beta$. 

%Similarly, we say that \emph{$\alpha$ changes sides with respect to $\beta$ (in $B_s(p)$) at $\alpha(t)$}, for some $t \in [0,1]$, if there exists $\veps > 0$ such that either
%\begin{enumerate}
      %  \item $\alpha((t-\veps, t)) \subset B_-$ and $\alpha((t, t+\veps)) \subset B_+$, or
      %  \item $\alpha((t-\veps, t)) \subset B_+$ and $\alpha((t, t+\veps)) \subset B_-$.
%\end{enumerate}

\begin{Cla}\label{Cla: side change}
Let $q := \alpha(t) = \beta(t')$ be an intersection point of $\alpha$ and $\beta$. Then $\alpha$ changes sides with respect to $\beta$ at $\alpha(t)$ if and only if $\beta$ changes sides with respect to $\alpha$ at $\beta(t')$.
\end{Cla}

\begin{proof}
There are four cases, depending on whether each of $\alpha$ and $\beta$ is
locally a geodesic or a corner at $q$. We consider only the case where $\alpha$ is locally a geodesic at $q$ and $\beta$ is locally a corner at $q$; the other cases are analogous. 

In this case, for all $s' > 0$ sufficiently small, $\gamma := \alpha \cap \overline B_{s'}(q)$ is a geodesic of length $2s'$ passing through $q$, and $\nu:=\beta \cap \overline B_{s'}(q)$ is a corner with vertex $q$ and two edges of length $s'$. Parameterize $\gamma$ and $\nu$ with constant speed on on $[-1,1]$ so that $\gamma(0) = \nu(0) = q$.

By Theorem~\ref{Thm: good components existence 2}, if $s' > 0$ sufficiently small, then the set $B_{s'}(q) \setminus \alpha = B_{s'}(q) \setminus \gamma$ has exactly two components, and the topological boundaries of both components in $B_{s'}(q)$ coincide with $\gamma((-1,1))$. Let $G_\pm$ denote these two components. 

Since $q = \alpha(t)$ lies in the topological boundaries of $A_\pm$ in $B_s(p)$, both $A_\pm$ intersect $B_{s'}(q)$. By choosing $s' > 0$ sufficiently small, we may assume that $B_{s'}(q) \subset B_{s}(p)$. Thus Lemma~\ref{Lem: big set small set comp} implies $G_\pm \subset A_\pm$, up to relabeling $G_\pm$.

Similarly, Theorem~\ref{Thm: good components existence corner} implies that if $s' > 0$ is sufficiently small, then the set $B_{s'}(q) \setminus \beta = B_{s'}(q) \setminus \nu$ has exactly two components, and the topological boundaries of both components in $B_{s'}(q)$ coincide with $\nu((-1,1))$. Let $H_\pm$ denote these components. Arguing as above, up to relabeling $H_\pm$, we have $H_\pm \subset B_\pm$.

Since $\alpha$ and $\beta$ have finite intersection, by choosing $s' > 0$ sufficiently small, we may assume that $\gamma$ and $\nu$ intersect only at $q$. By Lemma~\ref{Lem: local config}, $\gamma((-1,0))$ and $\gamma((0,1))$ are contained in the same one of $H_\pm$ if and only if $\nu((-1,0))$ and $\nu((0,1))$ are contained in the same one of $G_\pm$.  In view of the inclusion relations above, this is equivalent to saying that $\alpha$ changes sides with respect to $\beta$ at $\alpha(t)$ if and only if $\beta$ changes sides with respect to $\alpha$ at $\beta(t')$.
\end{proof}

Since $\alpha$ is a loop and intersects $\beta$ only finitely many times, it must change sides with respect to $\beta$ an even number of times. By Claim~\ref{Cla: side change}, $\beta$ must also change sides with respect to $\alpha$ an even number of times. Since $A_\pm \subset \tilde A_\pm$, this means $x_\pm = \beta(\pm 1)$ must both be contained in the same one of $\tilde A_\pm$, contradicting our assumption that $x_\pm \in A_\pm$.

Therefore, all points $x \in B_R(p)$ such that $\dd_X(x, p) > r$ lie in the same component of $B_R(p) \setminus \alpha$. Up to relabeling $\tilde A_\pm$, we may assume that this component is $\tilde A_-.$ Then $\tilde A_+ \subset \overline B_r(p)$.

We now turn the topology of $X \setminus \alpha$. We first show that $X \setminus \alpha$ has at most two components. Let $x \in X \setminus \alpha$, and let $\rho:[0,1] \to X$ be a geodesic from $x$ to $\alpha(0)$. There is a smallest $t \in (0,1]$ such that $\rho(t) \in \alpha \subset B_R(p)$. By continuity, there exists $t' \in (0,t)$ such that $\rho(t') \in B_R(p)$. Since $\rho(t') \in B_R(p) \setminus  \alpha$, it lies in one of $\tilde A_\pm$. Therefore, $\rho \lvert_{[0,t']}$ is a curve in $X \setminus \alpha$ from $x$ to one of $\tilde A_\pm$. Since $x$ is arbitrary and each of $\tilde A_\pm$ is path-connected, $X \setminus \alpha$ has at most two components. Moreover, if it has two components, then $\tilde A_\pm$ lie in different ones. 

Next, we show that $X \setminus \alpha$ has at least two components. Fix any two points $x_\pm \in \tilde A_\pm$. We claim that every curve from $x_+$ to $x_-$ must intersect $\alpha$. Assume for the sake of contradiction that there is a curve $\sigma:[0,1] \to X$ from $x_+$ to $x_-$ that does not intersect $\alpha$. Since $\tilde A_+$ is open, there exists a smallest $t \in (0,1]$ such that $\sigma(t) \notin \tilde A_+$. Therefore, $\sigma([0,t)) \subset \tilde A_+ \subset \overline B_{r}(p)$. 

By the definition of $t$, we have $\sigma(t)\notin \tilde A_+$. Moreover,
$\sigma(t)\notin \alpha$ since $\sigma$ does not intersect $\alpha$. By the continuity of $\sigma$, we have $\sigma(t) \in \overline B_r(p) \subset B_R(p)$. 
Since $B_R(p) = \tilde A_+ \sqcup \tilde A_- \sqcup \alpha$, it follows that $\sigma(t)$ must lie in $\tilde A_-$. Thus $\sigma \lvert_{[0,t]}$ is a curve in $B_R(p) \setminus \alpha$ from a point in $\tilde A_+$ to a point in $\tilde A_-$, yielding a contradiction. This proves the previous claim. 

Therefore, $X \setminus \alpha$ has at least two components. Since we have already shown that $X \setminus \alpha$ has at most two components, it has exactly two. Denote by $X_\pm$ the component containing $\tilde A_\pm$. We claim that 
\[
X_+ = \tilde A_+ \quad \text{and} \quad X_- = \tilde A_- \sqcup (X \setminus B_R(p)).
\]
Let $x \in X \setminus B_R(p)$, and let $\rho:[0,1] \to X$ be a geodesic from $x$ to $p$. Then there exists $t > 0$ such that $\rho(t) \in B_R(p)$ and $\rho([0,t]) \cap \overline B_r(p) = \emptyset$. 
Since $\tilde A_- \cup \alpha \subset \overline B_r(p)$ and $\rho(t) \notin \overline B_r(p)$, it must be that that $\rho(t) \in \tilde A_-$. Thus $\rho \lvert_{[0,t]}$ is a curve from $x$ to $\tilde A_-$ that does not intersect $\alpha$, and so $x \in X_-$. The claim follows. 

Finally, since the topological boundaries of $\tilde A_\pm$ in $B_R(p)$ coincide with $\alpha$, and $\tilde A_\pm \subset X_\pm$, the topological boundaries of $X_\pm$ must also contain $\alpha$. Since $X_\pm$ are clopen in $X \setminus \alpha$, the reverse inclusion follows.
\end{proof}

\section{Polygonal Domains}\label{Sec: poly dom}
In this section, we introduce the notion of an interior polygonal domain (see Definition~\ref{Def: poly dom}) and show that every interior point admits a neighborhood basis consisting of these domains. In Section~\ref{Sec: man struct}, we will show that interior polygonal domains are homeomorphic to open disks, which will allow us to deduce that the interior set has the structure of a topological manifold.

For the remainder of the section, let $(X,\dd_X,\mm_X)$ be an $\RCD(K,N)$ space of essential dimension $2$.

\begin{Def}[Interior polygonal domain]\label{Def: poly dom}
We say that a subset $Q \subset X$ is an \emph{interior polygonal domain} if the following conditions hold:
\begin{enumerate}
    \item $Q$ is open, connected, and precompact.
    
    \item The topological boundary $\partial Q:= \overline{Q}\setminus Q$ is the support of a simple polygonal loop $\alpha_Q$.

    \item The closure $\overline Q$ is contained in the interior set $\Int(X)$ (see Definition~\ref{Def: ext-int points}).
\end{enumerate}
\end{Def}

Since $\RCD$ spaces are locally path-connected, interior polygonal domains are automatically path-connected.

The main tool for proving the existence of interior polygonal domains is Jordan Theorem~\ref{Thm: Jordan theorem}. Let $\alpha\subset B_r(p)\cap \Int(X)$ be a simple polygonal loop. If the inclusions $B_{R} \hookrightarrow X$ and $B_r(p)\hookrightarrow B_{R/2}(p)$ induce trivial maps on $\pi_1$ for some $R\geq 2r$, and there exists a point $q\in X\setminus B_R(p)$, then, by Jordan Theorem~\ref{Thm: Jordan theorem}, the complement $X\setminus \alpha$ has exactly two components. The component contained in $B_R(p)$ is an interior polygonal domain $Q$ with $\alpha_Q=\alpha$.

\begin{Rem}\label{Rem: polydom consistency1}
Conversely, if $Q$ is an interior polygonal domain such that $\overline Q\subset B_r(p)$ and the geometric assumptions above are satisfied, then $\alpha_Q$ disconnects the space into two connected components, and $Q$ is precisely the component contained in $B_R(p)$. Indeed, it is clear that $Q$ is contained in one of the two components of $X\setminus \alpha_Q$. If it were strictly smaller, then we could find a continuous curve in the same connected component joining a point of $Q$ to a point in $X\setminus Q$. Such a curve would have to intersect $\partial Q=\alpha_Q$, which is impossible.
\end{Rem}

\begin{Thm}[Interior polygonal domain neighborhood basis]
\label{Thm: poly dom neigh basis}
Let $(X,\dd_X,\mm_X)$ be an $\RCD(K,N)$ space of essential dimension $2$. Then, for every interior point $p\in \Int(X)$ and every $R>0$, there exists a simple polygonal loop $\alpha$ such that:
\begin{enumerate}
    \item $\alpha \subset B_R(p) \cap \mathcal R_2$;

    \item $X \setminus \alpha$ has exactly two components, and the topological boundaries of both components coincide with $\alpha$. Moreover, one of these components is an interior polygonal domain $Q$ satisfying $p \in Q \subset B_R(p)$.
\end{enumerate}
In particular, every interior point $p\in \Int(X)$ admits a neighborhood basis consisting of interior polygonal domains.
\end{Thm}

As a corollary, we obtain two important consequences. The first asserts that geodesics starting from an interior point are regular.

\begin{Cor}\label{Cor: int geod reg}
Let $(X, \dd_X, \mm_X)$ be an $\RCD(K,N)$ space of essential dimension $2$, and let $\gamma:[0,1] \to X$ be a geodesic such that $\gamma(0)$ is an interior point. Then $\gamma$ is regular.
\end{Cor}

\begin{proof}
The proof is immediate. Apply Theorem~\ref{Thm: poly dom neigh basis} with $p:=\gamma(0)$ and $R:=\frac{1}{2}\dd_X(p,\gamma(1))$. We obtain a simple closed loop $\alpha$ and an interior polygonal domain $Q$ as in the statement. Since $\dd_X(p,\gamma(1))=2R$, we have $\gamma(1)\notin \overline{Q}$. Therefore, $\gamma(1)$ lies in the other connected component of $X\setminus \alpha$. It follows that $\gamma$ must intersect $\alpha$. Since $\alpha$ is contained in the regular set, Theorem~\ref{Thm: geodesic interior tangent} implies that $\gamma$ is a regular geodesic.
\end{proof}

The second corollary shows that the extremal set does not disconnect metric balls.

\begin{Cor}\label{Lem: ext set no disconnect}
Let $q\in X$ and let $r>0$. If $S\subset \Ext(X)$, then $B_r(q)\setminus S$ is path-connected.
\end{Cor}
\begin{proof}
Let $x, y \in B_r(q) \setminus S$. Since $\Int(X)$ is dense in $X$ by Theorem~\ref{Thm: int set open}, we may choose a point $x' \in \Int(X)$ sufficiently close to $x$, together with a geodesic $\gamma_1:[0,1]\to X$ joining $x$ to $x'$ and entirely contained in $B_r(q)$. It follows from Corollary~\ref{Cor: int geod reg} that $\gamma_1$ is a regular geodesic. By Corollary~\ref{Cor: int geod no ext}, we have $\gamma_1 \subset B_r(q) \setminus S$.
Similarly, we may choose $y' \in \Int(X)$ close to $y$ and a regular geodesic $\gamma_2$ from $y$ to $y'$ such that $\gamma_2 \subset B_r(q) \setminus S$.

Fix geodesics $\rho_1, \rho_2:[0,1] \to X$ from $x'$ and $y'$ to $q$, respectively. Since $x', y' \in \Int(X)$, by arguing as above, we obtain $\rho_1([0,1)), \rho_2([0,1)) \subset B_r(q) \setminus S$. Set $x'' := \rho_1(1/2)$ and $y'' := \rho_2(1/2)$. Then $\dd_X(x'',q), \dd_X(y'', q) < r/2$. Fix a geodesic $\sigma:[0,1] \to X$ from $x''$ to $y''$. Then $\sigma \subset B_r(q)$. Since $x'', y'' \in \Int(X)$, we obtain $\sigma \subset B_r(q) \setminus S$ by arguing as above.

By concatenating the curves obtained above between $x$, $x'$, $x''$, $y''$, $y'$, and $y$, we obtain a continuous curve in $B_{r}(q) \setminus S$ connecting $x$ to $y$. 
\end{proof}

%{\color{red}
%\begin{proof}
%Applying Theorem~\ref{Thm: poly dom neigh basis} for $p := \gamma(0)$ and $R := \dd_X(p, \gamma(1))$, we obtain a simple polygonal loop $\alpha:[0,1] \to X$ such that the following hold:
%\begin{enumerate}
   % \item $\alpha([0,1]) \subset B_R(p) \cap \mathcal R_2$.
   % \item $X \setminus \alpha([0,1])$ has exactly two components, and their topological boundaries coincide with $\alpha([0,1])$. Moreover, one of the components, denoted by $Q$, satisfies \[p \in Q \subset B_R(p).\]
%\end{enumerate}

%Since $\dd_X(\gamma(1),p)=R$, we have $\gamma(1)\notin Q \cup \alpha([0,1])$,
%so $\gamma(1)\in Q'$, where $Q'$ denotes the other component of
%$X\setminus \alpha([0,1])$.

%Since $\gamma$ is continuous curve connecting a point in $Q$ and a point in $Q'$, it must intersect %$\alpha([0,1])$. Let $\gamma(t)$, where $t \in (0,1)$, be such an intersection point. 

%By construction, $\alpha([0,1]) \subset \mathcal R_2$, and so $\gamma(t) \in \mathcal R_2$. It follows by Theorem~\ref{Thm: geodesic interior tangent} that $\gamma$ is a regular geodesic.
%\end{proof}
%}

\subsection{Proof of Theorem~\ref{Thm: poly dom neigh basis}}
\label{sec:proofplydom1}
Let $p \in \Int(X)$. By Proposition~\ref{Pro: every pt have extend reg geod}, there exists a regular geodesic $\gamma:[0,1] \to X$ with $\gamma(0) = p$. Let $\bar R := \dd_X(p, \gamma(1))$. Since the interior set is open by Theorem~\ref{Thm: int set open}, there exists $\bar R' > 0$ such that $B_{10\bar R'}(p)$ is contained in the interior set. 

Let $0 < R \leq \min(\bar R, \bar R')$. It suffices to prove the theorem for all such $R$. Let $r=r(p,R) \leq R$ be as in Corollary~\ref{Cor: Jordan theorem}. Our goal is to construct a polygonal loop $\alpha \subset B_r(p) \cap \mathcal R_2$ enclosing an interior polygonal domain containing $p$. 

Fix $t_0 \in [0,1]$ such that $\dd_X(\gamma(t_0), p)=r/4$. By Theorem~\ref{Thm: good components existence 2}, there exists $\bar s > 0$ such that for every $0 < s \leq \bar s$, the set $B_s(\gamma(t_0)) \setminus \gamma$ has two good components. Fix $s \leq \bar s$ so that in addition $B_s(\gamma(t_0)) \subset B_{r/2}(p)$, and denote the two good components by $B_\pm$.

Since $\gamma(t_0)$ lies in the topological boundaries of $B_\pm$, each of $B_\pm$ must intersect $B_{s/2}(\gamma(t_0))$. By our choice of $s$, the set $B_{s/2}(\gamma(t_0)) \setminus \gamma$ also has exactly two components. It follows from Lemma~\ref{Lem: big set small set comp} that these components are precisely $B_\pm \cap B_{s/2}(\gamma(t_0))$.

Fix $x \in B_+ \cap B_{s/2}(\gamma(t_0)) \cap \mathcal R_2$. By Proposition~\ref{Pro: every pt have extend reg geod}, we can choose $y \in B_- \cap B_{s/2}(\gamma(t_0)) \cap \mathcal R_2$ such that the geodesic $\rho:[0,1] \to X$ from $x$ to $y$ is unique and strongly regular, i.e., $\rho \subset \mathcal R_2$. By the triangle inequality, $\rho \subset B_s(\gamma(t_0))$. Since $\rho$ connects $B_\pm$, it must intersect $\gamma$.

We claim that $\rho$ does not intersect $\gamma$ more than once. Indeed, the endpoints of $\rho$ are contained in $B_s(\gamma(t_0)) \setminus \gamma$ and therefore do not lie on $\gamma$. If $\rho$ intersects $\gamma$ more than once, then it would follow from the non-branching property (Theorem~\ref{Thm: non-branching}) that $\gamma$ is a subsegment of $\rho$. Since $\rho \subset B_{s}(\gamma(t_0))$ and $\gamma(1) \notin B_{s}(\gamma(t_0))$, this would yield a contradiction. 

Therefore, $\rho$ intersects $\gamma$ exactly once. Let $s_0 \in [0,1]$ be such that $\rho(s_0) \in \gamma$. Then $\rho([0,s_0)) \subset B_+$ and $\rho((s_0, 1]) \subset B_-$.

Since $p \in \Int(X)$, the set $B_{r/2}(p) \setminus \gamma$ is path-connected. In particular, there exists a curve $\sigma:[0,1] \to B_{r/2}(p) \setminus \gamma$ from $x$ to $y$. Since $\sigma([0,1])$ and $\gamma([0,1])$ are compact and disjoint, we may fix $\veps > 0$ such that $\dd_X(\sigma([0,1]), \gamma([0,1])) > \veps$. 

Since $\sigma$ is uniformly continuous, for $m$
sufficiently large, we have
\[\dd_X(\sigma((i-1)/m), \sigma(i/m)) < \veps/10\quad \text{for every $i = 1,\dots,m$.}\]
For $i = 0, \dots, m$, set $x_i := \sigma(i/m)$. By Proposition~\ref{Pro: every pt have extend reg geod}, we may (inductively) choose approximations $\tilde x_i$ of $x_i$ such that the following properties hold:
\begin{enumerate}
    \item $\tilde x_0 = x_0 = x$ and $\tilde x_m = x_m = y$;
    \item $\dd_X(\tilde x_i, x_i) < \veps/10$, for all $i = 0, \dots, m$;
    \item for all $i = 1, \dots, m$, the geodesic from $\tilde x_{i-1}$ to $\tilde x_i$ is unique, strongly regular, and contained in $B_r(p) \setminus \gamma$.
\end{enumerate}

%Our construction will be inductive. Set \[\tilde x_0 := x_0 = x \in B_{r/2}(p) \cap \mathcal R_2 \quad \text{and} \quad \tilde x_m := x_m = y \in B_{r/2}(p) \cap \mathcal R_2.\]
%For each $i = 1, \dots, m-2$, assume that $\tilde x_{i-1} \in B_{\veps/10}(x_{i-1}) \cap \mathcal R_2$ has been chosen. By Proposition~\ref{Pro: every pt have extend reg geod}, we can choose $\tilde x_i \in B_{\veps/10}(x_i) \cap \mathcal R_2$ such that the geodesic from $\tilde x_{i-1}$ to $\tilde x_i$ is unique and strongly regular. Finally, once $\tilde x_0, \dots, \tilde x_{m-2}, \tilde x_m$ have been chosen, by applying Proposition~\ref{Pro: every pt have extend reg geod} simultaneously to $\tilde x_{m-2}$ and $\tilde x_m$, we can choose $\tilde x_{m-1} \in B_{\veps/10}(x_{m-1}) \cap \mathcal R_2$ so that the geodesics from $\tilde x_{m-1}$ to both $\tilde x_{m-2}$ and $\tilde x_{m}$ are unique and strongly regular. This completes the inductive construction. 

For each $i = 1, \dots m$, we have $\dd_X(\tilde x_{i-1}, x_i) < \veps/5$ by the triangle inequality. Therefore, the geodesic from $\tilde x_{i-1}$ to $\tilde x_i$ lies in $B_{3\veps/10}(x_i) \subset B_{3\veps/10}(\sigma([0,1]))$, and thus does not intersect $\gamma$. Moreover, this geodesic lies in $B_r(p)$, since $x_i \in B_{r/2}(p)$ and $\veps$ is arbitrarily small. 

Let $\tilde \sigma:[0,1] \to X$ be such that, for each $i = 1, \dots, m$, $\sigma \lvert_{[(i-1)/m, i/m]}$ is equal to the geodesic from $\tilde x_{i-1}$ to $\tilde x_{i}$. Then $\tilde \sigma$ is a (possibly non-simple) polygonal curve contained in $B_r(p)$. 

Since $\tilde \sigma(0) = \rho(0)$, $\tilde \sigma(1) = \rho(1)$, and $\tilde \sigma$ does not intersect $\rho(s_0) \in \gamma([0,1])$, there is a largest $t_1 \in [0,1)$ such that $\tilde \sigma(t_1) \in \rho([0,s_0))$. Repeating this argument, there is a smallest $t_2 \in (t_1, 1]$ such that $\tilde \sigma(t_2) \in \rho((s_0,1])$. We replace 
\begin{enumerate}
    \item $x$ by $\tilde \sigma(t_1)$, $y$ by $\tilde \sigma(t_2)$;
    \item $\rho$ by the subsegment of itself from $\tilde \sigma(t_1)$ to $\tilde \sigma(t_2)$, reparameterized to be constant-speed on $[0,1]$;
    \item $\tilde \sigma$ by $\tilde \sigma \lvert_{[t_1,t_2]}$, reparameterized to be constant-speed on $[0,1]$.
\end{enumerate}
With these choices, $\tilde \sigma$ is a polygonal curve from $x$ to $y$ contained in $B_r(p) \setminus \gamma$, and it intersects $\rho$ exactly twice, namely at $x$ and $y$.

Since $\tilde \sigma$ is a polygonal curve, there exists $0 = s_0 < \dots < s_m = 1$, such that $\tilde \sigma \lvert_{[s_{i-1}, s_i]}$ are geodesics. In particular, $\tilde \sigma$ is a concatenation of finitely many simple curves. It is not difficult to construct a simple polygonal curve from $\sigma$, retaining all the properties established previously, by a finite process that excises any overlap. Replace $\sigma$ with this curve. 

%We start by checking if any point in $\tilde \sigma((0,s_1))$ coincides with a point in $\tilde \sigma([s_1, 1))$ (note that $\tilde \sigma(0)$ and $\tilde \sigma(1)$ do not intersect with any other point in $\tilde \sigma$ by our construction). If there is an intersection, then we take the smallest time $t \in (0, s_1)$ such that $\tilde \sigma(t)$ coincides with a point in $\sigma([s_1, 1))$. This point is of the form $\tilde \sigma(t')$ for some $t' \in [s_{i-1}, s_i)$, for some $i > 1$. We replace $\tilde \sigma \lvert_{(t,t')}$ with the constant map $\tilde \sigma(t)$, and check if a point in $\tilde \sigma([t', s_i))$ coincides with a point in $\tilde \sigma ([s_i, 1))$ in the next step. If there is no intersection, then we move on to $\tilde \sigma_{[s_1, s_2)}$ and check if any point on it coincides with a point in $\sigma([s_2, 1])$, and so on. It is not difficult to see that this process terminates after at most $m$ steps, producing a simple polygonal curve. Replacing $\tilde \sigma$ by this curve, we see that $\tilde \sigma$ retains all of the properties established in the previous paragraph.

Let $\alpha:[0,1]\to X$ be the simple polygonal loop obtained by concatenating $\rho$ from $x$ to $y$ with $\sigma$ from $y$ to $x$. Then $\alpha \subset B_r(p)$. By our choice of $r$, this implies that $X \setminus \alpha$ has two components, at least one of which is contained in $B_R(p)$. Since $\dd_X(\gamma(1), p) = R$, it follows that exactly one of the components, denoted by $Q$, is supported in $B_R(p)$. Let $Q'$ be the other component. Then $Q$ is an interior polygonal domain and $\gamma(1) \in Q'$. 

To finish the proof, we only have to show that $p \in Q$.

By construction, $\gamma$ intersects $\alpha$ exactly once, in the interior of the edge $\rho$. Since $\gamma(0)$ and $\gamma(1)$ do not lie on $\alpha$, this intersection also occurs in the interior of $\gamma$. Thus there exist $t,t'\in(0,1)$ such that $\rho(t)=\gamma(t')$.

By Theorem~\ref{Thm: good components existence 2}, for all sufficiently small $s > 0$, the set $B_s(\rho(t)) \setminus \rho$ has two components $A_\pm$. Fix such an $s$ so that, in addition, $\alpha \cap B_s(\rho(t)) = \rho \cap B_s(\rho(t))$. Since the topological boundaries of $Q$ and $Q'$ coincide with $\alpha$, both $Q$ and $Q'$ intersect $B_s(\rho(t))$. By Lemma~\ref{Lem: big set small set comp}, up to relabeling $A_\pm$, we may assume that $A_+= Q \cap B_s(\rho(t))$ and $A_- = Q' \cap B_s(\rho(t))$. 

Since $\gamma(1) \in Q'$ and $\gamma((t', 1])$ is disjoint $\alpha$, we have $\gamma((t',1]) \subset Q'$. In particular, $\gamma((t',1]) \cap B_s(\rho(t)) \subset A_-$. By Lemma~\ref{Lem: geod no glance}, this implies that $\gamma([0,t')) \cap B_s(\rho(t)) \subset A_+$. Since $A_+ \subset Q$ and $\gamma([0,t'))$ is disjoint from $\alpha$, we conclude that $\gamma([0,t')) \subset Q$. Therefore, $p \in Q$. \qed

\section{Parameterization of the extremal set}\label{Sec: ext set top}

Let $(X, \dd_X, \mm_X)$ be an $\RCD(K,N)$ space of essential dimension $2$. In this section, we make the first step toward understanding neighborhoods of extremal points by proving the following parametrization-type property.

\begin{Pro}\label{Pro: ext pt curve}
Let $p \in \Ext(X)$ be an extremal point. Then there exists a simple continuous curve $\mathcal E:[-1,1] \to X$ such that $\mathcal E([-1,1]) \subset \Ext(X)$ and $\mathcal E(0) = p$.
\end{Pro}

For the remainder of this section, we consider the following framework. Let $\gamma:[0,10] \to X$ be a unit-speed regular geodesic with $\gamma(0)=p$, and assume that the following hold:
\begin{enumerate}
    \item $B_{10}(p) \setminus \gamma$ has exactly two components $A_\pm$.
    \item $p$ is an extremal point.
    \item The topological boundaries of $A_\pm$ in $B_{10}(p)$ coincide with $\gamma \cap B_{10}(p) = \gamma([0,10))$.
    \item For every $0 < r \leq 10$, $B_r(p) \setminus \gamma$ has exactly two components, and they are precisely $A_\pm \cap B_r(p)$. Moreover, the topological boundaries of both components in $B_r(p)$ coincide with $\gamma \cap B_r(p) = \gamma([0,r))$.
\end{enumerate}
By Definition~\ref{Def: ext-int points}, Lemma~\ref{Lem: ext-pt comp boundary}, and Lemma~\ref{Lem: ext point disconnect all small ball}, (2)--(4) are direct consequences of (1).

Suppose further that there exists another unit-speed regular geodesic $\rho:[0,10] \to X$ with $\rho(0) = p$, which intersects $\gamma$ only at $p$, and such that the following hold:
\begin{enumerate}
    \item[(5)] $B_{10}(p) \setminus \rho$ has exactly two components $B_\pm$.
    \item[(6)] The topological boundaries of $B_\pm$ in $B_{10}(p)$ coincide with $\rho \cap B_{10}(p) = \rho([0,10))$.
    \item[(7)] For any $0 < r \leq 10$, $B_r(p) \setminus \rho$ has exactly two components, and they are precisely $B_\pm \cap B_r(p)$. Moreover, the topological boundaries of both components in $B_r(p)$ coincide with $\rho \cap B_r(p) = \rho([0,r))$.
\end{enumerate}
Set $q:=\gamma(1)$. We consider geodesics $\sigma$ from $q$ to points on the support of $\rho$, with the aim of extending them until they stop being minimized. Under suitable assumptions, we will show that the resulting cut-locus point is extremal.

\begin{Lem}\label{Lem: extend no intersect}
Every geodesic $\sigma$ emanating from $q$ intersects $\rho([0,1])$ at most once. 
\end{Lem}

\begin{proof}
Suppose for the sake of contradiction that $\sigma$ intersects $\rho([0,1])$ at more than one point. Let $t \in [0,1]$ be such that $\rho(t)$ is a point of intersection. Since $q \notin \rho$, it follows from the non-branching property (Theorem~\ref{Thm: non-branching}) that one of $\rho \lvert_{[0,t]}$ or $\rho \lvert_{[t,10]}$ is a subsegment of $\sigma$. 

If $\rho \lvert_{[0,t]}$ is a subsegment of $\sigma$, then $\sigma$ must coincide with the concatenation of the reverse of $\gamma \lvert_{[0,1]}$ and $\rho \lvert_{[0,t]}$. On the other hand, Lemma~\ref{Lem: geod at ext-pt not extend} and property~(1) above imply that $\gamma \lvert_{[0,1]}$ is not extendible past $\gamma(0) = p$, yielding a contradiction. 

If $\rho \lvert_{[t,10]}$ is a subsegment of $\sigma$, then $\sigma$ must have length at least $10-t \geq 9$. On the other hand, the concatenation of the reverse of $\gamma \lvert_{[0,1]}$ with $\rho \lvert_{[0,t]}$ is a continuous curve from $q$ to $\rho(t)$ with length $1+t \leq 2$, which contradicts the assumption that $\sigma$ is a geodesic. 
\end{proof}

Since $\gamma((0,10))$ does not intersect $\rho$, it must be contained entirely in one of the components $B_\pm$. Similarly, $\rho((0,10))$ must be contained entirely in one of the components $A_\pm$. Therefore, up to relabeling $A_\pm$ and $B_\pm$, we may assume that
\begin{equation}\label{Eq: Apm Bpm assumption}
\gamma((0,10)) \subset B_- \quad \text{and} \quad \rho((0,10)) \subset A_+.
\end{equation}

\begin{Lem}\label{Lem: exp map 1}
For every $0 < \eps \le 10$, there exists $\delta > 0$ such that the following holds: For every $x \in B_+ \cap B_\delta(p)$ and every geodesic $\sigma$ from $q$ to $x$, the curve $\sigma$ intersects $\rho$ exactly once. Moreover, this intersection point belongs to $\rho([0,\eps))$, and the subsegment of $\sigma$ joining it to $x$ is contained in $B_\eps(p)$.
\end{Lem}

\begin{proof}
It suffices to prove the lemma for $0 < \veps \leq 1$. Suppose for the sake of contradiction that there does not exist $\delta > 0$ such that for every $x \in B_+ \cap B_\delta(p)$, every geodesic from $q$ to $x$ intersects $\rho([0,\veps))$. Then there exist a sequence $\delta_i \downarrow 0$ and for every $i \in \N$, a points $x_i \in B_+ \cap B_{\delta_i}(p)$ such that a geodesic $\sigma_i:[0,1] \to X$ from $q$ to $x_i$ does not intersect $\rho([0, \veps))$. 

Since $\dd_X(p,q) = 1$ and $\dd_X(p,x_i) < \delta_i$, the triangle inequality implies that $\sigma_i \subset B_{1+\delta_i}(p)$. Since $\delta_i \downarrow 0$, by passing to a subsequence we may assume that $\sigma_i \subset B_{10}(p)$ for all $i$. 

By our assumptions, $q \in \gamma((0,10)) \subset B_-$ and $x_i \in B_+$. Hence $\sigma_i$ is a curve in $B_{10}(p)$ joining the two components of $B_{10}(p) \setminus \rho$. Therefore, $\sigma_i$ must intersect $\rho$. Since $\sigma_i$ does not intersect $\rho([0,\veps))$ by assumption, it must intersect $\rho([\veps, 10])$. 

By the non-branching property, $\gamma \lvert_{[0,1]}$ is the unique geodesic from $p$ to $q$, since it is a proper subsegment of a geodesic. It follows by compactness that $\sigma_i$ converges uniformly to the reverse of $\gamma \lvert_{[0,1]}$. On the other hand, any uniform limit of $\sigma_i$ must intersect $\rho([\veps,10])$. This yields a contradiction, since $\gamma$ and $\rho$ intersect only at $p = \rho(0)$. 

Therefore, for every $0 < \veps \leq 1$, there exists $\delta > 0$ such that for every $x \in B_+ \cap B_\delta(p)$, every geodesic from $q$ to $x$ intersects $\rho([0,\veps))$. By Lemma~\ref{Lem: extend no intersect}, such geodesics intersect $\rho$ exactly once. This proves the first part of the lemma. 

Applying the first part of the lemma for $\veps/2$, it follows that for $\delta > 0$ sufficiently small, every geodesic $\sigma$ from $q$ to any $x \in B_+ \cap B_{\delta}(p)$ intersects $\rho$ at $\rho(t)$ for some $t \in [0, \veps/2)$. By further imposing that $\delta \leq \veps/2$, the triangle inequality then implies that the subsegment of $\sigma$ from $\rho(t)$ to $x$ must be contained in $B_{\veps}(p)$. This proves the second part of the lemma. 
\end{proof}

For every constant-speed geodesic $\sigma$ emanating from $q$, its maximal extension in the sense of Definition~\ref{Def: max extension} is either a geodesic or a ray. In the former case, we denote it by $\tilde \sigma:[0,1]\to X$, parameterized with constant speed and satisfying $\tilde \sigma(0)=q$.

\begin{Lem}\label{Lem: exp map 2}
For every $\eps > 0$, there exists $0 < \delta \le 10$ such that, for every $t \in (0,\delta]$ and every geodesic $\sigma$ from $q$ to $\rho(t)$, the maximal extension of $\sigma$ from $q$ is a geodesic. Moreover, the following hold:
\begin{enumerate}
    \item[(i)] $1-\eps < \dd_X(\tilde \sigma(0), \tilde \sigma(1)) < 1+\eps$;

    \item[(ii)] the subsegment of $\tilde \sigma$ from $\rho(t)$ to $\tilde \sigma(1)$ is contained in $B_{\eps}(p)$.
\end{enumerate}
\end{Lem}
\begin{proof}
Fix $0 < \veps \leq 10$. Since $\dd_X(q, \rho(t)) \to \dd_X(q, p) = 1$ as $t \downarrow 0$, for $t$ sufficiently small, the length of any geodesic $\sigma$ from $q$ to $\rho(t)$ is greater than $1 -\veps$. In particular, the length of the maximal extension $\tilde \sigma$ must also be greater than $1-\veps$. 

Next, suppose for the sake of contradiction that $\dd_X(\tilde \sigma(0), \tilde \sigma(1))<1+\eps$ does not hold for $t$ sufficiently small. Then there exists a sequence $t_i\downarrow 0$, and geodesics $\sigma_i:[0,1] \to X$ from $q$ to $\rho(t_i)$ such that, for each $i$, the maximal extension of $\sigma_i$ from $q$ is either a ray or a geodesic of length at least $1+\veps$. For each $i$, by taking a subsegment of this maximal extension, we obtain a geodesic $\tilde \sigma_i:[0,1] \to X$ of length $1+\veps$ such that $\tilde \sigma(0) = q$. Since $\dd_X(q, \rho(t_i)) \to \dd_X(q, p) = 1$, the point $\rho(t_i)$ lies in $\tilde \sigma$ for $i$ sufficiently large.

Passing to a convergent subsequence and taking a limit, we obtain a geodesic $\tilde \sigma:[0,1] \to X$ of length $1 + \veps$ such that $\tilde \sigma(0) = q$ and $p \in \tilde \sigma((0,1))$. In particular, $\tilde \sigma$ extends the reverse of $\gamma \lvert_{[0,1]}$ past the point $p$. On the other hand, Lemma~\ref{Lem: geod at ext-pt not extend} together with property~(1) imply that $\gamma \lvert_{[0,1]}$ is not extendible past $p$, yielding a contradiction. 

Finally, it follows from the previous steps that, 
\[
\dd_X(q, \rho(t)) > 1-\veps/10, \quad \dd_X(q, \tilde \sigma(1)) < 1+\veps/10, \quad \dd_X(\rho(t), p) < \veps/10.
\]
for $t$ sufficiently small. For any such $t$,
\begin{align*}
\dd_X(p, \tilde \sigma(1)) &\leq \dd_X(p, \rho(t)) + \dd_X(\rho(t), \tilde \sigma(1))\\
&= \dd_X(p, \rho(t)) + \dd_X(\tilde \sigma(1),q) - 
\dd_X(\rho(t), q)\\
&< \veps/10 + (1+\veps/10) - (1-\veps/10) < \veps/2.
\end{align*}
Since $\tilde \sigma(1)$ and $\rho(t)$ are both within $\veps/2$ of $p$, it follows from the triangle inequality that the subsegment of $\tilde \sigma$ from $\rho(t)$ to $\tilde \sigma(1)$ is contained in $B_{\veps}(p)$. 
\end{proof}

\subsection{Cut-locus map}
By Theorem~\ref{Thm: semi-locally simply connected}, we may fix $0 < r_0 \le 1/100$ such that the inclusion $B_{10r_0}(p)\hookrightarrow B_{1/10}(p)$ induces the trivial map on $\pi_1$. Let $r_1<r_0$ be smaller than the $\delta$ given by Lemma~\ref{Lem: exp map 1} for $\eps=r_0$ and $r_2<r_1/10$ smaller than the $\delta$ given by Lemma~\ref{Lem: exp map 2} for $\eps=r_1$.

For each $t \in [0,r_2]$, choose a geodesic $\sigma_t:[0,1] \to X$ from $q$ to $\rho(t)$. By the choice of $r_2$ and Lemma~\ref{Lem: exp map 2}, the maximal extension of $\sigma_t$ from $q$ is a geodesic, which we denote by $\tilde \sigma_t:[0,1] \to X$. Here $\tilde \sigma_t$ is parametrized with constant speed on $[0,1]$ and satisfies $\tilde \sigma_t(0)=q$.

\begin{Def}\label{Def: cut locus}
The \emph{cut-locus map} $F:[0,r_2]\to X$ is defined by $F(t):=\tilde \sigma_t(1)$.
\end{Def}

Notice that $F(t)$ is well defined, namely it does not depend on the choice of the geodesic $\sigma_t$. Indeed, if $\sigma_t=\tilde \sigma_t$, then necessarily $F(t)=\rho(t)$. Otherwise, $\sigma_t$ is a proper subsegment of $\tilde \sigma_t$, and in this case $\sigma_t$ is the unique geodesic from $q$ to $\rho(t)$ by the non-branching property.

\begin{Lem}\label{Lem: exp image inclusion}
It holds that
\[
F([0,r_2]) \subset \rho([0,r_2]) \sqcup \big(B_+ \cap B_{r_1}(p)\big).
\]
Moreover, $F(0) = p$, and if $F(t) \in \rho([0, r_2])$, then $F(t) = \rho(t)$.
\end{Lem}

\begin{proof}
Let $t \in [0,r_2]$. To shorten notation we set $\sigma:=\sigma_t$ and $\tilde \sigma:=\tilde \sigma_t$. 

If $\sigma = \tilde \sigma$, then $F(t) = \rho(t)$, and so $F(t) \in \rho([0,r_2])$. This occurs in particular when $t = 0$ as $\sigma$ is the reverse of $\gamma \lvert_{[0,1]}$, which is not extendible past $p$. 

Suppose that $\sigma \neq \tilde \sigma$. By our choice of $r_2$ and Lemma~\ref{Lem: exp map 2}, we know that $\tilde \sigma(1) \in B_{r_1}(p)$, so it suffices to show that $\tilde \sigma(1) \in B_+$. 

Since $\rho$ is regular and $t \in (0, r_2] \subset (0,10)$, it follows from Theorem~\ref{Thm: good components existence 2} that for sufficiently small $s > 0$, the set
$B_s(\rho(t)) \setminus \rho$ has two components $G_\pm$. Fix such an $s > 0$, chosen small enough that also $B_s(\rho(t))\subset B_{10}(p)$.

By property~(6) of $\rho$, the topological boundaries of $B_\pm$ in $B_{10}(p)$ coincide with $\rho([0,10))$. Therefore, both $B_\pm$ intersect $B_{s}(\rho(t))$. By~Lemma~\ref{Lem: big set small set comp}, up to relabeling $G_\pm$, we have
\begin{equation}\label{Eq: exp image inclusion 1}
G_\pm = B_\pm \cap B_{s}(\rho(t)).
\end{equation}

By our choice of $r_2$ and Lemma~\ref{Lem: exp map 2}, the maximal extension $\tilde \sigma$ has length at most $1+r_1 < 2$. Since $\dd_X(p,q) = 1$, the triangle inequality implies $\tilde \sigma \subset B_{10}(p)$. By Lemma~\ref{Lem: extend no intersect}, $\tilde \sigma$ intersects $\rho$ exactly once, namely at $\rho(t)$. Since $\tilde \sigma \neq \sigma$, this intersection occurs in the interior of $\tilde \sigma$. Let $t' \in (0,1)$ be such that $\tilde \sigma(t') = \rho(t)$. Therefore, the sets $\tilde \sigma([0, t'))$ and $\tilde \sigma((t',1])$ are each contained in one of the components $B_\pm$. 

Since $\tilde \sigma(0) = q \in B_-$, we have $\tilde \sigma([0,t')) \subset B_-$. By continuity, there exists $\veps > 0$ such that $\tilde \sigma([t'-\veps, t'+\veps]) \subset B_{s/10}(\rho(t))$.
The inclusion relation \eqref{Eq: exp image inclusion 1} then implies that $\tilde \sigma((t' -\veps, t')) \subset G_- \cap B_{s/10}(\rho(t))$.
By Theorem~\ref{Thm: good components existence 2}, no geodesic joining two points in $G_- \cap B_{s/10}(\rho(t))$ can intersect $\rho$, and so 
$\tilde \sigma(t' + \veps) \in G_+ \cap B_{s/10}(\rho(t))$.
Applying \eqref{Eq: exp image inclusion 1} again, we obtain $\tilde \sigma(t'+\veps) \in B_+$. This implies that $\tilde \sigma((t',1])$, and so in particular $\tilde \sigma(1)$, lies in $B_+$, as required. 

Since $B_+$ does not intersect $\rho([0,r_2])$, if $F(t) \in \rho([0,r_2])$, then $\sigma = \tilde \sigma$ and $F(t) = \rho(t)$. 
\end{proof}

\begin{Def}\label{Def: foc equiv}
Let $t,t' \in [0,1]$. We write $t\sim t'$ if either $t=t'$, or $t\neq t'$ and there exist geodesics $\sigma$ and $\sigma'$ from $q$ to $\rho(t)$ and $\rho(t')$, respectively, such that
\begin{enumerate}
    \item[(i)] the maximal extensions of $\sigma$ and $\sigma'$ from $q$ are geodesics;

    \item[(ii)] these maximal extensions have the same endpoints.
\end{enumerate}
\end{Def}

\begin{Rem}
If $t \sim t'$ and $t \neq t'$, then it is not difficult to check from Lemma~\ref{Lem: extend no intersect} that $\sigma$ and $\sigma'$ (as in Definition~\ref{Def: foc equiv}) must be proper subsegments of their respective maximal extensions from $q$. The non-branching property then implies that $\sigma$ and $\sigma'$ are the unique geodesics from $q$ to $\rho(t)$ and $\rho(t')$, respectively. It follows immediately that $\sim$ is an equivalence relation. 
\end{Rem}

\begin{Lem}\label{Lem: foc time bound}
Let $t \in [0,r_2]$, and let $t' \in [0,1]$ with $t' \neq t$ and $t \sim t'$. Then:
\begin{enumerate}
    \item[(i)] $t \in (0,r_2]$ and $t' \in (0, r_0]$.
    \item[(ii)] $F(t) \in B_+ \cap B_{r_1}(p)$.
\end{enumerate}
\end{Lem}

\begin{proof}
Let $\sigma$ and $\sigma'$ be geodesics from $q$ to $\rho(t)$ and $\rho(t')$, respectively, and let $\tilde \sigma$ and $\tilde \sigma'$ denote their maximal extension from $q$. By Definition~\ref{Def: foc equiv}, both $\tilde \sigma$ and $\tilde \sigma'$ are geodesics, so we may parameterize them with constant speed on $[0,1]$ such that $\tilde \sigma(0) = \tilde \sigma'(0) = q$.

We first show that $t \neq 0$. Suppose that $t = 0$. Then by Lemma~\ref{Lem: exp image inclusion}, we have $F(0) = \tilde \sigma(1) = p$. Since $\gamma \lvert_{[0,1]}$ is a proper subsegment of a geodesic, it is the unique geodesic between $p$ and $q$ (by the non-branching property). Hence $\tilde\sigma'$ must coincide with the reverse of $\gamma|_{[0,1]}$. However, $\gamma$ does not intersect $\rho$ other than at $p$, so $\tilde \sigma'$ cannot pass through $\rho(t')$, yielding a contradiction. A similar argument shows that $t' \neq 0$.

Next, we show that $F(t) \neq  \rho(t)$. Suppose otherwise. Then $\tilde \sigma(1) = \rho(t)$, and so $\tilde \sigma'(1) = \rho(t)$ since $t \sim t'$. Therefore, $\tilde \sigma'$ is a geodesic from $q$ that intersects $\rho$ at two distinct points, namely $\rho(t)$ and $\rho(t')$. Since $t \in [0,r_2] \subset [0,1]$, this contradicts Lemma~\ref{Lem: extend no intersect}. Therefore, $F(t) \neq \rho(t)$. It follows immediately from Lemma~\ref{Lem: exp image inclusion} that $F(t) = \tilde \sigma(1) \in B_+ \cap B_{r_1}(p)$.

Finally, we show that $t' \in (0, r_0]$. Since $\tilde \sigma'(1) = \tilde \sigma(1)$, we have $\tilde \sigma'(1) \in B_+ \cap B_{r_1}(p)$. By our choice of $r_1$ and Lemma~\ref{Lem: exp map 1}, it follows that $\tilde \sigma'$ intersects $\rho$ exactly once, and this intersection lies in $\rho([0, r_0])$. Since we already know that $\tilde \sigma'$ intersects $\rho$ at $\rho(t')$, we conclude that $t' \in [0, r_0]$. We have already shown that $t' \neq 0$. Therefore, $t' \in (0, r_0]$. 
\end{proof}

\begin{Lem}\label{Lem: ext pt sim}
Let $t_0 \in [0,r_2]$. Let $t_1, t_1' \in [0, 1]$ with $t_1 \leq t_0 \leq t_1'$ be such that $t_1 \sim t_1'$. If $F(t_0) \in \Ext(X)$, then $t_0 \sim t_1$. 
\end{Lem}

\begin{proof}
If $t_0 = 0$, then the statement is immediate, since in this case $t_1 = 0$, and so $t_0 \sim t_1$. 

Assume now that $t_0 \neq 0$ and that $F(t_0) \in \Ext(X)$. Set $\sigma:=\sigma_{t_0}$ and $\tilde \sigma:=\tilde\sigma_{t_0}$. 

Since $t_0 \in (0, r_2]$, the point $\rho(t_0)$ lies in the interior of the regular geodesic $\rho$, and hence is not extremal by Corollary~\ref{Cor: int geod no ext}. Therefore, $F(t_0) \neq \rho(t_0)$, and $\sigma$ contains $\rho(t_0)$ in its interior. Since $\rho(t_0) \in \mathcal R_2$, it follows from Theorem~\ref{Thm: geodesic interior tangent} that $\tilde \sigma$ is a regular geodesic. 

By our choice of $r_2$ and Lemma~\ref{Lem: exp map 2}, we have $\dd_X(F(t_0), p) < r_1$.  Since $r_1 \leq r_0 < 1/100$, the triangle inequality yields
\[
B_{5r_0}(F(t_0)) \subset B_{10r_0}(p) \quad \text{and} \quad B_{1/10}(p) \subset B_{1/5}(F(t_0)).
\]
It follows from our choice of $r_0$ that the inclusion $B_{5r_0}(F(t_0)) \hookrightarrow B_{1/5}(F(t_0))$ induces the trivial map on $\pi_1$. 

By our choice of $r_2$ and Lemma~\ref{Lem: exp map 2}, we also have $\dd_X(q,F(t_0)) > 1-r_1 > 1/5$. Since $F(t_0) \in \Ext(X)$ and $\tilde\sigma$ is a regular geodesic, it follows from Proposition~\ref{Pro: comp exist geod ep ext quant} that $B_{5r_0}(F(t_0)) \setminus \tilde \sigma$ has exactly two components.

Let $t_1, t_1' \in [0,1]$ with $t_1 \leq t_0 \leq t_1'$ and $t_1 \sim t_1'$. If either $t_1 = t_0$ or $t_1' = t_0$, then it is immediate that $t_0 \sim t_1 \sim t_1'$. Thus assume that $t_1 < t_0 < t_1'$.  

Let $\nu$ and $\nu'$ be geodesics from $q$ to $\rho(t_1)$ and $\rho(t_1')$, respectively, and let $\tilde \nu$ and $\tilde \nu'$ denote their maximal extensions from $q$. Parameterize $\tilde \nu$ and $\tilde \nu'$ with constant speed on $[0,1]$ so that
$\tilde \nu(0) = \tilde \nu'(0) = q$. 
Since $t_1 < t_0$, we have $t_1 \in [0,r_2]$. Lemma~\ref{Lem: foc time bound} implies
\[
\tilde \nu'(1) = \tilde \nu(1) = F(t_1) \in B_+ \cap B_{r_1}(p). 
\]
It follows from our choice of $r_1$ and Lemma~\ref{Lem: exp map 1} that the subsegment of $\tilde \nu$ from $\rho(t_1)$ to $\tilde \nu(1)$ and the subsegment of $\tilde \nu'$ from $\rho(t_1')$ to $\tilde \nu'(1)$ both lie in $B_{r_0}(p)$. 

If $\tilde \sigma(1) = \tilde \nu(1)$, then $t_0 \sim t_1$ and the lemma follows. Thus assume that this is not the case. 
We claim that $\tilde \sigma$ and $\tilde \nu$ intersect only at $\tilde \sigma(0) = \tilde \nu(0) = q$. Suppose that they intersect at a point other than $q$. Since $\tilde \sigma(1) \neq \tilde \nu(1)$, this other point of intersection cannot coincide simultaneously with both endpoints $\tilde \sigma(1)$ and $\tilde \nu(1)$. By the non-branching property, it follows that one of $\tilde \sigma$ and $\tilde \nu$ must be a subsegment of the other. In either case, one of the geodesics must intersect $\rho$ at least twice. Since $t_0, t_1 \in [0, r_2] \subset [0,1]$, this contradicts Lemma~\ref{Lem: extend no intersect}. Since $t_1' \in [0,1]$ as well, the same argument shows that $\tilde \sigma$ and $\tilde \nu'$ intersect only at $q$.

Therefore, the concatenation of $\tilde \nu'$ from $\rho(t_1')$ to $\tilde \nu'(1)$ with the reverse of $\tilde \nu$ from $\tilde \nu(1)$ to $\rho(t_1)$ is a continuous curve contained in $B_{r_0}(p) \subset B_{5r_0}(F(t))$, and disjoint from $\tilde \sigma$. It follows that $\rho(t_1)$ and $\rho(t_1')$ must be contained in the same component of $B_{5r_0}(F(t_0)) \setminus \tilde \sigma$.

We now show that $\rho(t_1)$ and $\rho(t_1')$ lie in different components of $B_{5r_0}(F(t_0)) \setminus \tilde \sigma$ for a contradiction. As observed above, $\rho$ intersects $\tilde \sigma$ exactly once, at $\rho(t_0)$. By considering the good components of $B_{s}(\rho(t_0))$ with respect to $\tilde \sigma$ for $s$ sufficiently small and applying the same argument as in the proof of Lemma~\ref{Lem: exp image inclusion}, we conclude that $\rho([t_1,t_0))$ must be contained in one component of $B_{5r_0}(F(t_0)) \setminus \tilde \sigma$ and $\rho((t_0, t_1'])$ in the other. In particular,  $\rho(t_1)$ and $\rho(t_1')$ lie in different components.
\end{proof}

\begin{Lem}\label{Lem: int pt sim}
Let $t_0, t_1 \in [0, r_2]$. Let $t_0', t_1' \in [0,1]$ be such that $t_0 \sim t_0'$ and $t_1 \sim t_1'$. If $t_1 \leq t_0 \leq t_1' \leq t_0'$ and $F(t_0) \in \Int(X)$, then $t_0 \sim t_1$. 
\end{Lem}

\begin{proof}
If $t_0 = 0$, then the statement is immediate, since in this case $t_1 = 0$ and so $t_0 \sim t_1$. 

Assume that $t_0 \neq 0$ and $F(t_0) \in \Int(X)$. Let $\sigma, \sigma'$ be geodesics from $q$ to $\rho(t_0)$ and $\rho(t_0')$, respectively, and let $\tilde \sigma, \tilde \sigma'$ be their maximal extensions from $q$. Parameterize $\tilde\sigma$ and $\tilde\sigma'$ with constant speed on $[0,1]$ so that $\tilde \sigma(0) = \tilde \sigma'(0) = q $ and $\tilde \sigma(1) = \tilde \sigma'(1) = F(t_0)$. 

If any of the inequalities in $t_1 \leq t_0 \leq t_1' \leq t_0'$ is an equality, then the statement is immediate. Thus assume that $t_1 < t_0 < t_1' < t_0'$. 

By arguing as in the proof of Lemma~\ref{Lem: ext pt sim}, we see that $\tilde \sigma$ and $\tilde \sigma'$ intersect $\rho$ exactly once, at $\rho(t_0)$ and $\rho(t_0')$, respectively. Moreover, these intersections occur in the interiors of $\tilde \sigma$ and $\tilde \sigma'$. Since $\rho(t_0), \rho(t_0') \in \mathcal R_2$, it follows from Theorem~\ref{Thm: geodesic interior tangent} that $\tilde \sigma$ and $\tilde \sigma'$ are regular geodesics. 

By our choice of $r_2$ and Lemma~\ref{Lem: exp map 2}, we have $\dd_X(q, F(t_0)) > 1-r_1$. Since $\tilde \sigma$ and $\tilde \sigma'$ are both geodesics from $q$ to $F(t_0)$, it follows that both of their lengths are greater than $1-r_1$. 

Let $\hat\sigma$ and $\hat\sigma'$ denote the geodesic subsegments obtained by removing from $\tilde\sigma$ and $\tilde\sigma'$ an initial subsegment of length $r_1$ starting at $q$. Since $r_1 < 1/100$, both $\hat \sigma$ and $\hat \sigma'$ have length at least $1-2r_1 > 1/5$. Since $\tilde \sigma$ and $\tilde \sigma'$ share the same endpoints but do not coincide (as they intersect $\rho$ at different points), it follows from the non-branching property that they intersect only at their endpoints. Therefore, $\hat \sigma$ and $\hat \sigma'$ intersect exactly once, namely at $F(t_0)$. 

Let $\alpha:[-1,1] \to X$ be the simple polygonal curve obtained by concatenating $\hat \sigma$ with $\hat \sigma'$ at $F(t_0)$, parameterized with constant speed and satisfying $\alpha(0)=F(t_0)$. There are two cases: either $\alpha$ is a corner, or there exists $\veps > 0$ such that $\alpha \lvert_{[-\veps, \veps]}$ is a geodesic.

In the first case, as in the proof of Lemma~\ref{Lem: ext pt sim}, the inclusion $B_{5r_0}(F(t_0)) \hookrightarrow B_{1/5}(F(t_0))$ induces the trivial map on $\pi_1$. Since $\alpha$ is a corner with an interior vertex $F(t_0)$ and two edges of length at least $1/5$, it follows from Proposition~\ref{Pro: comp exist corner int quant 2} that $B_{5r_0}(F(t_0)) \setminus \alpha$ has exactly two components. In the second case, it is not difficult to prove minor variants of Propositions~\ref{Pro: corner neigh} and \ref{Pro: comp exist extend corner quant} to show that $B_{5r_0}(F(t_0)) \setminus \alpha$ again has exactly two components. 

The argument then proceeds exactly as in the proof of Lemma~\ref{Lem: ext pt sim}. 
\end{proof}

The next lemma is the converse of Lemma~\ref{Lem: ext pt sim}. 

\begin{Lem}\label{Lem: ext pt sim conv}
Let $t_0 \in [0,r_2]$. Suppose that for every $t_1, t_1' \in [0,1]$ with $t_1 \leq t_0 \leq t_1'$ and $t_1 \sim t_1'$, it holds that $t_0 \sim t_1$. Then $F(t_0)$ is an extremal point.  
\end{Lem}

\begin{proof}
If $t_0 = 0$ then the statement is immediate, since Lemma~\ref{Lem: exp image inclusion} implies that $F(0) = p$, which is an extremal point. Thus assume that $t_0 \in (0, r_2]$.
To shorten the notation, set $\sigma:=\sigma_{t_0}$ and $\tilde \sigma:=\tilde \sigma_{t_0}$

Suppose for the sake of contradiction that $F(t_0)$ is an interior point. 
First, assume that $F(t_0) = \rho(t_0)$. By the proof of Lemma~\ref{Lem: exp image inclusion}, this implies that $\sigma = \tilde\sigma$.
Since $\rho(t_0)$ lies in the interior of $\rho$, it follows from Theorem~\ref{Thm: good components existence 2} that for all sufficiently small $s > 0$, the set $B_s(\rho(t_0)) \setminus \rho$ has exactly two components $G_\pm$. Assume that $B_s(\rho(t_0)) \subset B_{r_1}(p)$.

By property~(6) of $\rho$, the point $\rho(t_0)$ lies in the topological boundaries of $B_\pm$, where $B_\pm$ are the two components of $B_{10}(p) \setminus \rho$. Therefore, both $B_\pm$ intersect $B_s(\rho(t_0))$. It follows from Lemma~\ref{Lem: big set small set comp} that, up to relabeling $G_\pm$, we have $G_\pm = B_\pm\cap B_{s}(\rho(t_0))$

Since $\sigma(0) = q \in B_-$ and $\sigma$ does not intersect $\rho$ other than at $\sigma(1) = \rho(t_0)$, it follows that $\tilde \sigma([0, 1)) = \sigma([0,1)) \subset B_-$. Therefore, 
\begin{equation}\label{Eq: ext pt sim conv 1}
G_+ \cap \tilde \sigma([0,1]) = B_+ \cap \tilde \sigma([0,1]) = \emptyset.
\end{equation}

Set $t_2:= t_0 - s/20$ and $t_2' = t_0+s/20$. Then $t_2 < t_0 < t_2'$ and, since $\rho$ is unit-speed, $\rho(t_2), \rho(t_2') \in B_{s/10}(\rho(t_0))$.
By Theorem~\ref{Thm: good components existence 2}, the point $\rho(t_0)$ lies in the topological boundary of $G_+$, hence $G_+ \cap B_{s/10}(\rho(t_0)) \neq \emptyset$, and we can fix a point $x \in G_+ \cap B_{s/10}(\rho(t_0))$. 

Let $\nu:[0,1] \to X$ be a geodesic from $x$ to $\rho(t_2)$. By the triangle inequality, we have $\nu \subset B_s(\rho(t_0))$. Moreover, it follows from the non-branching property that $\nu$ does not intersect $\rho$ other than at $\nu(1) = \rho(t_2)$. Since $\nu(0) = x \in G_+$, it follows that $\nu([0,1)) \subset G_+$. Combined with \eqref{Eq: ext pt sim conv 1}, this implies that $\nu([0,1))$ and $\tilde \sigma$ are disjoint. Since $\nu(1) = \rho(t_2) \notin \tilde \sigma$, we conclude that $\nu$ and $\tilde \sigma$ are disjoint. Similarly, let $\nu':[0,1] \to X$ be a geodesic from $x$ to $\rho(t_2')$. Then $\nu'$ and $\tilde \sigma$ are disjoint.  
Let $\alpha:[-1,1] \to X$ be a continuous curve from $\rho(t_2)$ to $\rho(t_2')$ obtained by joining the reverse of $\nu$  to $\nu'$.
Notice that $\alpha$ is disjoint from $\tilde \sigma$ and 
\[
\alpha((-1,1)) \subset G_+ \cap B_{s}(\rho(t_0)) \subset B_+ \cap B_{r_1}(p).
\]

We define a map $g:[-1,1] \to [0, r_0]$ as follows: Set $g(-1) := t_2$ and $g(1) := t_2'$. For each $t \in (-1,1)$, choose a geodesic from $\alpha(t)$ to $q$. Since $\alpha(t) \in B_+ \cap B_{r_1}(p)$, it follows from our choice of $r_1$ and Lemma~\ref{Lem: exp map 1} that this geodesic intersects $\rho$ exactly once, at some point $\rho(t')$ with $t' \in [0, r_0]$. Define $g(t) := t'$.
We emphasize that $g(t)$ depends on the choice of the geodesic from $\alpha(t)$ to $q$. Since $\alpha \cap \tilde \sigma = \emptyset$, we have $g(t) \neq t_0$ for all $t \in [-1,1]$. 

Set
\[
S_- := \{t \in [-1,1] : g(t) < t_0\} \quad \text{and} \quad S_+ := \{t \in [-1,1]: g(t) > t_0\}.
\]
Then $S_+ \sqcup S_- = [-1,1]$ with $\pm 1 \in S_\pm$. By the connectedness of $[-1,1]$, there exists $a \in [-1,1]$ such that $a \in \overline S_+ \cap \overline S_-$.

Let $a_i \in S_-$ be a sequence such that $a_i \to a$, and let $\omega_i:[0,1] \to X$ be the geodesic from $\alpha(a_i)$ to $q$ used in the definition of $g$ (if $a_i=0$, we take $\omega_i$ to be any geodesic from $\rho(t_2)$ to $q$). Passing to a subsequence, $\omega_i$ converges uniformly to a geodesic $\omega$ from $\alpha(a)$ to $q$. Since each $\omega_i$ intersects $\rho$ at $\rho(g(a_i)) \in \rho([0,t_0])$, it follows that $\omega$ must also intersect $\rho$ at $\rho(t_1)$ for some $t_1 \in [0,t_0]$.  Since $\alpha(a) \notin \tilde \sigma$, this intersection cannot occur at $\rho(t_0)$. Therefore, $t_1 \in [0, t_0)$.

By considering a sequence $a_i' \in  S_+$ such that $a_i' \to a$, we also obtain a geodesic $\omega'$ from $\alpha(a)$ to $q$, such that $\omega'$ intersects $\rho$ at $\rho(t_1')$ for some $t_1' \in (t_0, r_0].$ In particular, $t_1' > t_0 > t_1$. 

Since $t_1, t_1' \in [0, r_0] \subset [0, 1]$, it follows from Lemma~\ref{Lem: extend no intersect} that both $\omega$ and $\omega'$ intersect $\rho$ exactly once, at $\rho(t_1)$ and $\rho(t_1')$ respectively. Since $t_1' > t_1$, we have $\rho(t_1') \neq \rho(t_1)$, and so the geodesics $\omega$ and $\omega'$ are distinct geodesics from $q$ to $\alpha(a)$. By the non-branching property, neither $\omega$ nor $\omega'$ is extendible past $\alpha(a)$. Therefore, $\omega$ (resp. $\omega'$) is the maximal extension from $q$ of its subsegment from $q$ to $\rho(t_1)$ (resp. $\rho(t_1')$). By definition, this implies $t_1 \sim t_1'$. 

By construction, $\omega(1) = \alpha(a) \neq \tilde \sigma(1)$. Therefore, $t_1 \nsim t_0$. Since $t_1' > t_0 > t_1$, this contradicts our initial assumptions. This concludes the proof in the case where $F(t_0) = \rho(t_0)$. 

Next assume that $F(t_0) \neq \rho(t_0)$. Then $\sigma$ is a proper subsegment of $\tilde \sigma$ by the proof of Lemma~\ref{Lem: exp image inclusion}. As in the previous case, let $s > 0$ be such that $B_s(\rho(t_0)) \subset B_{r_1}(p)$ and $B_s(\rho(t_0)) \setminus \rho$ has two components $G_\pm$. Set $t_2$ and $t_2'$ as before. We would like to again construct a curve $\alpha$ from $\rho(t_2)$ to $\rho(t_2')$ that is contained in $B_+ \cap B_{r_1}(p) \setminus \tilde \sigma$. However, the previous construction no longer works, since $\tilde \sigma$ intersects $G_-$. 

Let $\hat \sigma:[0,1] \to X$ be the subsegment of $\tilde \sigma$ from $\rho(t_0)$ to $\tilde \sigma(1)$, parameterized with constant speed and satisfying $\hat \sigma(0) = \rho(t_0)$. It is immediate from the proof of Lemma~\ref{Lem: exp image inclusion} that $\hat \sigma((0,1]) \subset B_+ \cap B_{r_1}(p)$.
Since $\rho(t) \in \mathcal{R}_2$ and lies in the interior of $\tilde \sigma$, it follows from Theorem~\ref{Thm: geodesic interior tangent} that $\tilde \sigma$ is a regular geodesic. In particular, $\hat \sigma$ is also a regular geodesic. By Proposition~\ref{Pro: geod neigh}, there exists a neighborhood $U$ of $\hat \sigma((0,1))$ such that $U \setminus \hat\sigma$ has exactly two components $U_\pm$, and the topological boundaries of both components in $U$ coincide with $\hat \sigma((0,1))$. Moreover, $U$ may be chosen so that $U\subset B_+\cap B_{r_1}(p)$ and so that $U$ does not intersect $\tilde \sigma$ outside of $\hat \sigma((0,1))$. In particular, $U_\pm$ do not intersect $\tilde \sigma$. 

Consider a geodesic from $\rho(t_2)$ to a point in $\hat \sigma((0,1)) \cap B_{s/10}(\rho(t_0))$. By the triangle inequality, this geodesic is contained in $B_{s}(\rho(t_0))$. By removing a suitable subsegment, we obtain a geodesic from $\rho(t_2)$ to a point $x_1$ in one of $U_\pm \cap B_{s/10}(\rho(t_0))$, which is disjoint from $\tilde \sigma$. Assume without loss of generality that $x_1 \in U_+$. 

Since $F(t_0) \in B_+ \cap B_{r_1}(p)$, we may choose $s' > 0$ such that $B_{s'}(F(t_0)) \subset B_+\cap B_{r_1}(p)$. Since the point $F(t_0) = \hat \sigma(1)$ lies in the topological boundary of $U_+$, we have $B_{s'}(F(t_0)) \cap U_+ \neq \emptyset$. Fix $x_2 \in B_{s'}(F(t_0)) \cap U_+$. Then $x_1$ can be connected to $x_2$ by a curve in $U_+$. 

Concatenating the two curves obtained above, we obtain a curve in $B_+ \cap B_{r_1}(p) \setminus \tilde \sigma$ from $\rho(t_2)$ to $x_2$. Repeating the same argument starting from $\rho(t_2')$, we obtain a point $x_2' \in B_{s'}(F(t_0))$ that can be connected to $\rho(t_2')$ by a curve contained in $B_+ \cap B_{r_1}(p) \setminus \tilde \sigma$. 

Since $F(t_0)$ is an interior point, the set $B_{s'}(F(t_0)) \setminus \tilde \sigma$ is path-connected. Therefore, the points $x_2$ and $x_2'$ can be joined by a curve in $B_{s'}(F(t_0)) \setminus \tilde \sigma \subset B_+ \cap B_{r_1}(p) \setminus \tilde \sigma$. 

The desired curve $\alpha$ from $\rho(t_2)$ to $\rho(t_2')$ can now be obtained by concatenating the curves constructed above: from $\rho(t_2)$ to $x_2$, from $x_2$ to $x_2'$, and finally from $x_2'$ to $\rho(t_2')$. Once the curve $\alpha$ has been constructed, the proof proceeds as in the previous case.
\end{proof}

\begin{Lem}\label{Lem: F lim ext pt}
Let $(t_i)_i$ be a sequence in $[0,r_2]$ converging to $t$. If $F(t_i) \in \Ext(X)$ for every $i \in \N$, then $F(t_i) \to F(t)$ as $i \to \infty$. Moreover, $F(t) \in \Ext(X)$. 
\end{Lem}

\begin{proof}
Since $F(t_i) \in \overline B_{r_1}(p)$ for all $i$ by Lemma~\ref{Lem: exp image inclusion}, and this set is compact, it suffices to show that every convergent subsequence of $(F(t_i))_i$ converges to $F(t)$. 

Pass to a subsequence such that $F(t_i) \to x$. For each $i \in \N$, set $\sigma_i:=\sigma_{t_i}$ and $\tilde \sigma_i:=\tilde \sigma_{t_i}$. After passing to a further subsequence, we may assume that $\tilde \sigma_i$ converges uniformly to some geodesic $\sigma:[0,1] \to X$ from $q$ to $x$. Since each $\sigma_i$ passes through $\rho(t_i)$, it follows by continuity that $\sigma$ passes through $\rho(t)$. 

Let $\tilde \sigma$ be the maximal extension of $\sigma$ from $q$. Since $\sigma$ passes through $\rho(t)$ and $t \in [0,r_2]$, it follows from Lemma~\ref{Lem: exp map 2} that $\tilde \sigma$ is a geodesic. Parameterize $\tilde \sigma$ with constant speed on $[0,1]$ so that $\tilde \sigma(0) = q$. Then $F(t) = \tilde \sigma(1)$.

Since $q$ lies in the interior of the regular geodesic $\gamma$, it is an interior point of $X$ by Corollary~\ref{Cor: int geod no ext}. Since $q$ is an endpoint of $\tilde \sigma$, Corollary~\ref{Cor: int geod reg} implies that $\tilde \sigma$ is a regular geodesic. Applying Corollary~\ref{Cor: int geod no ext} again, we conclude that the interior of $\tilde \sigma$ is contained in $\Int(X)$.

On the other hand, since $F(t_i) \in \Ext(X)$ and $\Ext(X)$ is closed by Theorem~\ref{Thm: int set open}, the limit point $x$ also lies in $\Ext(X)$. Therefore, $x$ cannot lie in the interior of $\tilde \sigma$. Moreover, since $\dd_X(q, F(t_i)) \geq 1 - r_1$ by our choice of $r_2$ and Lemma~\ref{Lem: exp map 2}, we also have $x \neq q = \tilde\sigma(0)$. Hence $x = \tilde\sigma(1) = F(t)$. In particular, $F(t_i) \to F(t) \in \Ext(X)$, as required. 
\end{proof}

For every $t \in [0,r_2]$, define 
\[
\mathcal T_t := \{s \in [0,t]: \text{there exists $s' \in [t,1]$ such that $s \sim s'$}\}.
\]
Clearly, $t \in \mathcal T_t$ and $\mathcal T_t \subset [0,t]$. 

\begin{Lem}\label{Lem: Tt closed}
For every $t \in [0,r_2]$, the set $\mathcal T_t$ is closed. Moreover, 
\begin{enumerate}
    \item[(i)] $\mathcal T_0 = \{0\}$;
    \item[(ii)] if $t \neq 0$, then $0 \notin \mathcal T_t$.
\end{enumerate}
\end{Lem}

\begin{proof}
Let $t \in [0,r_2]$, and let $s_i \in \mathcal T_t$ be a sequence converging to $s \in [0,t]$. We show that $s \in \mathcal T_t$. If $s = t$, then the statement is immediate since $t \in \mathcal T_t$. Thus assume that $s \neq t$. 

By definition, for every $i \in \N$, there exists $s_i' \in [t, 1]$ such that $s_i \sim s_i'$. 
To shorten notation, we write $\sigma_i:=\sigma_{s_i}$, $\sigma'_i:=\sigma_{s_i'}$, $\tilde \sigma_i:=\tilde \sigma_{s_i}$, and $\tilde \sigma_i':=\tilde \sigma_{s_i'}$.

By our choice of $r_2$ and Lemma~\ref{Lem: exp map 2}, we have $\tilde \sigma_i'(1) = \tilde \sigma_i(1) \in B_{r_1}(p)$ for all $i$. Passing to a subsequence, we may assume that $\tilde \sigma_i$ and $\tilde \sigma_i'$ converge uniformly to geodesics $\tilde \sigma, \tilde \sigma': [0,1] \to X$, respectively, with $\tilde \sigma(0) = \tilde \sigma'(0) = q$. 

It is immediate that $\tilde \sigma$ passes through $\rho(s)$, and $\tilde \sigma'$ passes through $\rho(s')$ for some $s' \in [t,1]$.  Moreover, since $\tilde \sigma_i$ and $\tilde \sigma_i'$ have the same endpoints for every $i$, the geodesics $\tilde \sigma$ and $\tilde \sigma'$ have the same endpoint as well. By the non-branching property, either $\tilde \sigma$ and $\tilde \sigma'$ are both not extendible past their common endpoint $\tilde \sigma(1) = \tilde \sigma'(1)$, or $\tilde \sigma = \tilde \sigma'$.

In the first case,  $\tilde \sigma$ and $\tilde \sigma'$ are maximal extensions of geodesics from $q$ to $\rho(s)$ and $\rho(s')$, respectively. Therefore, $s \sim s'$, and so $s \in \mathcal T_t$. In the second case, $\tilde \sigma$ passes through both $\rho(s)$ and $\rho(s')$. Since $s \in [0,t)$ and $s' \in [t, 1]$, we have $s \neq s'$.  Therefore, $\tilde \sigma$ intersects $\rho$ at least twice, contradicting Lemma~\ref{Lem: extend no intersect}. This finishes the proof that $\mathcal T_t$ is closed.

Statement~(i) is obvious, and Statement~(ii) follows immediately from Lemma~\ref{Lem: foc time bound}~(i).
\end{proof}

By Lemma~\ref{Lem: Tt closed}, for every $t \in [0,r_2]$, the set $\mathcal T_t \subset [0,r_2]$ admits a minimum. We define the function $m: [0,r_2] \to [0,r_2]$ by
\[
m(t) := \min(\mathcal T_t).
\]
The following lemma is immediate from the definition and Lemma~\ref{Lem: Tt closed}.
\begin{Lem}\label{Lem: m prop}
The following hold:
\begin{enumerate}
    \item[(i)] $m(0) = 0$.
    \item[(ii)] For every $t \in (0, r_2]$, $m(t) > 0$.
    \item[(iii)] For every $t \in [0,r_2]$, $m(t) \leq t$.
\end{enumerate}
\end{Lem}

\begin{Lem}\label{Lem: m ext}
Let $t \in [0,r_2]$. Then $m(m(t)) = m(t)$ and $F(m(t)) \in \Ext(X)$. 
\end{Lem}

\begin{proof}
Suppose for the sake of contradiction that $m(m(t)) \neq m(t)$. Then $m(m(t)) < m(t)$ by Lemma~\ref{Lem: m prop}~(iii). By definition of $m$, there exists $s_1 \in [t, 1]$ such that $m(t) \sim s_1$. Similarly, there exists $s_2 \in [m(t), 1]$ such that $m(m(t)) \sim s_2$. 

If $s_2 \geq t$, then $m(m(t)) < t \leq s_2$. Since $m(m(t)) \sim s_2$, it follows that $m(m(t)) \in \mathcal T_t$, contradicting the minimality of $m(t) = \min (\mathcal T_t)$. Therefore, we may assume that $s_2 < t$. Thus we have
$m(m(t)) < m(t) \leq s_2 < s_1$.

There are two cases, depending on whether $F(m(t))$ is an extremal point or an interior point. 
Lemmas~\ref{Lem: ext pt sim} and~\ref{Lem: int pt sim} imply in each case that $m(m(t)) \sim m(t) \sim s_1$. Since $m(m(t)) < t \leq s_1$, this implies that $m(m(t)) \in \mathcal T_t$, again contradicting the minimality of $m(t) = \min (\mathcal T_t)$. 

Since $\mathcal T_{m(t)} \subset [0,m(t)]$, it follows that $\mathcal T_{m(t)} = \{m(t)\}$. Lemma~\ref{Lem: ext pt sim conv} immediately implies that $F(m(t)) \in \Ext(X)$. 
\end{proof}

\begin{Lem}\label{Lem: m dec}
Let $t_1, t_2 \in [0,r_2]$ with $t_1 < t_2$. Then exactly one of the following holds:
\begin{enumerate}
    \item[(i)] $m(t_1) = m(t_2)$;    
    \item[(ii)] $m(t_1) \leq t_1 < m(t_2)$.
\end{enumerate}
In particular, $m$ is non-decreasing. 
\end{Lem}

\begin{proof}
First, assume that $m(t_2) \leq t_1$. By definition of $m(t_2)$, there exists $s_2 \in [t_2,1]$ such that $m(t_2) \sim s_2$. Since
$m(t_2) \leq t_1 < t_2 \leq s_2$,
it follows that $m(t_2) \in \mathcal T_{t_1}$. By the minimality of $m(t_1) = \min(\mathcal T_{t_1})$, we obtain $m(t_1) \leq m(t_2)$. Similarly, by definition of $m(t_1)$, there exists $s_1 \in [t_1, 1]$ such that $m(t_1) \sim s_1$. Since
$m(t_1) \leq m(t_2) \leq t_1 \leq s_1$,
it follows that $m(t_1) \in \mathcal T_{m(t_2)}$. By Lemma~\ref{Lem: m prop}~(iii), we have
$m(m(t_2)) \leq m(t_1) \leq m(t_2)$. Lemma~\ref{Lem: m ext} implies the previous inequalities are equalities, and so $m(t_1) = m(t_2)$, which is statement~(i).

If instead $m(t_2) > t_1$, then statement~(ii) follows immediately by Lemma~\ref{Lem: m prop}~(iii). Since statements~(i) and (ii) are mutually exclusive, we are finished.
\end{proof}

Define the map $\mathcal E:[0,r_2] \to X$ by \[\mathcal E(t) := F(m(t)).\]

\begin{Lem}\label{Lem: E prop}
The following hold:
\begin{enumerate}
    \item[(i)] $\mathcal E$ is non-constant;
    \item[(ii)] $\mathcal E((0,r_2]) \subset A_+$.
\end{enumerate}
\end{Lem}

\begin{proof}
By Lemma~\ref{Lem: m prop}, we have $m(0) = 0$ and $m(t) > 0$ for every $t \in (0,r_2]$. Let $t \in (0,r_2]$. By Lemma~\ref{Lem: exp image inclusion}, we have $F(m(0)) = p$, and either $F(m(t)) = \rho(m(t))$ or $F(m(t)) \in B_+ \cap B_{r_0}(p)$. In both cases, $F(m(t)) \neq F(m(0))$, and so $\mathcal E(t) \neq \mathcal E(0)$. Therefore, $\mathcal E$ is non-constant. 

Next, we prove statement~(ii). By \eqref{Eq: Apm Bpm assumption}, we have $\gamma((0,10)) \subset B_-$ and $\rho((0,10)) \subset A_+$.
The former implies that $\gamma$ does not intersect $B_+$. Since $B_+$ is path-connected, it must then be entirely contained in one of the components of $B_{10}(p) \setminus \gamma$, namely one of $A_\pm$. Since $\rho((0,10))$ coincides with the topological boundary of $B_+$ in $B_{10}(p)$ by property~(6) of $\rho$, and is contained in $A_+$, it is not difficult to show that $B_+ \subset A_+$. By Lemma~\ref{Lem: exp image inclusion}, we have  
\[\mathcal E((0,r_2]) \subset \rho((0,r_2]) \cup \big(B_+ \cap B_{r_0}(p)\big) \subset A_+. \qedhere\]
\end{proof}

\begin{Lem}\label{Lem: E ext}
The map $\mathcal E$ is continuous, and its image $\mathcal E([0,r_2]) \subset \Ext(X)$. Moreover, the level sets of $\mathcal E$ are closed intervals (a singleton in $\R$ is considered a closed interval). 
\end{Lem}

\begin{proof}
It is immediate from Lemma~\ref{Lem: m ext} that $\mathcal E([0,r_2]) = F(m([0,r_2])) \subset \Ext(X)$. 

Let $t \in [0,r_2]$. We will prove that $\mathcal E$ is continuous at $t$. It suffices to show that for any sequence $\{t_i\}$ in $[0,r_2]$, with $t_i \to t$, we have $\mathcal E(t_i) \to \mathcal E(t)$. 

First, assume that $t_i \downarrow t$. If there exist $s \in (t,1]$ such that $m(t) \sim s$, then for sufficiently large $i$, we have $m(t) \leq t < t_i \leq s$. Therefore, $m(t) \in \mathcal T_{t_i}$ and so $m(t) \geq \min(T_{t_i}) = m(t_i)$. Since $m$ is non-decreasing by Lemma~\ref{Lem: m dec}, we have $m(t_i) = m(t)$, which implies $\mathcal E(t_i) = \mathcal E(t)$. 

On the other hand, if there does not exist $s \in (t,1]$ such that $m(t) \sim s$, then $m(t) \sim t$.  We claim that for every $i \in \N$, $m(t_i) > t$. Indeed, if there exists $i$ such that $m(t_i) \leq t$, then Lemma~\ref{Lem: m dec} implies $m(t_i) = m(t)$. By definition of $m(t_i)$, there exists $s \in [t_i, 1] \subset (t, 1]$ such that $s \sim m(t_i) = m(t)$, contradicting our assumption. Therefore, $t_i \geq m(t_i) > t$ for each $i$. In particular, $m(t_i) \to t$. Since each $F(m(t_i)) \in \Ext(X)$ by Lemma~\ref{Lem: m ext}, Lemma~\ref{Lem: F lim ext pt} implies $F(m(t_i)) \to F(t).$ Since $m(t) \sim t$, this is exactly the statement that $\mathcal E(t_i) \to \mathcal E(t)$. 

Next, assume that $t_i \uparrow t$. If $m(t) < t$, then $m(t) < t_i$ for sufficiently large $i$. It follows from Lemma~\ref{Lem: m dec} that $m(t) = m(t_i)$, and so $\mathcal E(t_i) = \mathcal E(t)$. In particular, $\mathcal E(t_i) \to \mathcal E(t)$. 

If $m(t)=t$, then we claim that for every $i \in \N$, there does not exist $s \geq t$ such that $s \sim m(t_i)$. Indeed, suppose for some $i$ there exists $s \in [t,1]$ with $s \sim m(t_i)$. Since $m(t_i) \leq t_i < t$, this implies that $m(t_i) \in \mathcal T_t$, contradicting the minimality of $t = m(t) = \min(\mathcal T_t)$. Therefore, for each $i$, there exists $s_i \in [t_i, t)$ such that $s_i \sim m(t_i)$. In particular, $s_i \to t$ as $i \to \infty$. By Lemma~\ref{Lem: m ext}, we have $F(s_i) = F(m(t_i)) \in \Ext(X)$. Applying Lemma~\ref{Lem: F lim ext pt} to the sequence $s_i \to t$, we obtain $F(s_i) \to F(t)$. Since $F(s_i) = F(m(t_i)) = \mathcal E(t_i)$ and $F(t) = \mathcal E(t)$ (because $m(t) = t$), this is exactly the statement that $\mathcal E(t_i) \to \mathcal E(t)$.

Finally, let $x \in \mathcal E([0,r_2])$. Then there exists $t \in [0,r_2]$ such that $x = F(m(t))$. Arguing as in the proof of Lemma~\ref{Lem: Tt closed}, we can show that the set $\{s \in [m(t),1]: s \sim m(t)\}$ is closed and non-empty, and therefore admits a maximum $t'$. Using the monotonicity of $m(t)$, it is not difficult to see that  $\mathcal E^{-1}(x) = [m(t), t'] \cap [0, r_2]$.
\end{proof}

\subsection{Proof of Proposition~\ref{Pro: ext pt curve}}
By Proposition~\ref{Pro: every pt have extend reg geod}, there exists a regular geodesic $\gamma:[0,1] \to X$ such that $\gamma(0) = p$. Let $R > 0$ be the length of $\gamma$. 

By Theorem~\ref{Thm: semi-locally simply connected} and Proposition~\ref{Pro: comp exist geod ep ext quant}, there exists $0 < r \leq R$ such that for every $0 < s \leq r$, the set $B_{s}(p) \setminus \gamma$ has exactly two components. Denote by $\tilde A_\pm$ the components of $B_{r/2}(p) \setminus \gamma$. By Proposition~\ref{Pro: comp exist geod ep ext quant}, there exist points $x_\pm \in \tilde A_\pm$ such that the geodesics from $p$ to $x_\pm$ are unique and regular. Let $\rho_\pm: [0,1] \to X$ denote these geodesics, with $\rho_\pm(0) = p$. It is not difficult to check from the non-branching property that $\rho_\pm$ intersect $\gamma$ only at $p = \gamma(0)$. Therefore, $\rho_\pm((0,1]) \subset \tilde A_\pm$.

Set $R' := \min(\dd_X(p, x_+), \dd_X(p,x_-)) < r$. By Theorem~\ref{Thm: semi-locally simply connected}
and Proposition~\ref{Pro: comp exist geod ep ext quant}, there exists $0 < r' \leq R'$ such that for every $0< s \leq r'$, the sets $B_{s}(p) \setminus \rho_\pm([0,1])$ each have exactly two components. Denote by $A_\pm$ the components of $B_{r'}(p) \setminus \gamma$. By Lemma~\ref{Lem: ext point disconnect all small ball} and up to relabeling $A_\pm$, we have $A_\pm = \tilde A_\pm \cap B_{r'}(p)$.
In particular, $\rho_\pm((0,1]) \cap B_{r'}(p) \subset A_\pm$. 

Replace $\gamma$ and $\rho_\pm$ with their subsegments of length $r'$ starting from $p$. It is immediate that conditions (1), (2) and (5) in the initial setup of this section are satisfied for the pairs $(\gamma, \rho_+)$ and $(\gamma, \rho_-)$ (after rescaling by a factor of $1/(10r')$). By Lemmas~\ref{Lem: ext-pt comp boundary} and \ref{Lem: ext point disconnect all small ball}, all remaining conditions (3), (4), (6), and (7) are also satisfied. 

Applying Lemma~\ref{Lem: E ext} to the pair $(\gamma, \rho_+)$, we obtain a continuous map $\mathcal E_+:[0,1] \to X$ such that $\mathcal E_+ \subset \Ext(X)$ and $\mathcal E_+(0) = p$. Since the level sets of $\mathcal E_+$ are closed intervals and $\mathcal E_+$ is not a singleton by Lemma~\ref{Lem: E prop}~(i), it is a standard fact (see for instance \cite[Proposition 8.22]{N92}) that $\mathcal E_+$ is an arc, i.e. homeomorphic to the closed interval $[0,1]$. More precisely, there exists a continuous injective map $\tilde {\mathcal E}_+:[0,1] \to X$ with $\tilde{\mathcal E}_+(0) = p$ such that $\tilde {\mathcal E}_+([0,1]) = \mathcal E_+([0,1])$. By replacing $\mathcal E_+$ with $\tilde{\mathcal E}_+$, we may assume that $\mathcal E_+$ is injective.

Applying the previous argument to the pair $(\gamma, \rho_-)$, we also obtain a simple curve $\mathcal E_-:[0,1] \to X$ such that $\mathcal E_- \subset \Ext(X)$ and $\mathcal E_-(0) = p$. By Lemma~\ref{Lem: E prop}~(ii), we have $\mathcal E_\pm((0,1]) \subset A_\pm$. In particular, $\mathcal E_+$ and $\mathcal E_-$ intersect only at $p$. 

Finally, the continuous curve $\mathcal E:[-1,1] \to X$ defined by
\[
\mathcal E(t) :=
\begin{cases}
    \mathcal E_+(t) &\text{if $t\in[0,1]$,}\\
	\mathcal E_-(-t) &\text{if $t \in [-1,0)$}
	\end{cases}
\]
is simple and contained in the extremal set, with $\mathcal E(0) = p$. \qed

\medskip
The previous proof in fact gives the following stronger statement.
\begin{Cor}\label{Cor: ext pt curve}
Let $(X, \dd_X, \mm_X)$ be an $\RCD(K,N)$ space of essential dimension $2$, and let $\gamma:[0,1] \to X$ be a regular geodesic such that that $p := \gamma(0) \in \Ext(X)$. Let $R$ be the length of $\gamma$ and let $0 < r \leq R$ be such that the inclusion $B_r(p) \hookrightarrow B_R(p)$ induces the trivial map on $\pi_1$. Then the set $B_{r}(p) \setminus \gamma$ has exactly two components $A_\pm$, and there exists a simple continuous curve $\mathcal E:[-1,1] \to X$ with $\mathcal E(0) = p$ such that
\[
\mathcal E \subset \Ext(X), \quad \mathcal E((0,1]) \subset A_+, \quad \text{and} \quad \mathcal E([-1,0)) \subset A_-.
\]
\end{Cor}

\section{Boundary polygonal domains}\label{Sec: bd poly dom} 

In this section, we introduce the notion of a boundary polygonal domain (see Definition~\ref{Def: bd poly dom}; cf. Definition~\ref{Def: poly dom}) and show that every extremal point admits a neighborhood basis consisting of boundary polygonal domains. In Section~\ref{Sec: man struct}, we will prove that boundary polygonal domains are homeomorphic to open half-disks. This will allow us to deduce a manifold-with-boundary structure in a neighborhood of every extremal point.

For the rest of the section, let $(X,\dd_X,\mm_X)$ be an $\RCD(K,N)$ space of essential dimension~$2$.

\begin{Def}[Boundary polygonal loop]\label{Def: bd poly loop}
A continuous curve $\alpha:[-1,1]\to X$ is called a \emph{simple boundary polygonal loop} if the following conditions hold:
\begin{enumerate}
    \item $\alpha(-1) = \alpha(1)$ and $\alpha \lvert_{[-1,1)}$ is injective.  
    \item The restriction $\alpha \lvert_{[0,1]}$ is a simple polygonal curve.
    \item $\alpha((0,1)) \subset \Int(X)$ and $\alpha([-1,0]) \subset \Ext(X)$.
\end{enumerate}
\end{Def}
The term ``boundary'' reflects the fact that the portion $\alpha([-1,0])$, called the \emph{boundary curve} of $\alpha$, is contained in the extremal set. The notions of \emph{speed}, \emph{interior vertices}, \emph{extremal vertices}, \emph{edges}, and \emph{edges} of $\alpha$ are defined to be those of its restriction $\alpha|_{[0,1]}$.

The first result of this section is a version of the Jordan Theorem~\ref{Thm: Jordan theorem} for boundary loops. It will be the main tool to construct boundary polygonal domain.

\begin{Thm}[Jordan theorem, boundary]\label{Thm: Jordan theorem boundary}
Let $(X,\dd_X,\mm_X)$ be an $\RCD(K,N)$ space of essential dimension $2$. Let $\alpha \subset B_R(p)$ be a simple boundary polygonal loop. If the inclusion $B_R(p) \hookrightarrow X$ induces the trivial map on $\pi_1$, then $B_R(p) \setminus \alpha$ has exactly two components, and their topological boundaries in $B_R(p)$ are contained in $\alpha$.

If, in addition, there exist $0 < s \le r < R$ such that $\alpha \subset B_s(p)$ and the inclusion $B_s(p) \hookrightarrow B_r(p)$ induces the trivial map on $\pi_1$, then one of the connected components of $B_R(p) \setminus \alpha$ is contained in $\overline B_r(p)$, and its topological boundary is $\alpha$. In this case, $X \setminus \alpha$ also has exactly two components, and at least one of them coincides with a component of $B_R(p) \setminus \alpha$ contained in $\overline B_r(p)$. Moreover, the topological boundary of both components contain $\alpha([0,1])$.
\end{Thm}

The main difference with the Jordan theorem for interior domains is that the two connected components have different topological boundaries: the interior component is bounded by $\alpha$, while the exterior component is bounded only by the non-boundary part of $\alpha$. This is already clear in elementary examples, such as a half-disk in the Euclidean upper half-plane, where the boundary loop consists of an arc in the boundary line together with a curve in the interior. In that case, the interior component has boundary equal to the whole loop, whereas the exterior component has boundary given only by the interior part of the loop. Thus, in the presence of boundary, the statement above is optimal.

\begin{Def}[Boundary polygonal domain]
\label{Def: bd poly dom}
We say that a subset $Q \subset X$ is a \emph{boundary polygonal domain} if the following conditions hold:
\begin{enumerate}
    \item $Q$ is open, connected, and precompact.

    \item There exists a boundary polygonal loop $\alpha_Q$ (see Definition~\ref{Def: bd poly loop}) such that $\alpha_Q((-1,0))\subset Q$ and the topological boundary $\partial Q:=\overline Q\setminus Q$ coincides with $\alpha_Q([0,1])$.

    \item $\overline Q \cap \Ext(X) = \alpha_Q([-1,0])$ (see Definition~\ref{Def: ext-int points}).
\end{enumerate}
\end{Def}

Boundary polygonal domains play the role of neighborhoods of boundary points in manifolds with boundary. In our framework, they arise as the interior components appearing in the second conclusion of Theorem~\ref{Thm: Jordan theorem boundary}, that is, as connected components of the complement of a simple boundary polygonal loop, together with the corresponding boundary portion.

\begin{Rem}\label{Rem: polydom consistency2}
Assume that $Q\subset B_r(p)$ is a boundary polygonal domain. If the inclusions $B_R(p)\hookrightarrow X$ and $B_r(p)\hookrightarrow B_{R/2}(p)$ induce the trivial map on $\pi_1$ for some $R\ge 2r$, and there exists a point $q\in X\setminus B_R(p)$, then $\alpha_Q$ disconnects $X$ into two connected components, and $Q\cap \Int(X)$ is precisely the component contained in $B_R(p)$. Compare with Remark \ref{Rem: polydom consistency1}.
% Indeed, it is clear that $Q\cap \Int(X)$ is contained in one of the two components of $X\setminus \alpha_Q$. If it were strictly smaller, then we could find a continuous curve in the same connected component joining a point of $Q\cap \Int(X)$ to a point in the complement of $Q$. Such a curve would have to intersect the topological boundary of $Q$, namely $\partial Q=\alpha_Q([0,1])$, which is impossible since $\partial Q$ is disjoint from the connected component of $X\setminus \alpha_Q$ under consideration.
\end{Rem}

Conversely, if $Q$ is a boundary polygonal domain such that $\overline Q\subset B_r(p)$ and the geometric assumptions above are satisfied, then $\alpha_Q$ disconnects the space into two connected components, and $Q$ is precisely the component contained in $B_R(p)$.

\begin{Thm}[cf. Theorem~\ref{Thm: poly dom neigh basis}]\label{Thm: bd poly dom neigh basis}
Let $(X,\dd_X,\mm_X)$ be an $\RCD(K,N)$ space of essential dimension $2$. Then for every $p \in \Ext(X)$ and every $R>0$, there exists a simple boundary polygonal loop $\alpha:[-1,1] \to X$ such that the following hold:
\begin{enumerate}
\item $\alpha\subset B_R(p)$, $\alpha((0,1))\subset \mathcal{R}_2$, and $p \in \alpha((-1,0))$;

\item $X\setminus \alpha([0,1])$ has exactly two connected components. Moreover, one of these components $Q$ is a boundary polygonal domain contained in $B_R(p)$.
\end{enumerate}
In particular, every extremal point $p\in \Ext(X)$ admits a neighborhood basis consisting of boundary polygonal domains.
\end{Thm}

\subsection{Proof of the Jordan Theorem~\ref{Thm: Jordan theorem boundary}}

The proof is divided into two parts. The first, which contains the core of the argument, is a separation lemma for simple polygonal curves whose endpoints lie in the extremal set. The model example to keep in mind is the standard two-dimensional disk, where an arc joining two boundary points disconnects the disk. We now describe the precise setting.

Let $(X, \dd_X, \mm_X)$ be an $\RCD(K,N)$ space of essential dimension~$2$. Let $\alpha:[0,1] \to X$ be a simple polygonal curve such that
\begin{equation}\label{Eq: poly curve ext ep}
\alpha((0,1)) \subset \Int(X) \quad \text{and} \quad \alpha(\{0,1\}) \subset \Ext(X).
\end{equation}
We have the following separation lemma for $\alpha$; compare with Theorems~\ref{Thm: Jordan theorem 1} and~\ref{Thm: Jordan theorem 2}.

\begin{Lem}[Separation lemma]\label{Lem: sep lem}
Assume that $\alpha \subset B_R(p)$ for some $p \in X$ and $R > 0$. If the inclusion $B_R(p) \hookrightarrow X$ induces the trivial map on $\pi_1$, then $B_R(p) \setminus \alpha$ has exactly two components, and the topological boundary in $B_R(p)$ of each component coincides with $\alpha$.

If in addition there exist $0 < s \leq r < R$ such that $\alpha \subset B_s(p)$ and the inclusion $B_s(p) \hookrightarrow B_r(p)$ induces the trivial map on $\pi_1$, then at least one of the components of
$B_R(p) \setminus \alpha$ is contained in $\overline B_r(p)$.
In this case, $X \setminus \alpha$ has exactly two path-connected
components and at least one of these components coincides with a component of
$B_R(p) \setminus \alpha$ contained in $\overline B_r(p)$. Moreover, the topological boundaries of both components coincide with $\alpha$.
\end{Lem}

This result is analogous to Jordan Theorem~\ref{Thm: Jordan theorem}, and the arguments closely follow those in Section~\ref{Sec: poly loop}. Accordingly, we only list the main ingredients and omit the full proof.

Since $\alpha$ is simple, there exists $r_0>0$ such that, for all $0<r<r_0$, we have
\begin{equation}\label{Eq: ext pt neigh 1}
B_{r}(\alpha(0)) \setminus \alpha = B_{r}(\alpha(0)) \setminus \alpha([0,t_1]).
\end{equation}
Since $\alpha|_{[0,t_1]}$ is a regular geodesic with endpoint $\alpha(0)$, and $\alpha(0)$ is extremal, Proposition~\ref{Pro: comp exist geod ep ext quant} implies that, for all sufficiently small $0 < r \leq r_0$, the set $B_r(\alpha(0))\setminus \alpha([0,t_1])$
has exactly two connected components, and the topological boundary of each of them in $B_r(\alpha(0))$ coincides with $\alpha([0,t_1]) \cap B_r(\alpha(0))$. By \eqref{Eq: ext pt neigh 1}, the same conclusion holds with $\alpha$ in place of $\alpha([0,t_1])$. By an analogous argument, the same conclusion also holds for the set $B_{r}(\alpha(1)) \setminus \alpha$.

Arguing as in Proposition~\ref{Pro: corner neigh} from Subsection~\ref{Subsec: corner neigh}, together with the auxiliary lemmas leading up to it, we obtain the following variant of Proposition~\ref{Pro: corner neigh} (see also the proof of Proposition~\ref{Pro: poly loop neigh}). This proposition allows us to prove Lemma~\ref{Lem: sep lem} by following the same strategy as in the proof of Theorem~\ref{Thm: Jordan theorem}.

\begin{Pro}\label{Pro: bd poly loop neigh}
Let $V$ be an open neighborhood of $\alpha$. Then there exists an open path-connected neighborhood $U \subset V$ of $\alpha$ such that $U \setminus \alpha$ has exactly two connected components $U_\pm$. Moreover, the topological boundaries of $U_\pm$ in $U$ both coincide with $\alpha$.
\end{Pro}

We record a corollary of Lemma~\ref{Lem: sep lem} and Theorem~\ref{Thm: semi-locally simply connected}. Although it is not used in the proof of Theorem~\ref{Thm: Jordan theorem boundary}, it will play a role later in Section~\ref{sec:proof_bdpolydomneigh}.

\begin{Cor}[cf. Corollary~\ref{Cor: Jordan theorem}]\label{Cor: sep lem}
Let $R > 0$ and $p \in X$. Then there exists $0 < \bar r(p,R) < R$ such that for every $0 < r \leq \bar r$, if $\alpha \subset B_r(p)$, then $X \setminus \alpha$ has two components and at least one of them is contained in $B_R(p)$. Moreover, the topological boundaries of both components coincide with $\alpha$.
\end{Cor}

Let us now consider a simple boundary polygonal loop $\alpha\subset B_R(p)$ as in the statement of Theorem~\ref{Thm: Jordan theorem boundary}. Then the hypotheses of Lemma~\ref{Lem: sep lem} are satisfied by the restriction $\alpha|_{[0,1]}$. In particular, the set $X \setminus \alpha([0,1])$ has exactly two connected components. Since $\alpha((-1,0))$ is disjoint from $\alpha([0,1])$, it is contained in one of these two components, which we denote by $P$. The following lemma completes the proof of Theorem~\ref{Thm: Jordan theorem boundary}.

\begin{Lem}\label{Lem: bd poly loop no disconnect}
The set $P \setminus \alpha = P \setminus \alpha((-1,0))$ is path-connected.
\end{Lem}

\begin{proof}
For each $t \in (-1,0)$, let $r(t) > 0$ be such that $B_{r(t)}(\alpha(t)) \subset P$. Set 
\[
U := \bigcup_{t \in (-1,0)} B_{r(t)}(\alpha(t)) \subset P.
\]
We first show that $U \setminus \alpha$ is path-connected.

Let $t, t' \in (-1,0)$. By compactness, there exists a finite sequence $t = t_0 < \dots < t_m = t'$ such that for each $i = 1, \dots, m$, we have $B_{r(t_{i-1})}(\alpha(t_{i-1})) \cap B_{r(t_{i})}(\alpha(t_{i})) \neq \emptyset$. 
Since $\Int(X)$ is dense in $X$ (Theorem~\ref{Thm: int set open}), any such intersection contains an interior point, which in particular does not lie on $\alpha$. By Corollary~\ref{Lem: ext set no disconnect}, $B_{r(t)}(\alpha(t)) \setminus \alpha$
is path-connected. Since the union of two path-connected sets with non-empty intersection is path-connected, the set
\[
\bigcup_{i=1}^m B_{r(t_i)}(\alpha(t_{i})) \setminus \alpha \subset U \setminus \alpha
\]
is path-connected. It follows easily that $U \setminus \alpha$ is path-connected. 

Let $x, y \in P \setminus \alpha$. We now show that $x$ and $y$ can be connected by a curve in $P \setminus \alpha$. 

Since $P$ is path-connected, there exists a curve $\rho:[0,1] \to X$ from $x$ to $y$ contained in $P$. If $\rho$ does not intersect $\alpha$, then we are done. Otherwise, there must a smallest $t_1 \in (0,1]$ and a largest $t_2 \in [0,1)$ such that $\rho(t_1), \rho(t_2) \in \alpha$. By continuity, there exists $\veps > 0$ such that $\rho(t_1 - \veps), \rho(t_2 + \veps) \in U \setminus \alpha$. Since $U \setminus \alpha$ is path-connected, there exists a continuous curve $\sigma:[0,1] \to X$ from $\rho(t_1 - \veps)$ to $\rho(t_2 + \veps)$ that is contained in $U \setminus \alpha \subset P\setminus \alpha$. The union of $\rho$ (from $x$ to $\rho(t_1 - \varepsilon)$), $\sigma$, and $\rho$ (from $\rho(t_2 + \varepsilon)$ to $y$) is a curve in $P \setminus \alpha$ joining $x$ and $y$. 
\end{proof}

\subsection{Proof of Theorem~\ref{Thm: bd poly dom neigh basis}}\label{sec:proof_bdpolydomneigh}

Let $p \in X$ be an extremal point. By Proposition~\ref{Pro: every pt have extend reg geod}, there exists a regular geodesic $\gamma:[0,1] \to X$ with $\gamma(0) = p$. By rescaling $X$, we may assume that $\gamma$ is unit-speed. 
By Corollary~\ref{Cor: sep lem}, there exists $0 < \bar R \leq 1/2$ such that for every simple polygonal curve $\alpha \subset B_{2\bar R}(p)$ and satisfying \eqref{Eq: poly curve ext ep}, the set $X \setminus \alpha$ has exactly two components, at least one of which is contained in $B_1(p)$. 

Fix $0 < R \leq \bar R$. By Proposition~\ref{Pro: comp exist geod ep ext quant} and Corollary~\ref{Cor: sep lem}, there exists $0 < r \leq R$ such that the following hold:
\begin{enumerate}
    \item The set $B_r(p) \setminus \gamma$ has exactly two components, denoted by $A_\pm$.
    \item For every simple polygonal curve $\alpha \subset B_r(p)$ satisfying $\eqref{Eq: poly curve ext ep}$, the set $X \setminus \alpha$ has exactly two components, at least one of which is contained in $B_R(p)$. 
\end{enumerate}

By Corollary~\ref{Cor: ext pt curve}, there exists a simple continuous curve $\mathcal E:[-1,1] \to X$ such that
\[
\mathcal E(0) = p, \quad \mathcal E([-1,1]) \subset \Ext(X), \quad \mathcal E((0,1]) \subset A_+, \quad \mathcal E([-1,0)) \subset A_-.  
\]
By continuity, we may choose $0< \delta \leq 1$ such that
\[
\mathcal E((0, \delta]) \subset A_+ \cap B_{r/2}(p) \quad \text{and} \quad \mathcal E([-\delta, 0)) \subset A_- \cap B_{r/2}(p).
\]

Set $q := \gamma(r/2) \in \mathcal R_2$. Then $\dd_X(p,q) = r/2$, and $q \in \Int(X)$ by Corollary~\ref{Cor: int geod no ext}. Let $\rho_\pm$ be geodesics from $q$ to $\mathcal E(\pm\delta)$ parametrized on $[0,1]$. By the triangle inequality, $\rho_\pm \subset B_r(p)$.

\begin{Cla}\label{Cla: bd poly dom neigh basis}
The geodesics $\rho_+$, $\rho_-$, and $\gamma$ intersect only at $q$. Moreover, $\rho_\pm((0,1]) \subset A_\pm$.
\end{Cla}

\begin{proof}
We argue only for $\rho_+$, since the proof for $\rho_-$ is identical. Suppose, for contradiction, that $\rho_+$ and $\gamma$ intersect at a point $\gamma(t)$ with $t\neq r/2$.

First, consider the case $t\in [0,r/2)$. Since the geodesics $\rho_+$ and $\gamma|_{[0,r/2]}$ share exactly one endpoint, namely $q$, the non-branching property implies that one of them must be a proper subsegment of the other. On the other hand, both $\rho_+$ and $\gamma|_{[0,r/2]}$ are regular by Corollary~\ref{Cor: int geod reg}, since $q\in \Int(X)$. In particular, neither geodesic contains an extremal point in its interior by Corollary~\ref{Cor: int geod no ext}. Since the endpoints $\rho_+(1)= \mathcal E(\delta)$ and $\gamma(0)=p$ are both extremal, we obtain a contradiction.

Next, suppose that $t \in (r/2,1]$. Since $\gamma(1) \in X \setminus B_r(p)$ ($\dd_X(\gamma(1), p) = 1 \geq r$) and $\rho_+(1) \in B_r(p)$, we have $\gamma(1) \neq \rho_+(1)$. It follows again from the non-branching property that one of $\rho$ and $\gamma \lvert_{[r/2,1]}$ must be a proper subsegment of the other. On the other hand, $\gamma \lvert_{[r/2,1]}$ cannot be a subsegment of $\rho_+$, since $\rho_+ \subset B_r(p)$ while $\gamma(1) \in X \setminus B_r(p)$, and $\rho$ cannot be a proper subsegment of $\gamma \lvert_{[r/2,1]}$ by arguing as in the previous case, yielding a contradiction. 

It follows that $\rho_+$ and $\gamma$ intersect only at $q$. In particular $\rho_+((0,1]) \subset B_r(p) \setminus \gamma$. Since $\rho_+(1) \in A_+$, we conclude that $\rho_+((0,1]) \subset A_+$.
\end{proof}

Let $\alpha$ be the loop obtained by concatenating the reverse of $\rho_+$ from $\mathcal E(\delta)$ to $q$, $\rho_-$ from $q$ to $\mathcal E(-\delta)$, and $\mathcal E$ from $\mathcal E(-\delta)$ to $\mathcal E(\delta)$. It is straightforward to check that, after an appropriate reparameterization on $[-1,1]$, $\alpha$ is a simple boundary polygonal loop such that
\[
p \in \alpha((-1,0)), \quad \alpha(0) = \mathcal E(\delta), \quad \alpha(1) = \alpha(-1) = \mathcal E(-\delta), \quad  \alpha \subset B_{r}(p), \quad \alpha((0,1)) \subset \mathcal R_2.
\] 
In particular, $\alpha$ satisfies Theorem~\ref{Thm: bd poly dom neigh basis}~(1). 

By our choice $r$, the set $X \setminus \alpha([0,1])$ has exactly two components, and at least one is contained in $B_R(p)$. Since $\gamma(1) \in X \setminus B_R(p)$, exactly one of these components, denoted by $Q$, is contained in $B_R(p)$. Denote the other component by $Q'$. Then $\gamma(1) \in Q'$. 

\begin{Cla}
$Q \cap \Ext(X) = \alpha((-1,0))$.
\end{Cla}

\begin{proof}
We first show that $\alpha((-1,0)) \subset Q$. We know $\gamma$ intersects $\alpha$ exactly once, namely at $q$. Let $t_0 \in [-1,1]$ be such that $\alpha(t_0) = q$. Since $q \in \Int(X)$, we have $t_0 \in (0,1)$. 

There are two cases: either $\alpha$ is locally a geodesic at $\alpha(t_0)$, or $\alpha$ is locally a corner at $\alpha(t_0)$. In both cases, one argues exactly as in the proof of Theorem~\ref{Thm: poly dom neigh basis}, using Lemma~\ref{Lem: local config} in the corner case, to show that the sets $\gamma([0,r/2))$ and $\gamma((r/2,1])$ lie in different connected components of $X \setminus \alpha([0,1])$. Since $\gamma(1) \in Q'$, it follows that $p=\gamma(0)$ belongs to $Q$.
Since $p \in \alpha((-1,0))$ and the set $\alpha((-1,0))$ is disjoint from $\alpha([0,1])$, it follows that $\alpha((-1,0)) \subset Q$.

It remains to show that $\Ext(X) \cap Q \subset \alpha((-1,0))$. Suppose for the sake of contradiction that there exists $x \in \Ext(X) \cap Q$ such that $x \notin \alpha((-1,0))$. 

Let $\nu:[0,1] \to X$ be a geodesic from $q$ to $x$. Since $x \in Q \subset B_R(p)$ and $q \in B_r(p) \subset B_R(p)$, the triangle inequality implies that $\nu \subset B_{2R}(p)$. Let $\beta:[0,1] \to X$ be the simple polygonal curve satisfying \eqref{Eq: poly curve ext ep} obtained by concatenating the reverse of $\rho_+$ from $\mathcal E(\delta)$ to $q$ and $\nu$ from $q$ to $x$. Then $\beta \subset B_{2R}(p)$, and hence, by our choice of $R$, the set $X \setminus \beta$ has exactly two connected components. Since $\gamma(1) \in X \setminus B_1(p)$, exactly one of these components, denoted $P$, is contained in $B_1(p)$. We denote the other by $P'$, so $\gamma(1) \in P'$.

Arguing as in Claim~\ref{Cla: bd poly dom neigh basis} using the non-branching property, we see that $\nu$ does not intersect $\rho_+$ and $\rho_-$ other than at $q$. This implies that $\nu((0,1]) \subset X \setminus \alpha([0,1])$, and is therefore contained entirely in either $Q$ or $Q'$. Since $\nu(1) = x \in Q$, it follows that $\nu((0,1]) \subset Q$. Consequently, $\beta$ is disjoint from $Q'$. 
Since $Q'$ is path-connected, it must then be contained entirely in either $P$ or $P'$. Since $\gamma(1) \in P' \cap Q'$, we conclude that $Q' \subset P'$. 
Moreover, $\rho_-((0,1])$ lies in the topological boundary of $Q'$ by Corollary~\ref{Cor: sep lem}, and it does not intersect $\beta$. Therefore, $\rho_-((0,1]) \subset P'$. In particular, $\rho_-(1) = \alpha(-1) \in P'$. Since $\alpha([-1,0))$ is disjoint from $\beta$, we conclude that $\alpha([-1,0)) \subset P'$.

On the other hand, since $\rho_+(1) = \mathcal E(\delta) \in \Ext(X)$, it follows from Proposition~\ref{Pro: comp exist geod ep ext quant} that for sufficiently small $s > 0$, the set $B_{s}(\rho_+(1)) \setminus \rho_+$ has exactly two components. Choose such an $s > 0$ sufficiently small so that, in addition,
\[
B_s(\rho_+(1)) \cap \rho_- = \emptyset \quad \text{and} \quad B_s(\rho_+(1)) \cap \nu = \emptyset.
\]
Then 
\[
B_{s}(\rho_+(1)) \setminus \alpha([0,1]) = B_{s}(\rho_+(1)) \setminus \beta =  B_{s}(\rho_+(1)) \setminus \rho_+.
\]
We denote the components of the above set by $C_\pm$. 

Since the point $\rho_+(1) = \alpha(0)$ lies in the topological boundaries of both $Q$ and $Q'$ by Corollary~\ref{Cor: sep lem}, both $Q$ and $Q'$ intersect $B_s(\rho_+(1))$. By Lemma~\ref{Lem: big set small set comp}, up to relabeling $C_\pm$, we have $C_+ \subset Q$ and $C_- \subset Q'$. Similarly, since the point $\rho_+(1) = \beta(0)$ lies in the topological boundaries of both $P$ and $P'$, we see that $C_+$ must be contained in one of $P$ and $P'$, and $C_-$ in the other. Since we already know that $C_- \subset Q' \subset P'$, we obtain $C_+ \subset P$.

By continuity, there exists $\veps > 0$ such that $\alpha((-\veps, 0)) \subset B_{s}(\rho_+(1)) \setminus \rho_+$. Thus $\alpha((-\veps, 0))$ is contained entirely in one of $C_\pm$. We have already shown that $\alpha((-1,0)) \subset Q$, and so it follows from the above inclusion relations that  $\alpha((-\veps, 0)) \subset C_+ \subset P$. This contradicts the previous conclusion that $\alpha([-1,0)) \subset P'$. 
% Therefore, there cannot exist a point $x \in \Ext(X) \cap Q$ such that $x \notin \alpha((-1,0))$, and so $Q \cap \Ext(X) \subset \alpha((-1,0))$. 
\end{proof}

%By Lemma~\ref{Lem: bd poly loop no disconnect}, the sets $Q \setminus \alpha((-1,0))$ and $Q'$ are exactly the two components of $X \setminus \alpha$. It is immediate from the previous claim that $Q \setminus \alpha((-1,0))$ is a boundary polygonal domain satisfying Theorem~\ref{Thm: bd poly dom neigh basis}~(2). This completes the proof. 

\section{Intersections and general position}\label{Sec: general}

As usual, $(X,\dd_X,\mm_X)$ denotes an $\RCD(K,N)$ space of essential dimension $2$. In this section, we study intersections of polygonal curves and polygonal domains.

\begin{Def}[General position]
\label{Def: general position}
Let $\alpha,\beta$ be two simple polygonal curves contained in the interior except possibly at their endpoints. We say that $\alpha$ and $\beta$ are in \emph{general position} if they intersect in at most finitely many points, all contained in the interiors of their edges.
\end{Def}

We next prove a perturbation result, which makes precise the idea that general position is generic and can be achieved by a controlled perturbation. To this end, we work in the following setting.
Let $Q$ be either an interior polygonal domain (see Definition~\ref{Def: poly dom}) or a boundary polygonal domain (see Definition~\ref{Def: bd poly dom}). We denote by $\alpha_Q:[-1,1]\to X$ the associated polygonal loop. In order to treat these two cases in a unified way, we adopt the following conventions. If $Q$ is a boundary polygonal domain, then $\alpha_Q|_{[-1,0]}$ parametrizes the extremal set in $\overline Q$, while $\alpha_Q|_{(0,1)}$ is contained in the interior set (compare with Definition~\ref{Def: bd poly loop}). If $Q$ is an interior polygonal domain, then $\alpha_Q|_{[-1,0]}$ is constant, while $\alpha_Q|_{[0,1]}$ parametrizes the boundary of $Q$, which is contained in the interior. 

Assume that for some $p\in X$ and $0<r<R/2$ we have
\begin{enumerate}
    \item[(a)] $\overline Q\subset B_r(p)$;
    \item[(b)] the inclusions $B_R(p)\hookrightarrow X$ and $B_r(p)\hookrightarrow B_{R/2}(p)$ induce the trivial map on $\pi_1$;
    \item[(c)] there exists a point $q \in X$ such that $\dd_X(p,q) \ge R$.
\end{enumerate}

It follows from Jordan Theorems~\ref{Thm: Jordan theorem}, \ref{Thm: Jordan theorem boundary} and the Separation Lemma~\ref{Lem: sep lem}, together with assumptions (a)--(c), that every simple polygonal loop in $\overline Q$ (possibly a boundary polygonal loop) disconnects $X$ into two connected components, exactly one of which is entirely contained in $B_R(p)$.

\begin{Pro}[General Position]\label{Prop: general_pos}
Let $Q$ satisfy {\rm(a)--(c)}, and let $\beta$ be a polygonal curve contained in $\Int(X)$ except possibly at its endpoints. Then, for every $\eps>0$, there exists a polygonal domain $Q_\eps\subset Q$, which is a boundary polygonal domain if and only if $Q$ is a boundary polygonal domain, such that:
\begin{enumerate}
    \item the polygonal curves $\alpha_{Q_\eps}|_{[0,1]}$ and $\beta$ are in general position;

    \item  $Q_\eps=Q$ outside $B_\eps(\alpha_Q([0,1]))$ and $\overline Q_\veps \subset Q$.
    %$\dd_H(\alpha_{Q_\eps}([0,1]), \alpha_Q([0,1]))\le \eps$.
\end{enumerate}
\end{Pro}

Condition~(1) states that the topological boundary of $Q_\eps$ is in general position with respect to $\beta$. Condition~(2) ensures that $Q_\eps$ is a small perturbation of $Q$, concentrated in the $\eps$-tubular neighborhood of its topological boundary.

\subsection{Cutting polygonal domains}

In this section, we isolate the following lemma. It will play a key role in the perturbation argument leading to Proposition~\ref{Prop: general_pos}, and it will also be used in Section~\ref{Sec: man struct} to prove the manifold structure.

\begin{Lem}\label{lemma:two_polygonal_dom}
Let $Q$ be a polygonal domain satisfying {\rm(a)--(c)}. Let $\beta$ be a simple polygonal curve connecting two distinct points $x,y$ in the support of $\alpha_Q$, and contained in $Q\cap \Int(X)$ except possibly at its endpoints. If $Q$ is a boundary polygonal domain, assume in addition that at least one of $x$ and $y$ belongs to $\alpha_Q([0,1])$. Then $\beta$ disconnects $Q$ into two polygonal domains $Q_1$ and $Q_2$. Moreover, each of $\alpha_{Q_1}$ and $\alpha_{Q_2}$ is, up to reparameterization, the union of $\beta$ with exactly one of the two connected components of $\alpha_Q\setminus \{x,y\}$.
\end{Lem}

Set $\sigma_1:=\beta$ and let $\sigma_2$ and $\sigma_3$ be parametrizations on $[0,1]$ of the two connected components of $\alpha_Q\setminus \{x,y\}$. We then define the corresponding simple loops as follows:
\begin{enumerate}
    \item $\alpha_1$ is the concatenation of $\sigma_2$, from $x$ to $y$, and $\sigma_3$ from $y$ to $x$;
    \item $\alpha_2$ is the concatenation of $\sigma_3$, from $x$ to $y$, and $\sigma_1$ from $y$ to $x$;
    \item $\alpha_3$ is the concatenation of $\sigma_1$, from $x$ to $y$, and $\sigma_2$ from $y$ to $x$. 
\end{enumerate}
Notice that $\alpha_1$ coincides with $\alpha_Q$ up to reparametrization, while each of $\alpha_2$ and $\alpha_3$ is either a simple polygonal loop (Definition~\ref{Def: poly loop}) or, up to reparametrization, a simple boundary polygonal loop (Definition~\ref{Def: bd poly loop}). Here we use the assumption that at least one of the points $x$ and $y$ belongs to the interior set.

By assumptions {\rm(a)--(c)} and Jordan Theorems~\ref{Thm: Jordan theorem} and~\ref{Thm: Jordan theorem boundary}, we deduce that $X\setminus \alpha_i$ has exactly two connected components. Since $\sigma_i((0,1))$ is connected and disjoint from $\alpha_i$, it must be contained in one of them, which we denote by $A_+^i$. We denote the other component by $A_-^i$. From Remarks \ref{Rem: polydom consistency1} and \ref{Rem: polydom consistency2}, it follows that $Q\cap \Int(X) = A_+^1$ and $A_-^1 = X\setminus \overline Q$.

\begin{Cla}\label{Claim:A+}
\begin{equation}
    A_+^1 = A_-^2 \sqcup A_-^3 \sqcup \sigma_1((0,1)).
\end{equation}
\end{Cla}

\begin{proof}[Proof of Lemma \ref{lemma:two_polygonal_dom} given Claim \ref{Claim:A+}]

When $Q$ is an interior polygonal domain, the conclusion follows immediately from the fact that $Q=Q\cap \Int(X)=A_+^1$ together with Claim~\ref{Claim:A+}.

Assume now that $Q$ is a boundary polygonal domain. Then
$Q\setminus \sigma_1$ is the disjoint union of $(Q\cap \Int(X))\setminus \sigma_1 = A_-^2 \sqcup A_-^3$ and $\sigma_Q([-1,0])\setminus \sigma_1$. We distinguish two cases.
First, suppose that at least one of the points $x$ and $y$ belongs to $\sigma_Q([-1,0])$. Then $\sigma_Q([-1,0])\setminus \sigma_1$ has two connected components, and these are precisely the boundary parts of the loops $\alpha_2$ and $\alpha_3$. It follows that each of them combines with the corresponding set $A_-^2$ or $A_-^3$ to form an open connected set. This can be seen by using that the extremal set does not disconnects locally, see Corollary \ref{Lem: ext set no disconnect}. These are exactly the two connected components of $Q\setminus \sigma_1$.

Second, suppose that neither $x$ nor $y$ belongs to $\sigma_Q([-1,0])$. Then the whole boundary part $\sigma_Q([-1,0])$ is contained in one of the two loops, say $\alpha_2$. Hence $\sigma_Q([-1,0])$ joins naturally with $A_-^2$, and their union is again an open connected set. Together with $A_-^3$, this gives the two connected components of $Q\setminus \sigma_1$.
\end{proof}

\begin{proof}[Proof of Claim \ref{Claim:A+}]

Since $\sigma_2((0,1)) \subset A_+^2$, it follows that $\sigma_2((0,1))$ is disjoint from $A_-^2$. Moreover, $\sigma_3$ is also disjoint from $A_-^2$, since it is contained in $\alpha_2$. 
Therefore, $A_-^2$ is disjoint from the support of $\alpha_1$.
Being connected, $A_-^2$ must be contained entirely in one of the components of $X \setminus \alpha_1$, namely $A_+^1$ or $A_-^1$. We show that $A_-^2 \subset A_+^1$.

By Theorem~\ref{Thm: Jordan theorem} (or Theorem~\ref{Thm: Jordan theorem boundary} in the boundary case), the topological boundary of $A_-^2$ coincides with $\alpha_2$, in particular it contains $\sigma_1(1/2)$.  On the other hand, $\sigma_1(1/2) \in A_+^1$. Since $A_+^1$ is open (being a component of  $X\setminus \alpha_1$), it contains a neighborhood of $\sigma_1(1/2)$. 
In particular, $A_-^2$ intersects $A_+^1$ and hence $A_-^2 \subset A_+^1$.

Arguing in the same way for $A_-^3$, we obtain $A_-^3 \subset A_+^1$. Hence,
\begin{equation}\label{Eq: component relation 1}
  A_-^2 \cup A_-^3 \cup \sigma_1((0,1)) \subset A_+^1.
\end{equation}
By symmetry, we also have
\begin{equation}\label{Eq: component relation 2}
  A_-^1 \cup A_-^3 \cup \sigma_2((0,1)) \subset A_+^2 \quad \text{and} \quad A_-^1 \cup A_-^2 \cup \sigma_3((0,1)) \subset A_+^3.
\end{equation}

Next, we show that the unions in \eqref{Eq: component relation 1} are disjoint unions. Clearly, $\sigma_1((0,1))$ is disjoint from $A_-^2 \cup A^3_-$, since it belongs to the support of $\alpha_2$ and $\alpha_3$.
Moreover, we have $A_-^3 \subset A_+^2$ from \eqref{Eq: component relation 2}, which implies that $A_-^3 \cap A_-^2 = \emptyset$.

% {\color{red}Together, these show that the unions in \eqref{Eq: component relation 1} are disjoint unions. By symmetry, the unions in \eqref{Eq: component relation 2} are also disjoint unions. 
% }

It remains to prove the reverse inclusion. Let $x\in A_+^1\setminus \sigma_1$. We claim that $x\in A_-^2 \sqcup A_-^3$.
Since $A_+^1$ is path-connected and $\sigma_1((0,1)) \subset A_+^1$, there exists a continuous curve $\nu:[0,1] \to A_+^1$ joining $x$ to $\sigma_1(1/2)$. Let $t\in (0,1]$ be the smallest time such that $\nu(t) \in \sigma_1((0,1))$, which is well defined since $\sigma_1((0,1))$ is closed in $A_+^1$.

Choose $r>0$ small enough so that $B_r(\nu(t)) \subset A_+^1$ and $B_r(\nu(t))\setminus \sigma_1$ has exactly two connected components, denoted by $P_+$ and $P_-$. This follows from either Theorem~\ref{Thm: good components existence 2} or Proposition~\ref{Pro: good comp exist corner int}, according to whether $\nu(t)$ lies on an edge or at a vertex of the polygonal curve $\sigma_1$. Since $\sigma_1$ is contained in both $\alpha_2$ and $\alpha_3$, the two local components $P_+$ and $P_-$ lie in different connected components of $X\setminus \alpha_2$ and of $X\setminus \alpha_3$. Without loss of generality, assume that $P_+\subset A_+^2$ and $P_-\subset A_-^2$. By \eqref{Eq: component relation 2}, we have $P_-\subset A_-^2 \subset A_+^3$, and therefore $P_+\subset A_-^3$.
It follows that, for $\varepsilon>0$ sufficiently small, the point $\nu(t-\varepsilon)$ belongs either to $A_-^2$ or to $A_-^3$. Since $\nu([0,t))\subset A_+^1$ is connected and disjoint from both $\alpha_2$ and $\alpha_3$, the whole curve $\nu([0,t))$ must remain in the same connected component of $X\setminus \alpha_2$ and of $X\setminus \alpha_3$ as $\nu(t-\varepsilon)$. In particular, $x=\nu(0)$ belongs either to $A_-^2$ or to $A_-^3$, as claimed.
\end{proof}

The following corollary of Lemma~\ref{lemma:two_polygonal_dom} will be used in the proof of Theorem~\ref{thm:subordinated_grids}.

\begin{Cor}\label{cor:intersection}
Let $Q$ be a polygonal domain satisfying {\rm(a)--(c)}. Let $a,b,c,d\in \alpha_Q$ be distinct points, viewed as vertices of a quadrilateral and listed in clockwise order, with the convention that $\alpha_Q(-1)=d$ and $\alpha_Q(0)=a$ when $Q$ is a boundary polygonal domain. Let $\sigma_{ac}$ and $\sigma_{bd}$ be simple polygonal curves contained in $Q$ except at their endpoints, which are $a,c$ and $b,d$, respectively. Then $\sigma_{ac}$ and $\sigma_{bd}$ intersect.
\end{Cor}

\begin{proof}
By Lemma~\ref{lemma:two_polygonal_dom}, the set $Q\setminus \sigma_{ac}$ has exactly two components. Moreover, the closure of one component contains the boundary arcs $ab$ and $bc$, while the closure of the other contains the boundary arcs $cd$ and $da$.

If, by contradiction, $\alpha_{ac}$ and $\alpha_{bd}$ were disjoint, then $\alpha_{bd} \setminus \{b,d\}$ would be contained in a single connected component of $Q\setminus \alpha_{ac}$. In particular, the points $b$ and $d$ would both be contained in the closure that component, yielding a contradiction. 
\end{proof}

\subsection{Proof of Proposition \ref{Prop: general_pos}}
Fix a parameter $\eps>0$ sufficiently small so that $B_{\eps}(x) \subset B_r(p)$ for all $x \in \alpha_Q$. It follows from assumption~{\rm (b)} that $B_{\eps}(x) \hookrightarrow X$ induces the trivial map on $\pi_1$. By Theorem~\ref{Thm: semi-locally simply connected}, exists $0<r_\eps\le \veps/2$ such that, for every $x\in \alpha_Q$, the inclusion $B_{r_\eps}(x)\hookrightarrow B_{\eps/2}(x)$ induces the trivial map on $\pi_1$.

Assume first that $Q$ is an interior polygonal domain. By Theorem~\ref{Thm: good components existence 2} and Proposition~\ref{Pro: comp exist corner int quant 2}, we may choose points $x_i \in \alpha_Q$ and $0 < r_i \leq r_\eps$, $i = 1, \dots , k$, such that the following hold:
\begin{enumerate}
    \item[(i)] $B_{r_i}(x_i)$ covers $\alpha_Q$;
    \item[(ii)] $B_{r_i}(x_i)$ intersects only its predecessor $B_{r_{i-1}}(x_{i-1})$ and its successor $B_{r_{i+1}}(x_{i+1})$;
    \item[(iii)] $\alpha_Q \cap \overline B_{r_i}(x_i)$ is either a geodesic with length $2r_i$ or a corner with vertex $x_i$ and two edges of length $r_i$;
    \item[(iv)] $B_{r_i}(x_i) \setminus \alpha_Q$ has exactly components, and their topological boundaries in $B_{r_i}(x_i)$ coincide with $\alpha_Q \cap B_{r_i}(x_i)$;
\end{enumerate}
where $k + 1 = 1$ and $0 = k$ in (ii). 

Since each $x_i$ is contained in the topological boundary of $Q$, it follows from Lemma~\ref{Lem: big set small set comp} that exactly one of the components of $B_{r_i}(x_i) \setminus \alpha_Q$, denoted by $A_i$, is contained in $Q$, while the other is disjoint from $Q$. By (ii), we see that $\alpha_Q \cap B_{r_i}(x_i) \cap B_{r_{i+1}}(x_i)$ is non-empty for each $i$. Combining with (iv), it is not difficult to see that $A_i$ intersects only $A_{i-1}$ and $A_{i+1}$. 

For each $i = 1, \dots, k$, choose a point $y_i \in \alpha_Q \cap B_{r_{i-1}}(x_i) \cap B_{r_{i}}(x_i)$. Choose $y_i' \in A_{i-1} \cap A_{i}$ close to $y_i$. It follows from the non-branching property (Theorem~\ref{Thm: non-branching}) that if $y_i'$ is sufficiently close to $y_i$, then any geodesic from $y_i$ to $y_i'$ does not intersect $\alpha_Q$ except at $y_i$, and therefore lies in $A_{i-1} \cap A_{i}$. Since $y_i, y_i' \in \overline Q \subset \Int(X)$, it follows from Corollaries~\ref{Cor: int geod no ext} and \ref{Cor: int geod reg} that any such geodesic is regular and contained in $\Int(X)$. 

Applying Lemma~\ref{Lem: poly curve connect} below, we may join each $y_i'$ and $y_{i_1}'$ by a polygonal curve $\rho_i$ contained in $A_i$, which is in general position with respect to $\beta$. Let $\alpha_{Q_\veps}$ be a simple polygonal loop obtained by concatenating all the curves $\rho_i$ and excising any overlap. Then $\alpha_{Q_\veps}$ is in general position with respect to $\beta$ and thus satisfies claim~{\rm (1)} of the proposition. By excising carefully, it is not difficult to see that there exists finitely many polygonal loops $\alpha_j \subset B_{r}(p)$, $j = 0, \dots, n$, such that the following hold:
\begin{enumerate}
    \item[(i)] $\alpha_0 = \alpha_Q$ and $\alpha_n= \alpha_{Q_\veps}$.
    \item[(ii)] For each $j = 1, \dots, n$, the loop $\alpha_j$ is obtained from $\alpha_{j-1}$ by removing a polygonal curve, and adding in another polygonal curve with the same endpoints and contained in the interior polygonal domain bounded by $\alpha_n$. Moreover, the union of these two curves is a simple polygonal loop $\sigma_j$ contained in $B_{r_i}(x_i)$ for some $i$.  
\end{enumerate} 
Since $B_{r_i}(x_i) \subset B_{r_\veps}(x_i)$, and $B_{r_\eps}(x_i)\hookrightarrow B_{\eps/2}(x_i)$ and $B_{\eps}(x_i) \hookrightarrow X$ induce trivial maps on $\pi_1$, Jordan Theorem \ref{Thm: Jordan theorem} implies $\sigma_j$ bounds an interior polygonal domain contained in $B_\eps(x_i)$ and not containing $q$ (as in assumption~{\rm (c)}). It follows from induction and Lemma~\ref{lemma:two_polygonal_dom} (see also Claim \ref{Claim:A+}) that the interior polygonal domain (not containing $q$) bounded by $\alpha_j$ is contained in $Q$, and equal to $Q$ outside of $B_\veps(\alpha_Q([0,1]))$. In particular, this holds for the domain $Q_\veps$ bounded by $\alpha_{Q_\veps} = \alpha_n$. Since $\alpha_{Q_\veps} \subset Q$ by construction, we obtain $\overline Q_\veps = Q_\veps \sqcup \alpha_{Q_\veps} \subset Q$. Therefore, claim~{\rm(2)} of the proposition holds for $Q_\veps$.

Next, assume that $Q$ is a boundary polygonal domain. By applying in addition Proposition~\ref{Pro: comp exist geod ep ext quant}, we may choose points $x_i \in \alpha_Q([0,1])$ and $0 < r_i \leq r_\eps$, $i = 1, \dots, k$, with $x_1 = \alpha_Q(0)$ and $x_k = \alpha_Q(1)$, such that the following hold:
\begin{enumerate}
    \item[(i)] $B_{r_i}(x_i)$ covers $\alpha_Q([0,1])$;
    \item[(ii)] $B_{r_i}(x_i)$ intersects only its predecessor $B_{r_{i-1}}(x_{i-1})$ and its successor $B_{r_{i+1}}(x_{i+1})$;
    \item[(iii)] $\alpha_Q([0,1]) \cap \overline B_{r_i}(x_i)$ is either a geodesic with length $2r_i$ or a corner with vertex $x_i$ and two edges of length $r_i$ for $i \neq 1, k$, and is a geodesic of length $r_i$ with an endpoint $x_i$ for $i = 1, k$;
    \item[(iv)] $B_{r_i}(x_i) \setminus \alpha_Q([0,1])$ has exactly components, and their topological boundaries in $B_{r_i}(x_i)$ coincide with $\alpha_Q([0,1]) \cap B_{r_i}(x_i)$;
    \item[(v)] $B_{r_i}(x_i) \subset \Int(X)$ for $i \neq 1, k$. 
 \end{enumerate}

We choose $y_i$ and $y_i'$ as before for $i = 2, \dots, k$. We also choose $y_1 = x_0$ and $y_{k+1} = x_k$, and $y_1'$ and $y_{k+1}'$ to be points on $\alpha_Q((-1,0))$ close to $y_1$ and $y_{k+1}$. Arguing as in the previous case, making the obvious adjustments and replacing Jordan Theorem~\ref{Thm: Jordan theorem} by the Separation Lemma~\ref{Lem: sep lem} where needed, we arrive at the desired conclusion.

\begin{Lem}\label{Lem: poly curve connect}
Let $U \subset X$ be open and path-connected, and let $x, y \in U$. Then $x$ and $y$ can be joined by a simple polygonal curve $\rho \subset U$ contained in the interior set except possibly at its endpoints. Moreover, if $\beta$ is a simple polygonal curve and $x$ and $y$ do not lie on vertices of $\beta$, then $\rho$ can be chosen to be in general position with respect to $\beta$. 
\end{Lem}
\begin{proof}
We only prove the more general claim (i.e. with $\beta$); the same argument works without $\beta$. We will use Corollaries~\ref{Cor: int geod no ext} and \ref{Cor: int geod reg} repeatedly in the proof without further mention. 

Let $C \subset X$ be the set of vertices of $\beta$. Since $\Int(X)$ is dense (by Theorem~\ref{Thm: int set open}), there exists $x \in \Int(X) \cap U \setminus C$. Let $V$ be the set of all point $y \in U \setminus C$ such that $x$ can be connected to $y$ by a polygonal curve as in the lemma, except we do not require it to be simple, since such a curve can always be made simple by excision. It suffices to show that $V = U \setminus C$. Since $C$ is a finite and $X$ is locally path-connected, $U \setminus C$ is open and path-connected. Therefore, it suffices to show that $V$ is clopen in $U \setminus C$. 

We first show that $V$ is open. Let $y \in V$ and let $\nu$ be a polygonal curve from $x$ to $y$ as in the definition of $V$. We show that $B_r(y) \subset V$ for some $r > 0 $. There are three cases:

\noindent \textbf{Case 1:} $y \in \Int(X) \setminus \beta$. There exists $r > 0$ such that $B_r(y) \subset U \setminus \beta$. Since $y \in \Int(X)$, any geodesic from $y$ to any $z \in B_r(y)$ is regular and lies in $\Int(X)$, except possibly at $z$. The union of $\nu$ and such a geodesic gives a polygonal curve from $x$ to $z$ as in the definition of $V$.

\noindent \textbf{Case 2:} $y \in \Ext(X) \setminus \beta$. Let $r$ and $z$ be as in case 1. Let $y'$ be a point close to $y$ that lies in the interior of an edge of $\nu$. Then $y' \in \Int(X)$. Arguing as in case~1, the union of the restriction $\nu$ from $x$ to $y'$ and a geodesic from $y'$ to $z$ gives a polygonal curve from $x$ to $z$ as in the definition of $V$. 

\noindent \textbf{Case 3:} $y \in \beta$. Since $y \notin C$, there exists $r > 0$ such that $B_{2r}(y) \subset U \setminus C$ and $\beta \cap B_{2r}(y)$ is a geodesic of length $4r$. Moreover, $y \in \Int(X)$, as it lies in the interior of an edge of $\beta$. Since $\Int(X)$ is open (by Theorem~\ref{Thm: int set open}), we may choose $r$ so that $B_r(y) \subset \Int(X)$. Let $z \in B_r(y)$. 
    
In the subcase where $z \in B_r(y) \setminus \beta$: any geodesic from $y$ to $z$ is regular and lies in $B_r(y) \cap \Int(X)$. Moreover, it does not intersect $\beta$, except at $y$, by the non-branching property. 
    
In the subcase where $z \in \beta \cap B_r(y)$: Since $\dim(X) = 2$, it follows from \cite[Theorem 1.1]{KL16} that $X$ is not locally isometric to a geodesic at any point. In particular, $B_r(y) \setminus \beta$ is non-empty. By the density of $\Int(X)$, we may choose $z' \in B_r(y) \cap \Int(X)$. Then geodesics from $y$ to $z'$ and from $z'$ to $z$ are regular, contained in $B_{2r}(y) \cap \Int(X)$, and do not intersect $\beta$, except at $y$ and $z$ respectively, by the non-branching property. 

In both subcases, the union of $\nu$ and the obtained geodesics gives a polygonal curve from $x$ to $z$ as in the definition of $V$. This finishes the proof that $V$ is open. 

We now show that $V$ is closed in $U \setminus C$. Let $y \in U \setminus C$ and suppose that $y \in \overline V$. Let $r > 0$ be such that $B_r(y) \subset U \setminus C$, and so that $\beta \cap B_r(y)$ is either empty or a geodesic of length $2r$ (depending on whether $y \in \beta$). Since $\Int(X)$ is dense, we may choose $z \in B_r(y) \cap V$ such that $z \in \Int(X)$. Moreover, since $X$ is not locally isometric to a geodesic, we may choose $z \notin \beta$. By the non-branching property, any geodesic from $z$ to $y$ does not intersect $\beta$, except possibly at $y$. Since $z \in \Int(X)$, any such geodesic is regular, and contained in $B_r(y) \cap \Int(X)$, except possibly at $y$. The union of the polygonal curve from $x$ to $z$ as in the definition of $V$ and such a geodesic gives a polygonal curve from $x$ to $y$ as in the definition of $V$. In particular, $y \in V$. Therefore, $V$ is closed in $U \setminus C$. 
\end{proof}

\section{Manifold Structure}\label{Sec: man struct}

Let $(X,\dd_X,\mm_X)$ be an $\RCD(K,N)$ space of essential dimension $2$, and let $Q$ be either an interior polygonal domain (see Definition~\ref{Def: poly dom}) or a boundary polygonal domain (see Definition~\ref{Def: bd poly dom}). We denote by $\alpha_Q:[-1,1]\to X$ the associated polygonal loop. In order to treat these two cases in a unified way, we adopt the same convention as in Section \ref{Sec: general}. If $Q$ is a boundary polygonal domain, then $\alpha_Q|_{[-1,0]}$ parametrizes the extremal set in $\overline Q$, while $\alpha_Q|_{(0,1)}$ is contained in the interior set (compare with Definition~\ref{Def: bd poly loop}). If $Q$ is an interior polygonal domain, then $\alpha_Q|_{[-1,0]}$ is constant, while $\alpha_Q|_{[0,1]}$ parametrizes the topological boundary of $Q$, which is contained in the interior.

\begin{Thm}\label{thm:top1}
Assume that $Q$ satisfies {\rm(a)--(c)}. Then $\overline{Q}$ is homeomorphic to the $2$-disk $D^2$, via a homeomorphism that maps the loop $\alpha_Q$ onto the boundary $\partial D^2$.
\end{Thm}

The construction follows closely the classical two-sphere recognition theorem \cite{M16} (see also \cite{T10}) and proceeds as follows. We fix four distinguished points $a,b,c,d$ on $\alpha_Q$, which we regard as the vertices of a quadrilateral, listed in clockwise order. When $Q$ is a boundary polygonal domain, following the notation of Definition~\ref{Def: bd poly dom}, we always choose $\alpha(-1)=b$ and $\alpha(0)=a$. These four points decompose the simple polygonal loop $\alpha_Q$ into four edges. We refer to $ab$ and $cd$ as the \emph{horizontal} edges, and to $bc$ and $da$ as the \emph{vertical} edges.

We say that a polygonal curve in $\overline Q$ is \emph{horizontal} if it is simple, it is contained in $Q$ except for its endpoints, and its endpoints lie one on $ad$ and the other on $bc$. Similarly, we say that a polygonal curve is \emph{vertical} if it is simple, it is contained in $Q$ except for its endpoints, and its endpoints lie one on $ab$ and the other on $cd$.

\begin{Def}[Polygonal grid]
An $m\times n$ polygonal grid $\mathcal{G}$ on $Q$ is a collection of $m$ horizontal curves and $n$ vertical curves such that any two distinct horizontal curves are disjoint, any two distinct vertical curves are disjoint, and each horizontal curve intersects each vertical curve in exactly one point.
\end{Def}

As a model space, we consider the square $[0,1]^2$, whose vertices are labeled clockwise, starting from $a=(1,0)$. We denote by $\mathcal{G}_{m\times n}$ the $m\times n$ {\it regular grid} in $[0,1]^2$, consisting of equally spaced horizontal and vertical segments, all parallel to the coordinate axes.

\begin{Pro}[Cells of a grid]\label{prop:cells}
Assume that $Q$ satisfies {\rm(a)--(c)}, and let $\mathcal{G}$ be an $m\times n$ grid in $Q$. The following hold:
\begin{enumerate}
    \item $\mathcal{G}$ is {\it combinatorially equivalent} to $\mathcal{G}_{m\times n}$, in the sense that there exists a bijection between the vertices, edges, horizontal curves, and vertical curves of the two grids which preserves incidences and the order of intersection points along each horizontal and each vertical curve.

    \item The complement in $Q$ of the horizontal and vertical curves of $\mathcal{G}$ is the disjoint union of $(m+1)(n+1)$ polygonal domains (possibly boundary polygonal domains). Moreover, these domains are in bijection with the cells of $\mathcal{G}_{m\times n}$, in such a way that corresponding domains are bounded by corresponding horizontal and vertical edges.
\end{enumerate}
\end{Pro}

We call the closures of the polygonal domains appearing in item~{\rm(2)} of Proposition~\ref{prop:cells} the \emph{cells} of the grid.

The goal is to construct arbitrarily fine grids on $\overline Q$ and use them to define a homeomorphism onto $[0,1]^2$ by mapping each such grid to the corresponding combinatorially equivalent regular grid in $[0,1]^2$. The following extension lemma will be the technical tool for inductively refining grids.

\begin{Lem}[Extension property]\label{thm:extension}
Assume that $Q$ satisfies (a)--(c). Then, for every $a\neq b$ in the support of $\alpha_Q$, there exists a simple polygonal curve that is contained in $Q$ except possibly for its endpoints, which are $a$ and $b$. 
\end{Lem}

Let $\mathcal{U}$ be an open cover of $\overline Q$. We say that a polygonal grid is \emph{subordinated} to $\mathcal{U}$ if each cell of the grid is contained in some $U\in\mathcal{U}$. The next result is the key ingredient for constructing, inductively, finer and finer grids.

\begin{Thm}[Subordinated grid]\label{thm:subordinated_grids}
Assume that $Q$ satisfies (a)--(c). Let $\mathcal{U}$ be an open cover of $\overline Q$. Then there exists a polygonal grid in $Q$ subordinated to $\mathcal{U}$.
\end{Thm}

In the rest of this section, we prove the three key ingredients: Proposition~\ref{prop:cells}, the extension lemma~\ref{thm:extension}, and the grid construction theorem~\ref{thm:subordinated_grids}. In Section~\ref{sec:grid_con}, we then explain how to combine these ingredients to prove Theorem~\ref{thm:top1}. Finally, in Section~\ref{sec:proofmain}, we prove our main manifold structure theorem, Theorem~\ref{thm:mainRCD}.

\subsection{Proof of Proposition~\ref{prop:cells}}

We argue by induction on $n+m$, using Lemma~\ref{lemma:two_polygonal_dom} as the key step. When $n+m=0$, there is nothing to prove, so we may assume that there exists at least one edge $\beta$. We describe the construction only in the case where $\beta$ is vertical, since the horizontal case is completely analogous.

Denote by $x$ the intersection of $\beta$ with $ab$, and by $y$ its intersection with $cd$. By our convention, $y$ always belongs to the interior set. Hence Lemma~\ref{lemma:two_polygonal_dom} applies and shows that $Q\setminus \beta$ disconnects into two polygonal domains (possibly boundary polygonal domains), namely $Q_1$, with vertices $x,b,c,y$, and $Q_2$, with vertices $a,x,y,d$.

Every vertical curve of $Q$ distinct from $\beta$ is connected and contained in $Q\setminus \beta$ except at its endpoints. Hence it is contained in either $Q_1$ or $Q_2$, and remains a vertical curve with respect to the corresponding choice of vertices. We denote by $n_1$ the number of vertical curves in $Q_1$, and by $n_2$ the number of those in $Q_2$. Notice that $n_1+n_2=n-1$.

Since every horizontal curve in $Q$ meets every vertical curve in exactly one point, its intersections with $Q_1$ and $Q_2$ are again horizontal curves with respect to the chosen vertices. We are therefore in a position to apply the induction hypothesis to the induced grids in $Q_1$ and $Q_2$. This yields combinatorial equivalences with the regular grids $\mathcal{G}_{m\times n_1}$ and $\mathcal{G}_{m\times n_2}$. Moreover, $Q_1$ contains exactly $(m+1)(n_1+1)$ cells and $Q_2$ contains exactly $(m+1)(n_2+1)$ cells. It follows that $Q$ is divided by $\mathcal{G}$ into exactly $(m+1)(n+1)$ cells.

It remains to construct the combinatorial equivalence between $\mathcal{G}$ and the regular grid $\mathcal{G}_{m\times n}$. To this end, we view $\mathcal{G}_{m\times n}$ as the union of two subgrids, combinatorially equivalent to $\mathcal{G}_{m\times n_1}$ and $\mathcal{G}_{m\times n_2}$, separated by one additional vertical segment corresponding to $\beta$. The induction hypothesis provides bijections on $Q_1$ and $Q_2$, and these bijections are compatible along the common boundary determined by $\beta$, since they preserve the order of the intersection points of $\beta$ with the horizontal curves. We may therefore glue them together and obtain the desired combinatorial equivalence between $\mathcal{G}$ and $\mathcal{G}_{m\times n}$.

\subsection{Proof of Lemma~\ref{thm:extension}}

For sufficiently small $r > 0$, the set $\alpha_Q \cap \overline B_r(a)$ is one of the following configurations. The last two occur only when $Q$ is a boundary polygonal domain and $a\in \alpha_Q([-1,0])$ is an extremal point:
\begin{enumerate}
    \item a geodesic of length $2r$;

    \item a corner with vertex $a$ and two edges of length $r$;

    \item a subset of the extremal set $\Ext(X)$. More precisely, $\alpha_Q$ is a boundary loop and $\alpha_Q\cap \overline B_r(a) \subset \Ext(X)$;

    % a segment contained in the extremal set $\Ext(X)$. More precisely, $\alpha_Q$ is a boundary loop and $\alpha_Q\cap \overline B_r(a)= \alpha_Q(I)$ for some closed interval $I\subset (-1,0)$;

    \item a corner meeting the extremal set. More precisely, $\alpha_Q$ is a boundary loop, $a=\alpha_Q(0)$, $\alpha_Q([0,1])\cap \overline B_r(a)$ is a geodesic at $a$.

    % More precisely, $\alpha_Q$ is a boundary loop, $a=\alpha_Q(0)$, $\alpha_Q\cap \overline B_r(a)= \alpha_Q([-\eta,\eta])$ for some $\eta>0$, and $\alpha|_{[0,\eta]}$ is a geodesic.
    
\end{enumerate}

In all cases, there exists a point $x \in B_{r/4}(a) \cap Q$.
Let $\sigma_1$ be a geodesic from $a$ to $x$. We claim that $\sigma_1$ is contained in $Q$ except at the endpoint $a$. It is enough to show that $\sigma_1$ does not intersect $B_r(a)\cap \alpha_Q$ except at the endpoint $a$.
Since $x\in Q$ is an interior point, Corollary~\ref{Cor: int geod reg} implies that $\sigma_1$ is a regular geodesic, and therefore it cannot intersect the extremal set except at the endpoint $a$. This already settles configuration~(3). In the remaining cases, we use the non-branching property. Suppose that $\sigma_1$ intersects $\alpha_Q\cap B_r(a)$ at some point other than $a$. Then the intersection must occur along one of the geodesics of $\alpha_Q\cap B_r(a)$ issuing from $a$. Replacing the final portion of $\sigma_1$ after the first intersection point with the corresponding boundary geodesic joining that point to $a$, we obtain a minimizing geodesic from $a$ to $x$ which, in a neighborhood of $a$, coincides with a segment of $\alpha_Q$. This yields branching of geodesics, contradicting Theorem~\ref{Thm: non-branching}.

Similarly, we can choose a point $y \in Q$ and a geodesic $\sigma_2$ from $y$ to $b$ that is contained in $Q$ except at the endpoint $b$. Since $x,y\in Q$ and $Q$ is path-connected, there exists a continuous curve $\nu$ joining $x$ to $y$ and contained in $Q$. Let $U\subset Q$ be a connected tubular neighborhood of $\nu$. By Lemma~\ref{Lem: poly curve connect}, there exists a simple polygonal curve $\rho$ in $U$ connecting $x$ to $y$. Concatenating $\sigma_1$, $\rho$, and $\sigma_2$, we obtain a polygonal curve from $a$ to $b$ contained in $Q$ except at its endpoints.

\subsection{Proof of Theorem~\ref{thm:subordinated_grids}}\label{sec:proof_grid}

Let $\mathcal U$ be an open cover of $\overline Q$. We first refine it. For every interior point $x\in \overline Q$, we apply Theorem~\ref{Thm: poly dom neigh basis} to find a polygonal domain $Q_x$ such that $x \in Q_x \subset U$ for some $U \in \mathcal U$. For every extremal point $x\in \overline Q$ (equivalently, when $Q$ is a boundary polygonal domain and $x\in \alpha_Q([-1,0])$), we apply Theorem~\ref{Thm: bd poly dom neigh basis} to find a boundary polygonal domain $Q_x$ such that $x \in Q_x \subset U$ for some $U \in \mathcal U$. By compactness of $\overline Q$, we may extract a finite subcover. This reduces the proof to the following statement.

\begin{Pro}\label{prop:intermediate1}
For every polygonal domain $Q$ satisfying {\rm(a)--(c)}, and every finite open cover $\mathcal U$ of $\overline Q$ consisting of polygonal domains (possibly boundary polygonal domains), there exists a polygonal grid in $Q$ subordinated to $\mathcal U$.
\end{Pro}

To prove Proposition~\ref{prop:intermediate1}, we argue by induction on the cardinality $n$ of the cover. If $n=1$, there is nothing to prove. Let $n\ge 2$, and assume that the statement holds for every polygonal domain $Q$ satisfying {\rm(a)--(c)} and every finite open cover of $\overline Q$ consisting of $n-1$ polygonal domains.

Fix a polygonal domain $Q$ and a cover $\mathcal U$ of $\overline Q$ consisting of $n$ polygonal domains. Let $P\in \mathcal U$ be an element containing $a$. If $\mathcal U\setminus\{P\}$ still covers $\overline Q$, then we may discard $P$ and conclude by the inductive hypothesis. Hence, we may assume that $\overline Q$ is not contained in
\begin{equation}
    V:=\bigcup_{U\in\mathcal U\setminus\{P\}} U.
\end{equation}

\begin{Cla}\label{claim:cell1}
 There exists a polygonal grid $\mathcal{G}$ subordinated to the two-set cover $\{P, V\}$. 
\end{Cla}

\begin{proof}[Proof of Proposition~\ref{prop:intermediate1} given Claim \ref{claim:cell1}]

Each cell of $\mathcal{G}$ is the closure of a polygonal domain, possibly a boundary polygonal domain, contained either in $P$ or in $V$. By Proposition~\ref{prop:cells}, each such cell comes with distinguished vertices, given by the intersections of its horizontal and vertical edges. 
For each cell of $\mathcal{G}$ contained in $V$, we can apply the induction hypothesis, since $V$ is covered by the $n-1$ elements of $\mathcal{U}\setminus\{P\}$. Thus, on each such cell we construct a polygonal grid subordinated to the restricted cover $\mathcal{U}\setminus\{P\}$. In this way, we obtain a subdivision of $Q$ subordinated to $\mathcal{U}$, although it is not yet a grid. Indeed, a horizontal (or vertical) curve in one cell may terminate on a common boundary without continuing as a horizontal (or vertical) curve in the adjacent cell. 

We remedy this by applying the extension property, Lemma~\ref{thm:extension}, to prolong such boundary arcs across neighboring cells and connect the corresponding endpoints by simple polygonal curves inside the union of the relevant cells. 
More precisely, let $\mathcal{G}_{m\times n}$ be the regular grid combinatorially equivalent to $\mathcal{G}$ (see Proposition~\ref{prop:cells}). For each cell $\overline C$ of $\mathcal{G}$, where $C\subset Q$ is a polygonal domain, let $\mathcal{G}_C$ be a grid in $C$ subordinated to $\mathcal{U}\setminus \{P\}$, obtained by the induction hypothesis. Let $\overline C'\subset [0,1]^2$ be the cell of $\mathcal{G}_{m\times n}$ corresponding to $\overline C$.

Inside $C'$, we now draw a grid with the same combinatorial pattern as $\mathcal{G}_C$. More precisely, we place the endpoints of its horizontal and vertical curves on the sides of $C'$ so that their order agrees with the order of the corresponding endpoints on the sides of $C$, taking into account both the curves coming from $\mathcal{G}_C$ and the edges already belonging to $\mathcal{G}_{m\times n}$. We then connect these endpoints inside $C'$ so as to reproduce the same incidence relations and intersection pattern as in $\mathcal{G}_C$.

This procedure yields a subdivision of $[0,1]^2$, combinatorially equivalent to the one in $Q$, which can be completed to a grid by adding finitely many further segments. For each segment added in this way, we construct a corresponding simple polygonal curve in the associated cell of $Q$ by applying the extension lemma~\ref{thm:extension}.
\end{proof}

\begin{proof}[Proof of Claim \ref{claim:cell1}]

To construct a grid $\mathcal{G}$ subordinated to the two-set cover $\{P, V\}$, we will proceed in two steps. 

\medskip
\noindent{\bf Step 1.} There exists a polygonal domain $R$, which is a boundary polygonal domain if and only if $P$ is a boundary polygonal domain, such that 
\begin{enumerate}
\item $\overline R\subset P$;

\item $a\in R$ and $\{R, V\}$ is still a covering of $\overline Q$;

    \item $\alpha_R|_{[0,1]}$ and $\alpha_Q|_{[0,1]}$ are in general position (see Definition \ref{Def: general position}) and the intersection set does not contain the vertices $a,b,c,d\in \alpha_Q$.
\end{enumerate}

\medskip

Fix $\eps>0$ small enough so that $V$ contains the $10\eps$-neighborhood of $\partial P$. We first shrink the polygonal domain $P$ to a new one whose closure is contained in $P$ and which, together with $V$, still covers $\overline Q$.
To this end, we apply Proposition~\ref{Prop: general_pos} to the polygonal domain $P$, with $\beta=\alpha_P([0,1])$, and with parameter $\eps>0$. This yields a polygonal domain $P_\eps\subset P$ such that: (1) $\alpha_{P_\eps}([0,1])$ and $\alpha_P([0,1])$ are in general position; (2) $P_\eps=P$ outside $B_\eps(\partial P)$.

Since $P_\eps\subset P$, condition~(1) implies that $\alpha_{P_\eps}([0,1])$ and $\alpha_P([0,1])$ are in fact disjoint, hence $\overline P_\eps\subset P$.
Indeed, being in general position with $\alpha_P$, the curve $\alpha_{P_\eps}$ switches side with respect to $P$ at each intersection point. This follows from the local structure of the intersection: if $\gamma_1$ and $\gamma_2$ are regular geodesics intersecting at a point $x$ contained in the interior of both, then, for $r>0$ sufficiently small, the set $B_r(x)\setminus \gamma_1$ has exactly two connected components, and $\gamma_2$ passes from one component to the other as it crosses $x$; see Theorem~\ref{Thm: good components existence 2} and the proof of Lemma~\ref{Lem: exp image inclusion}.

By condition~(2), $\{P_\eps,V\}$ still covers $\overline Q$.

We now perturb $\alpha_{P_\eps}|_{[0,1]}$ so as to put it in general position with respect to $\alpha_Q|_{[0,1]}$, and at the same time ensure that the intersection does not contain  $a,b,c,d$. To this end, we apply Proposition~\ref{Prop: general_pos} to the polygonal domain $P_\eps$ with $\beta=\alpha_Q([0,1])$, regarding $a,b,c,d$ as vertices of the polygonal curve $\alpha_Q|_{[0,1]}$.

\bigskip
\noindent{\bf Step 2.} There exists a grid $\mathcal{G}$ such that $\alpha_R|_{[0,1]}\cap \overline Q$ is contained in the union of all horizontal and vertical polygonal curves of the grid. In particular, $\mathcal{G}$ is subordinated to $\{P,V\}$.

\medskip

Since $\alpha_R|_{[0,1]}$ and $\alpha_Q|_{[0,1]}$ are in general position by Step~1, the curve $\alpha_R$ switches side with respect to $Q$ at each intersection point. Since $\alpha_R|_{[0,1]}$ crosses $Q$ only finitely many times, the set $\alpha_P([0,1])\cap \overline Q$ is the union of finitely many simple polygonal curves $\alpha_1,\ldots,\alpha_m\subset \overline Q$, whose supports are pairwise disjoint. Moreover, the endpoints of each $\alpha_\ell$ are distinct from the vertices $a,b,c,d$.

By Lemma~\ref{lemma:two_polygonal_dom}, these curves subdivide $Q$ into $m+1$ polygonal domains, at most one of which is a boundary polygonal domain, and each of which is contained either in $P$ or in $V$. Thus, to conclude the proof, it remains to refine this subdivision into an actual grid. More precisely, we need to construct a grid $\mathcal{G}$ whose collection of horizontal and vertical curves contains $\alpha_1,\ldots,\alpha_m$. To this end, we follow closely the argument of \cite[Proposition 4.1]{T10}.

Consider the standard square $[0,1]^2$ with the usual labeling of its vertices. Choose $2m$ points in $\partial[0,1]^2$ corresponding to the endpoints of $\alpha_1,\ldots,\alpha_m$ on $\partial Q$, preserving both their order along the boundary and the side of the square on which they lie. These points are distinct from the vertices of $[0,1]^2$, since the endpoints of $\alpha_\ell$ avoid the vertices of $Q$.

For each curve $\alpha_i$, draw in $[0,1]^2$ a corresponding polygonal path joining the associated endpoints and made of finitely many segments parallel to the coordinate axes. By Corollary~\ref{cor:intersection}, these paths can be chosen pairwise disjoint.

We then extend this collection of paths to a grid in $[0,1]^2$. Finally, by applying the extension property (Lemma~\ref{thm:extension}) finitely many times, we construct a polygonal grid in $Q$ which is combinatorially equivalent to the one obtained in the square and contains $\alpha_1,\ldots,\alpha_m$ among its horizontal and vertical curves.
\end{proof}

\subsection{Grid construction and proof of Theorem~\ref{thm:top1}}\label{sec:grid_con}

By repeatedly applying Theorem~\ref{thm:subordinated_grids} together with the extension Lemma~\ref{thm:extension}, we construct a family of polygonal grids $\mathcal{G}_k$ on $Q$ satisfying the following properties:
\begin{itemize}
    \item[(i)] the family is nested, meaning that each cell of $\mathcal{G}_{k+1}$ is contained in a cell of $\mathcal{G}_k$;

    \item[(ii)] every cell of $\mathcal{G}_k$ has diameter at most $2^{-k}$.
\end{itemize}

We first construct $\mathcal{G}_1$ by applying Theorem~\ref{thm:subordinated_grids} to an open cover $\mathcal{U}$ of $\overline Q$ by balls of radius at most $1/2$. We then refine inductively each cell to obtain finer and finer grids.

Assume that $\mathcal{G}_k$ has already been constructed. Since each cell is the closure of a polygonal domain, we may cover it by balls of radius at most $2^{-k-1}$ and apply Theorem~\ref{thm:subordinated_grids} to obtain a grid inside that cell subordinated to this cover. Repeating this construction independently on each cell produces a subdivision of $\overline Q$, but not necessarily a global grid. We then use the extension property, Lemma~\ref{thm:extension}, exactly as in Section~\ref{sec:proof_grid}, to complete the subdivision into a genuine grid. This yields $\mathcal{G}_{k+1}$ and concludes the inductive step.

\medskip
We now construct the homeomorphism $f:\overline Q\to [0,1]^2$.
For each $x\in \overline Q$, let $\{\mathcal{G}_k(x)\}_k$ be a nested family of cells, with $\mathcal{G}_k(x)$ a cell of $\mathcal{G}_k$ containing $x$. This choice need not be unique when $x$ lies on one of the horizontal or vertical curves, but this will not cause any difficulty.

Let $\{\mathcal{G}_k(x)'\}_k$ be the corresponding cells in the regular grid $\mathcal{G}_k'$ of $[0,1]^2$ combinatorially equivalent to $\mathcal{G}_k$; see Proposition~\ref{prop:cells}. We claim that $\bigcap_k \mathcal{G}_k(x)'$ consists of a single point, which we denote by $f(x)\in [0,1]^2$. Indeed, the intersection is nonempty because the cells are compact and nested, and it contains at most one point because the diameters of the corresponding cells in $[0,1]^2$ tend to zero as $k\to\infty$.

It is now straightforward to check that the resulting map $f:\overline Q\to [0,1]^2$ is continuous, injective, and surjective. In particular, $f$ is a homeomorphism. Moreover, $f$ maps $\alpha_Q$ to the boundary of $[0,1]^2$.

\subsection{Proof of Theorem~\ref{thm:mainRCD}}\label{sec:proofmain}

For every $x\in \Int(X)$, Theorem~\ref{Thm: poly dom neigh basis} provides a polygonal domain $Q$ which is a neighborhood of $x$. By Theorem~\ref{thm:top1}, $Q$ is homeomorphic to the open unit disk in $\mathbb{R}^2$.

For every $x\in \Ext(X)$, Theorem~\ref{Thm: bd poly dom neigh basis} provides a boundary polygonal domain $Q$ which is a neighborhood of $x$. Again by Theorem~\ref{thm:top1}, $\overline Q$ is homeomorphic to $D^2$. In particular, $Q$ is homeomorphic to the open ball of the two-dimensional half-space.

Combining these two conclusions, we deduce that $(X,\dd_X)$ is a surface with boundary, and that its boundary coincides with $\Ext(X)$.

\subsection{Proof of Theorem \ref{Thm:top_characterization}}

Let $(X,\dd_X,\mm_X)$ be an $\RCD(K,N)$ space of essential dimension $2$. By Theorem~\ref{thm:mainRCD}, the metric space $(X,\dd_X)$ is a topological surface with boundary. Let $(\widehat X,\dd_{\widehat X},\mm_{\widehat X})$ be its universal cover. By \cite{MW19,W24a}, the lifted space is again a simply connected $\RCD(K,N)$ space of essential dimension $2$, and $\pi_1(X)$ acts on $\widehat X$ by deck transformations. If $\widehat X$ is compact, then the classification of compact surfaces implies that $\widehat X$ is homeomorphic either to $S^2$ or to $D^2$. When $\widehat X$ is noncompact, we use the following topological classification. 

\begin{Pro}[Classification of noncompact simply connected surfaces \cite{CKS12}]\label{Pro: class noncompact surf}
Let $M$ be a connected, Hausdorff, second-countable, noncompact topological surface, possibly with boundary, and assume that $M$ is simply connected. Then either $M$ is homeomorphic to $\mathbb{R}^2$, or $M$ is homeomorphic to $D^2\setminus E$ for some nonempty closed totally disconnected set $E\subset \partial D^2$.
\end{Pro}

In the model $D^2\setminus E$, the boundary may have an arbitrary totally disconnected end structure, and in particular may have many connected components. In our setting, however, the splitting theorem rules out this general behavior and forces the boundary structure to be much simpler: in fact, at most two boundary components can occur.

\begin{Pro}\label{prop:open-classification-rcd0}
Let $(X,\dd_X,\mm_X)$ be an $\RCD(0,N)$ space of essential dimension $2$, and assume that $(X,\dd_X)$ is homeomorphic to $D^2\setminus E$, where $E\subset \partial D^2$ is a nonempty closed totally disconnected subset. Then $E$ has either one or two elements. Moreover, if $E$ has two elements, then $(X,\dd_X)$ is isometric to a flat strip $\mathbb{R}\times [a,b]$.
\end{Pro}

\begin{proof}
The boundary of $D^2\setminus E$ is $\partial D^2\setminus E$, so its connected components are exactly the connected components of $\partial D^2\setminus E$. In particular, if $E$ consists of a single point, then $\partial X$ is connected, while if $E$ has more than one point, then $\partial X$ has at least two connected components.

Assume now that $E$ has more than one point. 
Choose two distinct boundary components of $X$, and let $\gamma$ be a minimizing geodesic joining them.
Then $\gamma$ divides the space into two non-compact path-connected components. Since any geodesic between two points from different components must intersect $\gamma$, it follows immediately from a compactness argument that there exists a line in $X$. By the splitting theorem (Theorem~\ref{Thm: splitting thm}), $X$ splits isometrically as $\mathbb{R}\times Y$, where $Y$ is an $\RCD(0,N-1)$ space of essential dimension $1$. Since $D^2\setminus E$ is simply connected, it follows from Theorem~\ref{Thm: RCD dim 1 class} that $Y$ is isometric to one of $\mathbb{R}$, $[0,\infty)$, $[a,b]$. The case $Y=\mathbb{R}$ is impossible because $X$ has nonempty boundary, while $Y=[0,\infty)$ is impossible because then $X$ would have connected boundary. Therefore $Y=[a,b]$, and $(X,\dd_X)$ is isometric to the flat strip $\mathbb{R}\times [a,b]$.
\end{proof}

\begin{proof}[Proof of {\rm(i)} in Theorem~\ref{Thm:top_characterization}]
When $K>0$, the universal cover is compact. Hence by Theorem~\ref{thm:mainRCD}, $\widehat X$ is homeomorphic either to $S^2$ or to the disk $D^2$. The conclusion then follows from the classification of effective, discrete, free group actions on these two models.
\end{proof}

\begin{proof}[Proof of {\rm(ii)} in Theorem~\ref{Thm:top_characterization}]
By the splitting theorem, either the universal cover $\widehat X$ is compact, or it splits isometrically off a Euclidean factor. The compact case was already treated in {\rm(i)}. In the noncompact case, it follows from Theorem~\ref{Thm: RCD dim 1 class} that only the following models, endowed with the Euclidean metric, can occur: $\mathbb{R}^2$, the half-plane $\mathbb{R}^2_+$, or a strip $\mathbb{R}\times [a,b]$.

We claim that the half-plane cannot occur. Indeed, every deck transformation is an isometry of $\mathbb{R}^2_+$ preserving the boundary line. Since the action is free and properly discontinuous, such an isometry must be a translation parallel to the boundary. In particular, the quotient cannot be compact in the transverse direction, so the action is not cocompact. 

Therefore, the only possibilities are $\widehat X\cong \mathbb{R}^2$ or $\widehat X\cong \mathbb{R}\times [a,b]$. The conclusion then follows from the classification of effective, discrete, free, cocompact isometric group actions on these two models.
\end{proof}

\begin{proof}[Proof of {\rm(iii)} in Theorem~\ref{Thm:top_characterization}]
By Proposition~\ref{prop:open-classification-rcd0}, the universal cover $\widehat X$ is either homeomorphic to $\mathbb{R}^2$, or homeomorphic to the half-plane $\mathbb{R}^2_+$, or isometric to the strip $\mathbb{R}\times [a,b]$.

\medskip
\noindent{\bf Case 1:} $\widehat X$ is isometric to $\mathbb{R}\times [a,b]$.

Any nontrivial discrete free group of isometries of the strip is infinite cyclic, generated either by a translation along the $\mathbb{R}$-factor or by the composition of such a translation with the reflection exchanging the two boundary components. In the first case, the quotient is a compact cylinder, and in the second case it is a compact M\"obius strip. Since $X$ is noncompact, neither can occur. Hence this case is impossible when $X$ is not simply connected.

\medskip
\noindent{\bf Case 2:} $\widehat X$ is homeomorphic to $\mathbb{R}^2_+$.

Every deck transformation preserves the boundary $\partial \widehat X$, which is homeomorphic to $\mathbb{R}$. Hence the deck group acts freely and properly discontinuously on $\partial \widehat X\cong \mathbb{R}$. Since the quotient of a connected one-dimensional manifold by such an action is again a connected one-dimensional manifold, the deck group is either trivial or infinite cyclic. If it is nontrivial, the quotient of the boundary is a circle, and therefore $X$ has exactly one boundary component. It follows from Proposition~\ref{Pro: class noncompact surf} that $X$ is homeomorphic to the half-cylinder $S^1\times [0,\infty)$.

\medskip
\noindent{\bf Case 3:} $\widehat X$ is homeomorphic to $\mathbb{R}^2$.

By a classical result of Johansson \cite{J31} (see also \cite[pp.~142--144]{S93}), the fundamental group of an open surface is free. On the other hand, by \cite[Corollary 4.9]{MW19} (see also \cite{M68, G81}), every finitely generated subgroup of the fundamental group is virtually nilpotent. Therefore, it must be either trivial or isomorphic to $\mathbb{Z}$. It follows that $X$ is either simply connected, or homeomorphic to the cylinder $\mathbb{R}\times S^1$, or to the open M\"obius strip.

In the latter two cases, we claim that $X$ is in fact flat. Indeed, after passing to a double cover if necessary, we may reduce to the cylinder case. Since the cylinder has two ends, one can construct a line in its universal cover. The splitting theorem then implies that $\widehat X$ is isometric to $\mathbb{R}^2$. Consequently, $X$ is flat as well.
\end{proof}

\end{document}